\documentclass[11pt,letterpaper,twoside,reqno,nosumlimits]{amsart}

\usepackage{fixltx2e}                       
\usepackage{amsmath,amsfonts,amsbsy,amsgen,amscd,mathrsfs,amssymb,amsthm}
\usepackage{mathtools}

\mathtoolsset{showonlyrefs}
\allowdisplaybreaks

\usepackage[usenames,dvipsnames]{xcolor}
\usepackage{graphicx}
\graphicspath{{art/}}

\usepackage{tikz}
\usepackage{microtype}

\usepackage{fancyhdr}

\usepackage{enumitem}

\usepackage{url}
\usepackage[font=small,margin=0.25in,labelfont={sc},labelsep={colon}]{caption}
\usepackage{float}

\DeclareRobustCommand{\gobblefive}[5]{}
\newcommand*{\SkipTocEntry}{\addtocontents{toc}{\gobblefive}}

\definecolor{dark-gray}{gray}{0.3}
\definecolor{dkgray}{rgb}{.4,.4,.4}
\definecolor{dkblue}{rgb}{0,0,.5}
\definecolor{medblue}{rgb}{0,0,.75}
\definecolor{rust}{rgb}{0.5,0.1,0.1}

\usepackage[colorlinks=true]{hyperref}

\hypersetup{urlcolor=rust}
\hypersetup{citecolor=dkblue}
\hypersetup{linkcolor=dkblue}

\usepackage[utf8]{inputenc}
\usepackage[T1]{fontenc}

\usepackage{fourier}
\usepackage{bm} 
\usepackage[scaled = 0.92]{cabin}

\DeclareMathVersion{semibold}
\DeclareMathAlphabet{\mathsf}{T1}{Cabin-TLF}{m}{n}
\SetMathAlphabet\mathsf{normal}{T1}{Cabin-TLF}{m}{n}
\SetMathAlphabet\mathsf{bold}{T1}{Cabin-TLF}{b}{n}
\SetMathAlphabet\mathsf{semibold}{T1}{Cabin-TLF}{sb}{n}

\newtheorem{bigthm}{Theorem}

\newtheorem{theorem}{Theorem}[section]
\newtheorem{lemma}[theorem]{Lemma}
\newtheorem{sublemma}[theorem]{Sublemma}
\newtheorem{proposition}[theorem]{Proposition}
\newtheorem{fact}[theorem]{Fact}

\newtheorem{corollary}[theorem]{Corollary}

\theoremstyle{definition}

\newtheorem{definition}[theorem]{Definition}
\newtheorem{example}[theorem]{Example}
\newtheorem{remark}[theorem]{Remark}

\newtheorem{model}[theorem]{Model}

\newcommand{\lang}{\textit}

\newcommand{\term}{\emph}

\numberwithin{equation}{section} 
\numberwithin{figure}{section}
\numberwithin{table}{section}

\floatstyle{plaintop}
\newfloat{recipe}{thp}{lor}
\floatname{recipe}{Recipe}
\numberwithin{recipe}{section}

\providecommand{\mathbold}[1]{\bm{#1}}                                                                  
                                 
\newcommand{\sametprev}[1]{#1}
\newcommand{\joelprev}[1]{#1}

\newcommand{\joel}[1]{#1}

\renewcommand{\phi}{\varphi}

\newcommand{\eps}{\varepsilon}

\newcommand{\half}{\tfrac{1}{2}}

\newcommand{\cnst}[1]{\mathrm{#1}} 
\newcommand{\econst}{\mathrm{e}}

\newcommand{\Id}{\mathbf{I}}

\newcommand{\coll}[1]{\mathscr{#1}}

\providecommand{\mathbbm}{\mathbb} 
\newcommand{\R}{\mathbbm{R}}

\newcommand{\polar}{\circ}

\newcommand{\abs}[1]{\left\vert {#1} \right\vert}
\newcommand{\abssq}[1]{{\abs{#1}}^2}

\newcommand{\diff}[1]{\mathrm{d}{#1}}
\newcommand{\idiff}[1]{\, \diff{#1}}

\newcommand{\argmin}{\operatorname*{arg\; min}}
\newcommand{\argmax}{\operatorname*{arg\; max}}

\newcommand{\Prob}[1]{\mathbbm{P}\left\{{#1}\right\}}

\newcommand{\Expect}{\operatorname{\mathbb{E}}}

\newcommand{\normal}{\textsc{normal}}

\DeclareMathOperator{\Var}{Var}

\newcommand{\vct}[1]{\mathbold{#1}}
\newcommand{\mtx}[1]{\mathbold{#1}}

\newcommand{\adj}{*}

\newcommand{\range}{\operatorname{range}}

\newcommand{\nullsp}{\operatorname{null}}

\newcommand{\supp}[1]{\operatorname{supp}(#1)}

\newcommand{\smax}{\sigma_{\max}}
\newcommand{\smin}{\sigma_{\min}}

\newcommand{\norm}[1]{\left\Vert {#1} \right\Vert}
\newcommand{\normsq}[1]{\norm{#1}^2}

\newcommand{\fnorm}[1]{\norm{#1}_{\mathrm{F}}}
\newcommand{\fnormsq}[1]{\fnorm{#1}^2}

\newcommand{\pnorm}[2]{\norm{#2}_{#1}}

\newcommand{\cone}{\operatorname{cone}}

\newcommand{\conv}{\operatorname{conv}}

\newcommand{\minimize}{\text{minimize}\quad}
\newcommand{\subjto}{\quad\text{subject to}\quad}

\newcommand{\muavg}[1]{\left\langle {#1} \right\rangle}

\title[Universality of Randomized Dimension Reduction]{Universality Laws for Randomized Dimension Reduction, \\ with Applications}

\author[S.~Oymak and J.~A.~Tropp]{Samet Oymak and Joel~A.~Tropp}

\date{30 November 2015.  Revised 7 August 2017 and 5 September 2017 and 14 September 2017.}



\begin{document}

\pagestyle{headings}
\setcounter{page}{1}
\pagenumbering{roman}

\begin{abstract}
Dimension reduction is the process of embedding high-dimensional data
into a lower dimensional space to facilitate its analysis.
In the Euclidean setting,
one fundamental technique for dimension reduction is to apply a random linear map
to the data.  This dimension reduction procedure succeeds when it preserves
certain geometric features of the set.
 The question is how large the embedding dimension must be
to ensure that randomized dimension reduction succeeds with high probability.

This paper studies a natural family of randomized dimension reduction maps
and a large class of data sets.  It proves that there is a phase
transition in the success probability of the dimension reduction map as the embedding
dimension increases.  For a given data set, the location of the phase transition
is the same for all maps in this family.
Furthermore, each map has the same stability properties, as quantified
through the restricted minimum singular value.  These results
can be viewed as new universality laws in high-dimensional stochastic geometry.

Universality laws for randomized dimension reduction have many
applications in applied mathematics, signal processing, and statistics.
They yield design principles for numerical linear algebra algorithms,
for compressed sensing measurement ensembles, and for random linear
codes.  Furthermore, these results have implications for the performance
of statistical estimation methods under a large class of random experimental designs.
\end{abstract}

\maketitle
\thispagestyle{empty}

\newpage 

\tableofcontents
\SkipTocEntry\listoffigures

\newpage

\setcounter{page}{1}
\pagenumbering{arabic}

\section{Overview of the Universality Phenomenon}

\noindent
This paper concerns a fundamental question in high-dimensional stochastic geometry:

\begin{enumerate}
\item[(Q1)]	Is it likely that a random subspace of fixed dimension does \emph{not} intersect a given set?

\end{enumerate}

\noindent
This problem has its roots in the earliest research on spherical integral
geometry~\cite{San52:Integral-Geometry,San76:Integral-Geometry},
and it also arises in asymptotic convex geometry~\cite{Gor88:Milmans-Inequality}.
In recent years, this question has attracted fresh
attention~\cite{Don06:High-Dimensional-Centrally,RV08:Sparse-Reconstruction,
Sto09:Various-Thresholds,DT09:Counting-Faces,CRPW12:Convex-Geometry,
ALMT14:Living-Edge,Sto13:Regularly-Random,MT13:Achievable-Performance,
OH13:Asymptotically-Exact,OTH13:Squared-Error,TOH15:Gaussian-Min-Max,
TAH15:High-Dimensional-Error}
because it is central to the analysis of
randomized dimension reduction.

This paper establishes that a striking universality phenomenon
takes place in the stochastic geometry problem (Q1).
For a given set, the answer to this question is essentially the same
for every distribution on random subspaces that is induced
by a natural model for random \joelprev{linear map}s.  Universality also manifests itself in metric variants of (Q1),
where we ask how far the random subspace lies from the set.
We discuss the implications of these results in
high-dimensional geometry, random matrix theory, numerical analysis,
optimization, statistics, signal processing, and beyond.

\subsection{Randomized Dimension Reduction}

\term{Dimension reduction} is the operation of mapping a set from a
large space into a smaller space.  Ideally,
this action distills the ``information'' in the set, and it allows
us to develop more efficient algorithms for processing that information.
In the setting of Euclidean spaces, a fundamental method for dimension
reduction is to apply a random linear map to each point in the set.
It is important that the random \joelprev{linear map}
preserve geometric features of the set.
In particular, we do \emph{not} want the
\joelprev{linear map} to map a point in the set to the origin.
Equivalently, the null space of the random \joelprev{linear map} should \emph{not}
intersect the set.  We see that (Q1) emerges naturally in the context
of randomized dimension reduction. 
\subsection{Technical Setting}

Let us introduce a framework in which to study this problem.
It is natural to treat (Q1) as a question in spherical geometry because
it is scale invariant.  Fix the \term{ambient dimension} $D$,
and consider a closed subset $\Omega$ of the Euclidean unit sphere in $\R^D$.
For the moment, we also assume that $\Omega$
is \term{spherically convex}; that is, $\Omega$ is the intersection
of a convex cone\footnote{A \term{convex cone} is a convex set $K$ that satisfies
$\alpha K = K$ for all $\alpha > 0$.}
with the unit sphere.
Construct a random linear map $\mtx{\Pi} : \R^D \to \R^d$,
where the \term{embedding dimension} $d$ does not
exceed the ambient dimension $D$.
As we vary the distribution of the random \joelprev{linear map} $\mtx{\Pi}$, the map $\mtx{\Pi} \mapsto \nullsp(\mtx{\Pi})$
induces different distributions on the subspaces in $\R^D$
with codimension at most $d$.
We may now reformulate (Q1) in this language:

\begin{enumerate}
\item[(Q2)]	For a given embedding dimension $d$, what is the probability that $\Omega \cap \nullsp(\mtx{\Pi}) = \emptyset$?  Equivalently, what is the probability that $\vct{0} \notin \mtx{\Pi}(\Omega)$?

\end{enumerate}

\noindent
We say that the random projection \term{succeeds} when $\vct{0} \notin \mtx{\Pi}(\Omega)$.
Conversely, when $\vct{0} \in \mtx{\Pi}(\Omega)$, we say that the random projection
\term{fails}.  See Figure~\ref{fig:intro-geometry} for an illustration.  

We have the intuition that, for a fixed choice of $\Omega$,
the projection is more likely to succeed as the embedding dimension $d$
increases.  Furthermore, a random \joelprev{linear map} $\mtx{\Pi}$
with fixed embedding dimension $d$ is less likely to succeed
as the size of the set increases.
We will justify these heuristics in complete detail.

\begin{figure}
\vspace{0.25in}
\begin{tikzpicture}[auto,scale=1.5]

\draw[gray, ultra thin] (0,0) circle (4pc);

\draw[red!80!black, very thick] (0,0) +(75:5pc) -- +(-105:5pc);
\draw[red!80!black] (0,0) +(-105:4.5pc) node[below right]{$\nullsp(\mtx{P})$};

\draw[blue!80!black, very thick] (0,0) +(-15:5pc) -- +(165:5pc);
\draw[blue!80!black] (0,0) +(165:4.5pc) node[below left]{$\range(\mtx{P})$};

\draw[black, fill] (0,0) circle (1pt);

\fill[gray, nearly transparent] (0,0) +(-15:1pc) -- ++(-15:3.46pc) -- ++(75:2pc) arc (15:60:4pc) -- cycle;
\draw[black, very thin, dashed] (0,0) +(-15:1pc) -- ++(-15:3.46pc) -- ++(75:2pc) arc (15:60:4pc) -- cycle;

\draw[black, ultra thick] (0,0) +(15:4pc) arc (15:60:4pc);
\draw[black] (0,0) +(30:4.25pc) node[right]{$\Omega$};

\draw[black, ultra thick] (0,0) + (-15:1pc) -- (-15:3.46pc);
\draw[black] (0,-0.25pc) + (-15:2.25pc) node[below]{$\mtx{P}(\Omega)$};

\draw[gray, ultra thin] (5.5,0) circle (4pc);

\draw[red!80!black, very thick] (5.5,0) +(45:5pc) -- +(-135:5pc);
\draw[red!80!black] (5.5,0) +(-135:4.5pc) node[above left]{$\nullsp(\mtx{P})$};

\draw[blue!80!black, very thick] (5.5,0) +(-45:5pc) -- +(135:5pc);
\draw[blue!80!black] (5.5,0) +(135:4.5pc) node[below left]{$\range(\mtx{P})$};

\draw[black, fill] (5.5,0) circle (1pt);

\fill[gray, nearly transparent] (5.5,0) +(135:1pc) -- +(-45:2pc) -- ++(15:4pc) arc (15:60:4pc) -- cycle;
\draw[black, very thin, dashed] (5.5,0) +(135:1pc) -- +(-45:2pc) -- ++(15:4pc) arc (15:60:4pc) -- cycle;

\draw[black, ultra thick] (5.5,0) +(15:4pc) arc (15:60:4pc);
\draw[black] (5.5,0) +(30:4.25pc) node[right]{$\Omega$};

\draw[black, ultra thick] (5.5,0) + (135:1pc) -- +(-45:2pc);
\draw[black] (5.5,0.2pc) + (-45:1.75pc) node[below left]{$\mtx{P}(\Omega)$};
\end{tikzpicture}

\caption[Geometry of a Random \joelprev{Linear Map}]{
\textsl{Geometry of a Random \joelprev{Linear Map}.} \label{fig:intro-geometry}
We can identify a \joelprev{linear map} $\mtx{\Pi} : \R^D \to \R^d$ with
\joelprev{an orthogonal projector} $\mtx{P} : \R^D \to \R^D$ whose range
is the orthogonal complement of $\nullsp(\mtx{\Pi})$.  In this diagram,
$\mtx{P} : \R^2 \to \R^2$ is a random orthogonal \joelprev{projector} applied to
a closed, spherically convex set $\Omega$.  The likelihood of a given configuration
depends on the statistical dimension $\delta(\Omega)$ of the set.
\textsc{Left:} \textsl{Success Regime.}
The null space, $\nullsp(\mtx{P})$, does not intersect $\Omega$,
and the image $\mtx{P}(\Omega)$ does not contain the origin.
\textsc{Right:} \textsl{Failure Regime.}
The null space, $\nullsp(\mtx{P})$, intersects the set $\Omega$,
and the image $\mtx{P}(\Omega)$ contains the origin.}
\end{figure}

\subsection{A Phase Transition for Uniformly Random \joelprev{Partial Isometries}}
\label{sec:intro-phase-transition}

We begin with a case where the literature already contains
a comprehensive answer to the question (Q2).

The most natural type of random embedding
is a \term{uniformly random \joelprev{partial isometry}}.
That is, $\mtx{\Pi} : \R^D \to \R^d$ is a \joelprev{partial isometry}\footnote{A \term{\joelprev{partial isometry}} $\mtx{\Pi}$ satisfies the condition
$\mtx{\Pi} \mtx{\Pi}^\adj = \Id$, where ${}^\adj$ is the transpose operation and $\Id$ is the identity map.}
whose null space, $\nullsp(\mtx{\Pi})$, is drawn uniformly
at random from the Haar measure on the Grassmann manifold of subspaces in $\R^D$ with codimension $d$.
The invariance properties of the distribution of $\mtx{\Pi}$ allow for a
complete analysis of its action on $\Omega$, the spherically
convex set~\cite[Chap.~6.5]{SW08:Stochastic-Integral}.
Recent research~\cite{ALMT14:Living-Edge,MT14:Steiner-Formulas,GNP14:Gaussian-Phase}
has shown how to convert the complicated exact
formulas into interpretable results.

The modern theory is expressed in terms of a geometric functional
$\delta(\Omega)$, called the \term{statistical dimension}:
$$
\delta(\Omega) := \Expect\left[ \left( \max_{\vct{t} \in \Omega} \vct{g} \cdot \vct{t} \right)_+^2 \right],
\quad\text{where $\vct{g}$ is $\normal(\vct{0}, \Id)$.}
$$
The statistical dimension is increasing with respect to set inclusion,
and its values range from zero (for the empty set) up to $D$ (for the whole sphere).
Furthermore, the functional can be computed accurately in many cases of interest.
See Section~\ref{sec:stat-dim} for more details.

The statistical dimension demarcates a phase transition
in the behavior of a uniformly random \joelprev{partial isometry} $\mtx{\Pi}$
as the embedding dimension $d$ varies.
For a closed, spherically convex set $\Omega$,
the results~\cite[Thm.~I and Prop.~10.2]{ALMT14:Living-Edge}
demonstrate that
\begin{equation} \label{eqn:intro-phase-transition}
\begin{aligned}
d &\leq \delta(\Omega) - \cnst{C} \sqrt{\delta(\Omega)}
\quad\text{implies}\quad
\vct{0} \in \mtx{\Pi}(\Omega)
\quad\text{with high probability}; \\
d &\geq \delta(\Omega) + \cnst{C} \sqrt{\delta(\Omega)}
\quad\text{implies}\quad
\vct{0} \notin \mtx{\Pi}(\Omega)
\quad\text{with high probability}.
\end{aligned}
\end{equation}
The number $\cnst{C}$ is a positive universal constant.  In other terms,
a uniformly random projection $\mtx{\Pi}(\Omega)$ of a spherically
convex set $\Omega$ is likely to succeed precisely
when the embedding dimension $d$
is larger than the statistical dimension $\delta(\Omega)$.
See Figure~\ref{fig:intro-orthant-universal} for a plot of
the exact probability that a uniformly random \joelprev{partial isometry}
annihilates a point in a specific set $\Omega$.

\begin{remark}[Related Work] \label{rem:intro-phase-details}
The results~\cite[Thm.~7.1]{ALMT14:Living-Edge}
and~\cite[Thm.~A]{MT13:Achievable-Performance}
contain good bounds for the probabilities in~\eqref{eqn:intro-phase-transition}.
The probabilities can be approximated more precisely by introducing a second geometric
functional~\cite{GNP14:Gaussian-Phase}.  These estimates depend on the
spherical Crofton formula~\cite[Eqns.~(6.62), (6.63)]{SW08:Stochastic-Integral},
which gives the \emph{exact} probabilities in a less interpretable form.
Related phase transition results can also be obtained via the Gaussian
Minimax Theorem; see~\cite[Cor.~3.4]{Gor88:Milmans-Inequality},
~\cite[Rem.~2.9]{ALMT14:Living-Edge}, \cite{Sto13:Regularly-Random},
or~\cite[Thm.~II.1]{TOH15:Gaussian-Min-Max}.
See~\cite{TH15:Isotropically-Random} for other
results on uniformly random \joelprev{partial isometries}.
\end{remark}

\subsection{Other Types of Random \joelprev{Linear Map}s?}
\label{sec:intro-other-projs}

The research outlined in Section~\ref{sec:intro-phase-transition}
delivers a complete account of how a uniformly random
\joelprev{partial isometry behaves} in the presence of some convexity.  In contrast,
the literature contains almost no precise information about the
behavior of \joelprev{other} random \joelprev{linear map}s.

Nevertheless, in applications,
we may prefer---or be forced---to work with \joelprev{other types of} random \joelprev{linear maps}.
Here is a motivating example.  Many algorithms for numerical linear algebra
now depend on randomized dimension reduction.
In this context, uniformly random \joelprev{partial isometries}
are expensive to construct, to store, and to perform arithmetic with.
It is more appealing to implement a random \joelprev{linear map}
that is discrete, or sparse, or structured.  The lack of
detailed theoretical information about how these \joelprev{linear map}s behave
makes it difficult to design numerical methods with
guaranteed performance.

\begin{figure}
\includegraphics[width=\textwidth]{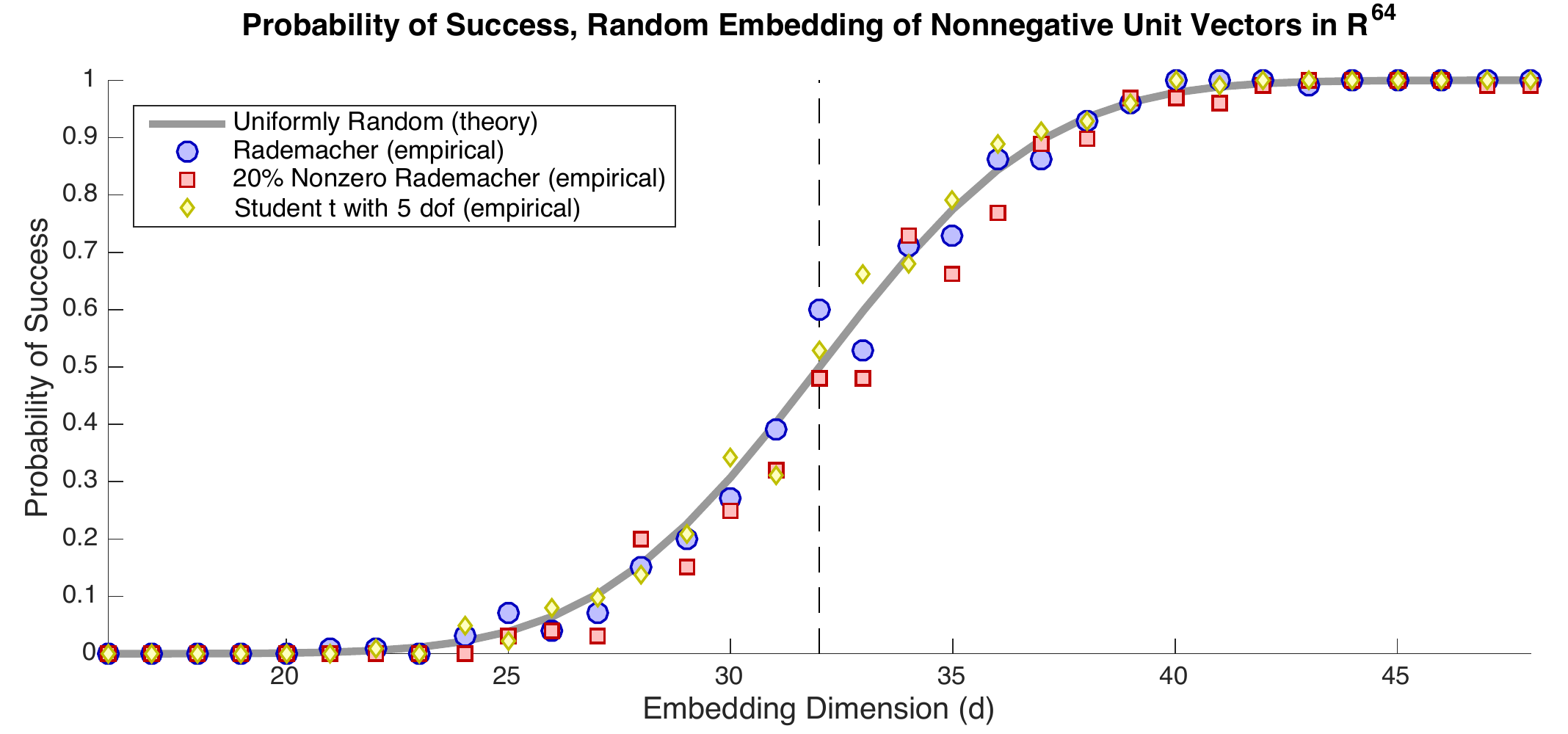}
\caption[Universality of the Embedding Dimension]{\textsl{Universality of the Embedding Dimension.} \label{fig:intro-orthant-universal}
This plot describes the behavior of four types of random \joelprev{linear map}s applied to the
set \joelprev{$\Omega$} of unit vectors with nonnegative coordinates:
$\Omega := \{ \vct{t} \in \R^{64} : \norm{\vct{t}} = 1, t_i \geq 0 \}$.
The \textbf{dashed line} marks the statistical dimension of the set:
$\delta(\Omega) = 32$.
The \textbf{gray curve} interpolates the exact probability that a uniformly
random \joelprev{partial isometry} $\mtx{\Pi} : \R^{64} \to \R^d$ succeeds
(i.e., $\vct{0} \notin \mtx{\Pi}(\Omega)$)
as a function of the embedding dimension $d$.
The \textbf{markers} indicate the empirical probability (over 100 trials) that dimension
reduction succeeds for a random \joelprev{linear map} with the specified distribution.
See Sections~\ref{sec:intro-other-projs} and~\ref{sec:repro-research} for further details.}
\end{figure}

We can, however, use computation to investigate
the behavior of \joelprev{other types of} random \joelprev{linear map}s.
Figure~\ref{fig:intro-orthant-universal}
presents the results of the following experiment.
Consider the set $\Omega$ of unit vectors in $\R^{64}$ with nonnegative coordinates:
$$
\Omega := \big\{ \vct{t} \in \R^{64} : \norm{\vct{t}} = 1
	\text{ and } t_i \geq 0 \text{ for $i = 1, \dots, 64$} \big\}.
$$
According to~\eqref{eqn:sdim-orthant}, below, the statistical dimension
$\delta(\Omega) = 32$, so the formula~\eqref{eqn:intro-phase-transition}
tells us to expect a phase transition in the behavior of a uniformly
random \joelprev{partial isometry} when the embedding dimension $d = 32$.
Using~\cite[Ex.~5.3 and Eqn.~(5.10)]{ALMT14:Living-Edge},
we can compute the exact probability that a uniformly
random \joelprev{partial isometry} $\mtx{\Pi} : \R^{64} \to \R^d$
succeeds as a function of $d$. Against this baseline, we compare
the empirical probability (over 100 trials) that a random
\joelprev{linear map} with independent Rademacher\footnote{A \term{Rademacher} random variable takes the two values $\pm 1$
with equal probability.}
entries succeeds.  We also display experiments
for a 20\% nonzero Rademacher \joelprev{linear map} and for
a \joelprev{linear map} with Student $t_5$ entries.
See Section~\ref{sec:repro-research} for
more \joelprev{details}.

From this experiment, we discover that all three \joelprev{linear map}s act
almost exactly the same way as a uniformly random \joelprev{partial isometry}!  This universality
phenomenon is remarkable because the four \joelprev{linear map}s have rather different
distributions.  At present, the literature contains no information about
when---or why---this phenomenon occurs.

\subsection{A Universality Law for the Embedding Dimension}
\label{sec:intro-univ-law1}

The central goal of this paper is to show that there is a substantial class
of random \joelprev{linear map}s for which the phase transition in the embedding dimension
is universal.  Here is a rough statement of the main result.

Let $\Omega$ be a closed, spherically convex set in $\R^D$.
Suppose that the entries of the matrix of the random \joelprev{linear map}
$\mtx{\Pi} : \R^D \to \R^d$ are independent, standardized,\footnote{A \term{standardized} random variable has mean zero and variance one.}
and symmetric,\footnote{A \term{symmetric} random variable $X$ has the same distribution as its negation $-X$.}
with a modest amount of regularity.\footnote{For concreteness, we may assume that the entries of $\mtx{\Pi}$
have five uniformly bounded moments.}
In particular, we may consider random \joelprev{linear map}s that have an arbitrarily small,
but constant, proportion of nonzero entries.
For this class of random \joelprev{linear map}s, we will demonstrate that
\begin{equation} \label{eqn:intro-phase-universal}
\begin{aligned}
d &\leq \delta(\Omega) - o(D)
\quad\text{implies}\quad
\vct{0} \in \mtx{\Pi}(\Omega)
\quad\text{with high probability}; \\
d &\geq \delta(\Omega) + o(D)
\quad\text{implies}\quad
\vct{0} \notin \mtx{\Pi}(\Omega)
\quad\text{with high probability}.
\end{aligned}
\end{equation}
The little-o notation suppresses constants that depend only on
the regularity of the random variables that populate $\mtx{\Pi}$.
See Theorem~\ref{thm:univ-embed} in Section~\ref{sec:thm-univ-embed}
for a more complete statement.

The result~\eqref{eqn:intro-phase-universal} states that a
random \joelprev{linear map} $\mtx{\Pi}$ is likely to succeed
for a spherically convex set $\Omega$ precisely
when the embedding dimension $d$ exceeds the
statistical dimension $\delta(\Omega)$ of the set.
We learn that the phase transition in the embedding dimension is universal
over our class of random \joelprev{linear map}s,
provided that $\Omega$ is not too small
as compared with the ambient dimension $D$.
Note that a random \joelprev{linear map} $\mtx{\Gamma} : \R^D \to \R^d$ with
standard normal entries has the same behavior as
a $d \times D$ uniformly random \joelprev{partial isometry} because, almost surely,
the null space of $\mtx{\Gamma}$ is a uniformly
random subspace of $\R^D$ with codimension $d$.
This analysis explains the dominant features
of the experiment in Figure~\ref{fig:intro-orthant-universal}!

\begin{figure}[t!]
\vspace{0.25in}
\begin{tikzpicture}[auto,scale=1.5]

\draw[red!80!black, very thick] (0,-0.5) -- (0,2.75);
\draw[red!80!black] (0,2.75) node[below left]{$\nullsp(\mtx{P})$};

\draw[blue!80!black, very thick] (-1.5,0) -- (6,0);
\draw[blue!80!black] (-1,0) node[above]{$\range(\mtx{P})$};

\draw[black, fill] (0,0) circle (1pt);

\fill[gray, nearly transparent] (3,0) -- (3,2) -- (5,2) -- (5,0) -- cycle;
\draw[black, very thin, dashed] (3,0) -- (3,2) -- (5,2) -- (5,0) -- cycle;

\draw[fill, blue!20!white] (3,2) to[out=90,in=180] (4,2.5) to[out=0,in=90] (5,2) to[out=-90,in=0] (3.5,1) to[out=180,in=-90] (3,2);
\draw[black, ultra thick] (3,2) to[out=90,in=180] (4,2.5) to[out=0,in=90] (5,2) to[out=-90,in=0] (3.5,1) to[out=180,in=-90] (3,2);
\draw[black] (4,2) node{$E$};

\draw[black, ultra thick] (3,0) -- (5,0);
\draw[black] (4,-0.1) node[below]{$\mtx{P}(E)$};

\draw[black,thin,>=latex,|<->|] (0,.1) -- (3,.1);
\draw[black] (1.5,0.1) node[above] {$\smin(\mtx{P}; E)$};

\end{tikzpicture}

\caption[Geometry of the Restricted Minimum Singular Value]{
\textsl{Geometry of the Restricted Minimum Singular Value.} \label{fig:intro-rsv}
We can identify a \joelprev{partial unitary linear map} $\mtx{\Pi} : \R^D \to \R^d$ with
\joelprev{an orthogonal projector} $\mtx{P} : \R^D \to \R^D$ whose range
is the orthogonal complement of $\nullsp(\mtx{\Pi})$.  In this diagram,
$\mtx{P} : \R^2 \to \R^2$ is \joelprev{an orthogonal projector} applied to
a compact convex set $E$.  The restricted minimum singular value
$\smin(\mtx{P}; E)$ is the distance from the origin to the image
$\mtx{P}(E)$.}
\end{figure}

\subsection{A Universality Law for the Restricted Minimum Singular Value}
\label{sec:intro-univ-law2}

It is also a matter of significant interest to understand the
\emph{stability} of randomized dimension reduction.  We quantify the
stability of the random \joelprev{linear map} $\mtx{\Pi} : \R^D \to \R^d$
on a compact, convex set $E$ in $\R^D$ using the
\term{restricted minimum singular value}:
$$
\smin(\mtx{\Pi}; E) := \min_{\vct{t} \in E} \norm{\mtx{\Pi}\vct{t}}.
$$
When the restricted minimum singular value $\smin(\mtx{\Pi}; E)$
is large, the random \joelprev{image} $\mtx{\Pi}(E)$ is far from the origin,
so the embedding is very stable.  That is, we can deform
either the \joelprev{linear map} $\mtx{\Pi}$ or the set $E$
and still be sure that the embedding succeeds.
When the restricted minimum singular value is small, the random \joelprev{dimension reduction map}
is unstable.  When it is zero, the random \joelprev{dimension reduction map} fails.
See Figure~\ref{fig:intro-rsv} for a diagram.

Our second major theorem is a universality law for
the restricted minimum singular values of a random \joelprev{linear map}.
This result is expressed \joelprev{using} a geometric functional
$\coll{E}_d(E)$, called the \term{$d$-excess width} of $E$:
$$
\coll{E}_d(E) := \Expect \min_{\vct{t} \in E} \big( \sqrt{d} \norm{\vct{t}} + \vct{g} \cdot \vct{t} \big),
\quad\text{where $\vct{g}$ is $\normal(\vct{0}, \Id)$.}
$$
The $d$-excess width increases with the parameter $d$, and it
decreases with respect to set inclusion.
The typical scale of $\coll{E}_d(E)$ is $O(\sqrt{d})$.
In addition, the excess width can be evaluated precisely
in many situations of interest.
See Section~\ref{sec:excess-width} for more details.

Now, suppose that the entries of the matrix of $\mtx{\Pi} : \R^D \to \R^d$
are independent, standardized, and symmetric, with a modest amount
of regularity.  For a compact, convex subset $E$ of the unit ball in $\R^D$,
we will establish that
\begin{equation} \label{eqn:intro-rsv-univ}
\smin(\mtx{\Pi}; E) = \big( \coll{E}_d(E) \big)_+ + o(\sqrt{d})
\quad\text{with high probability.}
\end{equation}
The little-o notation suppresses constants that depend
only on the regularity of the random variables.
Theorem~\ref{thm:univ-rsv} in Section~\ref{sec:thm-univ-rsv}
contains a more detailed statement of~\eqref{eqn:intro-rsv-univ}.

In summary, provided that the set $E$ is not too small,
the restricted minimum singular value $\smin(\mtx{\Pi};E)$
depends primarily on the geometry of the set $E$ and the embedding dimension
$d$, rather than on the distribution of the random \joelprev{linear map}
$\mtx{\Pi}$.  See Figure~\ref{fig:intro-orthant-rsv} for a numerical illustration of this fact.

\begin{figure}
\includegraphics[width=\textwidth]{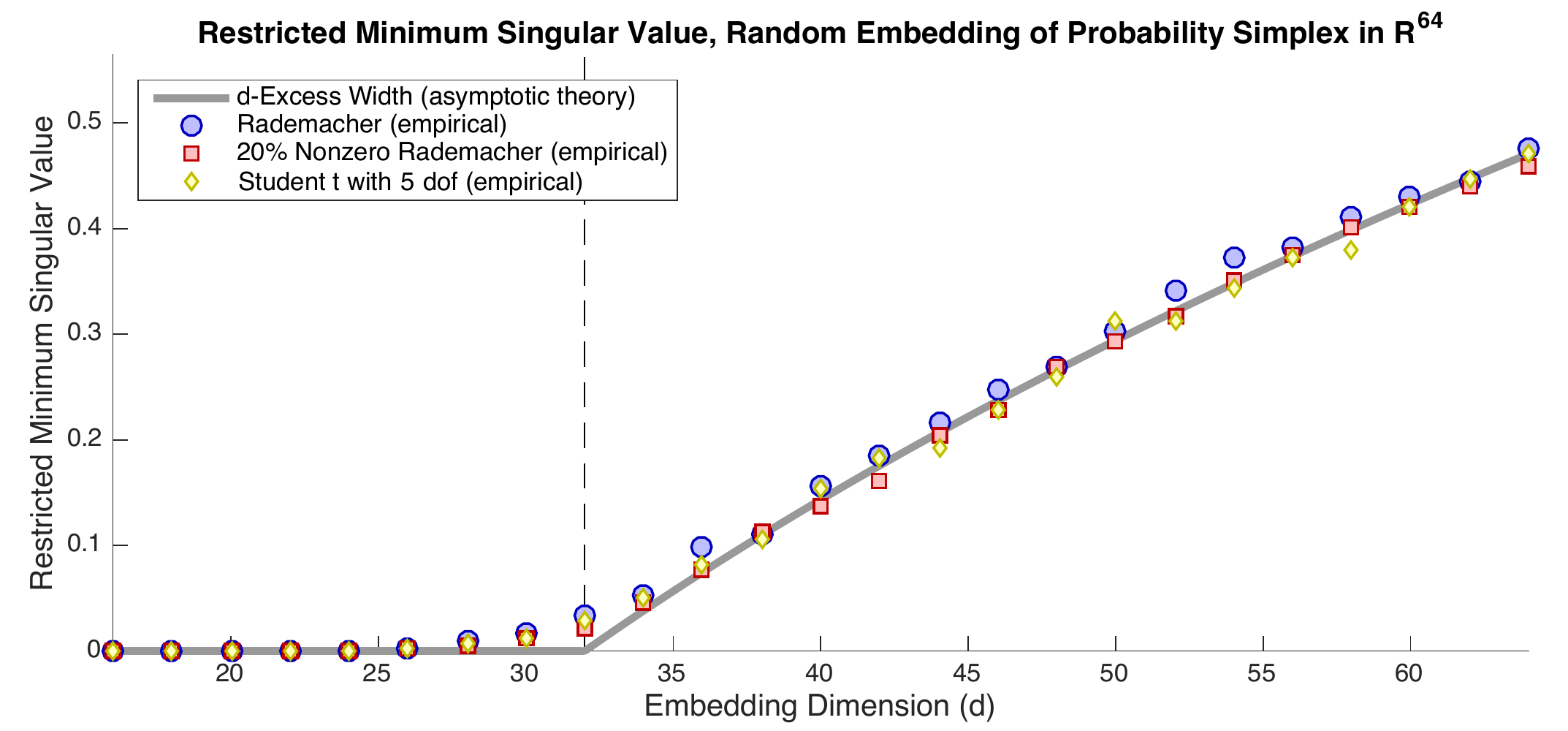}
\caption[Universality of the Restricted Minimum Singular Value]
{\textsl{Universality of the Restricted Minimum Singular Value.} \label{fig:intro-orthant-rsv}
This plot describes the behavior of three types of random \joelprev{linear map}s applied to the
probability simplex
$\Delta_{64} := \{ \vct{t} \in \R^{64} : \sum_i t_i = 1, t_i \geq 0 \}$.
The \textbf{dashed line} marks the minimum dimension $d = 32$ where
uniformly random embedding of the set $\Delta_{64}$ is likely to succeed.
The \textbf{gray curve} interpolates the value
of the positive part $\big( \coll{E}_d(\Delta_{64}) \big)_+$ of the $d$-excess width
of $\Delta_{64}$, obtained from the asymptotic calculation~\eqref{eqn:excess-width-orthant}.
The \textbf{markers} give an empirical estimate (over 100 trials)
for the restricted minimum singular value $\smin(\mtx{\Pi}; \Delta_{64})$
of a random \joelprev{linear map} $\mtx{\Pi} : \R^{64} \to \R^d$ drawn from the specified distribution.
See Section~\ref{sec:repro-research} for more information.}
\end{figure}

\subsection{Applications of Universality}

Randomized dimension reduction and, more generally,
random matrices have become ubiquitous in the information sciences.
As a consequence, the universality laws that we outlined
in Section~\ref{sec:intro-univ-law1} and~\ref{sec:intro-univ-law2}
have a wide range of implications.

\begin{description} \setlength{\itemsep}{0.25pc}
\item[Signal Processing]
The main idea in the field of compressed sensing is that
we can acquire information about a structured signal by
taking random linear measurements.
The literature contains extensive empirical evidence that
many types of random measurements behave in an
indistinguishable fashion.
Our work gives the first explanation of this phenomenon.
(Section~\ref{sec:signal-recovery})

\item[Stochastic Geometry]
Our results also indicate that the facial structure
of the convex hull of independent random vectors,
drawn from an appropriate class, does not depend
heavily on the distribution. (Section~\ref{sec:faces})

\item[Coding Theory]
Random linear codes provide an efficient way to protect
transmissions against error.  We demonstrate that
a class of random codes is resilient against sparse
corruptions.  The number of errors that can be
corrected does not depend on the specific choice
of codebook.  (Section~\ref{sec:coding})

\item[Numerical Analysis]
Our research provides an engineering design principle for
numerical algorithms based on randomized dimension reduction.
We can select the random \joelprev{linear map} that is most favorable for
implementation and still be confident about the detailed behavior
of the algorithm.
This approach allows us to develop efficient numerical methods
that also have rigorous performance guarantees.
(Section~\ref{sec:rand-nla})

\item[Random Matrix Theory]
Our work leads to a new proof of the Bai--Yin law for
the minimum singular value of a random matrix
with independent entries.  (Section~\ref{sec:subspace-embed})

\item[High-Dimensional Statistics]
The LASSO is a widely used method for performing
regression and variable selection.
We demonstrate that the prediction error associated
with a LASSO estimator is universal across a large class
of random designs and statistical error models.
We also show that least-absolute-deviation (LAD)
regression can correct a small number of arbitrary
statistical errors for a wide class of random designs.
(Section~\ref{sec:lasso-err} and Remark~\ref{rem:lad})

\item[Neuroscience]
Our universality laws may even have broader scientific significance.
It has been conjectured, with some experimental evidence, that the
brain may use dimension reduction to compress information~\cite{GG15:Simplicity-Complexity}.
Our universality laws suggest that many types of
uncoordinated (i.e., random) activity
lead to dimension reduction methods with the same behavior.
This result indicates that the hypothesis of neural dimensionality
reduction may be biologically plausible.
\end{description}

\subsection{The Scope of the Universality Phenomenon}

The universality phenomenon developed in this paper extends beyond
the results that we establish, but there are some (apparently) related
problems where universality does not hold.  Let us say a few words
about these examples and non-examples.

First, it does not appear important that the random \joelprev{linear map}
$\mtx{\Pi}$ has independent entries.
There is extensive evidence that structured
random \joelprev{linear map}s also have some universality properties;
for example, see~\cite{DT09:Observed-Universality}.

Second, the restricted minimum singular value is not the only
type of functional where universality is visible.  For instance,
suppose that $f$ is a convex, Lipschitz function.  Consider the
quantity
$$
\min_{\vct{t} \in E}\big( \normsq{ \mtx{\Pi} \vct{t} } + f(\vct{t}) \big).
$$
Optimization problems of this form play a central role in contemporary
statistics and machine learning.  It is likely that the value of this
optimization problem is universal over a wide class of random \joelprev{linear map}s.
Furthermore, we believe that our techniques can be adapted to address
this question. 
On the other hand, geometric functionals involving non-Euclidean norms
need not exhibit universality.  Consider the \term{$\ell_1$ restricted
minimum singular value}
\begin{equation} \label{eqn:l1-min-rsv}
\min_{\vct{t} \in E} \pnorm{\ell_1}{\mtx{\Pi} \vct{t}}.
\end{equation}
There are nontrivial sets $E$ where the value of the
optimization problem~\eqref{eqn:l1-min-rsv}
varies a lot with the choice of the random \joelprev{linear map} $\mtx{\Pi}$.
For instance, let $\mathbf{e}_1 \in \R^D$ be the first standard basis vector,
and define the shifted Euclidean ball
$$
E_{\alpha} := \big\{ \vct{t} \in \R^{D} : \norm{ \vct{t} - \mathbf{e}_1 } \leq \alpha \big\}
\quad\text{for $0 \leq \alpha \leq 1$.}
$$
Using Theorem~\ref{thm:univ-embed}
and the calculation~\cite[Sec.~3.4]{ALMT14:Living-Edge}
of the statistical dimension of a circular cone, we can verify
that there is a universal phase transition for successful embedding
of the set $E_{\alpha}$ when the embedding dimension $d = \alpha^2 D + O(1)$.
The result~\eqref{eqn:intro-rsv-univ}
implies that the minimum restricted singular value
of $E_{\alpha}$ also takes a universal value.
At the same time, Figure~\ref{fig:l1-rsv-nonuniversal} illustrates that the
functional~\eqref{eqn:l1-min-rsv} is not universal for the set $E_{\alpha}$.

Finally, functionals involving maximization do not necessarily exhibit
universality.  The \term{restricted maximum singular value}
is defined as
$$
\smax(\mtx{\Pi}; E) := \max_{\vct{t} \in E} \norm{\mtx{\Pi} \vct{t}}.
$$
It is not hard to produce examples where the restricted maximum
singular value depends on the choice of the random matrix $\mtx{\Pi}$.
For instance, Figure~\ref{fig:max-rsv-nonuniversal} demonstrates that
the random \joelprev{linear map} $\mtx{\Pi}$ has a substantial impact
on the maximum singular value $\smax(\mtx{\Pi}; \Delta_{D})$
restricted to the probability simplex $\Delta_D$ in $\R^D$.
This observation may surprise researchers in random matrix theory
because the ordinary maximum singular value is universal over the
class of random matrices with independent entries~\cite[Thm.~3.10]{BS10:Spectral-Analysis}.

\begin{figure}
\includegraphics[width=\textwidth]{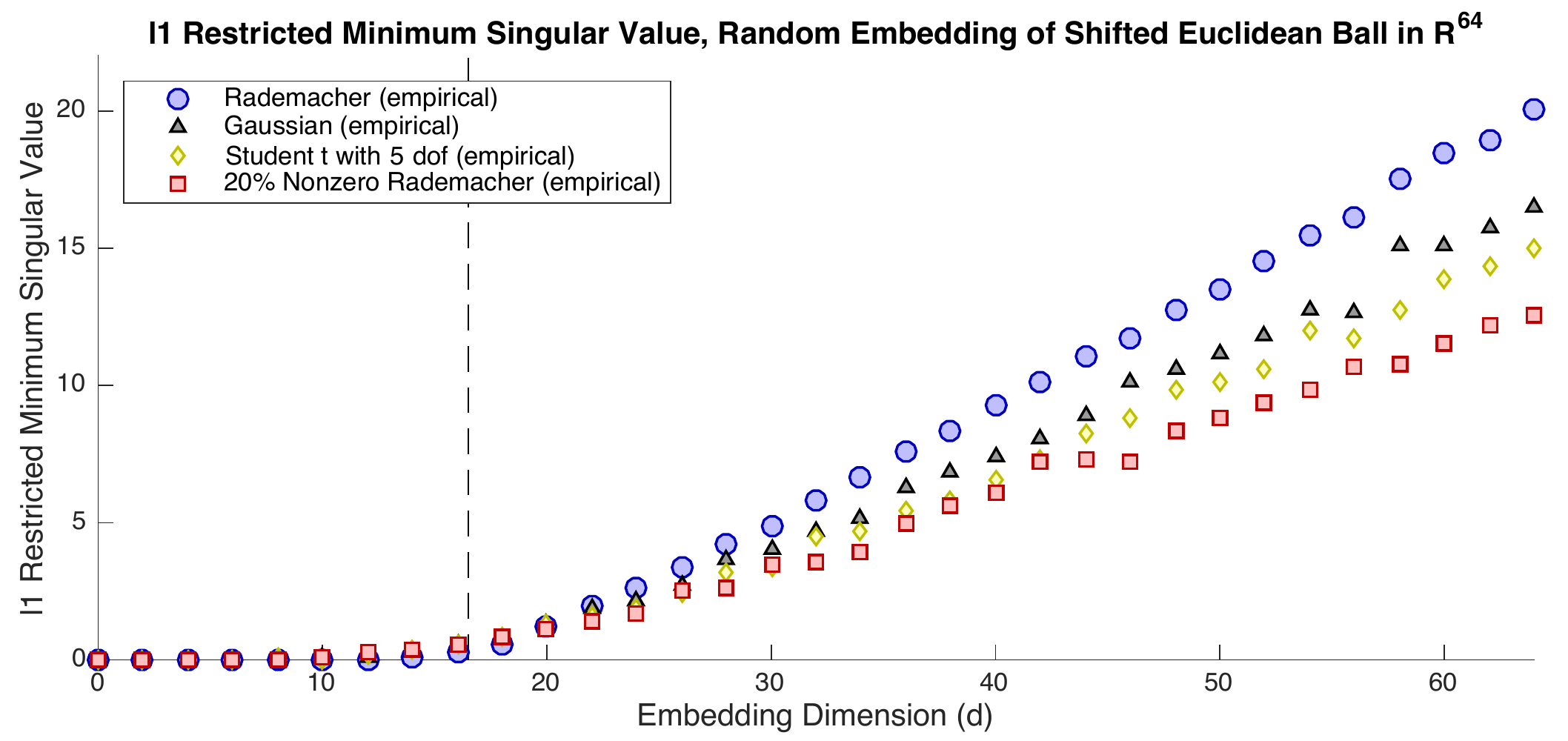}
\caption[Non-Universality of the $\ell_1$ Restricted Minimum Singular Value]
{\textsl{Non-Universality of the $\ell_1$ Restricted Minimum Singular Value.} \label{fig:l1-rsv-nonuniversal}
This plot describes the behavior of four types of random \joelprev{linear map}s applied to the
set $E_{1/2} := \{ \vct{t} \in \R^{64} : \norm{ \vct{t} - \mathbf{e}_1 } \leq 1/2 \}$,
where $\mathbf{e}_1$ is the first standard basis vector.
The \textbf{dashed line} stands at the phase transition
$d \approx 16$ for the embedding dimension of the set $E_{1/2}$.
The \textbf{markers} give an empirical estimate (over 100 trials)
for the $\ell_1$ restricted minimum singular value
$\min_{\vct{t} \in E_{1/2}} \pnorm{\ell_1}{\mtx{\Pi}\vct{t}}$
of a random \joelprev{linear map} $\mtx{\Pi} : \R^{64} \to \R^{d}$ with the specified
distribution as a function of the embedding dimension $d$.
See Section~\ref{sec:repro-research} for more details.}
\end{figure}

\begin{figure}
\includegraphics[width=\textwidth]{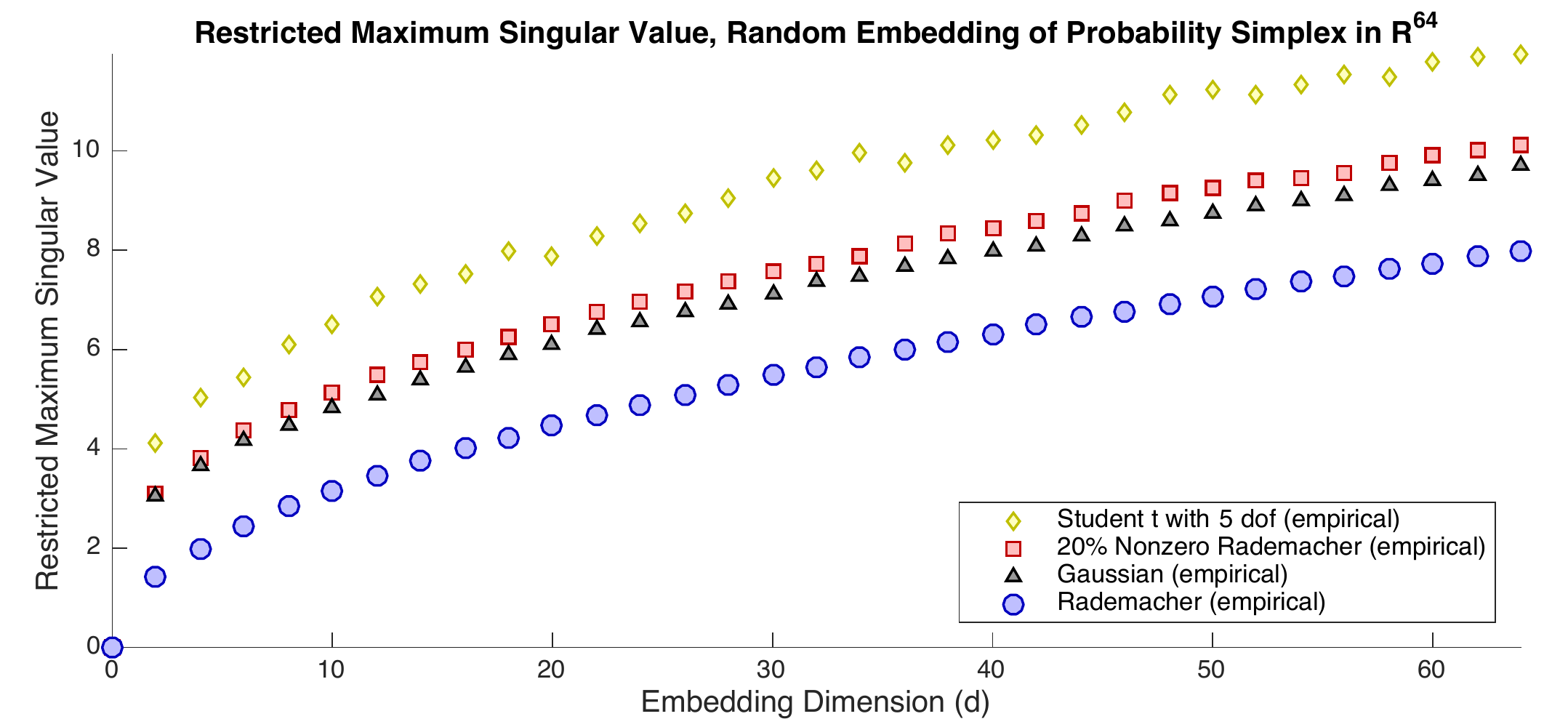}
\caption[Non-Universality of the Restricted Maximum Singular Value]
{\textsl{Non-Universality of the Restricted Maximum Singular Value.} \label{fig:max-rsv-nonuniversal}
This plot describes the behavior of four types of random \joelprev{linear map}s applied to the
probability simplex
$\Delta_{64} := \{ \vct{t} \in \R^{64} : \sum_i t_i = 1, t_i \geq 0 \}$.
The \textbf{markers} give an empirical estimate (over 100 trials)
for the restricted maximum singular value $\smax(\mtx{\Pi}; \Delta_{64})$
of a random \joelprev{linear map} $\mtx{\Pi} : \R^{64} \to \R^d$ with the specified
distribution.  See Section~\ref{sec:repro-research} for more details.}
\end{figure}

\subsection{Reproducible Research}
\label{sec:repro-research}

This paper is accompanied by \texttt{MATLAB} code~\cite{Tro15:Universality-Code}
that reproduces each figure from stored data.  This software can also repeat
the numerical experiments to obtain new instances of each figure.
By modifying the parameters in the code, the reader may explore how changes
affect the universality phenomenon.  We omit meticulous descriptions of
the numerical experiments from the text because these recitations are
tiresome for the reader and the code offers superior documentation.

\subsection{Roadmap}

This paper is divided into five parts.  Part~\ref{part:main}
offers a complete presentation of our universality laws, some
comments about the proofs, and some prospects for further research.
Part~\ref{part:applications} outlines the applications of universality
in several disciplines, and it contains more empirical confirmation
of our analysis.  Part~\ref{part:rsv} presents the proof that the restricted minimum singular value exhibits universal behavior;
this argument also yields the condition in which randomized embedding
is likely to succeed.
Part~\ref{part:rap} contains the proof of the condition in which
randomized embedding is likely to fail.
Finally, Part~\ref{part:back-matter} includes \joelprev{background
results}, the acknowledgments,
and the list of works cited.

\subsection{Notation}

Let us summarize our notation.
We use italic lowercase letters (for example, $a$) for scalars,
boldface lowercase letters ($\vct{a}$) for vectors,
and boldface uppercase letters ($\mtx{A}$) for matrices.
Uppercase italic letters ($A$) may denote scalars,
sets, or random variables, depending on the context.
Roman letters ($\cnst{c}$, $\cnst{C}$) denote
universal constants that may change from appearance
to appearance.  We sometimes delineate specific
constant values with subscripts ($\cnst{C}_{\rm rad}$).

Given a vector $\vct{a}$ and a set $J$ of indices,
we write $\vct{a}_J$ for the vector
restricted to those indices.  In particular,
$a_j$ is the $j$th coordinate of the vector.
Given a matrix $\mtx{A}$ and sets $I$ and $J$
of row and column indices, we write $\mtx{A}_{IJ}$
for the submatrix indexed by $I$ and $J$.
In particular, $a_{ij}$ is the component
in the $(i, j)$ position of $\mtx{A}$.  If there
is a single index $\mtx{A}_J$, it always refers to
the \emph{column} submatrix indexed by $J$.  

We always work in a real Euclidean space.
The symbol $\mathsf{B}^n$ is the unit ball
in $\R^n$, and $\mathsf{S}^{n-1}$ is the
unit sphere in $\R^n$.
The unadorned norm $\norm{\cdot}$
refers to the $\ell_2$ norm of a vector or the
$\ell_2$ operator norm of a matrix.
We use the notation $\vct{s} \cdot \vct{t}$
for the standard inner product of vectors
$\vct{s}$ and $\vct{t}$ with the same length.
We write ${}^\adj$ for the transpose of a
vector or a matrix.

For a real number $a$, we define the positive-part
and negative-part functions:
$$
(a)_+ := \max\{a, 0\}
\quad\text{and}\quad
(a)_- := \max\{-a, 0\}.
$$
These functions bind before powers,
so $(a)_+^2$ is the square of the positive part of $a$.

The symbols $\Expect$ and $\Var$ refer to the expectation and variance
of a random variable, and $\Prob{\cdot}$ returns the probability of an event.
We use the convention that powers bind before the expectation,
so $\Expect X^2$ returns the expectation of the square.  We write
$\mathbb{1}A$ for the 0--1 indicator random variable of the event $A$.

A \term{standardized} random variable has mean zero and variance one.
A \term{symmetric} random variable $X$ has the same distribution as
its negation $-X$.  We reserve the letter $\gamma$ for a standard normal random variable;
the boldface letters $\vct{\gamma}, \vct{g}, \vct{h}$ are always standard normal vectors;
and $\mtx{\Gamma}$ is a standard normal matrix.
The dimensions are determined by context.

\newpage

\part{Main Results}
\label{part:main}

This part of the paper introduces two new universality laws,
one for the phase transition in the embedding dimension
and a second one for the restricted minimum singular values
of a random \joelprev{linear map}.
We also include some high-level remarks about the proofs, but
we \joelprev{postpone the details} to Parts~\ref{part:rsv}
and~\ref{part:rap}.

In Section~\ref{sec:rdm-mtx}, we introduce two models for
random \joelprev{linear map}s that we use throughout the paper.
Section~\ref{sec:univ-embed} presents the universality
result for the embedding dimension,
and Section~\ref{sec:univ-rsv}
presents the result for restricted singular values.

\section{Random Matrix Models}
\label{sec:rdm-mtx}

To begin, we present two models for random \joelprev{linear map}s that
arise in our study of universality.  One model includes
bounded random matrices with independent entries,
while the second allows random matrices
with heavy-tailed entries.

\subsection{Bounded Random Matrix Model}

Our first model contains matrices whose entries are uniformly bounded.
This model is useful for some applications, and it plays a central
role in the proofs of our universality results.

\begin{model}[Bounded Matrix Model] \label{mod:bdd-mtx}
Fix a parameter $B \geq 1$.
A random matrix in this model has the following properties:

\begin{itemize}
\item	\textbf{Independence.}  The entries are stochastically independent random variables.

\item	\textbf{Standardization.}  Each entry has mean zero and variance one.

\item	\textbf{Symmetry.}  Each entry has a symmetric distribution.

\item	\textbf{Boundedness.}  Each entry $X$ of the matrix is uniformly bounded: $\abs{X} \leq B$.
\end{itemize}

\noindent
Identical distribution of entries is not required.
In some cases, which we will note, the symmetry requirement can be dropped.
\end{model}

This model includes several types of random matrices that appear frequently in practice.

\begin{example}[Rademacher Matrices]
Consider a random matrix whose entries are independent, Rademacher random variables.
This type of random matrix meets the requirements of Model~\ref{mod:bdd-mtx} with $B = 1$.
Rademacher matrices provide the simplest example of a random \joelprev{linear map}.
They are appealing for many applications because they are discrete.
\end{example}

\begin{example}[Sparse Rademacher Matrices]
Let $\alpha \in (0, 1]$ be a thinning parameter.  Consider a random variable $X$ with distribution
$$
X = \begin{cases}
	+ \alpha^{-1/2}, & \text{with probability $\alpha/2$}; \\
	- \alpha^{-1/2}, & \text{with probability $\alpha/2$}; \\
	0, & \text{otherwise.}
\end{cases}
$$
A random matrix whose entries are independent copies of $X$ satisfies Model~\ref{mod:bdd-mtx}
with $B = \alpha^{-1/2}$.  These random matrices are useful because we can control the sparsity.
\end{example}

\subsection{Heavy-Tailed Random Matrix Model}

Next, we introduce a \joelprev{more} general class of random matrices
that includes heavy-tailed examples.  Our main results concern
random \joelprev{linear map}s from this model.

\begin{model}[$p$-Moment Model] \label{mod:p-mom-mtx}
Fix parameters $p > 4$ and $\nu \geq 1$.
A random matrix in this model has the following properties:

\begin{itemize}
\item	\textbf{Independence.}  The entries are stochastically independent random variables.

\item	\textbf{Standardization.}  Each entry has mean zero and variance one.

\item	\textbf{Symmetry.}  Each entry has a symmetric distribution.

\item	\textbf{Bounded Moments.}  Each entry $X$ has a uniformly bounded $p$th moment:
$\Expect \abs{X}^{p} \leq \nu^{p}$.
\end{itemize}

\noindent
Identical distribution of entries is not required.
\end{model}

Model~\ref{mod:p-mom-mtx} subsumes Model~\ref{mod:bdd-mtx},
but it also encompasses many other types of random matrices.

\begin{example}[Gaussian Matrices]
Consider an $m \times n$ random matrix $\mtx{\Gamma}$ whose entries are independent,
standard normal random variables.
The matrix $\mtx{\Gamma}$ satisfies the requirements of Model~\ref{mod:p-mom-mtx}
for each $p > 4$ with $\nu \leq \sqrt{p}$.

In some contexts, we can use a Gaussian random matrix to study the behavior
of a uniformly random \joelprev{partial isometry}.  Indeed, the null space, $\nullsp(\mtx{\Gamma})$,
of the standard normal matrix is a uniformly random subspace
of $\R^n$ with codimension $\min\{m, n\}$, almost surely.
\end{example}

Model~\ref{mod:p-mom-mtx} contains several well-studied
classes of random matrices.

\begin{example}[Subgaussian Matrices]
Suppose that the entries of the random matrix are independent, and
each entry $X$ is symmetric, standardized,
and uniformly subgaussian.  That is, there is a parameter $\alpha > 0$ where
$$
\left( \Expect \abs{X}^p \right)^{1/p} \leq \alpha \sqrt{p}
\quad\text{for all $p \geq 1$.}
$$
These matrices are included in Model~\ref{mod:p-mom-mtx} 
for each $p > 4$ with $\nu \leq \alpha \sqrt{p}$.
Rademacher, sparse Rademacher, and Gaussian matrices fall in this category.
\end{example}

\begin{example}[Log-Concave Entries]
Suppose that the entries of the random matrix are independent,
and each entry $X$ is a symmetric, standardized, log-concave random variable.
Recall that a real log-concave random variable $X$ has a density $f$ of the form
$$
f(x) := \frac{1}{Z} \econst^{-h(x)}
$$
where $h : \R \to \R$ is convex and $Z$ is a normalizing constant.
It can be shown~\cite[Thm.~2.4.6]{BGVV14:Geometry-Isotropic}
that these matrices are included in Model~\ref{mod:p-mom-mtx}
for any $p > 4$.
\end{example}

In contrast with most research on randomized dimension reduction,
we allow the random \joelprev{linear map} to have entries with heavy tails.
Here is one such example.

\begin{example}[Student $t$ Matrices]
Suppose that each entry of the random matrix is an independent
Student $t$ random variable with $\alpha$ degrees of freedom, for $\alpha > 4$. 
This matrix also follows Model~\ref{mod:p-mom-mtx} for each $p < \alpha$.
\end{example}

Finally, we present a general construction that takes a matrix from
Model~\ref{mod:p-mom-mtx} and produces a sparse matrix that
also satisfies the model, albeit with a larger value of the parameter $\nu$.

\begin{example}[Sparsified Random Matrix]
Let $\mtx{\Phi}$ be a random matrix that satisfies Model~\ref{mod:p-mom-mtx}
for some value $p > 4$ and $\nu \geq 1$.  Let $\alpha \in (0, 1]$ be a
thinning parameter, and construct a new random matrix $\mtx{\Phi}^{(\alpha)}$
whose entries $\phi^{(\alpha)}_{ij}$ are independent random variables with the distribution
$$
\phi^{(\alpha)}_{ij} = \begin{cases}
	\alpha^{-1/2} \phi_{ij}, & \text{with probability $\alpha$} \\
	0, & \text{with probability $1 - \alpha$.}
\end{cases}
$$
Then the sparsified random matrix $\mtx{\Phi}^{(\alpha)}$
still follows Model~\ref{mod:p-mom-mtx}
with the same value of $p$ and with a modified value $\nu'$ of the other parameter:
$\nu' = \alpha^{-1/p} \nu$. 
\end{example}

\section{A Universality Law for the Embedding Dimension}
\label{sec:univ-embed}

In this section, we present detailed results which show that,
for a large class of sets, the embedding dimension is universal
over a large class of \joelprev{linear map}s.

\subsection{Embedding Dimension: Problem Formulation}
\label{sec:univ-embed-problem}

Let us begin with a more rigorous statement of the problem.

\begin{itemize}
\item	Fix the ambient dimension $D$. 
\item	Let $E$ be a nonempty, compact subset of $\R^D$ that does not contain the origin.

\item	Let $\mtx{\Pi} : \R^D \to \R^d$ be a random \joelprev{linear map} with embedding dimension $d$.
\end{itemize}

\noindent
We are interested in understanding the probability that the random projection
$\mtx{\Pi}(E)$ does not contain the origin.
That is, we want to study the following predicate:
\begin{equation} \label{eqn:univ-embed-predicate}
\vct{0} \notin \mtx{\Pi}(E)
\quad\text{or, equivalently,}\quad
E \cap \nullsp(\mtx{\Pi}) = \emptyset.
\end{equation}
We say that the random projection \term{succeeds} when the property~\eqref{eqn:univ-embed-predicate}
holds; otherwise, the random projection \term{fails}.

Our goal is to argue that there is a large class of sets
and a large class of random \joelprev{linear map}s where the probability
that~\eqref{eqn:univ-embed-predicate} holds depends primarily
on the choice of the embedding dimension $d$ and an appropriate
measure of the size of the set $E$.  In particular, the probability
does \emph{not} depend significantly on the distribution
of the random \joelprev{linear map} $\mtx{\Pi}$.

\subsection{Some Concepts from Spherical Geometry}

The property~\eqref{eqn:univ-embed-predicate} does not reflect
the scale of the points in the set $E$.  As a consequence,
it is appropriate to translate the problem into a question
about spherical geometry.  We begin with a definition.

\begin{definition}[Spherical Retraction]
The \term{spherical retraction} map $\vct{\theta} : \R^D \to \R^D$ is defined as
$$
\vct{\theta}(\vct{t}) :=
\begin{cases}
	\vct{t}/\norm{\vct{t}}, & \vct{t} \neq \vct{0} \\
	\vct{0}, & \vct{t} = \vct{0}.
\end{cases}
$$
\end{definition}

\noindent
For every set $E$ in $\R^D$, we have the equivalence
\begin{equation} \label{eqn:univ-embed-spherical}
\vct{0} \notin \mtx{\Pi}(E)
\quad\text{if and only if}\quad
\vct{0} \notin \mtx{\Pi}(\vct{\theta}(E)).
\end{equation}
Therefore, we may pass to the spherical retraction
$\vct{\theta}(E)$ of the set $E$ without loss of generality.

To obtain a complete analysis of when random projection succeeds or fails,
we must restrict our attention to sets that have a convexity property.

\begin{definition}[Spherical Convexity]
A nonempty subset $\Omega$ of the unit sphere in $\R^D$ is \term{spherically convex}
when the pre-image $\vct{\theta}^{-1}(\Omega) \cup \{\vct{0}\}$ is a convex cone.
By convention, the empty set is also spherically convex.
\end{definition}

\noindent
In particular, suppose that $T$ is a compact, convex set that does not contain
the origin.  Then the retraction $\vct{\theta}(T)$ is compact and spherically convex.

Next, we introduce a polarity operation for spherical sets
that supports some crucial duality arguments.

\begin{definition}[Spherical Polarity]
Let $\Omega$ be a subset of the unit sphere $\mathsf{S}^{D-1}$ in $\R^D$.
The \term{polar} of $\Omega$ is the set
$$
\Omega^\circ := \big\{ \vct{x} \in \mathsf{S}^{D-1} : \vct{x} \cdot \vct{t} \leq 0
\text{ for all $\vct{t} \in \Omega$} \big\}.
$$
By convention, the polar of the empty set is the whole sphere: $\emptyset^\polar := \mathsf{S}^{D-1}$.
\end{definition}

\noindent
This definition is simply the spherical analog of polarity for cones.
It can be verified that $\Omega^\circ$ is always closed and spherically convex.
Furthermore, the polarity operation is an involution on the class of closed,
spherically convex sets in $\mathsf{S}^{D-1}$.

\subsection{The Statistical Dimension Functional}
\label{sec:stat-dim}

We have the intuition that, for a given compact subset $E$ of $\R^D$,
the probability that a random \joelprev{linear map} $\mtx{\Pi} : \R^D \to \R^d$
succeeds decreases with the ``content'' of the set $E$.
In the present context, the correct notion of content involves
a geometric functional called the statistical dimension.

\begin{definition}[Statistical Dimension] \label{def:stat-dim}
Let $\Omega$ be a nonempty subset of the unit sphere in $\R^D$.  The
\term{statistical dimension} $\delta(\Omega)$ is defined as
\begin{equation} \label{eqn:stat-dim-sphere}
\delta(\Omega) := \Expect \bigg[
	\bigg( \sup_{\vct{t} \in \Omega} \vct{g} \cdot \vct{t} \bigg)_+^2 \bigg]
\end{equation}
In addition, define $\delta(\emptyset) := 0$.
We extend the statistical dimension to a general subset $T$ in $\R^D$
using the spherical retraction:
\begin{equation} \label{eqn:stat-dim-unbdd}
\delta(T) := \delta( \vct{\theta}(T) ).
\end{equation}
Recall that the random vector $\vct{g} \in \R^D$ has the standard normal distribution.
\end{definition}

The statistical dimension has a number of striking properties.
We include a short summary;
see~\cite[Prop.~3.1]{ALMT14:Living-Edge} and the citations there
for further information.

\begin{itemize}
\item	For a set $T$ in $\R^D$, the statistical dimension $\delta(T)$
takes values in the range $[0, D]$.

\item	The statistical dimension is increasing with respect to set inclusion:
$S \subset T$ implies that $\delta(S) \leq \delta(T)$.

\item	The statistical dimension agrees with the linear dimension on subspaces:
\begin{equation*} \label{eqn:sdim-subspace}
\delta(L) = \dim(L)
\quad\text{for each subspace $L$ in $\R^D$.}
\end{equation*}

\item	The statistical dimension interacts nicely with polarity:
\begin{equation} \label{eqn:sdim-polar}
\delta(\Omega^\polar) = D - \delta(\Omega)
\quad\text{for each spherically convex set $\Omega \in \R^D$.}
\end{equation}
The same relation holds if we replace $\Omega$ by a convex cone $K$ in $\R^D$
and use conic polarity.
\end{itemize}

As a specific example of~\eqref{eqn:sdim-polar},
we can evaluate the statistical dimension of the nonnegative orthant $\R_+^D$.
Indeed, the orthant is a self-dual cone, so
\begin{equation} \label{eqn:sdim-orthant}
\delta(\R_+^D) = D/2.
\end{equation}
There is also powerful machinery,
developed in~\cite{Sto09:Various-Thresholds,OH10:New-Null-Space,CRPW12:Convex-Geometry,ALMT14:Living-Edge,FM14:Corrupted-Sensing},
for computing the statistical dimension of a descent cone
of a convex function.
Finally, we mention that the statistical dimension can often be evaluated
by sampling Gaussian vectors and approximating the expectation
in~\eqref{eqn:stat-dim-sphere} with an empirical average.

\begin{remark}[Gaussian Width]
The statistical dimension is related to the \term{Gaussian width} functional.
For a bounded set $T$ in $\R^D$, the Gaussian width $\coll{W}(T)$ is defined as
\begin{equation} \label{eqn:gauss-width}
\coll{W}(T) := \Expect \sup_{\vct{t} \in T} \vct{g} \cdot \vct{t}.
\end{equation}
For a subset $\Omega$ of the unit sphere in $\R^D$, we have the inequalities
\begin{equation} \label{eqn:sdim-width}
\coll{W}^2(\Omega) \leq \delta(\Omega) \leq \coll{W}^2(\Omega) + 1.
\end{equation}
See~\cite[Prop.~10.3]{ALMT14:Living-Edge} for a proof.
The relation~\eqref{eqn:sdim-width} allows us to pass between
the squared width and the statistical dimension.

The Gaussian width of a set is closely related to its
mean width~\cite[Sec.~1.3.5]{Ver15:Estimation-High}.
The mean width is a canonical functional in the Euclidean
setting~\cite[Sec.~7.3]{Gru07:Convex-Discrete},
but it is not quite the right choice for spherical geometry.
We have chosen to work with the statistical dimension
because it has many geometric properties that the width
lacks in the spherical setting.
\end{remark}

\begin{remark}[Convex Cones]
The papers~\cite{ALMT14:Living-Edge,MT14:Steiner-Formulas}
define the statistical dimension of a convex cone
using intrinsic volumes.  It can be shown that
the statistical dimension is the canonical
(additive, continuous) extension of the
linear dimension to the class of
closed convex cones~\cite[Sec.~5.6]{ALMT14:Living-Edge}.
Our general definition~\eqref{eqn:stat-dim-unbdd}
agrees with the original definition on this class.
\end{remark}

\subsection{Theorem~\ref{thm:univ-embed}: Universality of the Embedding Dimension}
\label{sec:thm-univ-embed}

We are now prepared to state our main result on the probability that
a random \joelprev{linear map} succeeds or fails for a given set.

\begin{bigthm}[Universality of the Embedding Dimension] \label{thm:univ-embed}
Fix the parameters $p > 4$ and $\nu \geq 1$ for Model~\ref{mod:p-mom-mtx}.
Choose parameters $\varrho \in (0,1)$ and $\eps \in (0, 1)$.
There is a number $N := N(p, \nu, \varrho, \eps)$ for which the
following statement holds.
Suppose that

\begin{itemize}
\item	The ambient dimension $D \geq N$.

\item	$E$ is a nonempty, compact subset of $\R^D$ that does not contain the origin.

\item	The statistical dimension of $E$ is proportional to the ambient dimension: $\delta(E) \geq \varrho D$.

\item	The $d \times D$ random \joelprev{linear map} $\mtx{\Pi}$ obeys Model~\ref{mod:p-mom-mtx}
with parameters $p$ and $\nu$.
\end{itemize}
Then
\begin{align} \label{eqn:univ-embed-succ}
d &\geq (1 + \eps) \, \delta(E)
\quad\text{implies}\quad
\Prob{ \vct{0} \notin \mtx{\Pi}(E) } \geq 1 - \cnst{C}_p D^{1 - p/4}. \tag{a}
\intertext{Furthermore, if $\vct{\theta}(E)$ is spherically convex, then}
\label{eqn:univ-embed-fail}
d &\leq (1 - \eps) \, \delta(E)
\quad\text{implies}\quad
\Prob{ \vct{0} \in \mtx{\Pi}(E) } \geq 1 - \cnst{C}_p D^{1-p/4}.	\tag{b}
\end{align}
The constant $\cnst{C}_p$ depends only on
the parameter $p$ in the random matrix model.
\end{bigthm}

\noindent
Section~\ref{sec:thm1-strategy} summarizes our strategy for establishing
Theorem~\ref{thm:univ-embed}.  The detailed proof of
Theorem~\ref{thm:univ-embed}\eqref{eqn:univ-embed-succ} appears in
Section~\ref{sec:rsv-four-to-embed-univ};
the detailed proof of Theorem~\ref{thm:univ-embed}\eqref{eqn:univ-embed-fail}
appears in Section~\ref{sec:car-to-embed}.
Let us mention that stronger probability bounds hold when the random
\joelprev{linear map} is drawn from Model~\ref{mod:bdd-mtx}.

For any random \joelprev{linear map} $\mtx{\Pi}$ drawn from Model~\ref{mod:p-mom-mtx},
Theorem~\ref{thm:univ-embed}\eqref{eqn:univ-embed-succ} ensures that
the random \joelprev{dimension reduction} $\mtx{\Pi}(E)$
is likely to succeed when the embedding dimension $d$
exceeds the statistical dimension $\delta(E)$.
Similarly, Theorem~\ref{thm:univ-embed}\eqref{eqn:univ-embed-fail}
shows that $\mtx{\Pi}(E)$ is likely to fail
when the embedding dimension $d$ is smaller
than the statistical dimension $\delta(E)$,
provided that $\vct{\theta}(E)$ is spherically convex.
Note that both of these interpretations require that
the statistical dimension $\delta(E)$ is not too
small as compared with the ambient dimension $D$.

We have already seen a concrete example of Theorem~\ref{thm:univ-embed} at work.
In view of the calculation~\eqref{eqn:sdim-orthant}
of the statistical dimension of the orthant,
the theorem provides a satisfying explanation
of Figure~\ref{fig:intro-orthant-universal}!

\begin{remark}[Prior Work]
When the random \joelprev{linear map} $\mtx{\Pi}$ is Gaussian,
the result Theorem~\ref{thm:univ-embed}\eqref{eqn:univ-embed-succ}
follows from Gordon~\cite[Cor.~3.4]{Gor88:Milmans-Inequality},
while the conclusion~\eqref{eqn:univ-embed-fail}
seems to be more recent~\cite[Thm.~I]{ALMT14:Living-Edge}.
To our knowledge, the literature contains no precedent for
Theorem~\ref{thm:univ-embed} for general sets and for non-Gaussian \joelprev{linear map}s.
We can identify only a few sporadic special cases. 
Donoho \& Tanner~\cite{DT10:Counting-Faces} studied the problem
of recovering a ``saturated vector'' from random measurements
via $\ell_{\infty}$ minimization.
Their work can be interpreted as a statement about random embeddings of the set
of unit vectors with nonnegative coordinates.
They demonstrate that the phase transition
in the embedding dimension is universal when the rows of the \joelprev{linear map} matrix
are independent, symmetric random vectors with a density;
this result actually follows from classical work of Sch{\"afli}~\cite{Sch50:Gesammelte}
and Wendel~\cite{Wen62:Problem-Geometric}.
Let us emphasize that this result does not apply to discrete random \joelprev{linear map}s.

Bayati et al.~\cite{BLM15:Universality-Polytope} have studied
the problem of recovering a sparse vector from random measurements
via $\ell_1$ minimization.
Their work can be interpreted as a result on the embedding dimension
of the set of descent directions of the $\ell_1$ norm at a sparse vector.
They showed that, asymptotically, the phase transition in the embedding dimension
is universal.
Their result requires the \joelprev{linear map} matrix to contain independent,
subgaussian entries that are absolutely continuous with respect
to the Gaussian density.  See Section~\ref{sec:l1-min} for more
discussion of this result.

There are also many papers in asymptotic convex geometry
and mathematical signal processing that contain theory about
the order of the embedding dimension for \joelprev{linear map}s from Model~\ref{mod:p-mom-mtx}; see~\cite{MPT07:Reconstruction-Subgaussian,Men10:Empirical-Processes,Men13:Remark-Diameter,Tro15:Convex-Recovery,Ver15:Estimation-High}.  These
results do not allow us to reach any conclusions about the
existence of a phase transition or its location.  There is also
an extensive amount of empirical work,
such as~\cite{DT09:Observed-Universality,Sto09:Various-Thresholds,OH10:New-Null-Space, CSPW11:Rank-Sparsity-Incoherence, DGM13:Phase-Transition, MT14:Sharp-Recovery},
that suggests that phase transitions are ubiquitous in high-dimensional
signal processing problems,
but there has been no theoretical explanation
of the universality phenomenon until now.  
\end{remark}

\subsection{Theorem~\ref{thm:univ-embed}: Proof Strategy}
\label{sec:thm1-strategy}

The proof of Theorem~\ref{thm:univ-embed} depends on converting
the geometric question to an analytic problem.  First, recall
the equivalence~\eqref{eqn:univ-embed-spherical},
which allows us to pass from the compact set $E$
to its spherical retraction $\Omega := \vct{\theta}(E)$.
Next, we identify two analytic quantities that determine
whether a linear map annihilates a point in the set $\Omega$.

\begin{proposition}[Analytic Conditions for Embedding] \label{prop:annihilate}
Let $\Omega$ be a \joelprev{nonempty,} closed subset of the unit sphere $\mathsf{S}^{D-1}$ in $\R^D$,
and let $\mtx{A} : \R^D \to \R^d$ be a linear map.  Then 
\begin{align} \label{eqn:annihilate-primal}
\min_{\vct{t} \in \Omega} \norm{\mtx{A}\vct{t}} &> 0
\quad\text{implies}\quad
\vct{0} \notin \mtx{A}(\Omega).
\intertext{Furthermore, if $\Omega$ is spherically convex \sametprev{and $\cone(\Omega)$ is not a subspace},}
\label{eqn:annihilate-dual}
\min_{\norm{\vct{t}} = 1} \min_{\vct{s} \in \cone(\Omega^{\polar})} \norm{\smash{\vct{s} - \mtx{A}^\adj \vct{t}} } &> 0
\quad\text{implies}\quad
\vct{0} \in \mtx{A}(\Omega).
\intertext{\joelprev{Finally, if $\cone(\Omega)$ is a subspace, select an arbitrary subset $\Omega_{0}\subset \Omega$
with the property that $\cone(\Omega_{0})$ is a $(\dim(\Omega)-1)$-dimensional subspace.  Then}}
\label{eqn:annihilate-dual_sub}
\min_{\norm{\vct{t}} = 1} \min_{\vct{s} \in \cone(\Omega_{0}^{\polar})} \norm{\smash{\vct{s} - \mtx{A}^\adj \vct{t}} } &> 0
\quad\text{implies}\quad
\vct{0} \in \mtx{A}(\Omega).
\end{align}

\noindent
Recall that $\cone(S)$ is the smallest convex cone containing the set $S$.
\end{proposition}

\begin{proof}
The implication~\eqref{eqn:annihilate-primal} is quite easy to check.
At each point $\vct{t} \in \Omega$, we have $\norm{\mtx{A}\vct{t}} > 0$,
which implies that $\mtx{A}\vct{t} \neq \vct{0}$.  In other words,
$\vct{0} \notin \mtx{A}(\Omega)$.

The second implication~\eqref{eqn:annihilate-dual} follows from a spherical
duality principle.  Suppose that $\Omega$ and $\Upsilon$ are \joelprev{closed and spherically convex,
and assume $\cone(\Omega)$ is not a subspace.}
If $\Omega^\polar$ and \sametprev{$-\Upsilon^\polar$} do not intersect, then
$\Omega$ and $\Upsilon$ must intersect;
see~\cite[Thm~(2.7)]{Kle55:Separation-Properties}.
The analytic condition
$$
\min_{\norm{\vct{t}} = 1} \min_{\vct{s} \in \cone(\Omega^\polar)}
	\norm{ \smash{\vct{s} - \mtx{A}^\adj \vct{t}} } > 0
$$	
ensures that $\cone(\Omega^\polar)$ lies at a positive distance from $\vct{\theta}(\range(\mtx{A}^\adj))$.  It follows that $\Omega^\polar$
does not intersect \joelprev{$\vct{\theta}(\range(\mtx{A}^\adj))$}.
By duality, we conclude that $\Omega \cap \vct{\theta}(\nullsp(\mtx{A})) \neq \emptyset$,
which yields~\eqref{eqn:annihilate-dual}.

\joelprev{Finally, if $\cone(\Omega)$ is a subspace, we use a dimension counting argument.
As before, the analytic condition~\eqref{eqn:annihilate-dual_sub} implies that $\Omega_0^\circ \cap \vct{\theta}(\range(\mtx{A}^\adj)) = \emptyset$.  Since these sets do not intersect, the sum of their dimensions
cannot exceed the ambient dimension:
$$
\begin{aligned}
D \geq \dim( \Omega_0^\circ ) + \dim(\range(\mtx{A}^\adj))
	&= (D - \dim(\Omega_0)) + (D - \dim(\nullsp(\mtx{A}))) \\
	&= 2D + 1 - \dim(\Omega) - \dim(\nullsp(\mtx{A})).
\end{aligned}
$$
We see that $\dim(\Omega) + \dim(\nullsp(\mtx{A})) \geq D + 1$.
As a consequence, the subspace $\cone(\Omega)$ and $\null(\mtx{A})$ must have a nontrivial intersection.}
\end{proof}

In view of Proposition~\ref{prop:annihilate},
we can establish Theorem~\ref{thm:univ-embed}\eqref{eqn:univ-embed-succ}
by showing that
$$
d \geq (1 - \eps)\, \delta(\Omega)
\quad\text{implies}\quad
\min_{\vct{t} \in \Omega} \norm{\mtx{\Pi} \vct{t}} > 0
\quad\text{with high probability.}
$$
This condition follows from a universality result,
Corollary~\ref{cor:rsv-four}, for the restricted minimum singular value.
The proof appears in Section~\ref{sec:rsv-four-to-embed-univ}.
Similarly, we can establish Theorem~\ref{thm:univ-embed}\eqref{eqn:univ-embed-fail}
by showing that
$$
d \leq (1 + \eps)\, \delta(\Omega)
\quad\text{implies}\quad
\min_{\norm{\vct{t}}=1} \min_{\vct{s} \in \cone(\Omega^\polar)}
	\norm{\smash{\vct{s} - \mtx{\Pi}^\adj \vct{t}} } > 0
\quad\text{with high probability.}
$$
This condition follows from a specialized argument that
culminates in Corollary~\ref{cor:car-four}.
The proof appears in Section~\ref{sec:car-to-embed}.
\joelprev{We need not lavish extra attention on the case where $\cone(\Omega)$ is a subspace.
Indeed, the dimension of $\Omega$ and $\Omega_0$ only differ by one, which is invisible in the final results.}

\subsection{Theorem~\ref{thm:univ-embed}: Extensions}

It is an interesting challenge to delineate the scope of the
universality phenomenon described in Theorem~\ref{thm:univ-embed}.
We believe that there remain many opportunities for improving
on this result.

\begin{itemize}
\item	Theorem~\ref{thm:univ-embed} only shows that
the width of the phase transition is $o(D)$.
It is known~\cite[Thm.~7.1]{ALMT14:Living-Edge}
that the width of the phase transition has order
$\min\big\{ \sqrt{\delta(E)}, \sqrt{D - \delta(E)} \big\}$
for a \joelprev{linear map} with standard normal entries.
How wide is the phase transition for more general random \joelprev{linear map}s?

\item	A related question is whether Theorem~\ref{thm:univ-embed}
holds for those sets $E$ whose statistical dimension $\delta(E)$
is much smaller than the ambient dimension.

\item	There is empirical evidence that the location
of the phase transition is universal over a class wider than Model~\ref{mod:p-mom-mtx}.
In particular, results like Theorem~\ref{thm:univ-embed} may be valid
for structured random \joelprev{linear map}s.

\item	Figure~\ref{fig:intro-orthant-universal} suggests
that the \emph{probability} of successful embedding is universal.
Under what conditions can this observation be formalized?
\end{itemize}

\noindent
In summary, Theorem~\ref{thm:univ-embed} is just the
first step toward a broader theory of universality in high-dimensional
stochastic geometry.

\section{A Universality Law for the Restricted Minimum Singular Value}
\label{sec:univ-rsv}

This section describes a quantitative universality law for random \joelprev{linear map}s.
We show that the restricted minimum singular value of a random \joelprev{linear map} takes
the same value for every \joelprev{linear map} in a substantial class.  This type of result
provides information about the stability of randomized dimension reduction.

\subsection{Restricted Minimum Singular Value: Problem Formulation}

Let us frame our assumptions:

\begin{itemize}
\item	Fix the ambient dimension $D$.

\item	Let $E$ be a nonempty, compact subset of the unit ball $\mathsf{B}^D$. 
\item	Let $\mtx{\Pi} : \R^D \to \R^d$ be a random \joelprev{linear map} with embedding dimension $d$.
\end{itemize}

In this section, our goal is to understand the distance from the random projection
$\mtx{\Pi}(E)$ to the origin.  The following definition captures this property.

\begin{definition}[Restricted Minimum Singular Value]
Let $\mtx{A} : \R^D \to \R^d$ be a linear map, and let $T$ be a nonempty subset of $\R^D$.
The \term{restricted minimum singular value} of $\mtx{A}$ with respect to
the set $T$ is the quantity
$$
\smin(\mtx{A}; T) := \inf_{\vct{t} \in T} \norm{ \mtx{A} \vct{t} }.
$$
More briefly, we write \term{restricted singular value} or \term{RSV}.
\end{definition}

\noindent
Proposition~\ref{prop:annihilate} shows that
the restricted minimum singular value is a quantity
of interest when studying the embedding dimension
of a \joelprev{linear map}.
It is also productive to think about $\smin(\mtx{A}; T)$ as
a measure of the stability of inverting the map $\mtx{A}$
on the image $\mtx{A}(T)$.  In particular,
note that the restricted singular value $\smin(\mtx{A}; T)$ is a generalization of the ordinary
minimum singular value $\smin(\mtx{A})$, which we obtain from the selection $T = \mathsf{S}^{D-1}$.

\subsection{The Excess Width Functional}
\label{sec:excess-width}

It is clear that the restricted singular value $\smin(\mtx{A}; \cdot)$
decreases with respect to set inclusion.  In other words, the
restricted singular value depends on the ``content'' of the set $T$.
In the case of a random \joelprev{linear map}, the following geometric functional
provides the correct notion of content.

\begin{definition}[Excess Width] \label{def:excess-width}
Let $m$ be a positive number,
and let $T$ be a nonempty, bounded subset of $\R^D$.
The \term{$m$-excess width} of $T$ is the quantity
\begin{equation} \label{eqn:excess-width}
\coll{E}_m(T) :=
\Expect \inf_{\vct{t} \in T} \big( \sqrt{m} \norm{\vct{t}} + \vct{g} \cdot \vct{t} \big).
\end{equation}
\end{definition}

\noindent
Versions of the excess width appear as early as the work of
Gordon~\cite[Cor.~1.1]{Gor88:Milmans-Inequality}.  It has
also come up in recent papers~\cite{Sto13:Regularly-Random,OTH13:Squared-Error,TOH15:Gaussian-Min-Max,TAH15:High-Dimensional-Error}
on the analysis of Gaussian random \joelprev{linear map}s.

The excess width has a number of useful properties.  These results
are immediate consequences of the definition.

\begin{itemize}
\item	For a subset $T$ of the unit ball in $\R^D$, the $m$-excess width satisfies the bounds
$$
\sqrt{m} - \sqrt{D} \leq \coll{E}_m(T) \leq \sqrt{m}.
$$
In particular, the excess width can be positive or negative.

\item	The $m$-excess width is weakly increasing in $m$.  That is,
$m \leq n$ implies $\coll{E}_m(T) \leq \coll{E}_n(T)$.

\item	The $m$-excess width is decreasing with respect to set inclusion: $S \subset T$
implies $\coll{E}_m(S) \geq \coll{E}_m(T)$.

\item	The $m$-excess width is absolutely homogeneous: $\coll{E}_m(\alpha T) = \abs{\alpha} \, \coll{E}_m(T)$ for $\alpha \in \R$.

\item	The excess width~\eqref{eqn:excess-width} is related to the Gaussian width~\eqref{eqn:gauss-width}:
\begin{equation} \label{eqn:excess-vs-width}
\coll{E}_m(\Omega) = \sqrt{m} - \coll{W}(\Omega)
\quad\text{for $\Omega \subset \mathsf{S}^{D-1}$.}
\end{equation}
Using~\eqref{eqn:sdim-width}, we can also relate the excess \joelprev{width} to the statistical dimension:
\begin{equation} \label{eqn:excess-vs-sdim}
\sqrt{m} - \sqrt{1 + \delta(\Omega)} \leq \coll{E}_m(\Omega) \leq \sqrt{m} - \sqrt{\delta(\Omega)}
\quad\text{for $\Omega \subset \mathsf{S}^{D-1}$.}
\end{equation}
\end{itemize}

The term ``excess width'' is not standard, but the formula~\eqref{eqn:excess-vs-width}
suggests that this moniker is appropriate.  According to Theorem~\ref{thm:univ-embed},
the sign $(\pm)$ of the excess width $\coll{E}_d(\Omega)$ indicates
that \joelprev{random dimension reduction} $\mtx{\Pi}(\Omega)$ with embedding dimension $d$
succeeds $(+)$ or fails $(-)$ with high probability.

The papers~\cite{Sto13:Regularly-Random,OTH13:Squared-Error,TOH15:Gaussian-Min-Max} develop methods
for computing the excess width in a variety of situations.  For instance, if $L$ is a $k$-dimensional
subspace of $\R^D$, then
$$
\coll{E}_m(L \cap \mathsf{S}^{D-1}) \approx \sqrt{m} - \sqrt{k}.
$$
For a more sophisticated example, consider the probability simplex
$$
\Delta_D := \left\{ \vct{t} \in \R^D : \sum_{i=1}^D t_i = 1, t_i \geq 0 \text{ for $i = 1, \dots, D$} \right\}.
$$
We can develop an asymptotic \joelprev{expression} for its excess width:
\begin{equation} \label{eqn:excess-width-orthant}
\lim_{\substack{D,d \to \infty \\ D/d = \varrho}} \coll{E}_d(\Delta_D)
	= \inf_{\alpha \geq 0} \bigg( \alpha - \inf_{s \in \R} \big(s + \alpha \sqrt{\varrho q(s)}\big) \bigg)
	\quad\text{where}\quad
	q(s) := \Expect\big[ (\gamma - s)_+^2 \big].
\end{equation}
The proof of the \joelprev{formula}~\eqref{eqn:excess-width-orthant} is involved,
so we must omit the details; see~\cite{TAH15:High-Dimensional-Error} for
a framework for making such computations.
See Part~\ref{part:applications} for some other examples.
It is also possible to estimate the excess width numerically by approximating
the expectation in~\eqref{eqn:excess-width} with an empirical average.

\subsection{Theorem~\ref{thm:univ-rsv}: Universality for the Restricted Minimum Singular Value}
\label{sec:thm-univ-rsv}

With this preparation, we can present our main result about the universality
properties of the restricted minimum singular value of a random \joelprev{linear map}.

\begin{bigthm}[Universality for the Restricted Minimum Singular Value] \label{thm:univ-rsv}
Fix the parameters $p$ and $\nu$ for Model~\ref{mod:p-mom-mtx}.
Choose parameters $\lambda \in (0,1)$ and $\varrho \in (0, 1)$
and $\eps \in (0,1)$.  There is a number $N := N(p,\nu,\lambda,\varrho,\eps)$
for which the following statement holds.  Suppose that

\begin{itemize}
\item	The ambient dimension $D \geq N$.

\item	$E$ is a nonempty, compact subset of the unit ball $\mathsf{B}^D$ in $\R^D$.

\item	The embedding dimension $d$ is in the range $\lambda D \leq d \leq D^{6/5}$.

\item	The $d$-excess width of $E$ is not too small: $\coll{E}_d(E) \geq \varrho \sqrt{d}$.

\item	The random \joelprev{linear map} $\mtx{\Pi} : \R^D \to \R^d$ obeys Model~\ref{mod:p-mom-mtx}
with parameters $p$ and $\nu$.
\end{itemize}

\noindent
Then
\begin{align} \label{eqn:univ-rsv-succ}
\Prob{ \smin(\mtx{\Pi}; E) \geq (1 - \eps) \big( \coll{E}_d(E) \big)_+ }
	&\geq 1 - \cnst{C}_p D^{1 - p/4}. \tag{a}
\intertext{Furthermore, if $E$ is convex, then} \label{eqn:univ-rsv-fail}
\Prob{ \smin(\mtx{\Pi}; E) \leq (1 + \eps) \big( \coll{E}_d(E) \big)_+ }
	&\geq 1 - \cnst{C}_p D^{1 - p/4}. \tag{b}
\end{align}
The constant $\cnst{C}_p$ depends only on the parameter $p$ in the random matrix model.
\end{bigthm}

\noindent
Section~\ref{sec:thm2-strategy} contains an overview of the argument.
The detailed proof of Theorem~\ref{thm:univ-rsv} appears
in Section~\ref{sec:rsv-four-to-rsv-univ}.
Stronger probability bounds hold when the random \joelprev{linear map} is drawn
from Model~\ref{mod:bdd-mtx}.

For any random \joelprev{linear map} $\mtx{\Pi}$ drawn from Model~\ref{mod:p-mom-mtx},
Theorem~\ref{thm:univ-rsv} asserts that the distance of the random projection $\mtx{\Pi}(E)$
from the origin is typically not much smaller than the $d$-excess width $\coll{E}_d(E)$.
Similarly, when $E$ is convex, the distance of the
random \joelprev{image} from the origin is typically not much larger than the $d$-excess width.
These conclusions require that
the excess width $\coll{E}_d(E)$ is not too small as compared with the root $\sqrt{d}$
of the embedding dimension.

Theorem~\ref{thm:univ-rsv} is vacuous when $\coll{E}_d(E) \leq 0$.
Nevertheless, in case $\vct{\theta}(E)$ is spherically convex,
the condition $\coll{E}_d(\vct{\theta}(E)) + o(\sqrt{D}) \leq 0$
implies that $\smin(\mtx{\Pi}; E) = 0$ with high probability because
of~\eqref{eqn:excess-vs-sdim}
and Theorem~\ref{thm:univ-embed}.

We have already seen an illustration of Theorem~\ref{thm:univ-rsv}
in action.  In view of the expression~\eqref{eqn:excess-width-orthant} for the
excess width of the probability simplex, the theorem explains the experiment
documented in Figure~\ref{fig:intro-orthant-rsv}!

\begin{remark}[Prior Work]
When the random \joelprev{linear map} $\mtx{\Pi}$ is Gaussian,
Gordon's work~\cite[Cor.~1.1]{Gor88:Milmans-Inequality} yields
the conclusion Theorem~\ref{thm:univ-rsv}\eqref{eqn:univ-rsv-succ},
while the second conclusion~\eqref{eqn:univ-rsv-fail}
appears to be more recent;
see~\cite{Sto13:Regularly-Random,OTH13:Squared-Error,TOH15:Gaussian-Min-Max,TAH15:High-Dimensional-Error}.

\joelprev{Korada \& Montanari~\cite{KM11:Applications-Lindeberg} have established
universality laws for some optimization problems that arise in communications.
The mathematical challenge in their work is similar to showing that the RSV
of a cube is universal for Model~\ref{mod:p-mom-mtx} with $p \geq 6$.
We have adopted some of their arguments in our work.  On the other hand,
the cube does not present any of the technical difficulties that arise for general sets.}

For other types of random \joelprev{linear maps and other types of sets},
we are not aware of any research that yields the precise value of the
restricted singular value, which is required to assert universality.
The literature on
asymptotic convex geometry
and mathematical signal processing,
e.g.,~\cite{MPT07:Reconstruction-Subgaussian,Men10:Empirical-Processes,Men13:Remark-Diameter,Tro15:Convex-Recovery,Ver15:Estimation-High},
contains many results on the restricted singular value
that provide bounds of the correct order with unspecified
multiplicative constants.
\end{remark}

\subsection{Theorem~\ref{thm:univ-rsv}: Proof Strategy}
\label{sec:thm2-strategy}

One standard template for proving a universality law
factors the problem into two parts.  First, we compare
a general random structure with a canonical structure
that has additional properties.  Then we use
these extra properties to obtain a complete
analysis of the canonical structure.  Because of the
comparison, we learn that the general structure inherits
some of the behavior of the canonical structure.
\joelprev{This recipe describes} Lindeberg's
approach~\cite{Lin22:Eine-Neue}
to the central limit theorem;
see~\cite[Sec.~2.2.4]{Tao12:Topics-Random}.
More recently, similar methods
have been directed at harder problems,
such as the universality of
local spectral statistics of a non-symmetric
random matrix~\cite{TV15:Random-Matrices}.
\joelprev{Our research is closest in spirit to the
papers~\cite{MOO10:Noise-Stability,Cha06:Generalization-Lindeberg,KM11:Applications-Lindeberg}.}

Let us imagine how we could apply this technique
to prove Theorem~\ref{thm:univ-rsv}.
Suppose that $E$ is a compact, convex set.
After performing standard discretization and smoothing steps,
we might invoke the Lindeberg exchange principle
to compare the restricted singular value of
the $d \times D$ random \joelprev{linear map} $\mtx{\Pi}$
with that of a $d \times D$ standard normal matrix
$\mtx{\Gamma}$:
$$
\Expect \smin(\mtx{\Pi}; E) \approx \Expect \smin(\mtx{\Gamma}; E).
$$
To evaluate the restricted minimum singular value of
a standard normal matrix, we can use the
Gaussian Minimax Theorem, as in~\cite[Cor.~1.1]{Gor88:Milmans-Inequality}:
$$
\Expect \smin(\mtx{\Gamma}; E) \gtrapprox \big( \coll{E}_d(E) \big)_+.
$$
Last, we can incorporate a convex duality argument,
as in~\cite{Sto13:Regularly-Random,OTH13:Squared-Error,TOH15:Gaussian-Min-Max,TAH15:High-Dimensional-Error},
to obtain the reverse inequality:
$$
\Expect \smin(\mtx{\Gamma}; E) \lessapprox \big( \coll{E}_d(E)  \big)_+.
$$
Unfortunately, this simple approach is not adequate.
Our argument ultimately follows a related pattern,
but we have to overcome a number of \joelprev{technical
obstacles}.

Let us explain the most serious issue and our mechanism
for handling it.  We would like to apply the Lindeberg
exchange principle to $\smin(\mtx{\Pi}; E)$ to replace
each entry of $\mtx{\Pi}$ with a standard normal variable.
The problem is that $E$ may contain a vector $\vct{t}_0$ with large
entries.  If we try to replace the columns of $\mtx{\Pi}$
associated with the large components of $\vct{t}_0$,
we incur an intolerably large error.
Moreover, for any given coordinate $j$, the set $E$ may
contain a vector $\vct{t}_j$ that takes a large value
in the distinguished coordinate $j$.  This fact seems to
foreclose the possibility of replacing any part of the matrix $\mtx{\Pi}$.

We address this challenge by dissecting the index set $E$. For each (small) subset $J$ of $\{1, \dots, D\}$, we define the set
$E_J$ of vectors in $E$ whose components in $J^c$ are small.  First,
we argue that the restricted singular value of a subset does not
differ much from the restricted singular value of the entire set:
$$
\Expect \smin(\mtx{\Pi}; E) \approx \min\nolimits_{J} \Expect \smin(\mtx{\Pi}; E_J).
$$
We may now use the Lindeberg method to make the comparison
$$
\Expect \smin(\mtx{\Pi}; E_J) \approx \Expect \smin(\mtx{\Psi}; E_J),
$$
where the matrix $\mtx{\Psi}$ is a hybrid of the form $\mtx{\Psi}_J = \mtx{\Pi}_J$
and $\mtx{\Psi}_{J^c} = \mtx{\Gamma}_{J^c}$.  That is,
we only replace the columns of $\mtx{\Pi}$ listed in the set $J^c$
with independent standard normal variables.
At this point, we need to compare the minimum singular value
of $\mtx{\Psi}$ restricted to the subset $E_J$
against the excess width of the full set:
$$
\Expect \smin(\mtx{\Psi}; E_J) \approx \big( \coll{E}_d(E) \big)_+.
$$
In other words, the remaining part \sametprev{$\mtx{\Psi}_J$} of the original
matrix plays a negligible role in determining the restricted
singular value.  We perform this estimate using the
Gaussian Minimax Theorem~\cite{Gor85:Some-Inequalities},
some convex duality arguments~\cite{TOH15:Gaussian-Min-Max},
and some coarse results from nonasymptotic
random matrix theory~\cite{Ver12:Introduction-Nonasymptotic,Tro15:Expected-Norm}.

See Part~\ref{part:rsv} for detailed statements
of the main technical results that support
Theorem~\ref{thm:univ-rsv} and the proofs
of these results.
Most of the ingredients are standard
tools from modern applied probability,
but we have combined them in a subtle way.
To make the long argument clearer,
we have attempted to present the proof
in a multi-resolution fashion where pieces
from each level are combined explicitly
at the level above.

\subsection{Theorem~\ref{thm:univ-rsv}: Extensions}

We expect that there are a number of avenues for
extending Theorem~\ref{thm:univ-rsv}.

\begin{itemize}
\item	Our analysis shows that the error in estimating
$\smin(\mtx{\Pi}; E)$ by the excess width $\coll{E}_d(E)$
is at most $o(\sqrt{D})$.  In case $\mtx{\Pi}$ is a standard
normal matrix, the error actually has constant order~\cite[Thm.~II.1]{TOH15:Gaussian-Min-Max}.
Can we improve the error bound for more general random \joelprev{linear map}s?

\item	A related question is whether the universality of $\smin(\mtx{\Pi}; E)$
persists when the excess width $\coll{E}_d(E)$ is very small in comparison
with $\sqrt{d}$.

\item	There is some empirical evidence that the restricted minimum singular
value may have universality properties for a class of random \joelprev{linear map}s
wider than Model~\ref{mod:p-mom-mtx}.

\item	Our proof can probably be adapted to show that other
types of functionals exhibit universal behavior.
For example, we can study
$$
\min_{\vct{t} \in E}\big(  \normsq{\mtx{\Pi}\vct{t}} + f(\vct{t}) \big)
$$
where $f$ is a convex, Lipschitz function.  This type of optimization problem plays
an important role in statistics and machine learning.
\end{itemize}

At the same time, we know that many natural functionals do \emph{not}
exhibit universal behavior.
In particular, consider the restricted \emph{maximum} singular value:
$$
\smax(\mtx{A}; T) := \sup_{\vct{t} \in T} \norm{\mtx{A}\vct{t}}.
$$
There are many sets $T$ where the maximum singular value $\smax(\cdot; T)$
does not take a universal value for \joelprev{linear map}s in Model~\ref{mod:p-mom-mtx};
for example, see Figure~\ref{fig:max-rsv-nonuniversal}.
It is an interesting challenge to determine the full scope of the universality
principle uncovered in Theorem~\ref{thm:univ-rsv}.

\newpage

\part{Applications}
\label{part:applications}

In this part of the paper, we outline some of the implications
of Theorems~\ref{thm:univ-embed} and~\ref{thm:univ-rsv} in the
information sciences.  We focus on problems involving least squares
and $\ell_1$ minimization to reduce the number of distinct calculations
that we need to perform; nevertheless the techniques apply broadly.
In an effort to make the presentation shorter and more intuitive,
we have also chosen to sacrifice some precision.

Section~\ref{sec:signal-recovery} discusses structured signal recovery
and, in particular, the compressed sensing problem.
Section~\ref{sec:coding} presents an application in coding theory.
Section~\ref{sec:rand-nla} describes how the results allow us to
analyze randomized algorithms for numerical linear algebra.
Finally, Section~\ref{sec:lasso-err} gives a universal formula for
the prediction error in a sparse regression problem.

\section{Recovery of Structured Signals from Random Measurements}
\label{sec:signal-recovery}

In signal processing, dimension reduction arises
as a mechanism for signal acquisition.  The idea is that the number of measurements
we need to recover a structured signal
is comparable with the number of degrees
of freedom in the signal, rather than the ambient dimension.
Given the measurements, we can reconstruct the signal using
nonlinear algorithms that take into account our prior
knowledge about the structure.
The field of compressed sensing is based on the idea
that \emph{random} linear measurements offer an efficient
way to acquire structured signals.  The practical
challenge in implementing this proposal
is to find technologies that can perform
random sampling.

Our universality laws have a significant implication for
compressed sensing.  We prove that the number of
random measurements required to recover a structured
signal does not have a strong dependence
on the distribution of the random measurements.
This result
is important because most applications offer
limited flexibility in the type of measurement
that we can take.  Our theory justifies the use
of a broad class of measurement ensembles.

\subsection{The Phase Transition for Sparse Signal Recovery}
\label{sec:l1-min}

In this section, we study the phase transition that appears
when we reconstruct a sparse signal from random measurements
via $\ell_1$ minimization.  For a large class of random
measurements, we prove that the distribution does not affect
the location of the phase transition.  This result resolves
a major open question~\cite{DT09:Observed-Universality}
in the theory of compressed sensing.

Let us give a more precise description of the problem.
Suppose that $\vct{x}_\star \in \R^n$ is a fixed vector with precisely
$s$ nonzero entries.  Let $\mtx{\Phi} : \R^n \to \R^m$ be a random linear map,
and suppose that we have access to the image $\vct{y} = \mtx{\Phi} \vct{x}_\star$.
We interpret this data as a list of $m$ samples of the unknown signal.

A standard
approach~\cite{CDS98:Atomic-Decomposition,DH01:Uncertainty-Principles,Tro06:Just-Relax,CRT06:Robust-Uncertainty,Don06:Compressed-Sensing}
for reconstructing the sparse signal $\vct{x}_{\star}$
from the data is to solve a convex optimization problem:
\begin{equation} \label{eqn:l1-min}
\underset{\vct{x} \in \R^n}{\quad\minimize} \pnorm{\ell_1}{\vct{x}}
\subjto \mtx{\Phi}\vct{x} = \vct{y}.
\end{equation}
Minimizing the $\ell_1$ norm promotes sparsity in the optimization variable
$\vct{x}$, and the constraint ensures that the virtual measurements
$\mtx{\Phi} \vct{x}$ are consistent with the observed data $\vct{y}$.
We say that the optimization problem~\eqref{eqn:l1-min} \term{succeeds}
if it has a unique solution $\widehat{\vct{x}}$ that coincides
with $\vct{x}_{\star}$.  Otherwise, it \term{fails}.
In this setting, the following challenge arises.

\begin{quotation}
\textbf{The Compressed Sensing Problem:}
Is the optimization problem~\eqref{eqn:l1-min} likely to succeed or
to fail to recover $\vct{x}_{\star}$
as a function of the sparsity $s$, the number $m$ of measurements,
the ambient dimension $n$, and the distribution of the random measurement matrix $\mtx{\Phi}$?
\end{quotation}

\noindent
This question has been a subject of inquiry in thousands of papers over
the last 10 years; see the books~\cite{EK12:Compressed-Sensing,FR13:Mathematical-Introduction}
for more background and references.

In a series of recent papers~\cite{DT09:Counting-Faces,Sto09:Various-Thresholds,CRPW12:Convex-Geometry,BLM15:Universality-Polytope,ALMT14:Living-Edge,Sto13:Regularly-Random,OTH13:Squared-Error,FM14:Corrupted-Sensing,GNP14:Gaussian-Phase},
the compressed sensing problem has been solved completely
in the case where the random measurement matrix $\mtx{\Phi}$ follows the standard
normal distribution. See Remark~\ref{rem:l1-prior} for a narrative of who did what when.
In brief, there exists a phase transition
function $\psi_{\ell_1} : [0, 1] \to [0,1]$ defined by
\begin{equation} \label{eqn:psi-l1}
\psi_{\ell_1}(\varrho) := \inf_{\tau \geq 0} \left( \varrho (1+\tau^2)
	+ (1-\varrho) 	\sqrt{\frac{2}{\pi}} \int_{\tau}^\infty (\zeta - \tau)^2 \econst^{-\zeta^2/2} \idiff{\zeta} \right).
\end{equation}
The phase transition function is increasing and convex,
and it satisfies $\psi_{\ell_1}(0) = 0$ and $\psi_{\ell_1}(1) = 1$.
When the measurement matrix $\mtx{\Phi}$ is standard normal,
\begin{equation} \label{eqn:l1-gauss}
\begin{aligned}
m/n &< \psi_{\ell_1}(s/n) - o(1)
\quad\text{implies}\quad
\text{\eqref{eqn:l1-min} fails with probability $1 - o(1)$;} \\
m/n &> \psi_{\ell_1}(s/n) + o(1)
\quad\text{implies}\quad
\text{\eqref{eqn:l1-min} succeeds with probability $1 - o(1)$}.
\end{aligned}
\end{equation}
In other words, as the number $m$ of measurements increases,
the probability of success jumps from zero to one at the
point $n \psi_{\ell_1}(s/n)$ over a range of $o(n)$ measurements.
The error terms in~\eqref{eqn:l1-gauss} can be improved substantially,
but this presentation suffices for our purposes.

Donoho \& Tanner~\cite{DT09:Observed-Universality} have performed an
extensive empirical investigation of the phase transition in the
$\ell_1$ minimization problem~\eqref{eqn:l1-min}.
Their work suggests that the Gaussian phase transition~\eqref{eqn:l1-gauss}
persists for many other types of random measurements.  See Figure~\ref{fig:l1-univ}
for a small illustration.  Our universality results provide
the first rigorous explanation of this phenomenon for measurement matrices
drawn from Model~\ref{mod:p-mom-mtx}.

\begin{figure}
\includegraphics[width=\textwidth]{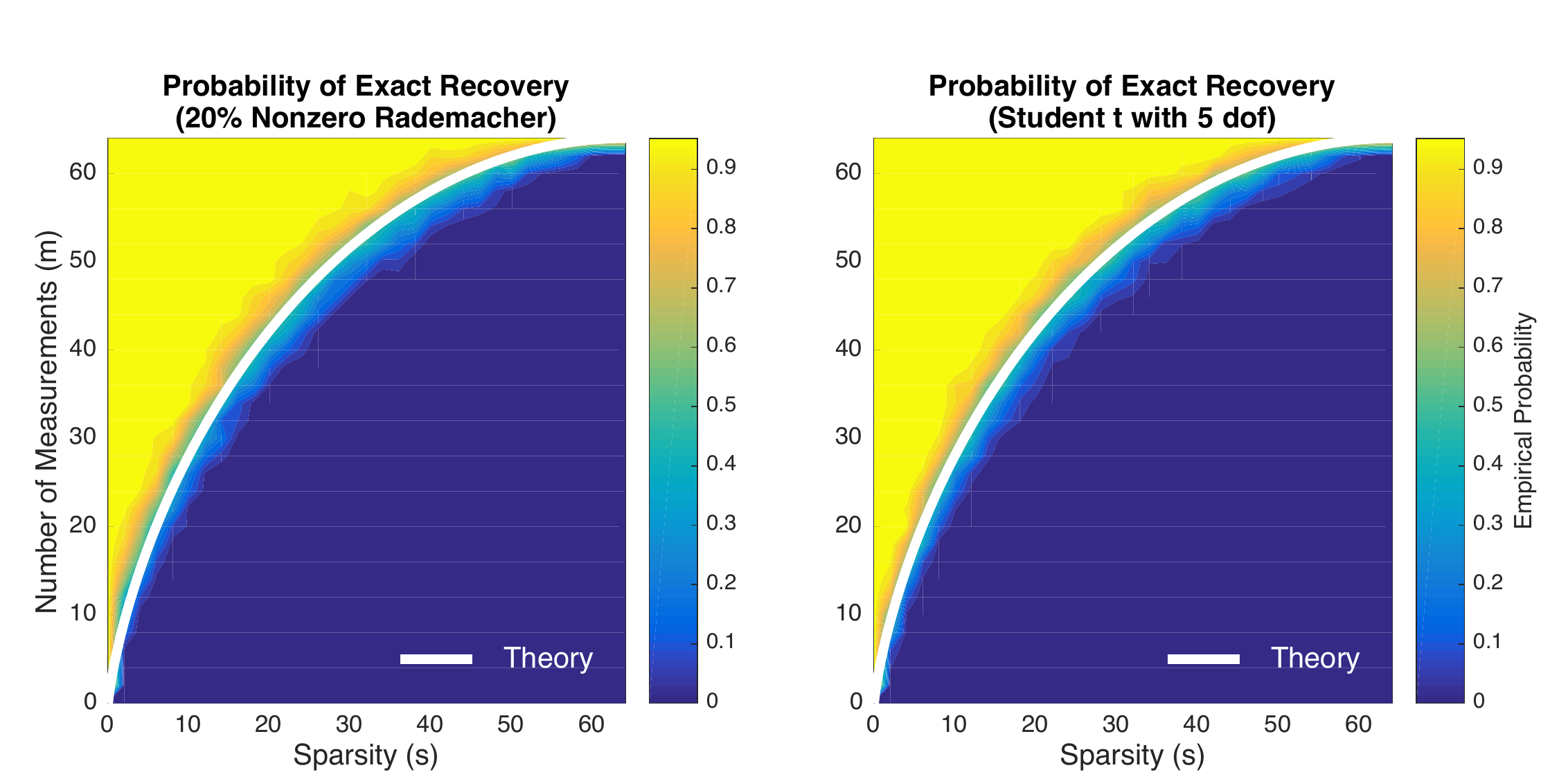}
\caption[Universality of the $\ell_1$ Recovery Phase Transition]
{\textsl{Universality of the $\ell_1$ Recovery Phase Transition.} \label{fig:l1-univ}
These plots depict the empirical probability that the $\ell_1$ minimization
problem~\eqref{eqn:l1-min} recovers a vector $\vct{x}_\star \in \R^{64}$ with $s$ nonzero entries
from a vector of random measurements $\vct{y} = \mtx{\Phi} \vct{x}_\star \in \R^m$.
The \textbf{heatmap} indicates the empirical probability, computed over 100 trials.
The \textbf{white curve} is the phase transition $n \psi_{\ell_1}(s/n)$
promised by Proposition~\ref{prop:l1-univ}.
\textsc{Left:} The random measurement matrix $\mtx{\Phi}$
is a sparse Rademacher matrix with an average of 20\% nonzero entries.
\textsc{Right:} The random measurement matrix $\mtx{\Phi}$ has independent Student $t_5$
entries.  See Section~\ref{sec:repro-research} for more details.}
\end{figure}

\begin{proposition}[Universality of $\ell_1$ Phase Transition] \label{prop:l1-univ}
Assume that

\begin{itemize}
\item	The ambient dimension $n$ and the number $m$ of measurements satisfy $m \leq n$.

\item	The vector $\vct{x}_\star \in \R^n$ has exactly $s$ nonzero entries.

\item	The random measurement matrix $\mtx{\Phi} : \R^n \to \R^m$ follows Model~\ref{mod:p-mom-mtx} with parameters $p$ and $\nu$.

\item	We observe the vector $\vct{y} = \mtx{\Phi} \vct{x}_\star$.

\end{itemize}

\noindent
Then, as the ambient dimension $n \to \infty$,
$$
\begin{aligned}
m/n &< \psi_{\ell_1}(s/n) - o(1)
\quad\text{implies}\quad
\text{\eqref{eqn:l1-min} fails with probability $1 - o(1)$;} \\
m/n &> \psi_{\ell_1}(s/n) + o(1)
\quad\text{implies}\quad
\text{\eqref{eqn:l1-min} succeeds with probability $1 - o(1)$.}
\end{aligned}
$$
The little-o suppresses constants that depend only on $p$ and $\nu$.
\end{proposition}

Proposition~\ref{prop:l1-univ} follows directly from our universality result,
Theorem~\ref{thm:univ-embed}, and the established calculation~\eqref{eqn:l1-gauss}
of the phase transition in the standard normal setting.

\begin{proof} The approach is quite standard.
Let $\Omega$ be the set of unit-norm descent directions of the $\ell_1$ norm
at the point $\vct{x}_\star$.  That is,
$$
\Omega := \big\{ \vct{u} \in \mathsf{S}^{n-1} :
	\pnorm{1}{\vct{x}_\star + \lambda \vct{u}} \leq \pnorm{1}{\vct{x}_\star}
	\text{ for some $\lambda > 0$} \big\}.
$$
The primal optimality condition for~\eqref{eqn:l1-min} demonstrates that
the reconstruction succeeds if and only if
$\Omega \cap \nullsp(\joelprev{\mtx{\Phi}}) = \emptyset$.
Since the $\ell_1$ norm is a closed convex function,
the set $\vct{\theta}^{-1}(\Omega) \cup \{ \vct{0} \}$
of all descent directions forms a closed, convex cone.
Therefore, the set $\Omega$ is closed and spherically convex.
It follows from Theorem~\ref{thm:univ-embed} that
the behavior of~\eqref{eqn:l1-min} undergoes a
phase transition at the statistical dimension $\delta(\Omega)$
for any random \joelprev{linear map} drawn from Model~\ref{mod:p-mom-mtx}.

The result~\cite[Prop.~4.5]{ALMT14:Living-Edge} contains the
first complete and accurate computation of the statistical
dimension of a descent cone of the $\ell_1$ norm:
\begin{equation} \label{eqn:sdim-l1}
n \psi_{\ell_1}(s/n) - O(\sqrt{n}) \leq \delta(\Omega) \leq n \psi_{\ell_1}(s/n).
\end{equation}
See also~\cite[Prop.~1]{FM14:Corrupted-Sensing}.  This fact completes the proof.
\end{proof}

Note that Proposition~\ref{prop:l1-univ} requires the
sparsity level $s$ to be proportional to the ambient
dimension $n$ before it provides any information.
By refining our argument, we can address the case when $s$ is
proportional to $n^{1-\eps}$ for a small number $\eps$ that
depends on the regularity of the random \joelprev{linear map}.
The empirical work~\cite{DT09:Observed-Universality}
of Donoho \& Tanner is unable to provide
statistical evidence for the universality hypothesis in the
regime where $s$ is very small.  It remains an open problem
to understand how rapidly $s/n$ can vanish before the
universality phenomenon fails.

The paper~\cite{DT09:Observed-Universality} also contains
numerical evidence that the $\ell_1$ phase transition persists for
random measurement systems that have more structure than Model~\ref{mod:p-mom-mtx}.
It remains an intriguing open question to understand these experiments.

\begin{remark}[Prior Work] \label{rem:l1-prior}
In early 2005, Donoho~\cite{Don06:High-Dimensional-Centrally}
and Donoho \& Tanner~\cite{DT06:Thresholds-Recovery}
observed that there is a phase transition
in the number of standard normal measurements needed to
reconstruct a sparse signal via $\ell_1$ minimization.
Using methods from integral geometry, they were able to
perform a heuristic computation of the location of the
phase transition function~\eqref{eqn:psi-l1}.
In subsequent work~\cite{DT09:Counting-Faces},
they proved that the transition~\eqref{eqn:l1-gauss}
is valid in some parameter regimes.  They later reported
extensive empirical evidence~\cite{DT09:Observed-Universality}
that the distribution of the random measurement map has little effect on the location of the phase transition.

In early 2005, Rudelson \& Vershynin~\cite{RV08:Sparse-Reconstruction}
proposed a different approach to studying $\ell_1$ minimization by adapting
results of Gordon~\cite{Gor88:Milmans-Inequality} that depend
on Gaussian process theory.  Stojnic~\cite{Sto09:Various-Thresholds}
refined this argument to obtain an empirically sharp success condition
for standard normal \joelprev{linear map}s, but his work did not establish
a matching failure condition.  Stojnic's calculations were
clarified and extended to other signal recovery problems in the
papers~\cite{OH10:New-Null-Space,CRPW12:Convex-Geometry,ALMT14:Living-Edge,FM14:Corrupted-Sensing}.

Bayati et al.~\cite{BLM15:Universality-Polytope} is the first paper
to rigorously demonstrate that the phase transition~\eqref{eqn:l1-gauss}
is valid for standard normal measurements.  The argument is based on
a state evolution framework for an iterative algorithm inspired by statistical physics.
This work also gives the striking conclusion that the
$\ell_1$ phase transition is universal over a class of random measurement
maps.  This result requires the measurement matrix to have independent, standardized,
subgaussian entries that are absolutely continuous with respect
to the Gaussian distribution.  As a consequence,
the paper~\cite{BLM15:Universality-Polytope} excludes
discrete and heavy-tailed models.  Furthermore, it only applies to $\ell_1$
minimization.  In contrast, Proposition~\ref{prop:l1-univ} and its
proof have a much wider compass.

The paper~\cite{ALMT14:Living-Edge} of
Amelunxen et al.~contains the first complete integral-geometric proof
of~\eqref{eqn:l1-gauss} for standard normal measurements.
This work is significant because it
established for the first time that phase transitions are
ubiquitous in signal reconstruction problems.  It also
forged the first links between the approach to phase
transitions based on integral geometry and those based on
Gaussian process theory.
The papers~\cite{MT14:Steiner-Formulas,GNP14:Gaussian-Phase}
build on these ideas to obtain precise estimates
of the success probability in~\eqref{eqn:l1-gauss}.

Subsequently, Stojnic~\cite{Sto13:Regularly-Random} refined
the Gaussian process methods to obtain more detailed
information about the behavior of the errors in noisy
variants of the $\ell_1$ minimization problem.  His work
has been \joelprev{extended} in a series~\cite{OTH13:Squared-Error,TOH15:Gaussian-Min-Max,TAH15:High-Dimensional-Error}
of papers by Abbasi, Oymak, Thrampoulidis, Hassibi,
and their collaborators.
Because of this research, we now have a very detailed
understanding of the behavior of convex signal recovery
methods with Gaussian measurements.

To our knowledge, the current paper is the first work
that extends the type of general analysis
in~\cite{ALMT14:Living-Edge,
OTH13:Squared-Error,TAH15:High-Dimensional-Error}
beyond the confines of the standard normal model.
\end{remark}

\subsection{Other Signal Recovery Problems}
\label{sec:other-signal-recovery}

The compressed sensing problem is the most prominent example
from a large class of related questions.  Our universality
results have implications for this entire class of problems.
We include a brief explanation.

Let $f : \R^n \to \R$ be a proper\footnote{A proper convex function takes at least one finite value
and never takes the value $-\infty$.}
convex function whose value increases with the ``complexity''
of its argument.  The $\ell_1$ norm is an example of
a complexity measure that is appropriate for sparse signals~\cite{CDS98:Atomic-Decomposition}.
Similarly, the Schatten 1-norm is a good complexity
measure for low-rank matrices~\cite{Faz02:Matrix-Rank}.

Let $\vct{x}_\star \in \R^n$ be a vector with ``low complexity.''
Draw a random linear map $\mtx{\Phi} : \R^n \to \R^m$,
and suppose we have access to $\vct{x}_\star$ only through
the measurements $\vct{y} = \mtx{\Phi} \vct{x}_\star$.  We
can attempt to reconstruct $\vct{x}_\star$ by solving the
convex optimization problem
\begin{equation*} \label{eqn:f-min}
\underset{\vct{x} \in \R^n}{\quad\minimize} f(\vct{x})
\subjto \mtx{\Phi}\vct{x} = \vct{y}.
\end{equation*}
In other words, we find the vector with minimum complexity that is consistent with the observed data.
We say that~\eqref{eqn:f-min} \term{succeeds} when it
has a unique optimal point that coincides with $\vct{x}_\star$;
otherwise, it \term{fails}.

The paper~\cite{ALMT14:Living-Edge} proves that there is a
phase transition in the behavior of~\eqref{eqn:f-min}
when $\mtx{\Phi}$ is standard normal.  Our universality
law, Theorem~\ref{thm:univ-embed}, allows us to extend
this result to include every random \joelprev{linear map} from
Model~\ref{mod:p-mom-mtx}.
Define the set $\Omega$ of unit-norm
descent directions of $f$ at the point $\vct{x}_\star$:
$$
\Omega := \big\{ \vct{u} \in \mathsf{S}^{n-1} :
f(\vct{x}_\star + \lambda\vct{u}) \leq f(\vct{x}_\star)
\text{ for some $\lambda > 0$} \big\}.
$$
Then, as the ambient dimension $n \to \infty$,
$$
\begin{aligned}
m &< \delta(\Omega) - o(n)
\quad\text{implies}\quad
\text{\eqref{eqn:f-min} fails with probability $1 - o(1)$}; \\
m &> \delta(\Omega) + o(n)
\quad\text{implies}\quad
\text{\eqref{eqn:f-min} succeeds with probability $1 - o(1)$}. \\
\end{aligned}
$$
In other words, there is a phase transition in the
behavior of~\eqref{eqn:f-min} when the number $m$ of measurements
equals the statistical dimension $\delta(\Omega)$ of the set of
descent directions of $f$ at the point $\vct{x}_\star$.
See the papers~\cite{CRPW12:Convex-Geometry,ALMT14:Living-Edge,FM14:Corrupted-Sensing} for
some general methods for computing the statistical
dimension of a descent cone.

\subsection{Complements}

We conclude this section with a few additional remarks about the scope
of our results on signal recovery.  First, we discuss some
geometric applications.  Second, we mention some other
signal processing problems that can be studied using the
same methods.

\subsubsection{Geometric Implications}
\label{sec:faces}

Proposition~\ref{prop:l1-univ} can be understood as
a statement about the facial structure of a random
projection of the $\ell_1$ ball.  Equivalently, it
provides information about the facial structure of
the convex hull of random points.

Suppose that $\mathsf{B}_1^n$ is the $n$-dimensional
$\ell_1$ ball, and fix an
$(s-1)$-dimensional face $F$ of $\mathsf{B}_1^n$.
Let $\mtx{\Phi} : \R^n \to \R^m$ be
a random \joelprev{linear map} from Model~\ref{mod:p-mom-mtx}.
We say that $\mtx{\Phi}$ \term{preserves}
the face $F$ when $\mtx{\Phi}(F)$ is
an $(s-1)$-dimensional face of
the projection $\mtx{\Phi}(\mathsf{B}_1^n)$.
We can reinterpret Proposition~\ref{prop:l1-univ}
as saying that
$$
\begin{aligned}
m/n &< \psi_{\ell_1}(s/n) - o(1)
\quad\text{implies}\quad
\text{$\mtx{\Phi}$ is unlikely to preserve $F$;} \\
m/n &> \psi_{\ell_1}(s/n) + o(1)
\quad\text{implies}\quad
\text{$\mtx{\Phi}$ is likely to preserve $F$.}
\end{aligned}
$$
The connection between $\ell_1$ minimization and the
facial structure of the $\ell_1$ ball was identified
in~\cite{Don06:Most-Large-II}; see also~\cite[Sec.~10.1.1]{ALMT14:Living-Edge}.

Here is another way of framing the same result.
Fix an index set $J \subset \{1, \dots, n\}$ with
cardinality $\#J = s$ and a vector $\vct{\eta} \in \R^n$
with $\pm 1$ entries.  Let $\vct{\phi}_1, \dots, \vct{\phi}_n \in \R^m$
be independent random vectors, drawn from Model~\ref{mod:p-mom-mtx},
and consider the absolute convex hull
$E := \conv\{ \pm \vct{\phi}_1, \dots, \pm \vct{\phi}_n \}$.
The question is whether the set $F := \conv\{ \eta_j \vct{\phi}_j : j \in J \}$
is an $(s-1)$-dimensional face of $E$.
We have the statements
$$
\begin{aligned}
s/n &< \psi_{\ell_1}^{-1}(m/n) - o(1)
\quad\text{implies}\quad
\text{$F$ is likely to be an $(s-1)$-dimensional face of $E$;} \\
s/n &> \psi_{\ell_1}^{-1}(m/n) + o(1)
\quad\text{implies}\quad
\text{$F$ is unlikely to be an $(s-1)$-dimensional face of $E$.}
\end{aligned}
$$
See the paper~\cite{DT09:Counting-Faces} for more
discussion of the connection between the facial
structure of polytopes and signal recovery.
Some universality results of this type
also appear in Bayati et al.~\cite{BLM15:Universality-Polytope}.

\subsubsection{Other Signal Processing Applications}

We often want to perform signal processing tasks on data after
reducing its dimension.  In this section, we have focused
on the problem of reconstructing a sparse signal from
random measurements.  Here are some related problems:

\begin{itemize}
\item	\textbf{Detection.} Does an observed signal consist of a template
corrupted with noise? Or is it just noise?

\item	\textbf{Classification.}  Does an observed signal belong to
class \textsf{A} or to class \textsf{B}?

\end{itemize}

\noindent
The literature contains many papers that propose methods for solving these problems
after dimension reduction; for example, see~\cite{DDW+07:Smashed-Filter}.
The existing analysis is either qualitative or it assumes that the
dimension reduction map is Gaussian.  Our universality laws can be used to
study the precise behavior of compressed detection and classification
with more general types of random \joelprev{linear map}s.  For brevity,
we omit the details.

\section{Decoding with Structured Errors}
\label{sec:coding}

One of the goals of coding theory is to design codes and decoding
algorithms that can correct gross errors in transmission.  In particular,
it is common that some proportion of the received symbols are corrupted.
In this section, we show that a large family of random codes can be
decoded in the presence of structured errors.  The number of errors
that we can correct is universal over this family.  This
result is valuable because it applies to random codebooks that are
closer to realistic coding schemes.

The result on random decoding can also be interpreted as a statement
about the behavior of the least-absolute-deviation (LAD) method for
regression.  We also discuss how our universality results apply
to a class of demixing problems.

\subsection{Decoding with Sparse Errors}

We work with a random linear code over the real field.
Consider a fixed message $\vct{x}_\star \in \R^m$.
Let $\mtx{\Phi} \in \R^{n \times m}$ be a random matrix,
which is called a \term{codebook} in this context.
Instead of transmitting the original message $\vct{x}_\star$,
we transmit the coded message $\mtx{\Phi} \vct{x}_\star$.
Suppose that we receive a version $\vct{y}$ of the message where
some number $s$ of the entries are corrupted.  That is,
$\vct{y} = \mtx{\Phi} \vct{x}_\star + \vct{z}_\star$
where the error vector $\vct{z}_\star$
has at most $s$ nonzero entries.
For simplicity, we assume that $\vct{z}_\star$
does not depend on the codebook $\mtx{\Phi}$.

In this setting, one can attempt to decode the message using
an $\ell_1$ minimization method~\cite{DH01:Uncertainty-Principles,CRTV05:Error-Correction,DT06:Thresholds-Recovery,MT14:Sharp-Recovery}.
We solve the optimization problem
\begin{equation} \label{eqn:l1-decode}
\underset{\vct{x} \in \R^m, \vct{z} \in \R^n}{\quad\minimize}
	\pnorm{\ell_1}{\vct{z}}
	\subjto \vct{y} = \mtx{\Phi} \vct{x} + \vct{z}.
\end{equation}
In other words, we search for a message $\vct{x}$ and a sparse corruption $\vct{z}$
that match the received data.  We say that the optimization~\eqref{eqn:l1-decode}
\term{succeeds} if it has a unique optimal point $(\widehat{\vct{x}}, \widehat{\vct{z}})$
that coincides with $(\vct{x}_\star, \vct{z}_\star)$; otherwise it \term{fails}.

\begin{figure}
\includegraphics[width=\textwidth]{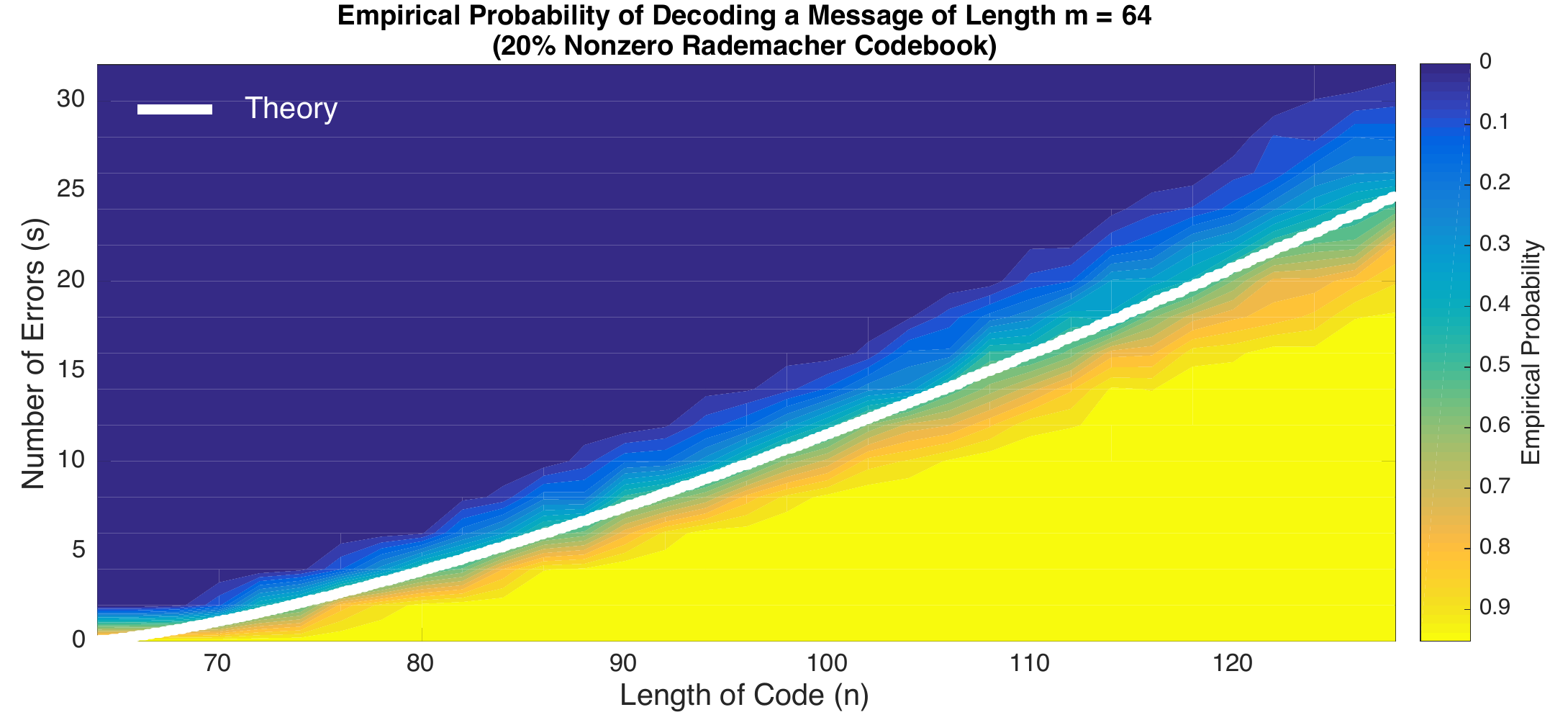}
\caption[Universality of the $\ell_1$ Decoding Phase Transition]
{\textsl{Universality of the $\ell_1$ Decoding Phase Transition.} \label{fig:l1-decode}
This plot shows the empirical probability that the $\ell_1$ method~\eqref{eqn:l1-decode}
decodes the message $\vct{x}_\star \in \R^{64}$
from the received message $\vct{y} = \mtx{\Phi} \vct{x}_\star + \vct{z}_\star \in \R^n$
where the corruption $\vct{z}_\star \in \R^n$ has exactly $s$ nonzero entries.
The codebook $\mtx{\Phi} \in \R^{n \times 64}$ is a sparse Rademacher matrix
with an average of 20\% nonzero entries.
The \textbf{heatmap} gives the empirical probability of correct decoding,
computed over 100 trials.
The \textbf{white curve} is the exact phase transition $n \psi_{\ell_1}^{-1}(1 - 64/n)$
promised by Proposition~\ref{prop:univ-decode}.
See Section~\ref{sec:repro-research} for more details.}
\end{figure}

The question is when the optimization problem~\eqref{eqn:l1-decode}
is effective at decoding the received transmission.  That is,
how many errors $s$ can we correct as a function of
the message length $m$ and the code length $n$?  The following result
gives a solution to this problem for any codebook drawn
from Model~\ref{mod:p-mom-mtx}.

\begin{proposition}[Universality of Sparse Error Correction] \label{prop:univ-decode}
Assume that

\begin{itemize}
\item	The message length $m$ and the code length $n$ satisfy $m \leq n$.

\item	The message  $\vct{x}_\star \in \R^m$ is arbitrary.

\item	\joelprev{The error vector $\vct{z}_\star \in \R^n$ has exactly $s$ nonzero entries,
where $s \leq (1 - \xi) n$ for some $\xi > 0$.}

\item	The random codebook $\mtx{\Phi} \in \R^{n \times m}$ follows Model~\ref{mod:p-mom-mtx}
with parameters $p$ and $\nu$.

\item	We observe the vector $\vct{y} = \mtx{\Phi} \vct{x}_\star + \vct{z}_\star$.

\end{itemize}

\noindent
Then, as the message length $m \to \infty$,
$$
\begin{aligned}
s/n &< \psi_{\ell_1}^{-1}\joelprev{(1 - m/n - o(1))}
\quad\text{implies}\quad
\quad\text{\eqref{eqn:l1-decode} succeeds with probability $1 - o(1)$;} \\
s/n &> \psi_{\ell_1}^{-1}\joelprev{(1 - m/n + o(1))}
\quad\text{implies}\quad
\quad\text{\eqref{eqn:l1-decode} fails with probability $1 - o(1)$.}
\end{aligned}
$$
The function $\psi_{\ell_1}$ is defined in~\eqref{eqn:psi-l1}.
The little-o suppresses constants that depend only on \joelprev{$\xi$} and $p$ and $\nu$.
\end{proposition}

This result is significant because it allows us to understand the
behavior of this method for a sparse, discrete codebook.
This type of code is somewhat closer to a practical coding mechanism
than the ultra-random codebooks that have been studied in the past;
see Remark~\ref{rem:decode-prior}.
Figure~\ref{fig:l1-decode} contains an illustration of how the theory
compares with the actual performance of this coding scheme.

\begin{proof} To analyze the decoding problem~\eqref{eqn:l1-decode},
we change variables:  $\vct{u} := \vct{x} - \vct{x}_\star$
and $\vct{v} := \vct{z} - \vct{z}_\star$.  We obtain the
equivalent optimization problem
\begin{equation} \label{eqn:l1-decode-uv}
\underset{\vct{u} \in \R^m, \vct{v} \in \R^n}{\quad\minimize}
	\pnorm{\ell_1}{\vct{z}_\star + \vct{v}}
	\subjto	\mtx{\Phi}\vct{u} = \vct{v}.
\end{equation}
The decoding procedure~\eqref{eqn:l1-decode}
succeeds if and only if the unique
optimal point $(\widehat{\vct{u}}, \widehat{\vct{v}})$
of the problem~\eqref{eqn:l1-decode-uv} is the pair $(\vct{0}, \vct{0})$.

Introduce the set $\Omega$ of unit-norm descent directions
of the $\ell_1$ norm at $\vct{z}_\star$:
$$
\Omega := \big\{ \vct{v} \in \mathsf{S}^{n-1} : \pnorm{\ell_1}{\vct{z}_\star + \lambda \vct{v}} \leq \pnorm{\ell_1}{\vct{z}_\star} \text{ for some $\lambda > 0$} \big\}.
$$
The primal optimality condition for~\eqref{eqn:l1-decode-uv}
shows that decoding succeeds if and only if
$\Omega \cap \range(\mtx{\Phi}) = \emptyset$.

\joelprev{
Let us compute the statistical dimension of $\Omega^{\polar}$:
\begin{equation} \label{eqn:sdim-l1-polar}
\delta( \Omega^\polar )
	= n - \delta( \Omega )
	= n - (n \psi_{\ell_1}(s/n) + o(n))
	= n( 1 - \psi_{\ell_1}(s/n) + o(1)).
\end{equation}
The first relation is the polar identity~\eqref{eqn:sdim-polar} for the statistical dimension,
and the value of the statistical dimension appears in~\eqref{eqn:sdim-l1}.
Since $s \leq (1 - \xi) n$, the properties of the function $\psi_{\ell_1}$
ensure that $\delta(\Omega^{\polar}) \geq \varrho n$ for some $\varrho > 0$.}

\joelprev{
First, we demonstrate that decoding fails when the number $s$ of errors is too large.
To do so, we must show that $\Omega \cap \range(\mtx{\Phi}) \neq \emptyset$.
By polarity~\cite[Thm.~(2.7)]{Kle55:Separation-Properties}, it suffices to check
that
$$
\Omega^\polar \cap \nullsp(\mtx{\Phi}^\adj) = \emptyset.
$$
With probability at least $1 - o(1)$, this relation follows from Theorem~\ref{thm:univ-embed}\eqref{eqn:univ-embed-succ},
provided that
$$
m > \delta(\Omega^{\polar}) + o(n)
	= n(1 - \psi_{\ell_1}(s/n)) + o(n).
$$
Finally, revert the inequality so that it is expressed in terms of $s$.
}

\joelprev{Last, we must check that decoding succeeds when the number $s$ of errors
is sufficiently small.  To do so, we must verify that $\Omega \cap \range(\mtx{\Phi}) = \emptyset$
with high probability.  This relation follows from ideas closely related to the proof of
Theorem~\ref{thm:univ-embed}, but it is not a direct consequence.
See Section~\ref{sec:decoding-success} for the details.}
\end{proof}

\begin{remark}[Least-Absolute-Deviation Regression] \label{rem:lad}
Proposition~\ref{prop:univ-decode} can also be interpreted as a statement
about the performance of the least-absolute deviation method for fitting
models with outliers. Suppose that we observe the data
$\vct{y} = \mtx{X} \vct{\beta}_\star + \vct{z}$.
Each of the $n$ rows of $\mtx{X}$ is interpreted
as a vector of $m$ measured variables for an independent subject
in an experiment.  The vector $\vct{\beta}_{\star}$ lists the
coefficients in the true linear model,
and the sparse vector $\vct{z}$ contains a small number $s$ of
arbitrary statistical errors.
The least-absolute deviation method fits a model by solving
$$
\underset{\vct{\beta} \in \R^m}{\quad\minimize}
	\pnorm{\ell_1}{\smash{\mtx{X} \vct{\beta} - \vct{y}}}.
$$
Proposition~\ref{prop:univ-decode} shows that the procedure identifies
the true model $\vct{\beta}_\star$ exactly, provided that the number $s$
of contaminated data points satisfies $s/n < \psi_{\ell_1}^{-1}(1 - m/n) - o(1)$.
\end{remark}

\begin{remark}[Prior Work] \label{rem:decode-prior}
The idea of using $\ell_1$ minimization for decoding in the presence of sparse errors
dates at least as far back as the paper~\cite{DH01:Uncertainty-Principles}.
This scheme received further attention in the work~\cite{CRTV05:Error-Correction}.
Later, Donoho \& Tanner~\cite{DT06:Thresholds-Recovery} applied phase transition
calculations to assess the precise performance of this coding scheme
for a standard normal codebook; the least-absolute-deviation interpretation
of this result appears in~\cite[Sec.~1.3]{DT09:Observed-Universality}.
The paper~\cite{MT14:Sharp-Recovery}
revisits the coding problem and develops a sharp analysis in the case where
the codebook is a random orthogonal matrix.  The current paper contains the
first precise result that extends to codebooks with more general
distributions.
\end{remark}

\subsection{Other Demixing Problems}

The decoding problem~\eqref{eqn:l1-decode} is an example of a convex
demixing problem~\cite{MT14:Sharp-Recovery,ALMT14:Living-Edge}.
Our universality results can be used to study other questions
of this species.

Let $f_0 : \R^n \to \R$ and $f_1 : \R^n \to \R$ be proper convex functions
that measure the complexity of a signal.  Suppose that $\vct{x}_{\star}^{0} \in \R^n$
and $\vct{x}_{\star}^{1} \in \R^n$ are signals with ``low complexity.''
Draw random matrices $\mtx{\Phi}_0 : \R^n \to \R^m$ and $\mtx{\Phi}_1 : \R^n \to \R^m$
from Model~\ref{mod:p-mom-mtx}.  Suppose that we observe the vector
$\vct{y} = \mtx{\Phi}_0 \vct{x}_{\star}^{1} + \mtx{\Phi}_1 \vct{x}_{\star}^{1}$.
We interpret the random matrices as known transformations of
the unknown signals.  For example, the matrices $\mtx{\Phi}_i$
might denote dictionaries in which the two components of $\vct{y}$ are sparse.

We can attempt to reconstruct the original signal pair
by solving
\begin{equation} \label{eqn:demix}
\underset{\vct{z}^0 \in \R^n, \vct{z}^1 \in \R^n}{\quad\minimize}
	\max\big\{ f_0(\vct{z}^0), f_1(\vct{z}^1) \big\}
	\subjto \mtx{\Phi}_0 \vct{z}^{0} + \mtx{\Phi}_1 \vct{z}^{1} = \vct{y}.
\end{equation}
In other words, we witness a superposition of two structured signals, and
we attempt to find the lowest complexity pair $(\vct{z}^0, \vct{z}^1)$
that reproduces the observed data.  The demixing problem \term{succeeds}
if it has a unique optimal point that equals $(\vct{x}_{\star}^{0}, \vct{x}_{\star}^{1})$.

To analyze this problem, we introduce two descent sets:
$$
\Omega_i := \big\{ \vct{u} \in \mathsf{S}^{n-1} :
	f_i(\vct{x}_\star^i + \lambda \vct{u}) \leq f_i(\vct{x}_\star^i)
	\text{ for some $\lambda > 0$} \big\}
	\quad\text{for $i = 1, 2$.}
$$
Up to scaling,
the descent directions of $\max\{f_0(\cdot), f_1(\cdot\cdot)\}$
at the pair $(\vct{x}_\star^0, \vct{x}_\star^1)$
coincide with the direct product $\Omega_0 \times \Omega_1$.
The statistical dimension of a direct product of two spherical sets satisfies
$\delta(\Omega_0 \times \Omega_1) = \delta(\Omega_0) + \delta(\Omega_1)$.
Therefore, Theorem~\ref{thm:univ-embed} demonstrates that
$$
\begin{aligned}
m &< \delta(\Omega_0) + \delta(\Omega_1) - o(n)
\quad\text{implies}\quad
\text{\eqref{eqn:demix} fails with probability $1 - o(1)$;} \\
m &> \delta(\Omega_0) + \delta(\Omega_1) + o(n)
\quad\text{implies}\quad
\text{\eqref{eqn:demix} succeeds with probability $1 - o(1)$.} \\
\end{aligned}
$$
In other words, the amount of information needed to extract
a pair of signals from the superposition equals the
total complexity of the two signals.  This result holds true
for a wide class of distributions on $\mtx{\Phi}_0$ and $\mtx{\Phi}_1$.

\section{Randomized Numerical Linear Algebra}
\label{sec:rand-nla}

Numerical linear algebra (NLA) is the study of computational
methods for problems in linear algebra, including the solution of linear systems,
spectral calculations, and matrix approximations.
Over the last 15 years, researchers have developed many new algorithms for
NLA that exploit randomness to perform these computations more efficiently.
See the surveys~\cite{Mah11:Randomized-Algorithms,HMT11:Finding-Structure,Woo14:Sketching-Tool} for an overview of this field.

In this section, we apply our universality techniques to obtain new
results on dimension reduction in randomized NLA.  This discussion shows
that a broad class of dimension reduction methods share the
same quantitative behavior.  Therefore, within some limits,
we can choose the random \joelprev{linear map} that is most computationally
appealing when we design numerical algorithms based on dimension reduction.

As an added bonus, the arguments here lead to a new proof of the Bai--Yin
limit for the minimum singular value of a random matrix drawn from
Model~\ref{mod:p-mom-mtx}.

\subsection{Subspace Embeddings}
\label{sec:subspace-embed}

In randomized NLA, one of the key primitives is a \term{subspace embedding}.
A subspace embedding is nothing more than a randomized \joelprev{linear map} that does
not annihilate any point in a fixed subspace.

\begin{definition}[Subspace Embedding]
Fix a natural number $k$, and let $L$ be an arbitrary $k$-dimensional subspace.
We say that a randomized \joelprev{linear map} $\mtx{\Pi} : \R^D \to \R^d$
is an \term{oblivious subspace embedding of order $k$} if
$$
\vct{0} \notin \mtx{\Pi}\big(L \cap \mathsf{S}^{D-1} \big)
\quad\text{with high probability.}
$$
The term ``oblivious'' indicates that the \joelprev{linear map} $\mtx{\Pi}$ is chosen without knowledge of the subspace $L$.
\end{definition}

\noindent
In the definition of a subspace embedding,
some authors include quantitative bounds on
the stability of the embedding.
These estimates are useful for analyzing certain algorithms,
but we have left them out because they are not essential.

A standard normal matrix provides an important theoretical and practical example
of a subspace embedding.

\begin{example}[Gaussian Subspace Embedding]
For any natural number $k$,
a standard normal matrix $\mtx{\Gamma} \in \R^{d \times D}$ is a subspace embedding
with probability one when the embedding dimension $d \geq k$.
In practice, it is preferable to select
the embedding dimension $d \geq k + 10$
to ensure that the restricted singular value
$\smin(\mtx{\Gamma}; L \cap \mathsf{S}^{D-1})$ is sufficiently positive, which makes
the embedding more stable.  See~\cite{HMT11:Finding-Structure} for more details.
\end{example}

A Gaussian subspace embedding has superb dimension reduction properties.
On the other hand, standard normal matrices are expensive to generate,
to store, and to perform arithmetic with.  Therefore, in most
randomized NLA algorithms, it is better to use subspace embeddings
that are discrete or sparse.

Our universality results demonstrate that, in a certain range of
parameters, every matrix that follows Model~\ref{mod:p-mom-mtx}
enjoys the same subspace embedding properties as a Gaussian matrix.

\begin{proposition}[Universality for Subspace Embedding] \label{prop:subspace-embed}
Suppose that

\begin{itemize}
\item	The ambient dimension $D$ is sufficiently large.

\item	The embedding dimension satisfies $d \leq D^{6/5}$.

\item	The random \joelprev{linear map} $\mtx{\Pi} : \R^D \to \R^d$ follows Model~\ref{mod:p-mom-mtx}
with parameters $p$ and $\nu$.
\end{itemize}

\noindent
Then, for each $k$-dimensional subspace $L$ of $\R^D$,
$$
\smin(\mtx{\Pi}; L \cap \mathsf{S}^{D-1})
	\geq \sqrt{d} - \sqrt{k} - o(\sqrt{D})
	\quad\text{with probability $1 - o(1)$.}
$$
In particular, $\mtx{\Pi}$ is a subspace embedding of order $k$ whenever $d \geq k + o(D)$.
In these expressions, the little-o suppresses constants that depend only on $p$ and $\nu$.
\end{proposition}

\begin{proof} Proposition~\ref{prop:subspace-embed} is a consequence of
Theorem~\ref{thm:univ-rsv}\eqref{eqn:univ-rsv-succ}
and~\eqref{eqn:excess-vs-sdim} because
$$
\coll{E}_d( L \cap \mathsf{S}^{D-1} )
	\geq \sqrt{d} - \sqrt{\delta(L \cap \mathsf{S}^{D-1})}
	= \sqrt{d} - \sqrt{\delta(L)}
	= \sqrt{d} - \sqrt{k}.
$$
The last identity holds because the $k$-dimensional
subspace $L$ has statistical dimension $k$.
We introduce the error term $o\big(\sqrt{D}\big)$ to make sure that
the stated result is only valid when the hypotheses of
Theorem~\ref{thm:univ-rsv} are in force.
\end{proof}

Note that Proposition~\ref{prop:subspace-embed}
applies to a sparse Rademacher \joelprev{linear map} with a fixed,
but arbitrarily small, proportion of nonzero entries.  This particular example has
received extensive attention in recent years~\cite{CW13:Low-Rank-Approximation,NN13:OSNAP-Faster,KN14:Sparser-Johnson-Lindenstrauss,BDN15:Toward-Unified},
although these works typically focus on the regime where
the subspace dimension $k$ is small and the sparsity
level of the random \joelprev{linear map} is a vanishing proportion
of the embedding dimension $d$.

\begin{remark}[Prior Work]
For the simple problem considered in Proposition~\ref{prop:subspace-embed},
much sharper results are available in the random matrix literature.
See the paper~\cite{KY14:Anisotropic-Local} for a recent analysis,
as well as additional references.
\end{remark}

\begin{remark}[The Bai--Yin Limit for the Minimum Singular Value]
One of the most important problems in random matrix theory is to
obtain bounds for the extreme singular values of a random matrix.
The Bai--Yin law~\cite{BY93:Limit-Smallest} gives a near-optimal
result in case the entries of the random matrix are independent
and standardized.  We can reproduce a slightly weaker version of the
Bai--Yin law for the minimum singular value by modifying
the proof of Proposition~\ref{prop:subspace-embed}.

Fix an aspect ratio $\varrho \in (0, 1)$.
For each natural number $d$, define $k := k(d) := \lfloor \varrho d \rfloor$.
Draw a $d \times k$ random matrix $\mtx{\Phi}^{(d)}$
from Model~\ref{mod:p-mom-mtx} with fixed parameters $p$ and $\nu$.
For each $\eps > 0$, we can apply Theorem~\ref{thm:univ-rsv}\eqref{eqn:univ-rsv-succ}
with $E = \mathsf{S}^{k - 1}$ to see that
$$
\Prob{ d^{-1/2} \smin\big(\mtx{\Phi}^{(d)}\big) \geq 1 - \sqrt{\varrho} - \eps }
	\to 1
	\quad\text{as $d \to \infty$}.
$$
Here, $\smin$ denotes the $k$th largest singular value of $\mtx{\Phi}^{(d)}$.

Under these assumptions, it is known~\cite{Yin86:Limiting-Spectral}
that the empirical distribution of the singular values
of $\mtx{\Phi}^{(d)}$ converges in probability to the Mar{\v c}enko--Pastur
density, whose support is the interval $1 \pm \sqrt{\varrho}$.
It follows that
$$
\Prob{ d^{-1/2} \smin\big( \mtx{\Phi}^{(d)} \big) \leq 1 - \sqrt{\varrho} + \eps }
	\to 1
\quad\text{as $d \to \infty$}.
$$
Therefore, we may conclude that
$$
d^{-1/2} \smin\big(\mtx{\Phi}^{(d)} \big) \to 1 - \sqrt{\varrho}
\quad\text{in probability.}
$$
In comparison, the Bai--Yin law~\cite[Thm.~2]{BY93:Limit-Smallest}
gives the same conclusion almost surely when the entries
of $\mtx{\Phi}^{(d)}$ have four finite moments.
See the recent paper~\cite{Tik15:Limit-Smallest}
for an optimal result.
\end{remark}

\subsection{Sketching and Least Squares}

In randomized NLA, one of the core applications of dimension
reduction is to solve over-determined least-squares problems,
perhaps with additional constraints.  This idea is attributed
to Sarl{\'o}s~\cite{Sar06:Improved-Approximation},
and it has been studied intensively over the last
decade;
see the surveys~\cite{Mah11:Randomized-Algorithms,Woo14:Sketching-Tool}
for more information.
In this section, we develop sharp bounds for the simplest
version of this approach.

Suppose that $\mtx{A}$ is a fixed $D \times n$ matrix
with full column rank.  Let $\vct{y} \in \R^D$ be a vector,
and consider the over-determined least-squares problem
\begin{equation} \label{eqn:ls-over}
\underset{\vct{x} \in \R^n}{\quad\minimize}
	\normsq{\smash{\mtx{A} \vct{x} - \vct{y}}}.
\end{equation}
This problem can be expensive to solve when $D \gg n$.
One remedy is to apply dimension reduction.
Draw a random \joelprev{linear map} $\mtx{\Pi} : \R^D \to \R^d$
from Model~\ref{mod:p-mom-mtx},
and consider the compressed problem
\begin{equation} \label{eqn:ls-redux}
\underset{\vct{x} \in \R^n}{\quad\minimize}
	\normsq{\smash{\mtx{\Pi}(\mtx{A} \vct{x} - \vct{y})}}.
\end{equation}
The question is how the quality of the solution of~\eqref{eqn:ls-redux}
depends on the embedding dimension $d$.  The following result provides
an optimal estimate.

\begin{proposition}[Randomized Least Squares: Error Bound] \label{prop:rand-ls}
Instate the prevailing notation.
Fix parameters $\lambda \in (0,1)$ and $\varrho \in (0,1)$
and $\iota \in (0, 1)$.
Assume that

\begin{itemize}
\item	The number $D$ of constraints is sufficiently large as a function of the parameters.
\item	The embedding dimension $d$ is comparable with the number $D$ of constraints: $\lambda D \leq d \leq D$.
\item	The embedding dimension $d$ is somewhat larger than the number $n$ of variables:
$d \geq (1 + \varrho) n$.
\end{itemize}

\noindent
With high probability, the solution $\widehat{\vct{x}}$ to the reduced least-squares
problem~\eqref{eqn:ls-redux} satisfies
\begin{equation} \label{eqn:ls-redux-err}
\frac{\normsq{ \mtx{A}(\widehat{\vct{x}} - \vct{x}_\star)}}
	{\normsq{\smash{\mtx{A} \vct{x}_{\star} - \vct{y}}}}
	\leq \frac{n + \iota d}{d - n}, 
\end{equation}
where $\vct{x}_{\star}$ is the solution to the original
least-squares problem~\eqref{eqn:ls-over}.  In particular,
$$
\frac{\normsq{ \smash{\mtx{A} \widehat{\vct{x}} - \vct{y}}}}
	{\normsq{\smash{\mtx{A} \vct{x}_{\star} - \vct{y}}}}
	\leq (1 + \iota) \frac{d}{d - n}
$$
\end{proposition}

In other words, the excess error
$\norm{\mtx{A}(\widehat{\vct{x}} - \vct{x}_{\star})}$
incurred in solving the compressed
least-squares problem~\eqref{eqn:ls-redux} is negligible as compared with
the optimal value of the least-squares problem~\eqref{eqn:ls-over}
if we choose the embedding dimension $d$ sufficiently large.
Proposition~\ref{prop:rand-ls} improves substantially on the
most recent work~\cite[Cor.~2(a)]{PW15:Randomized-Sketches},
both in terms of the error bound and in terms of the assumptions
on the randomized \joelprev{linear map}.

Let us remark that there is nothing special about ordinary least squares.
We can also solve least-squares problems
with a convex constraint set by dimension reduction.
For this class of problems, we can also obtain optimal
bounds by adapting the argument below.
For example, see the results in Section~\ref{sec:lasso-err}.

\begin{proof} Let $\vct{x}_{\star} \in \R^n$ be the solution to the original least-squares
problem~\eqref{eqn:ls-over}.  Define the optimal residual
$\vct{z}_{\star} := \vct{y} - \mtx{A}\vct{x}_{\star} \in \R^D$,
and recall that $\vct{z}_{\star}$ is orthogonal to $\range(\mtx{A})$.
Moreover,
$$
\min_{\vct{x} \in \R^n} \normsq{ \smash{\mtx{A} \vct{x} - \vct{y}} }
	= \normsq{\smash{\mtx{A}\vct{x}_{\star} - \vct{y} }}
	= \normsq{ \vct{z}_{\star} }.
$$
Without loss of generality, we may scale the problem so that $\normsq{\vct{z}_{\star}}=1$.

Next, change variables.  Define $\vct{w} := \mtx{A}( \vct{x} - \vct{x}_{\star})$,
and note that $\vct{w}$ is orthogonal to $\vct{z}_{\star}$.
We can write the reduced least-squares problem~\eqref{eqn:ls-redux} as
\begin{equation} \label{eqn:ls-redux-w}
\underset{\vct{w} \in \range(\mtx{A})}{\quad\minimize}
	\norm{ \mtx{\Pi}(\vct{w} - \vct{z}_{\star}) }.
\end{equation}
When dimension reduction is effective,
we expect the solution $\widehat{\vct{w}}$ to~\eqref{eqn:ls-redux-w}
to be close to zero.

Since $d \geq (1 + \varrho) n$, we can use the fact
(Proposition~\ref{prop:subspace-embed})
that $\mtx{\Pi}$ is a subspace embedding to obtain
an a priori bound $\norm{\widehat{\vct{w}}} \leq R_{\infty}$
that holds with high probability.  The number $R_{\infty}$
is a constant that depends on nothing but $\eps$.  We only need
this observation to ensure that we are optimizing
over a compact set with constant radius, so we omit the details.

Next, we invoke Theorem~\ref{thm:univ-rsv}\eqref{eqn:univ-rsv-fail}
to bound the optimal value of the reduced least-squares problem~\eqref{eqn:ls-redux-w}.
Define the compact, convex set
$$
E := \{ \vct{w} \in \range(\mtx{A}) : \norm{\vct{w}} \leq R_{\infty} \}.
$$
Let $\eps > 0$ be a parameter that will depend on the parameter $\iota$.
With high probability,
$$
\min_{\vct{w} \in \range(\mtx{A})} \norm{\mtx{\Pi}(\vct{w} - \vct{z}_{\star})}
	= \min_{\vct{w} \in E} \norm{\mtx{\Pi}(\vct{w} - \vct{z}_{\star})}
	\leq (1 + \eps)\, \coll{E}_d(E - \vct{z}_{\star}).
$$
By direct calculation, we can bound the excess width above.
Let $\mtx{P}$ be the orthogonal projector onto the range of $\mtx{A}$,
and observe that
$$
\begin{aligned}
\coll{E}_d(E - \vct{z}_{\star})
	&= \Expect \inf_{\vct{w} \in E} \big( \sqrt{d} \norm{\vct{w} - \vct{z}_{\star}}
		+ \vct{g} \cdot (\vct{w} - \vct{z}_{\star}) \big) \\
	&= \Expect \inf_{\vct{w} \in E} \big( \sqrt{d} (\normsq{\vct{w}} + 1)^{1/2}
		+\vct{g} \cdot \vct{w} \big) \\
	&= \Expect \inf_{\vct{w} \in E} \big( \sqrt{d} (\normsq{\vct{w}} + 1)^{1/2}
		- \norm{\smash{\mtx{P}\vct{g}}} \norm{\vct{w}} \big) \\
	&\leq \inf_{0 \leq \alpha \leq R_{\infty}} \big(\sqrt{d} (\alpha^2 + 1)^{1/2}
		- \sqrt{n} \alpha \big) \\
	&= \sqrt{ d - n }. \end{aligned}
$$
The first line is Definition~\ref{def:excess-width}, of the excess width.
Next, simplify via the orthogonality of $\vct{w}$ and $\vct{z}_\star$
and the scaling $\normsq{\vct{z}_{\star}} = 1$.  Use translation invariance of
the infimum to remove $\vct{z}_\star$ from the Gaussian term.  Apply
Jensen's inequality to draw the expectation inside the infimum,
and note that $\Expect \norm{\smash{\mtx{P}\vct{g}}} \leq \sqrt{n}$ because
$\mathrm{rank}(\mtx{P}) = n$.  Then make the change of variables
$\alpha = \norm{\vct{w}}$, and solve the scalar convex optimization problem.
The infimum occurs at the value
$$
\alpha_{\rm opt}^2 := \frac{n}{d - n}.
$$
Since $d \geq (1 + \varrho) n$, we may be confident that $\alpha_{\rm opt} < R_{\infty}$.

We have shown that, with high probability,
\begin{equation} \label{eqn:ls-redux-w-val}
\min_{\vct{w} \in \range(\mtx{A})} \norm{\mtx{\Pi}(\vct{w} - \vct{z}_{\star})}
	\leq (1 + \eps) \sqrt{d - n}.
\end{equation}
Furthermore, we have evidence that the norm of the minimizer
$\norm{\widehat{\vct{w}}} \approx \alpha_{\rm opt}$.
To prove the main result, we compute the value of the
optimization problem~\eqref{eqn:ls-redux-w}
restricted to points with $\norm{\vct{w}} \geq \alpha_{\rm opt} \sqrt{1 + \iota_0}$,
where $\iota_0$ is a small positive number to be chosen later.
Then we verify that the optimal value of the restricted problem
is usually larger than the bound~\eqref{eqn:ls-redux-w-val} for
the optimal value of~\eqref{eqn:ls-redux}. This argument implies that
$\norm{\widehat{\vct{w}}} < \alpha_{\rm opt} \sqrt{1 + \iota_0}$
with high probability.

To that end, define $R_+ := \alpha_{\rm opt} \sqrt{1 + \iota_0}$,
and introduce the compact set
$$
E_+ := \{\vct{w} \in \range(\mtx{A}) : R_+ \leq \norm{\vct{w}} \leq R_{\infty} \}.
$$
Theorem~\ref{thm:univ-rsv}\eqref{eqn:univ-rsv-succ} shows that,
with high probability,
\begin{equation} \label{eqn:ls-redux-w-E+}
\min_{\vct{w} \in E_+} \norm{\mtx{\Pi}(\vct{w} - \vct{z}_{\star})}
	\geq (1-\eps) \, \coll{E}_d(E_+ - \vct{z}_{\star}).
\end{equation}
Calculating the excess width as before,
\begin{equation} \label{eqn:ls-redux-w-restrict}
\begin{aligned}
\coll{E}_d(E_+ - \vct{z}_{\star})
	&= \Expect \inf_{R_+ \leq \alpha \leq R_{\infty}}
	\big( \sqrt{d} (\alpha^2 + 1)^{1/2} - \norm{\smash{\mtx{P}\vct{g}}} \alpha \big) \\
	&\geq \inf_{R_+ \leq \alpha \leq R_{\infty}}
	\big( \sqrt{d} (\alpha^2 + 1)^{1/2} - \sqrt{n} \alpha \big)
		- R_{\infty} \Expect \big(\norm{\smash{\mtx{P}\vct{g}}} - \sqrt{n} \big)_+ \\
	&\geq \big( \sqrt{d} (R_+^2 + 1)^{1/2} - \sqrt{n} R_+ \big)
		- R_{\infty}. \end{aligned}	
\end{equation}
Add and subtract $\sqrt{n} \alpha$ to reach the second line,
and use the fact that $\Expect \norm{\smash{\mtx{P}\vct{g}}} \leq \sqrt{n}$.
Then apply the Gaussian variance inequality, Fact~\ref{fact:gauss-variance},
to bound the expectation by one.  The infimum occurs at $R_+$ because the objective
is convex and $R_+$ exceeds the unconstrained minimizer $\alpha_{\rm opt}$.

Next, we simplify the expression involving $R_+$.  Setting $\iota_0 := \iota d/n$,
we find that
\begin{equation} \label{eqn:ls-redux-R+-simp}
\begin{aligned}
\sqrt{d} (R_+^2 + 1)^{1/2} - \sqrt{n} R_+
	&= \frac{d\sqrt{1 + \iota_0 n/d} - n\sqrt{1 + \iota_0}}{\sqrt{d-n}} \\
	&= \sqrt{\frac{1 + \iota}{d-n}} \left( d - n \sqrt{1 + \frac{\iota(d/n-1)}{1 + \iota}} \right) \\
	&\geq \sqrt{\frac{1 + \iota}{d-n}} \left( d - n - \frac{\iota(d-n)}{2(1 + \iota)} \right)
	= \frac{1 + \iota/2}{\sqrt{1 + \iota}} \sqrt{d - n}.
\end{aligned}
\end{equation}
The inequality follows from the linear upper bound for the square root at one.

Combine~\eqref{eqn:ls-redux-w-E+},~\eqref{eqn:ls-redux-w-restrict}, and~\eqref{eqn:ls-redux-R+-simp}
to arrive at
$$
\min_{\vct{w} \in E_+} \norm{\mtx{\Pi}(\vct{w} - \vct{z}_{\star})}
	\geq \frac{(1- \eps)(1 + \iota/2)}{\sqrt{1 + \iota}} \sqrt{d - n}.  
$$
Comparing~\eqref{eqn:ls-redux-w-val} with the last display,
we discover that the choice $\eps := \cnst{c} \iota^2$
is sufficient to ensure that, with high probability,
$$
\begin{aligned}
\min_{\vct{w} \in E_+} \norm{\mtx{\Pi}(\vct{w} - \vct{z}_{\star})}
	> \min_{\vct{w} \in E} \norm{\mtx{\Pi}(\vct{w} - \vct{z}_{\star})}.
\end{aligned}
$$
It follows that the minimum of~\eqref{eqn:ls-redux-w}
usually occurs on the set $E \setminus E_+$.
We determine that
$$
\normsq{\widehat{\vct{w}}} \leq (1 + \iota d/n) \, \alpha_{\rm opt}^2.
$$
Reinterpret this inequality to obtain the stated result~\eqref{eqn:ls-redux-err}.

We can obtain a matching lower bound for $\norm{\widehat{\vct{w}}}$ by considering
the set of vectors
$E_- := \{ \vct{w} \in \range(\mtx{A}) : \norm{\vct{w}} \leq R_- \}$
where $R_- := \alpha_{\rm opt} \sqrt{ 1 - \iota_0 }$.  We omit the details.
\end{proof}

\begin{remark}[Prior Work]
The idea of using random \joelprev{linear map}s to accelerate the solution
of least-squares problems appears in the work of
Sarl{\'o}s~\cite{Sar06:Improved-Approximation}.
This approach has been extended and refined in
the literature on randomized NLA;
see the surveys~\cite{Mah11:Randomized-Algorithms,Woo14:Sketching-Tool}
for an overview.  Most of this research is concerned
with randomized \joelprev{linear map}s that have favorable
computational properties, but the results are much
less precise than Proposition~\ref{prop:rand-ls}.
Recently, Pilanci \& Wainright~\cite{PW15:Randomized-Sketches}
have offered a more refined analysis of randomized
dimension reduction for constrained least-squares problems,
but it still falls short of describing the actual
behavior of these methods.  Parts of the argument
here is adapted from the work
of Oymak, Thrampoulidis, and Hassibi~\cite{OTH13:Squared-Error,TOH15:Gaussian-Min-Max}.
\end{remark}

\section{The Prediction Error for LASSO}
\label{sec:lasso-err}

Universality results have always played an important role
in statistics.  The most fundamental example is the law
of large numbers, which justifies the use of the sample average
to estimate the mean of a general distribution.  Similarly,
the central limit theorem permits us to build a confidence
interval for the mean of a distribution.

High-dimensional statistics relies on more sophisticated methods,
often based on optimization, to estimate population parameters.
In particular, applied statisticians frequently employ the
LASSO estimator~\cite{Tib96:Regression-Shrinkage}
to perform regression and variable selection in linear models. It is only recently that researchers have developed
theory~\joelprev{\cite{KM11:Applications-Lindeberg,BM12:LASSO-Risk,OTH13:Squared-Error,DJM13:Accurate-Prediction,JM14:Hypothesis-Testing,TAH15:High-Dimensional-Error}}
that can predict the precise behavior of the LASSO when the data are assumed to be
Gaussian.
It is a critical methodological challenge to
develop universality results that expand
the range of models in which we can make
confident assertions about the performance
of the LASSO.

In this section, we prove the first \joelprev{general} universality result
for the prediction error using a LASSO model estimate.
This theory offers a justification for using
a LASSO model to make predictions when the data
derives from a sparse model.  We expect that
further developments in this direction will play
an important role in applied statistics.

\subsection{The Sparse Linear Model and the LASSO}

The LASSO is designed to perform  simultaneous regression and variable
selection in a linear model.  Let us present a simple
statistical model in which to study the behavior
of the LASSO estimator.  Suppose that
the random variable $Y$ takes the form
\begin{equation} \label{eqn:lasso-model}
Y = \vct{x} \cdot \vct{\beta}_{\star} + \sigma Z
\end{equation}
where

\begin{itemize}

\item	The deterministic vector $\vct{\beta}_{\star} \in \R^p$ of model parameters has at most $s$ nonzero entries.

\item	The random vector $\vct{x} \in \R^p$ of predictor variables is drawn from Model~\ref{mod:p-mom-mtx}. 
\item	The noise variance $\sigma^2 > 0$ is known.

\item	The statistical error $Z$ is drawn from Model~\ref{mod:p-mom-mtx}, independent of $\vct{x}$.
\end{itemize}

\noindent
We interpret $\vct{x} = (X_1, \dots, X_p)$ as a family of predictor variables that we want to use to predict
the value of the response variable $Y$.  Only the variables $X_j$ where $(\vct{\beta}_\star)_j \neq 0$
are relevant to the prediction, while the others are confounding.  The observed value
$Y$ of the response is contaminated with a statistical error $\sigma Z$.
These assumptions are idealized,
but let us emphasize that our analysis holds
even when the predictors and the noise are heavy-tailed.

Suppose that we observe independent pairs $(\vct{x}_1, Y_1), \dots, (\vct{x}_n, Y_n)$
drawn from the model above, and let $\vct{z} \in \R^n$ be the unknown vector
of statistical errors.  One of the goals of sparse regression
is to use this data to construct an estimate
$\widehat{\vct{\beta}} \in \R^p$ of the model coefficients
so that we can predict future responses.
That is, given a fresh random vector $\vct{x}_0$ of predictor
variables, we can predict the (unknown) response
$Y_0 = \vct{x}_0 \cdot \vct{\beta}_{\star} + \sigma Z_0$
using the linear estimate
$$
\widehat{Y}_0 := \vct{x}_0 \cdot \widehat{\vct{\beta}}.
$$
We want to control the \term{mean squared error in prediction},
which is defined as
\begin{equation} \label{eqn:lasso-msep}
\mathsf{MSEP} := \Expect\big[ \abssq{ \smash{\widehat{Y}_0 - Y_0} }
	\, \big\vert \, \mtx{X}, \vct{z} \big].
\end{equation}
Using the statistical model~\eqref{eqn:lasso-model}, it is easy to verify that
\begin{equation} \label{eqn:msep-calc}
\mathsf{MSEP}
	= \Expect\big[ \abssq{ \smash{\vct{x}_0 \cdot(\widehat{\vct{\beta}} - \vct{\beta}_{\star})
		+ \sigma Z_0 }} \,\vert\, \mtx{X}, \vct{z} \big]
	= \normsq{\smash{\widehat{\vct{\beta}} - \vct{\beta}_{\star}}} + \sigma^2.
\end{equation}
In other words, the prediction error is controlled by the squared error
in estimating the model coefficients.

The LASSO uses convex optimization to produce an estimate $\widehat{\vct{\beta}}$
of the model coefficients.  The estimator is chosen arbitrarily from the
set of solutions to the problem
\begin{equation} \label{eqn:lasso}
\underset{\vct{\beta} \in \R^p}{\quad\minimize}
	\normsq{\smash{\mtx{X}\vct{\beta} - \vct{y}}}
	\subjto \pnorm{\ell_1}{\smash{\vct{\beta}}} \leq \pnorm{\ell_1}{\smash{\vct{\beta}_{\star}}}.
\end{equation}
In this formula, the rows of the $n \times p$ matrix $\mtx{X}$ are the
observed predictor vectors $\vct{x}_i$.
The entries of the vector $\vct{y} \in \R^n$
are the measured responses.  For simplicity,
we also assume that we have the exact side information
$\pnorm{\ell_1}{\smash{\vct{\beta}_{\star}}}$.

We can prove the following result on the squared error
in the LASSO estimate of the sparse coefficient model.
This is the first universal statement that offers a
precise analysis for this class of statistical models.

\begin{proposition}[Universality of LASSO Prediction Error] \label{prop:lasso-err}
Instate the prevailing notation. Choose parameters $\lambda \in (0,1)$ and $\varrho \in (0,1)$
and $\iota \in (0,1)$.
Assume that

\begin{itemize}
\item	The number $n$ of subjects is sufficiently large as a function of the parameters.
\item	The number $p$ of predictors satisfies $\lambda p \leq n \leq p^{6/5}$.

\item	The number $n$ of subjects satisfies $n \geq (1 + \varrho) \, p \psi_{\ell_1}(s/p)$. 

\end{itemize}

\noindent
With high probability over the observed data,
the mean squared error in prediction~\eqref{eqn:lasso-msep} satisfies
\begin{equation} \label{eqn:msep-theory}
\mathsf{MSEP} \leq (1 + \iota) \frac{ \sigma^2 n }{n - p \psi_{\ell_1}(s/p)}.
\end{equation}
The function $\psi_{\ell_1}$ is defined in~\eqref{eqn:psi-l1}.
Furthermore, the bound~\eqref{eqn:msep-theory} is sharp when $n \gg p \psi_{\ell_1}(s/p)$
or when $\sigma^2 \to 0$. 
\end{proposition}

Proposition~\ref{prop:lasso-err} gives an upper bound for the $\mathsf{MSEP}$, which matches
the low-noise limit ($\sigma \to 0$) obtained in the Gaussian case~\cite{OTH13:Squared-Error}.
See Figure~\ref{fig:lasso-msep} for a numerical experiment that confirms our
theoretical predictions.  The proof of the result appears below in Section~\ref{prop:lasso-err-pf}.

The assumptions in Proposition~\ref{prop:lasso-err} are somewhat restrictive,
in that the number $n$ of subjects must be roughly comparable with the number
$p$ of predictors.  This condition can probably be relaxed, but the error bound
in Theorem~\ref{thm:univ-rsv} does not allow
for a stronger conclusion.  The argument can also
be extended to give even more precise formulas for the $\mathsf{MSEP}$
under the same assumptions.  We also note that there
is nothing special about the $\ell_1$ constraint in~\eqref{eqn:lasso};
similar results are valid for many other convex constraints.

\begin{figure}
\includegraphics[width=\textwidth]{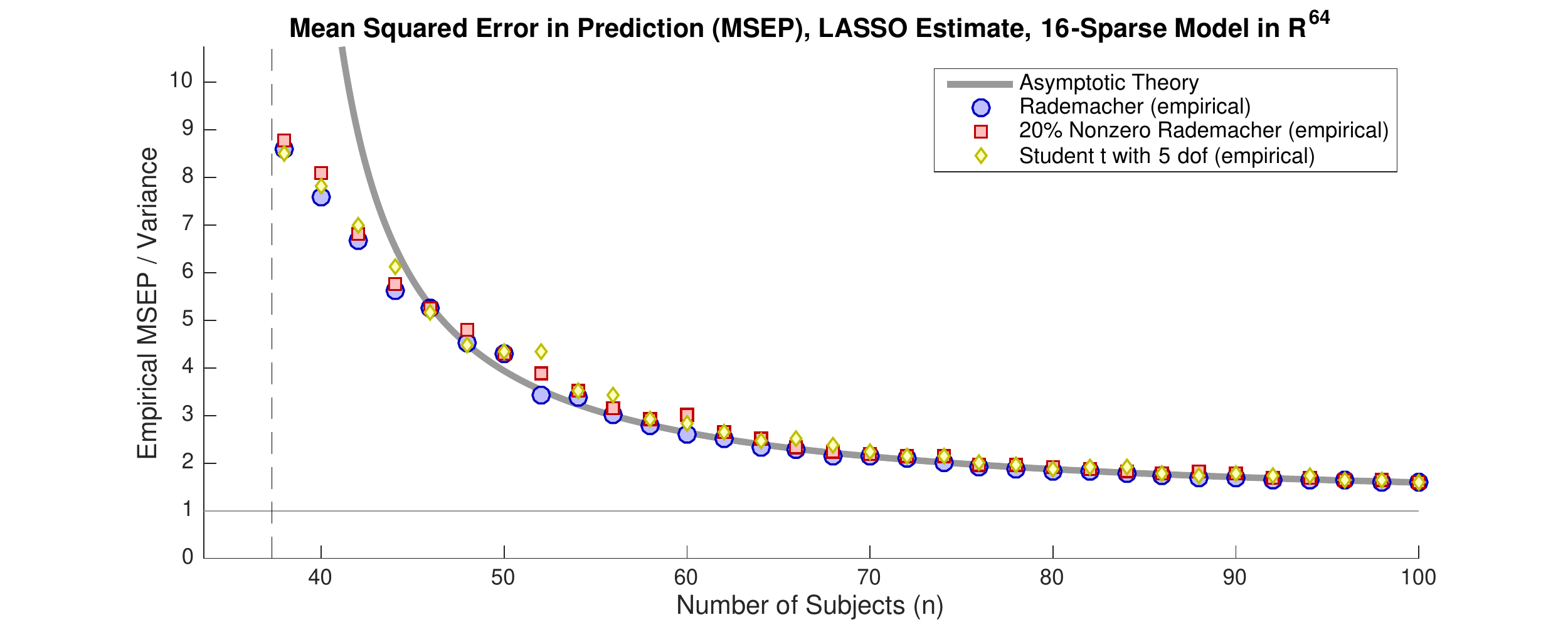}
\caption[Universality of the LASSO Prediction Error]
{\textsl{Universality of the LASSO Prediction Error.} \label{fig:lasso-msep}
This plot shows the $\mathsf{MSEP}$~\eqref{eqn:lasso-msep}
obtained with the LASSO estimator~\eqref{eqn:lasso},
averaged over design matrices $\mtx{X}$ and statistical errors $\vct{z}$.
In the linear model~\eqref{eqn:lasso-model},
the number of predictors $p = 64$; the number of active predictors $s = 16$;
each nonzero coefficient $(\vct{\beta}_{\star})_j$ in the model has unit
magnitude; the statistical error $Z$ is Gaussian;
the variance $\sigma^2 = 1$;
and the number $n$ of subjects varies.
The \textbf{dashed line} marks the location of the phase transition for
the number of subjects required to identify the model exactly
when the noise variance is zero.
The \textbf{gray curve} delineates the asymptotic upper bound $n/(n - p \psi_{\ell_1}(s/p))$
for the normalized $\mathsf{MSEP}$ from Proposition~\ref{prop:lasso-err}.
The \textbf{markers} give an empirical estimate (over 100 trials)
for the $\mathsf{MSEP}$ when the design matrix $\mtx{X}$ has the
specified distribution.}
\end{figure}

\subsection{Proof of Proposition~\ref{prop:lasso-err}}
\label{prop:lasso-err-pf}

Without loss of generality, we may assume that the statistical model is scaled
so the noise level $\sigma = 1$.  Define the sublevel set $E$ of the
$\ell_1$ norm at the true parameter vector $\vct{\beta}_{\star}$:
$$
E := \big\{ \vct{u} \in \R^p : \pnorm{\ell_1}{\smash{\vct{\beta}_{\star} + \vct{u}}} \leq \pnorm{\ell_1}{\smash{\vct{\beta}_{\star}}} \big\}.
$$
Note that $E$ is a compact, convex set that contains the origin.
Define the $n \times (p+1)$ random matrix
$\mtx{\Phi} := \begin{bmatrix} \mtx{X} & \vct{z} \end{bmatrix}$,
and note that $\mtx{\Phi}$ also follows Model~\ref{mod:p-mom-mtx}.
Making the change of variables $\vct{u} := \vct{\beta} - \vct{\beta}_{\star}$,
we can rewrite the LASSO problem~\eqref{eqn:lasso} in the form
\begin{equation} \label{eqn:lasso-pf}
\underset{\vct{u} \in E}{\quad\minimize}
	\norm{ \mtx{\Phi} \begin{bmatrix} \vct{u} \\ - 1 \end{bmatrix} }.
\end{equation}
This expression depends on the assumption that $\sigma = 1$.
Let $\widehat{\vct{u}}$ be an optimal point of the problem~\eqref{eqn:lasso-pf}.
Referring to~\eqref{eqn:msep-calc}, we see that a formula for
$\normsq{\widehat{\vct{u}}}$ leads to a formula for the $\mathsf{MSEP}$.

It will be helpful to introduce some additional sets.  For a parameter $\alpha > 0$, define the compact (but typically nonconvex) set
$$
E_{\alpha} := \big\{ \vct{u} \in E : \norm{\vct{u}} = \alpha \big\}.
$$
We also define the compact and convex set
$$
E_{\leq \alpha} := \big\{ \vct{u} \in E : \norm{\vct{u}} \leq \alpha \big\}.
$$
Observe that $E_{\alpha} \subset E_{\leq \alpha}$.  Furthermore,
\begin{equation} \label{eqn:Ea-vs-Eb}
\alpha \leq \alpha_+
\quad\text{implies}\quad
(1/\alpha_+) E_{\alpha_+} \subset (1 / \alpha) E_{\alpha}.
\end{equation}
The inclusion~\eqref{eqn:Ea-vs-Eb} holds because $E$ is convex and contains the origin.

Let $R_{\infty}$ be a constant that depends
only on the parameters $\varrho$ and $\iota$.
Suppose that $0 \leq R \leq R_+ \leq R_{\infty}$.
To prove that $\norm{\widehat{\vct{u}}} < R_+$,
it suffices to establish the inequality
\begin{equation} \label{eqn:lasso-claim-1}
\min_{\vct{u} \in E_{\leq R}}
	\norm{ \mtx{\Phi} \begin{bmatrix} \vct{u} \\ - 1 \end{bmatrix} }
	< \min_{\vct{u} \in E_{R_+}}
	\norm{ \mtx{\Phi} \begin{bmatrix} \vct{u} \\ - 1 \end{bmatrix} }.
\end{equation}
Indeed, recall that $\widehat{\vct{u}}$ is the minimizer of the objective
over $E$, and let $\vct{u}_0$ be the point in $E_{\leq R}$ where the left-hand
minimum in~\eqref{eqn:lasso-claim-1} is attained.  The objective is a convex
function of $\vct{u}$, so it does not decrease as
$\vct{u}$ traverses the line segment from $\widehat{\vct{u}}$ to $\vct{u}_0$.
If $\norm{\widehat{\vct{u}}} \geq R_+$, this line
segment muss pass through $E_{R_+}$, which is
impossible because the ordering~\eqref{eqn:lasso-claim-1}
forces the objective to decrease on the way from $E_{R_+}$
to $\vct{u}_0$.

Fix the parameter $R$ in the range $0 \leq R \leq R_{\infty}$;
we will select a suitable value later.
Theorem~\ref{thm:univ-rsv}\eqref{eqn:univ-rsv-fail}
demonstrates that, with high probability,
$$
\min_{\vct{u} \in E_{\leq R}}
	\norm{ \mtx{\Phi} \begin{bmatrix} \vct{u} \\ - 1 \end{bmatrix} }
	\leq \coll{E}_n\big(E_{\leq R} \times \{-1\} \big) + o(\sqrt{n})
	\leq \coll{E}_n\big(E_{R} \times \{-1\} \big) + o(\sqrt{n}).
$$
The second relation holds because the excess width decreases
with respect to set inclusion.  Observe that
$$
\begin{aligned}
\coll{E}_n\big(E_{R} \times \{-1\} \big)
	&= \Expect \inf_{\vct{t} \in E_R} \big( \sqrt{n} (\normsq{\vct{t}} + 1 )^{1/2}
	+ \vct{g} \cdot \vct{t} \big) \\
	&= \sqrt{n} ( R^2 + 1 )^{1/2}
	- \Expect \sup_{\vct{t} \in E_R} \vct{g} \cdot \vct{t} \\
	&= \sqrt{n} ( R^2 + 1 )^{1/2}
	- R \, \coll{W}\big((1/R) E_R \big).
\end{aligned}
$$
The last identity holds when we factor out $\norm{\vct{t}}$
and identify the Gaussian width~\eqref{eqn:gauss-width}.

Fix the second parameter $R_+$, such that $R \leq R_+ \leq R_{\infty}$.
Theorem~\ref{thm:univ-rsv}\eqref{eqn:univ-rsv-succ}
demonstrates that, with high probability,
$$
\min_{\vct{u} \in E_{R_+}}
	\norm{ \mtx{\Phi} \begin{bmatrix} \vct{u} \\ - 1 \end{bmatrix} }
	\geq \coll{E}_n\big(E_{R_+} \times \{-1\} \big) - o(\sqrt{n}).
$$
Much as before, we calculate the excess width:
$$
\begin{aligned}
\coll{E}_n\big(E_{R_+} \times \{-1\} \big)
	&= \sqrt{n} ( R_+^2 + 1 )^{1/2}
	- R_+ \coll{W}\big((1/R_+) E_{R_+} \big) \\
	&\geq \sqrt{n} ( R_+^2 + 1 )^{1/2}
	- R_+ \coll{W}\big((1/R) E_{R} \big).
\end{aligned}
$$
The last inequality holds because of~\eqref{eqn:Ea-vs-Eb}
and the fact that the Gaussian width is increasing with
respect to set inclusion.

Combine the last four displays to reach
\begin{align*}
\min_{\vct{u} \in E_{\leq R}}
	\norm{ \mtx{\Phi} \begin{bmatrix} \vct{u} \\ - 1 \end{bmatrix} }
	&\leq \sqrt{n} ( R^2 + 1 )^{1/2}
	- R \, \coll{W}\big((1/R) E_R \big) + o(\sqrt{n}) \\
	&\leq \sqrt{n} ( R_+^2 + 1 )^{1/2}
	- R_+ \coll{W}\big((1/R) E_{R} \big) - o(\sqrt{n})
	\leq \min_{\vct{u} \in E_{R_+}}
	\norm{ \mtx{\Phi} \begin{bmatrix} \vct{u} \\ - 1 \end{bmatrix} }.
\end{align*}
It follows that we can establish~\eqref{eqn:lasso-claim-1}
by finding parameters for which $0 \leq R \leq R_+ \leq R_\infty$ and
$$
\left[ \sqrt{n} ( R_+^2 + 1 )^{1/2}
	- R_+ \coll{W}\big((1/R) E_{R} \big) \right]
	- \left[ \sqrt{n} ( R^2 + 1 )^{1/2}
	- R \, \coll{W}\big((1/R) E_R \big) \right]
\geq o(\sqrt{n}).
$$
To that end, we replace the Gaussian width by a number that
does not depend on the parameter $R$:
$$
\coll{W}^2\big((1/R) E_R \big)
	\leq \delta\big((1/R) E_R \big)
	\leq \delta\big(\cone(E)\big)
	\leq p \psi_{\ell_1}(s/p)
	=: d.
$$
The first relation is~\eqref{eqn:sdim-width};
the second follows from Definition~\ref{def:stat-dim};
the last is the estimate~\eqref{eqn:sdim-l1}.
Moreover, these bounds are sharp when
$R$ is sufficiently close to zero.
Therefore, we just need to verify that
\begin{equation} \label{eqn:lasso-claim-2}
\left[ \sqrt{n} ( R_+^2 + 1 )^{1/2}
	- R_+ \sqrt{d} \right]
	- \left[ \sqrt{n} ( R^2 + 1 )^{1/2}
	- R \, \sqrt{d} \right]
	\geq o(\sqrt{n}).
\end{equation}
Once we choose $R$ and $R_+$ appropriately, we can
adapt the analysis in the proof of Proposition~\ref{prop:rand-ls}.

For $\alpha \geq 0$, introduce the function
$$
f(\alpha) := \sqrt{n} (\alpha^2 + 1 )^{1/2} - \alpha \sqrt{d}.
$$
As in the proof of Proposition~\ref{prop:rand-ls},
by direct calculation, $f$ is minimized at the value
$$
R := \frac{d}{n - d}.
$$
Note that $R$ is very close to zero when $n \gg d$,
in which case $d$ is an accurate bound for $\coll{W}^2\big((1/R)E_R\big)$.
Furthermore,
\begin{equation} \label{eqn:lasso-fR}
f(R) = \sqrt{n - d}.
\end{equation}
Now, make the selection
\begin{equation} \label{eqn:lasso-R+}
R_+^2 := (1 + \iota n / d) R^2 = \frac{d + \iota n}{n - d}.
\end{equation}
Since $n \geq (1 + \varrho)\, d$, we see that $R_+$ is bounded by a
constant $R_{\infty}$ that depends only on $\iota$ and $\varrho$.
Repeating the calculation in~\eqref{eqn:ls-redux-R+-simp}, \lang{mutatis mutandis}, we have
\begin{equation} \label{eqn:lasso-fR+}
f(R_+) \geq \frac{1 + \iota/2}{\sqrt{1 + \iota}} \sqrt{n - d}.
\end{equation}
Combining~\eqref{eqn:lasso-fR} and~\eqref{eqn:lasso-fR+},
we determine that
$$
f(R_+) - f(R)
	\geq \frac{1 + \iota/2 - \sqrt{1 + \iota}}{\sqrt{1 + \iota}} \sqrt{n - d}
	\geq o(\sqrt{n}).
$$
The last inequality holds because $\iota$ is a fixed positive constant
and we have assumed that $n \geq (1 + \varrho) d$.

In conclusion, we have confirmed the claim~\eqref{eqn:lasso-claim-2}
for $R$ and the value $R_+$ designated in~\eqref{eqn:lasso-R+}.
It follows that~\eqref{eqn:lasso-claim-1} holds with high probability,
and so the optimizer of~\eqref{eqn:lasso-pf} satisfies
$\norm{\widehat{\vct{u}}} \leq R_+$ with high probability.
Therefore, the formula~\eqref{eqn:lasso-msep} for the mean squared error
yields the bound
$$
\mathsf{MSEP} = \normsq{\widehat{\vct{u}}} + 1
	\leq R_+^2 + 1
	= (1 + \iota)  \frac{n}{n - d}.
$$
Once again, we have used the assumption that $\sigma = 1$.
By homogeneity, the $\mathsf{MSEP}$ must be proportional
to $\sigma^2$, which leads to the stated result~\eqref{eqn:msep-theory}.
Finally, if we allow $\sigma \to 0$ with the other parameters fixed,
the analysis here can be adapted to show that the error
bound~\eqref{eqn:msep-theory} is sharp.

\begin{remark}[Prior Work]
In the special case of a standard normal design
$\mtx{X}$ and a standard normal error $\vct{z}$,
the result of Proposition~\ref{prop:lasso-err}
appeared in the paper~\cite{OTH13:Squared-Error}.
Our extension to more general random models is new.
Nevertheless, our proof has a lot in common with
the analysis in~\cite{OTH13:Squared-Error,TAH15:High-Dimensional-Error}.
\end{remark}

\newpage

\part{Universality of the Restricted Minimum Singular Value:
Proofs of Theorem~\ref{thm:univ-rsv} and Theorem~\ref{thm:univ-embed}(\ref{eqn:univ-embed-succ})}
\label{part:rsv}

In this part of the paper, we present a detailed proof of
the universality law for the restricted singular value of
a random matrix, Theorem~\ref{thm:univ-rsv}, and the first
part of the universality law for the embedding dimension,
Theorem~\ref{thm:univ-embed}\eqref{eqn:univ-embed-succ}.

Section~\ref{sec:rsv-bdd} contains our main technical result,
which establishes universality for the bounded random matrix
model, Model~\ref{mod:bdd-mtx}.  Section~\ref{sec:rsv-four}
extends the result for bounded random matrices to
the heavy-tailed random matrix model, Model~\ref{mod:p-mom-mtx}.
In Section~\ref{sec:rsv-four-to-rsv-univ}, we obtain
Theorem~\ref{thm:univ-rsv} as an immediate consequence
of the result for heavy-tailed matrices.
Section~\ref{sec:rsv-four-to-embed-univ}
shows how to derive Theorem~\ref{thm:univ-embed}\eqref{eqn:univ-embed-succ}
as an additional consequence.
The remaining sections in this part lay out \joelprev{the} calculations
that undergird the result for bounded random matrices.

\section{The Restricted Singular Values of a Bounded Random Matrix}
\label{sec:rsv-bdd}

The key challenges in establishing universality for the restricted
minimum singular value are already present in the case where the random
matrix is drawn from the bounded model, Model~\ref{mod:bdd-mtx}.
This section presents a universality law for bounded random matrices,
and it gives an overview of the calculations that are required to
establish this result.

\subsection{Theorem~\ref{thm:rsv-bdd}: Main Result for the Bounded Random Matrix Model}

The main technical result in this paper is a theorem
on the behavior of the restricted minimum singular value
of a bounded random matrix.

\begin{theorem}[RSV: Bounded Random Matrix Model] \label{thm:rsv-bdd}
Place the following assumptions:

\begin{itemize}
\item	Let $m$ and $n$ be natural numbers with $m \leq n^{6/5}$.

\item	Let $T$ be a closed subset of the unit ball $\mathsf{B}^n$ in $\R^n$.

\item	Draw an $m \times n$ random matrix $\mtx{\Phi}$ from Model~\ref{mod:bdd-mtx}
with bound $B$.
\end{itemize}

\noindent
Then the squared restricted singular value $\smin^2(\mtx{\Phi}; T)$ has the following properties:

\begin{enumerate}
\item \label{it:rsv-bdd-conc}
The squared restricted singular value concentrates about its mean on a scale of \joel{$B^2 \sqrt{m + n}$}.
For each $\zeta \geq 0$,
$$
\begin{aligned}
\Prob{ \smin^2( \mtx{\Phi}; T ) \leq \Expect \smin^2(\mtx{\Phi}; T) - \cnst{C}B^2 \zeta }
	&\leq \econst^{-\zeta^2/m}, \quad\text{and} \\
\Prob{ \smin^2(\mtx{\Phi}; T) \geq \Expect \smin^2(\mtx{\Phi}; T) + \cnst{C}B^2 \zeta }
	&\leq \joel{\cnst{C} \econst^{-\zeta^2/(m + \zeta \sqrt{n})}}.
\end{aligned}
$$

\item \label{it:rsv-bdd-expect-lower}
The expectation of the squared restricted singular value is bounded below in terms of the excess width:
$$
\begin{aligned}
\Expect \smin^2(\mtx{\Phi}; T)
	\geq \big( \coll{E}_m(T) \big)_+^2 - \cnst{C} B^2 (m+n)^{0.92}.
\end{aligned}
$$

\item \label{it:rsv-bdd-expect-cvx}
If $T$ is a convex set, the squared restricted singular value is bounded above in terms of the excess width:
$$
\begin{aligned}
\Expect \smin^2(\mtx{\Phi}; T) 
	\leq \big( \coll{E}_m(T) \big)_+^2 + \cnst{C} B^4 (m+n)^{\sametprev{0.94}}.
\end{aligned}
$$
\end{enumerate}

\noindent
Furthermore, the entries of $\mtx{\Phi}$ need not be symmetric for this
result to hold.
\end{theorem}

\joelprev{The proof of this result will occupy us} for the
rest of this part of the paper.  This section summarizes the required
calculations, with cross-references to the detailed arguments.

\subsection{Proof of Theorem~\ref{thm:rsv-bdd}(\ref{it:rsv-bdd-conc}): Concentration}
\label{sec:rsv-pf-conc}

Theorem~\ref{thm:rsv-bdd}\eqref{it:rsv-bdd-conc} states that the squared restricted
singular value $\smin^2(\mtx{\Phi}; T)$ concentrates around its mean.
We prove this claim in Proposition~\ref{prop:rsv-concentration},
which appears below. The argument depends on some concentration inequalities that are
derived using the entropy method.  This approach is more or less
standard, so we move on to the more interesting part of the proof.

\subsection{Setup for Proof of Theorem~\ref{thm:rsv-bdd}(\ref{it:rsv-bdd-expect-lower})
and~(\ref{it:rsv-bdd-expect-cvx}): Dissection of the Index Set}
\label{sec:rsv-pf-strategy}

Let us continue with the proof of Theorem~\ref{thm:rsv-bdd}, conclusions~\eqref{it:rsv-bdd-expect-lower}
and~\eqref{it:rsv-bdd-expect-cvx}.  The overall approach is the same in both cases,
but some of the details differ.

The first step is to dissect the index set $T$ into appropriate subsets.
For each set $J \subset \{1, \dots, n\}$, we introduce a closed subset $T_J$ of $T$ via the rule
\begin{equation} \label{eqn:TK-defn}
T_J := \big\{ \vct{t} \in T : \abs{\smash{t_j}} \leq (\#J)^{-1/2} \text{ for all $j \in J^c$} \big\},
\end{equation}
where $J^c := \{1, \dots, n\} \setminus J$.  In other words, $T_J$ contains
all the vectors in $T$ where the coordinates listed in $J^c$ are sufficiently small.
Note that convexity of the set $T$ implies convexity of each subset $T_J$.

Fix a number $k \in \{1, \dots, n\}$.  Since $T$ is a subset of the unit ball $\mathsf{B}^n$,
every vector $\vct{t} \in T$ satisfies
$$
\# \big\{ j : \abs{\smash{t_j}} > k^{-1/2} \big\} \leq k.
$$
Therefore, $\vct{t}$ belongs to some subset $T_J$ where the cardinality $\#J = k$,
and we have the decomposition 
\begin{equation} \label{eqn:T-dissection}
T = \bigcup_{\#J = k} T_J.
\end{equation}
It is clear that the number of subsets $T_J$ in this decomposition satisfies
\begin{equation} \label{eqn:TK-count}
\# \{ T_J : \# J = k \} = {n \choose k} \leq \left( \frac{\econst n}{k} \right)^{k}.
\end{equation}
We must limit the cardinality $k$ of the subsets, so we can control the number
of subsets we need to examine.

\subsection{Proof of Theorem~\ref{thm:rsv-bdd}(\ref{it:rsv-bdd-expect-lower}): Lower Bound for the RSV}
\label{sec:rsv-pf-lower}

Let us proceed with the proof of Theorem~\ref{thm:rsv-bdd}\eqref{it:rsv-bdd-expect-lower},
the lower bound for the RSV.
For the time being, we fix the parameter $k$ that designates the cardinality of the
sets $J$.  We require that $k \geq m^{2/3}$.

First, we pass from the restricted singular value $\smin(\mtx{\Phi}; T)$ over the whole set $T$
to a bound that depends on $\smin(\mtx{\Phi}; T_J)$ for the subsets $T_J$.
Proposition~\ref{prop:rsv-dissection} gives the comparison
\begin{equation} \label{eqn:rsv-bdd-pf1}
\Expect \smin^2(\mtx{\Phi}; T)
	\geq \min_{\# J = k} \Expect \smin^2(\mtx{\Phi}; T_J)
	- \cnst{C} B^2 \sqrt{k m \log(n/k)}.
\end{equation}
We have used the decomposition~\eqref{eqn:T-dissection} and the bound~\eqref{eqn:TK-count}
on the number of subsets in the decomposition to invoke the proposition.  The main ingredient
in the proof of this estimate is the concentration inequality for restricted singular values,
Proposition~\ref{prop:rsv-concentration}.

Let us focus on a specific subset $T_J$.  To study the restricted singular value
$\smin(\mtx{\Phi}; T_J)$, we want to replace the entries of the random matrix $\mtx{\Phi}$
with standard normal random variables.  Proposition~\ref{prop:partial-replacement} allows
us to make partial progress toward this goal.
Let $\mtx{\Gamma}$ be an $m \times n$ standard normal matrix, independent from $\mtx{\Phi}$.
Define an $m \times n$ random matrix $\mtx{\Psi} := \mtx{\Psi}(J)$ where $\mtx{\Psi}_J = \mtx{\Phi}_J$ and $\mtx{\Psi}_{J^c} = \mtx{\Gamma}_{J^c}$.  Then
 \begin{equation} \label{eqn:rsv-bdd-pf2}
\Expect \smin^2(\mtx{\Phi}; T_J)
	\geq \Expect \smin^2(\mtx{\Psi}; T_J)
	- \frac{\cnst{C} B^2 m^{1/3} n \log(mn)}{k^{1/2}}.
\end{equation}
This bound requires the assumption that $k \geq m^{2/3}$.
The argument is based on the Lindeberg exchange principle,
but we have used this method in an unusual way.
For a vector $\vct{t} \in T_J$, the definition~\eqref{eqn:TK-defn} gives us control on the magnitude of the entries
listed in $J^c$, which we can exploit to replace the column submatrix $\mtx{\Phi}_{J^c}$
with $\mtx{\Gamma}_{J^c}$.  The coordinates of $\vct{t}$ listed in $J$ may be large,
so we cannot replace the column submatrix $\mtx{\Phi}_J$ without incurring
a significant penalty.  Instead, we just leave it in place.

Next, we want to compare the expected \joelprev{RSV} on the right-hand side of~\eqref{eqn:rsv-bdd-pf2}
with the excess width of the set $T_J$.  Proposition~\ref{prop:excess-width} provides the bound
\begin{equation} \label{eqn:rsv-bdd-pf3}
\Expect \smin^2(\mtx{\Psi}; T_J)
	\geq \left( \coll{E}_m(T_J) - \cnst{C}B^2\sqrt{k} \right)_+^2
	\geq \left( \coll{E}_m(T) - \cnst{C}B^2 \sqrt{k} \right)_+^2.
\end{equation}
The second inequality is a consequence of the facts that $T_J \subset T$
and that the excess width is decreasing with respect to set inclusion.
The calculation in Proposition~\ref{prop:excess-width} uses the Gaussian Minimax Theorem
(see Fact~\ref{fact:gauss-minmax}) to simplify the average with respect to the
standard normal matrix $\mtx{\Gamma}$.  We also invoke standard results
from nonasymptotic random matrix theory to control the expectation over $\mtx{\Phi}_J$.

We can obtain a simple lower bound on the last term in~\eqref{eqn:rsv-bdd-pf3}
by linearizing the convex function $(\cdot)_+^2$ at the point $\coll{E}_m(T)$:
\begin{equation} \label{eqn:rsv-bdd-pf3.5}
\begin{aligned}
\left( \coll{E}_m(T) - \cnst{C}B^2\sqrt{k} \right)_+^2
	\geq \big( \coll{E}_m(T) \big)_+^2 - \cnst{C} B^2 \sqrt{k} \, \big(\coll{E}_m(T)\big)_+
	\geq \big( \coll{E}_m(T) \big)_+^2 - \cnst{C} B^2 \sqrt{km}.
\end{aligned}
\end{equation}
The last estimate follows from the observation that
\begin{equation} \label{eqn:excess-width-trivial-bd}
\coll{E}_m(T)
	= \Expect \inf_{\vct{t} \in T} \left( \sqrt{m} \norm{\vct{t}} + \vct{g} \cdot \vct{t} \right)
	\leq \Expect \left( \sqrt{m} \norm{\vct{t}_0} + \vct{g} \cdot \vct{t}_0 \right)
	\leq \sqrt{m}.
\end{equation}
In this calculation, $\vct{t}_0$ is an arbitrary point in $T$.

Finally, we sequence the four displays~\eqref{eqn:rsv-bdd-pf1},
\eqref{eqn:rsv-bdd-pf2}, \eqref{eqn:rsv-bdd-pf3}, and~\eqref{eqn:rsv-bdd-pf3.5} and combine error terms to obtain
\begin{equation} \label{eqn:rsv-bdd-pf4}
\Expect \smin^2(\mtx{\Phi}; T)
	\geq \big(\coll{E}_m(T)\big)_+^2
	- \frac{\cnst{C} B^2 m^{1/3} n \log(mn)}{k^{1/2}} - \cnst{C}B^2 \sqrt{km\log(n/k)}.
\end{equation}
Up to logarithms, an optimal choice of the cardinality parameter is $k = \lceil m^{-1/6} n \rceil$.
Since $m \leq n^{6/5}$, this choice ensures that $k \geq m^{2/3}$.  We conclude that
$$
\begin{aligned}
\Expect \smin^2(\mtx{\Phi}; T)
	&\geq \big( \coll{E}_m(T) \big)_+^2 - \cnst{C} B^2 m^{5/12} n^{1/2} \log(mn) \\
	&\geq \big( \coll{E}_m(T) \big)_+^2 - \cnst{C} B^2 (m+n)^{11/12} \log(m+n).
\end{aligned}
$$
Combine the logarithm with the power, and adjust the constants to complete
the proof of Theorem~\ref{thm:rsv-bdd}\eqref{it:rsv-bdd-expect-lower}.

\subsection{Proof of Theorem~\ref{thm:rsv-bdd}(\ref{it:rsv-bdd-expect-cvx}): Upper Bound for the RSV of a Convex Set}
\label{sec:rsv-pf-upper}

At a high level, the steps in the proof of Theorem~\ref{thm:rsv-bdd}\eqref{it:rsv-bdd-expect-cvx}
are similar with the argument in the last section.  Many of the technical details, however,
depend on convex duality arguments.

As before, we fix the cardinality parameter $k$,
with the requirement $k \geq m^{2/3}$.
The first step is to pass from $\smin(\mtx{\Phi}; T)$
to $\smin(\mtx{\Phi}; T_J)$.  We have
\begin{equation} \label{eqn:rsv-bdd-pf1b}
\Expect \smin^2(\mtx{\Phi}; T)
	\leq \min_{\#J = k} \Expect \smin^2(\mtx{\Phi}; T_J).
\end{equation}
This bound is a trivial consequence of the facts that $T_J \subset T$
for each subset $J$ of $\{1, \dots, n\}$ and that $\smin(\mtx{\Phi}; \cdot)$
is decreasing with respect to set inclusion.

Select any subset $T_J$.  We invoke Proposition~\ref{prop:partial-replacement}
to exchange most of the entries of the random matrix $\mtx{\Phi}$ for standard
normal variables.  Using the same random matrix $\mtx{\Psi}$ from the last section,
we have
\begin{equation} \label{eqn:rsv-bdd-pf2b}
\Expect \smin^2(\mtx{\Phi}; T_J)
	\leq \Expect \smin^2(\mtx{\Psi}; T_J)
	+ \frac{\cnst{C} B^2 m^{1/3} n \log(mn)}{k^{1/2}}.
\end{equation}
The bound~\eqref{eqn:rsv-bdd-pf2b} requires the assumption $k \geq m^{2/3}$,
and the proof is identical with the proof of the lower bound~\eqref{eqn:rsv-bdd-pf2}.

Next, we compare the expected \joelprev{RSV} on the right-hand side of~\eqref{eqn:rsv-bdd-pf2b}
with the excess width.  Proposition~\ref{prop:excess-width} delivers
\begin{equation} \label{eqn:rsv-bdd-pf3b}
\Expect \smin^2(\mtx{\Psi}; T_J)
	\leq \left( \coll{E}_m(T_J) + \cnst{C} B^2 \sqrt{k} \right)_+^2.
\end{equation}
\joelprev{This argument is more complicated than the analogous lower bound~\eqref{eqn:rsv-bdd-pf3},
and it depends on the convexity of $T_J$.} We also apply the assumption that $k \geq m^{2/3}$ here.

Proposition~\ref{prop:excess-width-dissection} allows us to replace the excess width of $T_J$
in~\eqref{eqn:rsv-bdd-pf3b} with the excess width of $T$:
\begin{equation} \label{eqn:rsv-bdd-pf4b}
\coll{E}_m(T_J)
	\leq \coll{E}_m(T) + \cnst{C} \sqrt{k \log(n/k)}.
\end{equation}
We use the decomposition~\eqref{eqn:T-dissection} and the bound~\eqref{eqn:TK-count} on the number of sets $T_J$ to invoke the proposition.  This proof depends on the Gaussian concentration inequality.  

Sequence the bounds~\eqref{eqn:rsv-bdd-pf1b}, \eqref{eqn:rsv-bdd-pf2b},
\eqref{eqn:rsv-bdd-pf3b}, and~\eqref{eqn:rsv-bdd-pf4b}, and combine the error terms
to arrive at
$$
\Expect \smin^2(\mtx{\Phi}; T)
	\leq \left( \coll{E}_m(T) + \cnst{C} B^2 \sqrt{k \log(n/k)} \right)_+^2
	+ \frac{\cnst{C}B^2 m^{1/3} n \log(mn)}{k^{1/2}}.
$$
Expand the square using~\eqref{eqn:excess-width-trivial-bd} to reach
$$
\Expect \smin^2(\mtx{\Phi}; T)
	\leq \big( \coll{E}_m(T) \big)_+^2 + \cnst{C}B^2 \left( \frac{m^{1/3} n \log(mn)}{k^{1/2}}
	+ \sqrt{km \log(n/k)} + B^2 k \log(n/k) \right).
$$
\sametprev{Select $k = \lceil n^{4/5} \rceil$ to arrive at
$$
\begin{aligned}
\Expect \smin^2(\mtx{\Phi}; T)
	&\leq \big( \coll{E}_m(T) \big)_+^2 + \cnst{C} B^2 (m^{1/3} n^{3/5}+n^{2/5}m^{1/2}) \log(n)
	+ \cnst{C} B^4 n^{4/5} \log(mn) \\
	&\leq \big( \coll{E}_m(T) \big)_+^2 + \cnst{C} B^4 (m+n)^{14/15} \log(m+n).
\end{aligned}
$$
Combine the power with the logarithm, and adjust constants to finish the
proof of Theorem~\ref{thm:rsv-bdd}\eqref{it:rsv-bdd-expect-cvx}.}

\section{The Restricted Singular Values of a Heavy-Tailed Random Matrix}
\label{sec:rsv-four}

In this section, we present an extension of Theorem~\ref{thm:rsv-bdd}
to the heavy-tailed matrix model, Model~\ref{mod:p-mom-mtx}.
In Section~\ref{sec:rsv-four-to-rsv-univ}, we explain how the
universality result for the restricted singular value,
Theorem~\ref{thm:univ-rsv}, is an immediate consequence.
In Section~\ref{sec:rsv-four-to-embed-univ}, we show how to
derive the first half of the universality result for the embedding
dimension, Theorem~\ref{thm:univ-embed}\eqref{eqn:univ-embed-succ}.

\subsection{Corollary~\ref{cor:rsv-four}: Main Result for the $p$-Moment Random Matrix Model}

We can extend Theorem~\ref{thm:rsv-bdd} to the heavy-tailed random
matrix model, Model~\ref{mod:p-mom-mtx} using a truncation argument.
The following corollary contains a detailed statement of the result.

\begin{corollary}[RSV: $p$-Moment Random Matrix Model] \label{cor:rsv-four}
Fix parameters $p > 4$ and $\nu \geq 1$.  Place the following assumptions:

\begin{itemize}
\item	Let $m$ and $n$ be natural numbers with $m \leq n^{6/5}$. 
\item	Let $T$ be a closed subset of the unit ball $\mathsf{B}^n$ in $\R^n$.

\item	Draw an $m \times n$ random matrix $\mtx{\Phi}$ that satisfies Model~\ref{mod:p-mom-mtx}
with given $p$ and $\nu$.
\end{itemize}

\noindent
Then the restricted singular value $\smin(\mtx{\Phi}; T)$ has the following properties:

\begin{enumerate}

\item	With high probability, the restricted singular value is bounded below by the excess width:
\label{it:rsv-four-lower}
\begin{equation} \label{eqn:smin-four-lower}
\Prob{ \smin(\mtx{\Phi}; T)
	\leq \big( \coll{E}_m(T) \big)_+ - \cnst{C}_p \nu (m+n)^{1/2 - \kappa(p)}  }
	\leq \cnst{C}_p (m+n)^{1 - p/4}.
\end{equation}

\item	If $T$ is a convex set, with high probability, the restricted singular value is bounded above by the excess width:
\label{it:rsv-four-upper}
\begin{equation} \label{eqn:smin-four-upper}
\Prob{ \smin(\mtx{\Phi}; T) \geq \big( \coll{E}_m(T) \big)_+ + \cnst{C}_p \nu^2 (m+n)^{1/2 - \kappa(p)} }
	\leq \cnst{C}_p (m+n)^{1 - p/4}.
\end{equation}
\end{enumerate}
\noindent
The function $\kappa(p)$ is strictly positive for $p > 4$,
and the constant $\cnst{C}_{p}$ depends only on $p$.
\end{corollary}

\noindent
The proof of Corollary~\ref{cor:rsv-four} appears below in Section~\ref{sec:rsv-truncation}.
The main idea is to truncate each entry of the heavy-tailed random matrix $\mtx{\Phi}$
individually.  We can treat the bounded part of the random matrix
using Theorem~\ref{thm:rsv-bdd}.  We show that the tails are negligible
by means of a relatively simple norm bound for random matrices,
Fact~\ref{fact:heavy-tail-norm}.

\subsection{Proof of Theorem~\ref{thm:univ-rsv} from Corollary~\ref{cor:rsv-four}}
\label{sec:rsv-four-to-rsv-univ}

Theorem~\ref{thm:univ-rsv} is an easy consequence of Corollary~\ref{cor:rsv-four}.
Recall the assumptions of the theorem:

\begin{itemize}
\item	The embedding dimension satisfies $\lambda D \leq d \leq D^{6/5}$.

\item	$E$ is a closed subset of the unit ball in $\R^D$.

\item	The $d$-excess width of $E$ satisfies $\coll{E}_d(E) \geq \varrho \sqrt{d}$.

\item	The $d \times D$ random \joelprev{linear map} $\mtx{\Pi}$ follows Model~\ref{mod:p-mom-mtx}
with parameters $p > 4$ and $\nu \geq 1$.
\end{itemize}

\noindent
Therefore, we may apply Corollary~\ref{cor:rsv-four} with $\mtx{\Phi} = \mtx{\Pi}$
and $T = E$ to see that
$$
\begin{aligned}
\Prob{ \smin(\mtx{\Pi}; E)
	\leq \big( \coll{E}_d(E) \big)_+ - \cnst{C}_p \nu (d+D)^{1/2 - \kappa(p)} }
	&\leq \cnst{C}_p D^{1 - p/4},
	\quad\text{and} \\
\Prob{ \smin(\mtx{\Pi}; E)
	\geq \big( \coll{E}_d(E) \big)_+ + \cnst{C}_p \nu^2 (d+D)^{1/2 - \kappa(p)} }
	&\leq \cnst{C}_p D^{1 - p/4}\phantom{,}
	\quad\text{if $E$ is convex.}
\end{aligned}
$$
For simplicity, we have dropped the embedding dimension $d$ from the right-hand
side of the bounds in the last display.

To prove Theorem~\ref{thm:univ-rsv}, it suffices to check that we can make the error term
$\cnst{C}_p \nu^2 (d + D)^{1/2 - \kappa(p)}$
smaller than $\eps \coll{E}_d(E)$ if we select the ambient dimension $D$ large enough.
Indeed, the conditions $D \leq \lambda^{-1} d$ and $d^{1/2} \leq \varrho^{-1} \coll{E}_d(E)$ ensure that
$$
\begin{aligned}
(d + D)^{1/2 - \kappa(p)}
	&\leq (1 + \lambda^{-1} )^{1/2} d^{1/2 - \kappa(p)} \\
	&\leq (1 + \lambda^{-1} )^{1/2} \varrho^{-1} \coll{E}_d(E) d^{-\kappa(p)} \\
	&\leq \frac{(1 + \lambda^{-1} )^{1/2} \varrho^{-1} \lambda^{-\kappa(p)}}{D^{\kappa(p)}}
	\coll{E}_d(E)
\end{aligned}
$$
Since $\kappa(p)$ is positive,
there is a number $N := N(p, \nu, \lambda, \varrho, \eps)$
for which $D \geq N$ implies
$$
\cnst{C}_p \nu^2 (d + D)^{1/2 - \kappa(p)} \leq \eps \coll{E}_d(E).
$$
This is what we needed to show.

\subsection{Proof of Theorem~\ref{thm:univ-embed}(\ref{eqn:univ-embed-succ}) from Corollary~\ref{cor:rsv-four}}
\label{sec:rsv-four-to-embed-univ}

Theorem~\ref{thm:univ-embed}\eqref{eqn:univ-embed-succ} is also an easy consequence of Corollary~\ref{cor:rsv-four}.  Recall the assumptions of the theorem:

\begin{itemize}
\item	$E$ is a compact subset of $\R^D$ that does not contain the origin.

\item	The statistical dimension of $E$ satisfies $\delta(E) \geq \varrho D$.

\item	The $d \times D$ random \joelprev{linear map} $\mtx{\Pi}$ follows Model~\ref{mod:p-mom-mtx}
with parameters $p > 4$ and $\nu \geq 1$.
\end{itemize}

\noindent
In this section, we consider the regime where the embedding dimension
$d \geq (1 + \eps) \, \delta(E)$.  We need to demonstrate that
\begin{equation} \label{eqn:embed-a-pf1}
\Prob{ \vct{0} \notin \mtx{\Pi}(E) }
	= \Prob{ E \cap \nullsp(\mtx{\Pi}) = \emptyset }
	\geq 1 - \cnst{C}_p D^{1 - p/4}.
\end{equation}

We begin with a reduction to a specific choice of the embedding dimension $d$.
Let $\mtx{\Pi}_m$ be the $m \times D$ matrix formed from the first $m$ rows
of the random \joelprev{linear map} $\mtx{\Pi}$.  The function
$m \mapsto \Prob{ E \cap \nullsp(\mtx{\Pi}_m) = \emptyset }$
is weakly increasing because $m \mapsto \nullsp(\mtx{\Pi}_m)$ is a decreasing sequence of sets.
Therefore, it suffices to verify~\eqref{eqn:embed-a-pf1} in the case
where $d = \lceil (1 + \eps) \, \delta(E) \rceil$.  Note that
$d \leq 2D + 1$ because the statistical dimension $\delta(E) \leq D$
and $\eps  < 1$.

Introduce the spherical retraction $\Omega := \vct{\theta}(E)$.
Proposition~\ref{prop:annihilate} and~\eqref{eqn:univ-embed-spherical}
demonstrate \joelprev{that} $$
\smin(\mtx{\Pi}; \Omega) > 0
\quad\text{implies}\quad
\vct{0} \notin \mtx{\Pi}(\Omega)
\quad\text{implies}\quad
\vct{0} \notin \mtx{\Pi}(E).
$$
Therefore, to check~\eqref{eqn:embed-a-pf1},
it suffices to produce a high-probability lower bound
on the restricted singular value $\smin(\mtx{\Pi}; \Omega)$.
With the choices $\mtx{\Phi} = \mtx{\Pi}$ and $T = \Omega$,
Corollary~\ref{cor:rsv-four} yields
$$
\Prob{ \smin(\mtx{\Pi}; \Omega) \geq \big( \coll{E}_d(\Omega) \big)_+
	- \cnst{C}_p \nu (d + D)^{1/2 - \kappa(p)} }
	\geq 1-\cnst{C}_p D^{1 - p/4}.
$$
To complete the proof, we  need to verify that our hypotheses imply
\begin{equation} \label{eqn:embed-a-pf1.5}
\coll{E}_d(\Omega)
	> \cnst{C}_p \nu (d + D)^{1/2 - \kappa(p)}.
\end{equation}
This point follows from two relatively short calculations.

Since $\Omega$ is a subset of the unit sphere,
we quickly compute ite excess width:
\begin{equation} \label{eqn:embed-a-pf2}
\begin{aligned}
\coll{E}_d(\Omega) = \Expect \inf_{\vct{x} \in \Omega}
	\big( \sqrt{d} \norm{\vct{x}} + \vct{g} \cdot \vct{x} \big)
	&\geq \sqrt{d} - \sqrt{\delta(\Omega)} \\
	&= \sqrt{d} - \sqrt{\delta(E)}
	\geq \big(\sqrt{1 + \eps} - 1\big) \sqrt{\delta(E)}.
\end{aligned}
\end{equation}
The second identity holds because $\Omega$ is a subset of the unit sphere,
and we have used the relation~\eqref{eqn:excess-vs-sdim}.
Recall that the statistical
dimension of a general set $E$ is defined as
$\delta(E) = \delta(\vct{\theta}(E))$, and then introduce
the value $d = \lceil (1 + \eps)\, \delta(E) \rceil$ of the embedding dimension.

Meanwhile, we can bound the dimensional term in~\eqref{eqn:embed-a-pf1.5} above:
$$
(d + D)^{1/2 - \kappa(p)} \leq \cnst{C} D^{1/2 - \kappa(p)}
	\leq \cnst{C} \varrho^{-1/2} D^{-\kappa(p)} \sqrt{\delta(E)}.
$$
The first inequality holds because $d \leq 2D + 1$,
and the second inequality uses the assumption $D \leq \varrho^{-1} \delta(E)$.
Since $\kappa(p)$ is positive,
we can find a number $N := N(p, \nu, \varrho, \eps)$ for which
$D \geq N$ implies that
$$
\cnst{C}_p \nu (d + D)^{1/2 -\kappa(p)}
	< \big( \sqrt{1 + \eps} - 1 \big) \sqrt{\delta(E)}.
$$
Combine the last display with~\eqref{eqn:embed-a-pf2} to determine
that the claim~\eqref{eqn:embed-a-pf1.5} is valid.

\section{Theorem~\ref{thm:rsv-bdd}: Concentration for Restricted Singular Values}
\label{sec:rsv-concentration}

In this section, we demonstrate that the restricted minimum singular value
of a bounded random matrix concentrates about its mean value.  This result
yields Theorem~\ref{thm:rsv-bdd}\eqref{it:rsv-bdd-conc}.

\begin{proposition}[Theorem~\ref{thm:rsv-bdd}: Concentration] \label{prop:rsv-concentration}
Let $S$ be a closed subset of $\mathsf{B}^n$.  Let $\mtx{\Phi}$ be an $m \times n$
random matrix drawn from Model~\ref{mod:bdd-mtx} with
uniform bound $B$.  For all $\zeta \geq 0$,
\begin{align}
\Prob{ \smin^2(\mtx{\Phi}; S) \leq \Expect \smin^2(\mtx{\Phi}; S) - \cnst{C} B^2 \zeta }
	&\leq \econst^{-\zeta^2 / m} \label{eqn:rsv-lower-tail} \\
\Prob{ \smin^2(\mtx{\Phi}; S) \geq \Expect \smin^2(\mtx{\Phi}; S) + \cnst{C} B^2 \zeta }
	&\leq \econst^{-\zeta^2/(m + \zeta \sqrt{n})}. \label{eqn:rsv-upper-tail}
\end{align}
\end{proposition}

\noindent
The proof relies on some modern concentration inequalities
that are derived using the entropy method.  We establish the
bound~\eqref{eqn:rsv-lower-tail} on the lower tail in Section~\ref{sec:rsv-lower}; the
bound~\eqref{eqn:rsv-upper-tail} on the upper tail appears in Section~\ref{sec:rsv-upper}.

In several other parts of the paper, we rely on variants of Proposition~\ref{prop:rsv-concentration} that follow from essentially the same arguments.  We omit the details
of these proofs to avoid repetition.

\subsection{Proposition~\ref{prop:rsv-concentration}:
The Lower Tail of the RSV} \label{sec:rsv-lower}

First, we establish that the restricted singular
value is unlikely to be much smaller than its mean.
The proof depends on an exponential moment inequality
derived using Massart's modified logarithmic Sobolev inequality~\cite{Mas00:About-Constants}.
For instance, see~\cite[Thm.~6.27]{BLM13:Concentration-Inequalities}
or \cite[Thm.~19]{Mau12:Thermodynamics-Concentration}.

\begin{fact}[Self-Bounded Random Variable: Lower Tail] \label{fact:self-bdd-lower}
Let $(X_1, \dots, X_p)$ be an independent sequence of real random variables.
For a nonnegative function $f : \R^p \to \R_+$, define
$$
\begin{aligned}
Y &:= &&f(X_1, \dots, X_{i-1}, X_i, X_{i+1}, \dots, X_p), \quad\text{and} \\
\sup\nolimits_i Y &:= \sup_{\alpha \in \supp{X_i}} &&f(X_1, \dots, X_{i-1}, \alpha, X_{i+1}, \dots, X_p)
\quad\text{for $i = 1, \dots, p$.}
\end{aligned}
$$
Suppose that
$$
\Delta_{-} := \sum_{i=1}^p \big(\sup\nolimits_i Y - Y \big)^2
	\leq aY + b
	\quad\text{where $a, b \geq 0$.}
$$
Then, for $\zeta \geq 0$,
$$
\Prob{ Y \leq \Expect Y - \zeta } \leq \exp\left( \frac{-\zeta^2/2}{a \Expect Y + b} \right).
$$
The function $\supp{\cdot}$ returns the support of a random variable.
\end{fact}

\noindent
With this fact at hand, we may derive the lower tail bound.

\begin{proof}[Proof of Proposition~\ref{prop:rsv-concentration}, Eqn.~\eqref{eqn:rsv-lower-tail}]
Introduce the random variable
$$
Y := \smin^2(\mtx{\Phi}; S)
	= \min_{\vct{s} \in S} \normsq{\mtx{\Phi}\vct{s}}
	= \min_{\vct{s} \in S} \sum_{i=1}^m \left(\sum_{j = 1}^n \phi_{ij} s_j \right)^2.
$$
For each index pair $(i, j)$ and real number $\alpha$, define a random matrix $\mtx{\Phi}_{\alpha}^{(ij)}$
by replacing the $(i, j)$ entry $\phi_{ij}$ of $\mtx{\Phi}$ with the number $\alpha$.
Now, \joel{the random variable
$$
\sup\nolimits_{ij} Y = \max_{\abs{\alpha} \leq B} \smin^2\big(\mtx{\Phi}^{(ij)}_{\alpha}; S \big)
	= \max_{\abs{\alpha} \leq B} \min_{\vct{s} \in S} \normsq{\smash{\mtx{\Phi}^{(ij)}_{\alpha} \vct{s}}}.
$$}
We must evaluate the lower variance proxy
\joel{$$
\Delta_- := \sum_{i=1}^m \sum_{j = 1}^n \big(\sup\nolimits_{ij} Y - Y \big)^2.
$$}
To that end, select a point $\vct{t} \in \argmin_{\vct{s} \in S} \normsq{\mtx{\Phi}\vct{s}}$.
For each index pair $(i, j)$,
\joel{$$
\sup\nolimits_{ij} Y - Y
	= \max_{\abs{\alpha} \leq B} \min_{\vct{s} \in S} \normsq{\smash{\mtx{\Phi}^{(ij)}_{\alpha}} \vct{s}} -
	\min_{\vct{s} \in S} \normsq{\mtx{\Phi} \vct{s}}
	\leq \max_{\abs{\alpha} \leq B} \normsq{\smash{\mtx{\Phi}^{(ij)}_{\alpha} \vct{t}}} - \normsq{\mtx{\Phi} \vct{t}}.
$$}
The matrix \joel{$\mtx{\Phi}^{(ij)}_{\alpha}$} differs from $\mtx{\Phi}$ only in the $i$th row.
Therefore, when we expand the squared norms into components, all of the components
cancel except for the $i$th one.  We discover that
\joel{\begin{equation} \label{eqn:Yij-Y}
\begin{aligned}
\sup\nolimits_{ij} Y - Y
	&\leq \max_{\abs{\alpha} \leq B} \left( \big( (\alpha - \phi_{ij}) t_j + \sum_{\ell=1}^n \phi_{i\ell} t_\ell \big)^2
		- \big(\sum_{\ell=1}^n \phi_{i \ell} t_{\ell} \big)^2 \right) \\
	&= \max_{\abs{\alpha} \leq B} \left( (\alpha - \phi_{ij})^2 t_j^2 + 2 (\alpha - \phi_{ij}) t_j \sum_{\ell=1}^n \phi_{i\ell} t_\ell \right) \\
	&\leq \max_{\abs{\alpha} \leq B} \left( (\alpha - \phi_{ij})^2 t_j^2 + 2 \abs{\smash{\alpha - \phi_{ij}} } \abs{\smash{t_j}} \abs{\sum_{\ell=1}^n \phi_{i \ell} t_\ell} \right) \\
	&\leq 4B^2 t_j^2 + 4B \abs{\smash{t_j}} \abs{ \sum_{\ell=1}^n \phi_{i \ell} t_{\ell} }.
\end{aligned}
\end{equation}}
The last inequality holds because \joel{$\abs{\alpha} \leq B$} and $\abs{\smash{\phi_{ij}}} \leq B$.
\joel{As a consequence,}
$$
\begin{aligned}
\joel{\big( \sup\nolimits_{ij} Y - Y \big)^2}
	\leq \left( 4B^2 t_j^2 + 4B \abs{\smash{t_j}} \abs{ \sum_{\ell=1}^n \phi_{i \ell} t_{\ell} } \right)^2
	\leq 32 B^4 t_j^4 + 32 B^2 t_j^2 \left( \sum_{\ell=1}^n \phi_{i\ell} t_{\ell} \right)^2.
\end{aligned}
$$
Sum over pairs $(i, j)$ to arrive at
$$
\begin{aligned}
\joel{\sum_{i=1}^m \sum_{j=1}^n \big( \sup\nolimits_{ij} Y - Y \big)^2}
	&\leq \sum_{i=1}^m \sum_{j=1}^n \left( 32 B^4 t_j^4 + 32 B^2 t_j^2 \left( \sum_{\ell=1}^n \phi_{i\ell} t_{\ell} \right)^2 \right) \\
	&\leq 32 B^4 m + 32 B^2 \sum_{i=1}^m \left( \sum_{\ell=1}^n \phi_{i\ell} t_{\ell} \right)^2 \\
	&= 32 B^4 m + 32 B^2 \normsq{ \mtx{\Phi} \vct{t} } \\
	&= 32 B^4 m + 32 B^2 \min_{\vct{s} \in S} \normsq{\mtx{\Phi} \vct{s}}.
\end{aligned}
$$
To reach the \joelprev{second} line, we used the fact that $\pnorm{\ell_4}{\vct{t}} \leq \norm{\vct{t}} \leq 1$.  The last line depends on the definition of $\vct{t}$.

In summary, we have demonstrated that
$$
\Delta_- \leq 32 B^2 Y + 32 B^4 m.
$$
To apply Fact~\ref{fact:self-bdd-lower}, we need a bound for the expectation of $Y$.
Designate a point $\vct{s}_0 \in S$, and calculate that
$$
\Expect Y \leq \Expect \normsq{\mtx{\Phi} \vct{s}_0} = m \normsq{\vct{s}_0} \leq  m.
$$
The identity holds because $\mtx{\Phi}$ has independent, standardized entries.
Fact~\ref{fact:self-bdd-lower} ensures that
$$
\Prob{ Y \leq \Expect Y - \zeta } \leq \exp\left( \frac{-\zeta^2/2}{32B^2 \Expect Y + 32 B^4 m} \right)
	\leq \exp\left( \frac{-\zeta^2}{128 B^4 m} \right).
$$
Rewrite this formula to complete the proof of~\eqref{eqn:rsv-lower-tail}.
\end{proof}

\subsection{Proposition~\ref{prop:rsv-concentration}:
The Upper Tail of the RSV} \label{sec:rsv-upper}

Next, we establish that the restricted singular value of a bounded random matrix
is unlikely to be much larger than its mean.
The proof depends on another \joel{tail inequality} for self-bounded random variables.
This result also follows from Massart's modified log-Sobolev inequality~\cite{Mas00:About-Constants};
the argument is similar to the proofs of~\cite[Thms.~18, 19]{Mau12:Thermodynamics-Concentration}

\begin{fact}[Self-Bounded Random Variable: Upper Tail] \label{fact:self-bdd-upper}
Let $(X_1, \dots, X_p)$ be an independent sequence of real random variables.
For a function $f : \R^p \to \R$, define
$$
\begin{aligned}
Y &:= &&f(X_1, \dots, X_{i-1}, X_i, X_{i+1}, \dots, X_p), \quad\text{and} \\
\sup\nolimits_i Y &:= \sup_{\alpha \in \supp{X_i}} &&f(X_1, \dots, X_{i-1}, \alpha, X_{i+1}, \dots, X_p)
\quad\text{for $i = 1, \dots, p$.}
\end{aligned}
$$
Suppose that $\sup_i Y - Y \leq L$ uniformly for a fixed value $L$ and that
$$
\Delta_{-} := \sum_{i=1}^p \big(\sup\nolimits_i Y - Y \big)^2
	\leq aY + b
	\quad\text{where $a, b \geq 0$.}
$$
Then, for $\zeta \geq 0$,
$$
\Prob{ Y \geq \Expect Y + \zeta } \leq \exp\left( \frac{-\zeta^2/2}{(a \Expect Y + b) + (a + L) \zeta} \right).
$$
The function $\supp{\cdot}$ returns the support of a random variable.
\end{fact}

\joel{
\begin{proof}[Proof of Proposition~\ref{prop:rsv-concentration}, Eqn.~\eqref{eqn:rsv-upper-tail}]
We proceed as in the proof of~\eqref{eqn:rsv-lower-tail}.
It just remains to verify the uniform bound on $\sup_{ij} Y - Y$.
Starting from~\eqref{eqn:Yij-Y}, we calculate that
$$
\begin{aligned}
\sup\nolimits_{ij} Y - Y &\leq 4 B^2 t_j^2 + 4B \abs{\smash{t_j}} \abs{ \sum_{\ell=1}^n \phi_{il} t_{\ell} } \\
	&\leq 4B^2 + 4B \left( \sum_{\ell=1}^n \phi_{i\ell}^2 \right)^{1/2} \\
	&\leq 4B^2 + 4B^2 \sqrt{n} \\
	&\leq 8B^2 \sqrt{n}.
\end{aligned}
$$
This calculation relies on the fact that $\norm{\vct{t}} \leq 1$ and the entries of $\mtx{\Phi}$
have magnitude bounded by $B$.  Apply Fact~\ref{fact:self-bdd-upper} to complete the argument.
\end{proof}
}

\section{Theorem~\ref{thm:rsv-bdd}: Probability Bounds for Dissections}

In this section, we establish some results that compare the expectation of a minimum
of random variables indexed by a set with the expectations of the minima indexed by subsets.
These facts are easy consequences of concentration \joelprev{phenomena}.

\subsection{Dissection of the Excess Width}

First, we show that the excess width of a set can be related to the excess width
of a collection of subsets.

\begin{proposition}[Theorem~\ref{thm:rsv-bdd}: Dissection of Excess Width] \label{prop:excess-width-dissection}
Consider a closed subset $T$ of the unit ball $\mathsf{B}^n$ that has been decomposed into
a finite number of closed subsets $T_J$:
$$
T = \bigcup_{J \in \coll{J}} T_J.
$$
For each $m \geq 0$, it holds that
$$
\bigg( \Expect \min_{\vct{t} \in T}\big( \sqrt{m} \norm{\vct{t}} + \vct{g} \cdot \vct{t} \big) \bigg)_+
	\geq \min_{J \in \coll{J}} \Expect \min_{\vct{t} \in T_J}
	\big( \sqrt{m} \norm{\vct{t}} + \vct{g} \cdot \vct{t} \big)
	- \cnst{C} \sqrt{\log(\# \coll{J})}.
$$
In other words,
$$
\min_{J \in \coll{J}} \big( \coll{E}_m(T_J) \big)_+
	\leq \big( \coll{E}_m(T) \big)_+ 
	+ \cnst{C} \sqrt{\log(\# \coll{J})}.
$$
\end{proposition}

\begin{proof}
Since $T = \bigcup_{J \in \coll{J}} T_J$, we can stratify the minimum over $T$:
$$
\min_{\vct{t} \in T} \big(\sqrt{m} \norm{\vct{t}} + \vct{g} \cdot \vct{t} \big)
	= \min_{J \in \coll{J}} \min_{\vct{t} \in T_J}
	\big(\sqrt{m} \norm{\vct{t}} + \vct{g} \cdot \vct{t} \big).
$$
We will obtain a lower tail bound for the minimum over $T$ by combining
lower tail bounds for each minimum over $T_J$.  Then we integrate the tail
probability to obtain the required expectation bound.

Each subset $T_J$ is contained in $\mathsf{B}^n$, so the map
$$
\vct{g} \mapsto \min_{\vct{t} \in T_J} \big( \sqrt{m} \norm{\vct{t}} + \vct{g} \cdot \vct{t} \big)
$$
is 1-Lipschitz.  The Gaussian concentration inequality, Fact~\ref{fact:gauss-concentration},
provides a tail bound.  For each $J \in \coll{J}$ and for each $\lambda \geq 0$,
$$
\Prob{ \min_{\vct{t} \in T_J} \big( \sqrt{m} \norm{\vct{t}} + \vct{g} \cdot \vct{t} \big)
	\leq \Expect \min_{\vct{t} \in T_J}  \big(\sqrt{m} \norm{\vct{t}} + \vct{g} \cdot \vct{t} \big) - \lambda }
	\leq \econst^{-\lambda^2/2}.
$$
Apply the union bound over $J \in \coll{J}$ to see that, for all $\zeta \geq 0$,
\begin{multline*}
\Prob{ \min_{\vct{t} \in T} \big( \sqrt{m} \norm{\vct{t}} + \vct{g} \cdot \vct{t} \big) \leq
	\min_{J \in \coll{J}} \Expect \min_{\vct{t} \in T_J} \big( \sqrt{m} \norm{\vct{t}} + \vct{g} \cdot \vct{t} \big)
	- \sqrt{2 \log( \# \coll{J})} - \zeta } \\
	\leq (\# \coll{J}) \econst^{-(\sqrt{2 \log( \# \coll{J})} + \zeta)^2/2}
	\leq \econst^{- \zeta^2 / 2}.
\end{multline*}

Using the latter estimate, we quickly compute the excess width of $T$.
Abbreviate
$$
Y := \min_{\vct{t} \in T} \big( \sqrt{m} \norm{\vct{t}} + \vct{g} \cdot \vct{t} \big)
\quad\text{and}\quad
\alpha := \min_{J \in \coll{J}} \Expect \min_{\vct{t} \in T_J}
	\big( \sqrt{m} \norm{\vct{t}} + \vct{g} \cdot \vct{t} \big)
	- \sqrt{2 \log(\#\coll{J})}.
$$
If $\alpha \leq 0$, the stated result is trivial, so we may assume that $\alpha > 0$.
The integration by parts representation of the expectation yields
$$ 
( \Expect Y )_+ \geq \int_0^\alpha \Prob{ Y > \zeta } \idiff{\zeta}
	= \int_0^\alpha \Prob{Y > \alpha - \zeta} \idiff{\zeta}
	= \alpha - \int_0^\alpha \Prob{ Y \leq \alpha - \zeta } \idiff{\zeta}
	\geq \alpha - \int_0^\alpha \econst^{-\zeta^2/2} \idiff{\zeta}
	> \alpha - \sqrt{\pi/2}.
$$
Reintroduce the values of $Y$ and $\alpha$, and combine constants to complete the argument.
\end{proof}

\subsection{Dissection of the Restricted Singular Value}

Next, we show that the minimum singular value of a random matrix, restricted to a set,
is controlled by the minimum singular value, restricted to subsets.

\begin{proposition}[Theorem~\ref{thm:rsv-bdd}: Dissection of RSV] \label{prop:rsv-dissection}
Consider a closed subset $T$ of the unit ball $\mathsf{B}^n$, and assume that it has been decomposed into a finite number of closed subsets $T_J$:
$$
T = \bigcup_{J \in \coll{J}} T_J.
$$
Let $\mtx{\Phi}$ be an $m \times n$ random matrix that satisfies Model~\ref{mod:bdd-mtx}
with uniform bound $B$.  Then
$$
\min_{J \in \coll{J}} \Expect \smin^2(\mtx{\Phi}; T_J)
	\leq \Expect \smin^2(\mtx{\Phi}; T)
	+ \cnst{C} B^2 \sqrt{m \log(\#\coll{J})}.
$$
\end{proposition}

\begin{proof}
The argument is similar with the proof of Proposition~\ref{prop:excess-width-dissection}.
Since $T = \bigcup_{J \in \coll{J}} T_J$,
$$
\smin^2(\mtx{\Phi}; T)
	= \min_{\vct{t} \in T} \normsq{\mtx{\Phi} \vct{t}}
	= \min_{J \in \coll{J}} \min_{\vct{t} \in T_J} \normsq{\mtx{\Phi} \vct{t}}
	= \min_{J \in \coll{J}} \smin^2(\mtx{\Phi}; T_J).
$$
The lower tail bound~~\eqref{eqn:rsv-lower-tail} for restricted singular values implies that, for all $\lambda \geq 0$,
$$
\Prob{ \smin^2(\mtx{\Phi}; T_J)
	\leq \Expect \smin^{\joelprev{2}}(\mtx{\Phi}; T_J) - \lambda }
	\leq \econst^{-\lambda^2/ (\cnst{C} B^4 m) }.
$$
Next, we take a union bound.  For all $\zeta \geq 0$,
\begin{multline*}
\Prob{ \smin^2(\mtx{\Phi}; T)
	\leq \min_{J \in \coll{J}} \Expect \smin^2(\mtx{\Phi}; T_J)
	- \sqrt{\cnst{C} B^4 m \log(\#\coll{J})} - \zeta } \\
	\leq (\# \coll{J}) \econst^{-(\sqrt{ \cnst{C} B^4 m \log(\# \coll{J})} + \zeta)^2 / (\cnst{C} B^4 m)}
	\leq \econst^{-\zeta^2 / (\cnst{C} B^4 m)}.
\end{multline*}
Define a random variable $Y$ and a constant $\alpha$:
$$
Y := \smin^2(\mtx{\Phi}; T)
\quad\text{and}\quad
\alpha := \min_{J \in \coll{J}} \Expect \smin^2(\mtx{\Phi}; T_J) - \sqrt{\cnst{C} B^4 m \log(\#\coll{J})}.
$$
If $\alpha \leq 0$, the stated result is trivial, so we may assume that $\alpha > 0$.
Calculating as in Proposition~\ref{prop:excess-width-dissection},
\begin{multline*}
\Expect Y \geq \int_0^\alpha \Prob{ Y > \zeta } \idiff{\zeta}
	= \int_0^\alpha \Prob{Y > \alpha - \zeta} \idiff{\zeta} \\
	= \alpha - \int_0^\alpha \Prob{ Y \leq \alpha - \zeta } \idiff{\zeta}
	\geq \alpha - \int_0^\alpha \econst^{-\zeta^2/ (\cnst{C} B^4 m)} \idiff{\zeta}
	\geq \alpha - \cnst{C} B^2 \sqrt{m}.
\end{multline*}
Simplify this bound to complete the proof.
\end{proof}

\section{Theorem~\ref{thm:rsv-bdd}: Replacing Most Entries of the Random Matrix}
\label{sec:replacement}

In this section, we show that it is possible to replace most of the entries
of a random matrix $\mtx{\Phi}$ with standard normal random variables
without changing the restricted singular value
$\smin(\mtx{\Phi}; T_J)$ very much.

\begin{proposition}[Theorem~\ref{thm:rsv-bdd}: Partial Replacement] \label{prop:partial-replacement}
Let $\mtx{\Phi}$ be an $m \times n$ random matrix that satisfies Model~\ref{mod:bdd-mtx}
with magnitude bound $B$.
Let $J$ be a subset of $\{1, \dots, n\}$ with cardinality $k$,
and let $T_J$ be a closed subset of $\mathsf{B}^n$ for which
\begin{equation} \label{eqn:TK-hyp-replace}
\vct{t} \in T_J
\quad\text{implies}\quad
\abs{\smash{t_j}} \leq k^{-1/2}
\quad\text{for each index $j \in J^c$}.
\end{equation}
Define an $m \times n$ random matrix $\mtx{\Psi}$ where \begin{equation} \label{eqn:rsv-hybrid}
\mtx{\Psi}_J = \mtx{\Phi}_J
\quad\text{and}\quad
\mtx{\Psi}_{J^c} = \mtx{\Gamma}_{J^c}.
\end{equation}
Assuming that $k \geq m^{2/3}$, we have
\begin{equation} \label{eqn:replacement-error}
\abs{ \Expect \min_{\vct{t} \in T_J} \normsq{\mtx{\Phi} \vct{t}}
	- \Expect \min_{\vct{t} \in T_J} \normsq{\mtx{\Psi} \vct{t}} }
	\quad\leq\quad \frac{\cnst{C} B^2 m^{1/3}n \log(mn)}{k^{1/2}}.
\end{equation}
Equivalently,
\begin{equation} \abs{ \Expect \smin^2(\mtx{\Phi}; T_J)
	- \Expect \smin^2(\mtx{\Psi}; T_J) }
	\quad\leq\quad \frac{\cnst{C} B^2 m^{1/3}n \log(mn)}{k^{1/2}}.
\end{equation}
As usual, $\mtx{\Gamma}$ is an $m \times n$ standard normal matrix.
\end{proposition}

The hypothesis~\eqref{eqn:TK-hyp-replace} is an essential ingredient
in the proof of Proposition~\ref{prop:partial-replacement}.  Indeed,
we can only exchange the elements of the random matrix $\mtx{\Phi}$
in the columns indexed by $J^c$ because we depend on the uniform
bound $k^{-1/2}$ to control the replacement errors.

\subsection{Proof of Proposition~\ref{prop:partial-replacement}}

The proof of Proposition~\ref{prop:partial-replacement} involves \joelprev{several steps},
so it is helpful to give an overview before we begin in earnest.
Throughout this discussion, the index set $J$ is fixed.

Let $\eps \in (0, 1)$ be a discretization
parameter.  The first step in the argument is to replace
the index set $T_J$ by a finite subset $T_J^\eps$ with
cardinality $\log(\# T_J^\eps) \leq n \log (3/\eps)$.  We obtain
the bounds
\begin{equation} \label{eqn:pr-discretize}
\abs{ \Expect \min_{\vct{t} \in T_J} \normsq{\mtx{\Phi} \vct{t}}
	- \Expect \min_{\vct{t} \in T_J^\eps} \normsq{\mtx{\Phi} \vct{t}} } \leq 2 mn \eps
	\quad\text{and}\quad
\abs{ \Expect \min_{\vct{t} \in T_J} \normsq{\mtx{\Psi} \vct{t}}
	- \Expect \min_{\vct{t} \in T_J^\eps} \normsq{\mtx{\Psi} \vct{t}} } \leq 2 mn \eps.	
\end{equation}
See Lemma~\ref{lem:discretization} for details, which are quite standard.

Next, we introduce a smoothing parameter $\beta > 0$,
and we define the soft-min function:
\begin{equation} \label{eqn:soft-min}
F : \R^{m \times n} \to \R
\quad\text{where}\quad
F(\mtx{A}) := -\frac{1}{\beta} \joelprev{\log \sum_{\vct{t} \in T_{J}^\eps}
	\econst^{- \beta \normsq{\mtx{A}\vct{t}} }}.
\end{equation}
It is advantageous to work with the soft-min because
it is differentiable.  Lemma~\ref{lem:smoothing} demonstrates
that we pay only a small cost for passing to the soft-min:
\begin{equation} \label{eqn:pr-smooth}
\abs{ \Expect \min_{\vct{t} \in T_J^\eps} \normsq{\mtx{\Phi} \vct{t}}
	- \Expect F(\mtx{\Phi}) } \leq \beta^{-1} \log(\# T_J^\eps)
	\quad\text{and}\quad
\abs{ \Expect \min_{\vct{t} \in T_J^\eps} \normsq{\mtx{\Psi} \vct{t}}
	- \Expect F(\mtx{\Psi}) } \leq \beta^{-1} \log(\# T_J^\eps).
\end{equation}
This argument is also standard.

Finally, we compare the expectation of the soft-min, evaluated
at each of the random matrices:
\begin{equation} \label{eqn:pr-exchange}
\abs{ \Expect F(\mtx{\Phi}) - \Expect F(\mtx{\Psi}) }
	\leq \frac{\cnst{C} (\beta B^4 + \beta^2 B^6) mn}{k^{3/2}}.
\end{equation}
Lemma~\ref{lem:exchange} encapsulates the \joelprev{this argument},
which is based on the Lindeberg principle (Fact~\ref{fact:lindeberg}, below).
The idea is to replace one entry of $\mtx{\Phi}_{J^c}$
at a time with the corresponding entry of $\mtx{\Psi}_{J^c}$,
which is a Gaussian random variable.
We lose very little with each exchange because the function $F$
is smooth and the vectors in $T_{J}^\eps$ are bounded on the
coordinates in $J^c$.

Combine the relations~\eqref{eqn:pr-discretize},~\eqref{eqn:pr-smooth},
and~\eqref{eqn:pr-exchange} to obtain an estimate for the total error:
$$
\abs{ \Expect \min_{\vct{t} \in T_J} \normsq{\mtx{\Phi} \vct{t}}
	- \Expect \min_{\vct{t} \in T_J} \normsq{\mtx{\Psi} \vct{t}} }
	\quad\leq\quad \cnst{C} mn \eps
	\ +\ \cnst{C} \beta^{-1} n \log(1/\eps)
	\ +\ \frac{\cnst{C} (\beta B^4 + \beta^2 B^6) mn}{k^{3/2}}.
$$
We used the fact that $\log \#(T_J^\eps) \leq n \log(3/\eps)$
in the smoothing term.

It remains to identify appropriate values for the parameters.
Select $\eps = (mn)^{-1}$ so the discretization error is negligible.
We choose the smoothing parameter so that $\beta^3 = k^{3/2}/(B^6 m)$.  In
summary,
$$
\abs{ \Expect \min_{\vct{t} \in T_J} \normsq{\mtx{\Phi} \vct{t}}
	- \Expect \min_{\vct{t} \in T_J} \normsq{\mtx{\Psi} \vct{t}} }
	\quad\leq\quad \cnst{C}
	\ +\ \frac{\cnst{C} B^2 m^{2/3} n}{k}
	\ +\ \frac{\cnst{C} B^2 m^{1/3} n \log(mn)}{k^{1/2}}.
$$
Since $B \geq 1$ and $m^{2/3} \leq k \leq n$,
the third term dominates.
This is the bound stated in~\eqref{eqn:replacement-error}.

\subsection{Proposition~\ref{prop:partial-replacement}: Discretizing the Index Set}

The first step in the proof of Proposition~\ref{prop:partial-replacement}
is to replace the set $T_J$ with a finite subset $T_J^\eps$
without changing the restricted singular values of $\mtx{\Phi}$
and $\mtx{\Psi}$ substantially.

\begin{lemma}[Proposition~\ref{prop:partial-replacement}: Discretization] \label{lem:discretization}
Adopt the notation and hypotheses of Proposition~\ref{prop:partial-replacement}.
Fix a parameter $\eps \in (0,1]$.  We can construct a subset $T_J^\eps$
of $T_J$ with cardinality $\log(\# T_J^\eps) \leq n \log (3/\eps)$
that has the property
\begin{equation} \label{eqn:covering-error}
\abs{ \Expect \min_{\vct{t} \in T_J} \normsq{ \mtx{\Phi} \vct{t} }
	- \Expect \min_{\vct{t} \in T_J^\eps} \normsq{ \mtx{\Phi} \vct{t} } }
	\leq 2mn \eps.
\end{equation}
Furthermore, the bound~\eqref{eqn:covering-error}
holds if we replace $\mtx{\Phi}$ by $\mtx{\Psi}$.
\end{lemma}

\begin{proof}
We choose the discretization $T_J^\eps$ to be an $\eps$-covering of $T_J$
with minimal cardinality.  That is,
\begin{equation*} \label{eqn:S-covers-T}
T_J^\eps \subset T_J
\quad\text{and}\quad
\max_{\vct{t} \in T_J} \min_{\vct{t}_{\eps} \in T_J^\eps} \norm{ \vct{t}_{\eps} - \vct{t} } \leq \eps.
\end{equation*}
Since $\mathrm{Vol}(T_J) \leq \mathrm{Vol}(\mathsf{B}^n)$, we can use a classic volumetric argument
to ensure that the covering satisfies $\# T_J^\eps \leq (3/\eps)^n$.
For example, see~\cite[Lem.~5.2]{Ver12:Introduction-Nonasymptotic}.

When we replace the set $T_J$ with the covering $T_J^\eps$, we incur a discretization
error.
We establish the error bound for $\mtx{\Phi}$;
the argument for $\mtx{\Psi}$ is identical.
Since $T_J^\eps \subset T_J$, it is immediate that
$$
\Expect \min_{\vct{t} \in T_J} \normsq{\mtx{\Phi} \vct{t}}
	\leq \Expect \min_{\vct{t} \in T_J^\eps} \normsq{\mtx{\Phi} \vct{t}}.
$$
We claim that
\begin{equation} \label{eqn:covering-error-claim}
\Expect \min_{\vct{t} \in T_J^\eps} \normsq{\mtx{\Phi} \vct{t}}
	\leq \Expect \min_{\vct{t} \in T_J} \normsq{\mtx{\Phi} \vct{t}}
	+ 2 \eps \left( \Expect \normsq{\mtx{\Phi}} \right).
\end{equation}
Since $\mtx{\Phi}$ has standardized entries,
$$
\Expect \normsq{\mtx{\Phi}} \leq \Expect \fnormsq{\mtx{\Phi}} = mn,
$$
where $\fnorm{\cdot}$ denotes the Frobenius norm.
Combine the last three displays to verify~\eqref{eqn:covering-error}.

It is quite easy to establish the claim~\eqref{eqn:covering-error-claim}.
For all $\vct{s}, \vct{t} \in T_J$, we have
$$
\begin{aligned}
\normsq{\mtx{\Phi} \vct{s}} - \normsq{\mtx{\Phi} \vct{t}}
	&= \big( \norm{\mtx{\Phi} \vct{s}} + \norm{\mtx{\Phi} \vct{t}} \big)
		\cdot \big( \norm{\mtx{\Phi} \vct{s}} - \norm{\mtx{\Phi} \vct{t}} \big) \\
	&\leq \norm{\mtx{\Phi}} \big( \norm{\vct{s}} + \norm{\vct{t}} \big) \cdot
	\norm{\mtx{\Phi}(\vct{s} - \vct{t}) } \\
	&\leq 2 \normsq{\mtx{\Phi}} \norm{ \vct{s} - \vct{t} }.
\end{aligned}
$$
We have used standard norm inequalities and the fact that $T_J$ is a subset of the unit ball.
Now, let $\vct{t}_{\star} \in \argmin_{\vct{t} \in T_J} \normsq{\mtx{\Phi} \vct{t}}$, and let $\vct{t}_{\eps}$ be a point in $T_J^\eps$ for which $\norm{ \vct{t}_{\eps} - \vct{t}_{\star} } \leq \eps$.  Then
$$
\begin{aligned}
\min_{\vct{t} \in T} \normsq{\mtx{\Phi} \vct{t}}
	= \normsq{ \mtx{\Phi} \vct{t}_{\star} }
	&= \normsq{ \mtx{\Phi} \vct{t}_{\eps} }
	- \big( \normsq{ \mtx{\Phi} \vct{t}_\eps } - \normsq{\mtx{\Phi} \vct{t}_{\star}} \big) \\
	&\geq \min_{\vct{t} \in T_J^\eps} \normsq{\mtx{\Phi} \vct{t}}
	- 2 \normsq{\mtx{\Phi}} \norm{ \vct{t}_{\eps} - \vct{t}_{\star}} \\
	&\geq \min_{\vct{t} \in T_J^\eps} \normsq{\mtx{\Phi} \vct{t}}
	- 2 \eps \normsq{\mtx{\Phi}}.
\end{aligned}
$$
Taking expectations, we arrive at~\eqref{eqn:covering-error-claim}.
\end{proof}

\subsection{Proposition~\ref{prop:partial-replacement}: Smoothing the Minimum}

The second step in the proof of Proposition~\ref{prop:partial-replacement}
is to pass from the restricted minimum singular value over the discrete set $T_J^\eps$
to a smooth function.  We rely on an exponential smoothing technique
that is common in the mathematical literature on statistical physics.

\begin{lemma}[Proposition~\ref{prop:partial-replacement}: Smoothing the Minimum] \label{lem:smoothing}
Adopt the notation and hypotheses of Proposition~\ref{prop:partial-replacement},
and let $T_J^\eps$ be the set introduced in Lemma~\ref{lem:discretization}.
Fix a parameter $\beta > 0$, and define the soft-min function $F$ via~\eqref{eqn:soft-min}
Then
\begin{equation} \label{eqn:smoothing-error}
\abs{ \Expect \min_{\vct{t} \in T_J^\eps} \normsq{ \mtx{\Phi} \vct{t} } - \Expect F(\mtx{\Phi}) }
	\leq \frac{1}{\beta} \log( \# T_J^\eps ).
\end{equation}
The bound~\eqref{eqn:smoothing-error} also holds if we replace $\mtx{\Phi}$ by $\mtx{\Psi}$.
\end{lemma}

\begin{proof}
This result follows from trivial bounds on the sum in the definition~\eqref{eqn:soft-min}
of the soft-min function $F$.  The summands are nonnegative, so the sum exceeds its maximum term.
Thus,
$$
\log \sum_{\vct{t} \in T_J^\eps} \econst^{-\beta \normsq{\mtx{A} \vct{t}}}
	\geq \log \max_{\vct{t} \in T_J^\eps} \econst^{-\beta \normsq{\mtx{A}\vct{t}}}
	= - \beta \min_{\vct{t} \in T_J^\eps} \normsq{\mtx{A} \vct{t}}.
$$
Similarly, the sum does not exceed the number of summands times the maximum term, so
$$
\log \sum_{\vct{t} \in T_J^\eps} \econst^{-\beta \normsq{\mtx{A} \vct{t}}}
	\leq \log \bigg( (\#T_J^\eps) \max_{\vct{t} \in T_J^\eps} \econst^{-\beta \normsq{\mtx{A} \vct{t}}} \bigg)
	= \log (\#T_J^\eps) - \beta \min_{\vct{t} \in T_J^\eps} \normsq{\mtx{A} \vct{t}}.
$$
Combine the last two displays and multiply through by the negative number $-1/\beta$ to reach
$$
\min_{\vct{t} \in T_J^\eps} \normsq{\mtx{A} \vct{t}} - \frac{1}{\beta} \log(\#T_J^\eps)
	\leq - \frac{1}{\beta}\log \sum_{\vct{t} \in T_J^\eps} \econst^{-\beta \normsq{\mtx{A} \vct{t}}}
	\leq \min_{\vct{t} \in T_J^\eps} \normsq{\mtx{A}\vct{t}}.
$$
Replace $\mtx{A}$ by the random matrix $\mtx{\Phi}$
and take the expectation to reach~\eqref{eqn:smoothing-error}.
Similarly, we can take $\mtx{A} = \mtx{\Psi}$ to obtain the result for $\mtx{\Psi}$.
\end{proof}

\subsection{Proposition~\ref{prop:partial-replacement}: Exchanging the Entries of the Random Matrix}

The last step in the proof of Proposition~\ref{prop:partial-replacement} is the hardest.
We must demonstrate that the expectation $\Expect F(\mtx{\Phi})$ of the soft-min function
does not change very much when we replace the submatrix $\mtx{\Phi}_{J^c}$ with the submatrix
$\mtx{\Psi}_{J^c}$. 
\begin{lemma}[Proposition~\ref{prop:partial-replacement}: Exchanging Entries] \label{lem:exchange}
Adopt the notation and hypotheses of Proposition~\ref{prop:partial-replacement};
let $T_J^\eps$ be the set described in Lemma~\ref{lem:discretization};
and let $F$ be the soft-min function~\eqref{eqn:soft-min} with parameter $\beta > 0$.
Then
\begin{equation} \label{eqn:total-exchange-error}
\abs{ \Expect F(\mtx{\Phi}) - \Expect F( \mtx{\Psi} ) }
	\leq \frac{\cnst{C}(\beta B^4 + \beta^2 B^6) mn}{k^{3/2}}.
\end{equation}
The random matrix $\mtx{\Psi}$ is defined in~\eqref{eqn:rsv-hybrid}.
\end{lemma}

\begin{proof}
We establish the lemma by replacing the rows of the random matrix $\mtx{\Phi}$
by the rows of the random matrix $\mtx{\Psi}$ one at a time.
For each $i = 1, 2, \dots, m + 1$, let $\mtx{\Xi}(i)$ be the random matrix
whose first $i - 1$ rows are drawn from $\mtx{\Psi}$ and whose remaining rows
are drawn from $\mtx{\Phi}$.  By construction, $\mtx{\Xi}(1) = \mtx{\Phi}$
and $\mtx{\Xi}(m+1) = \mtx{\Psi}$.  It follows that
$$
\abs{ \Expect F(\mtx{\Phi}) - \Expect F( \mtx{\Psi} ) }
	\leq \sum_{i=1}^{m} \abs{ \Expect F\big(\mtx{\Xi}(i)\big) - \Expect F\big(\mtx{\Xi}(i+1)\big) }.
$$
We will demonstrate that
\begin{equation} \label{eqn:replacement-matrix}
\abs{ \Expect F\big(\mtx{\Xi}(i)\big) - \Expect F\big(\mtx{\Xi}(i+1)\big) }
	\leq \frac{\cnst{C} (\beta B^4 + \beta^2 B^6) n}{k^{3/2}}
	\quad\text{for $i = 1, \dots, m$.}
\end{equation}
Combining the last two displays, we arrive at the bound~\eqref{eqn:total-exchange-error}.

Fix an index $i$.  By construction, the random matrices $\mtx{\Xi}(i)$ and $\mtx{\Xi}(i+1)$
are identical, except in the $i$th row.  To perform the estimate~\eqref{eqn:replacement-matrix},
it will be convenient to suppress the dependence of the function $F$ on the remaining rows.
To that end, define the functions
$$
f_i : \R^n \to \R
\quad\text{via}\quad
f_i( \vct{a} ) := - \frac{1}{\beta} \log \sum_{\vct{t} \in T_J^\eps}
	\econst^{ - \beta (\vct{a} \cdot \vct{t})^2 + q_i(\vct{t}) }
$$
where
$$
q_i : \R^n \to \R
\quad\text{via}\quad
q_i(\vct{t}) := -\beta \left[ \sum_{j < i} (\vct{\psi}^j \cdot \vct{t})^2 + \sum_{j > i} (\vct{\phi}^j \cdot \vct{t})^2 \right].
$$
We have written $\vct{\phi}^{j}$ and $\vct{\psi}^{j}$ for the $j$th rows of $\mtx{\Phi}$ and $\mtx{\Psi}$.
With these definitions,
$$
F\big( \mtx{\Xi}(i) \big) = f_i\big( \vct{\phi}^{i} \big)
\quad\text{and}\quad
F\big( \mtx{\Xi}(i+1) \big) = f_i\big( \vct{\psi}^{i} \big).
$$
Therefore, the inequality~\eqref{eqn:replacement-matrix} is equivalent with
\begin{equation} \label{eqn:replacement-vector}
\abs{ \Expect f_i\big(\vct{\phi}^i\big) - \Expect f_i\big( \vct{\psi}^i \big) }
	\leq \frac{\cnst{C} (\beta B^4 + \beta^2 B^6) n}{k^{3/2}}
	\quad\text{for $i = 1, \dots, m$.}
\end{equation}
Since the matrix $\mtx{\Phi}$ is drawn from Model~\ref{mod:bdd-mtx},
the random vector $\vct{\phi}^{i}$
contains independent, standardized entries whose magnitudes are bounded by $B$.
Meanwhile, the form~\eqref{eqn:rsv-hybrid} of the matrix $\mtx{\Psi}$ shows that the random vector $\vct{\psi}^{i}$
coincides with $\vct{\phi}^{i}$ on the components indexed by $J$,
while the entries of $\vct{\psi}^{i}$ indexed by $J^c$ are independent standard normal variables.
Sublemma~\ref{slem:compare-one}, below,
contains the proof of the inequality~\eqref{eqn:replacement-vector}.
\end{proof}

\subsubsection{Lemma~\ref{lem:exchange}: Comparison Principle for One Row}

To complete the argument in Lemma~\ref{lem:exchange}, we need to control how much a certain
function changes when we replace some of the entries of its argument with standard normal variables.  The following result contains the required calculation.

\begin{sublemma}[Lemma~\ref{lem:exchange}: Comparison for One Row] \label{slem:compare-one}
Adopt the notation and hypotheses of Proposition~\ref{prop:partial-replacement},
and let $T_J^\eps$ be the set defined in Lemma~\ref{lem:smoothing}.
Introduce the function
$$
f : \R^n \to \R
\quad\text{given by}\quad
f(\vct{a}) := - \frac{1}{\beta} \log \sum_{\vct{t} \in T_J^\eps}
	\econst^{ - \beta (\vct{a} \cdot \vct{t})^2 + q(\vct{t}) },
$$
where $q : T_J^\eps \to \R$ is an arbitrary function.
Suppose that $\vct{\phi} \in \R^n$ is a random vector with independent, standardized entries
that are bounded in magnitude by $B$.
Suppose that $\vct{\psi} \in \R^n$ is a random vector with
$$
\vct{\psi}_J = \vct{\phi}_J
\quad\text{and}\quad
\vct{\psi}_{J^c} = \vct{\gamma}_{J^c},
$$
where $\vct{\gamma} \in \R^n$ is a standard normal vector.  Then
\begin{equation} \label{eqn:compare-one}
\abs{ \Expect f(\vct{\phi}) - \Expect f( \vct{\psi}) }
	\leq \frac{\cnst{C}(\beta B^4 + \beta^2 B^6) n}{k^{3/2}}.
\end{equation}
\end{sublemma}

The proof of Sublemma~\ref{slem:compare-one} is based
on a modern interpretation~\cite{MOO10:Noise-Stability,Cha06:Generalization-Lindeberg,KM11:Applications-Lindeberg} of the Lindeberg
exchange principle~\cite{Lin22:Eine-Neue,Tro59:Elementary-Proof,Rot73:Certain-Limit}.
It is similar in spirit with examples
from the paper~\cite{KM11:Applications-Lindeberg}.  We apply the
following version of the Lindeberg principle,
which is adapted from these works.

\begin{fact}[Lindeberg Exchange Principle] \label{fact:lindeberg}
Let $r : \R \to \R$ be a function with three continuous derivatives.
Let $\phi$ and $\psi$ be standardized random variables, not necessarily independent,
with three finite moments.  Then
$$
\abs{\Expect r(\phi) - \Expect r(\psi)}
	\leq \frac{1}{6}  \Expect \left[\abs{\smash{\phi}}^3 \max_{\abs{\alpha} \leq \abs{\smash{\phi}}} \abs{ r'''(\alpha) }
	+ \abs{\smash{\psi}}^3 \max_{\abs{\alpha} \leq \abs{\smash{\psi}}} \abs{ r'''(\alpha) } \right].
$$
\end{fact}

\noindent
The proof of Fact~\ref{fact:lindeberg} involves nothing more than a third-order
Taylor expansion of $r$ around the origin.  The first- and second-order terms
cancel because the random variables $\phi$ and $\psi$ have mean zero and variance one.
With this inequality at hand, we may proceed with the proof of the sublemma.

\begin{proof}[Proof of Sublemma~\ref{slem:compare-one}]
Without loss of generality, assume that the index set $J = \{1, \dots, k\}$.
Indeed, the entries of the random matrix $\mtx{\Phi}$ are independent,
standardized, and uniformly bounded so it does not matter which set $J$ of
columns we distinguish.  This choice allows more intuitive indexing.

For each fixed index $j \in \{ k+1, \dots, n\}$, define the interpolating vector
$$
\vct{\xi}_j : \R \to \R^n
\quad\text{where}\quad
\vct{\xi}_j(\alpha) := (\phi_1, \dots, \phi_{j-1}, \alpha, \psi_{j+1}, \dots, \psi_n).
$$
Introduce the function
$$
r_j : \R \to \R
\quad\text{where}\quad
r_j(\alpha) := f\big(\vct{\xi}_j(\alpha)\big).
$$
Observe that
$$
\begin{aligned}
\abs{\Expect f(\vct{\phi}) - \Expect f(\vct{\psi})}
	= \abs{ \sum_{j = k+1}^n \Expect f\big(\vct{\xi}_j(\phi_j)\big) - \Expect f\big(\vct{\xi}_j(\psi_j)\big)} 	\leq \sum_{j = k+1}^n \abs{ \Expect r_j(\phi_j) - \Expect r_j( \psi_j ) }.
\end{aligned}
$$
The first identity holds because the sum telescopes.  Fact~\ref{fact:lindeberg} shows that
$$
\abs{\Expect r_j(\phi_j) - \Expect r_j(\psi_j)}
	\leq \frac{1}{6} \Expect \left[\abs{\smash{\phi_j}}^3 \max_{\abs{\alpha}\leq\abs{\smash{\phi_j}}} \abs{ r_j'''(\alpha)  }
	+ \abs{\smash{\psi_j}}^3  \max_{\abs{\alpha} \leq \abs{\smash{\psi_j}}} \abs{ r_j'''(\alpha) } \right].
$$ 
We claim that, for each index $j \in \{ k+1, \dots, n\}$, 
\begin{equation} \label{eqn:compare-entry}
\begin{aligned}
\Expect \left[ \abs{\smash{\phi_j}}^3 \max_{\abs{\alpha}\leq\abs{\smash{\phi_j}}} \abs{ r_j'''(\alpha)  } \right]
	&\leq \cnst{C} (\beta B^4 + \beta^2 B^6) k^{-3/2}, \quad\text{and} \\
\Expect \left[ \abs{\smash{\psi_j}}^3 \max_{\abs{\alpha} \leq \abs{\smash{\psi_j}}} \abs{ r_j'''(\alpha) } \right]
	&\leq \cnst{C} (\beta B^4 + \beta^2 B^6) k^{-3/2}.
\end{aligned}
\end{equation}
Once we establish the bound~\eqref{eqn:compare-entry}, we can combine the last three displays to reach
$$
\abs{\Expect f(\vct{\phi}) - \Expect f(\vct{\psi})}
	\leq  \cnst{C} (\beta B^4 + \beta^2 B^6) k^{-3/2} (n - k).
$$
The main result~\eqref{eqn:compare-one} follows.

To establish~\eqref{eqn:compare-entry}, fix an index $j \in \{ k+1, \dots, n \}$.
The forthcoming Sublemma~\ref{slem:derivatives}
will demonstrate that
\begin{equation} \label{eqn:compare-bd1}
\abs{ r_j'''(\alpha) } \leq \cnst{C}
	\left( \max_{\vct{t} \in T_J^\eps} \abs{\vct{\xi}'(\alpha) \cdot \vct{t} }^3 \right)
	\left( \beta \Expect_{\vct{v}} \abs{ \vct{\xi}_j(\alpha) \cdot \vct{v} }
	+ \beta^2 \Expect_{\vct{v}} \abs{ \vct{\xi}_j(\alpha) \cdot \vct{v} }^3 \right)
\end{equation}
In this expression,
$\vct{v} \in T_J^\eps$ is a random vector that is independent from $\vct{\xi}_j$;
the precise distribution of $\vct{v}$ is immaterial.  Note that
\begin{equation} \label{eqn:compare-bd2}
\max_{\vct{t} \in T_J^\eps} \abs{\vct{\xi}'(\alpha) \cdot \vct{t} }^3
	= \max_{\vct{t} \in T_J^\eps} \abs{\smash{t_j}}^3
	\leq k^{-3/2}.
\end{equation}
Indeed, the derivative $\vct{\xi}'_j(\alpha) = \mathbf{e}_j$, where $\mathbf{e}_j$
is the $j$th standard basis vector.  The last inequality holds because the conditions
$T_J^\eps \subset T_J$ and $j \in J^c$ allow us to invoke the
assumption~\eqref{eqn:TK-hyp-replace}.

We arrive at the following bound on the quantity of interest from~\eqref{eqn:compare-entry}:
\begin{multline} \label{eqn:compare-bd2.5}
\Expect \left[ \abs{\smash{\phi_j}}^3 \max_{\abs{\alpha}\leq\abs{\smash{\phi_j}}} \abs{ r_j'''(\alpha)  } \right]
	\leq \cnst{C} k^{-3/2} \left( \beta \Expect \left[ \abs{\smash{\phi_j}}^3 \max_{\abs{\alpha} \leq \abs{\smash{\phi_j}}} \abs{\vct{\xi}_j(\alpha) \cdot \vct{v}} \right]
	+ \beta^2 \Expect \left[ \abs{\smash{\phi_j}}^3 \max_{\abs{\alpha} \leq \abs{\smash{\phi_j}}} \abs{ \vct{\xi}_j(\alpha) \cdot \vct{v} }^3 \right] \right).
\end{multline}
We have merged the bounds~\eqref{eqn:compare-bd1} and~\eqref{eqn:compare-bd2}
to control $\abs{r_j(\alpha)}$.  Next, we invoked Jensen's inequality to draw the
expectation over $\vct{v}$ out of the maximum over $\alpha$.  Last, we combined
the expectations to reach~\eqref{eqn:compare-bd2.5}.
It remains to bound the expectations on the right-hand side.

Let us begin with the term in~\eqref{eqn:compare-bd2.5} that is linear in $\beta$.
In view of the identity $\vct{\xi}_j(\alpha) = \alpha \mathbf{e}_j + \vct{\xi}_j(0)$,
\begin{equation} \label{eqn:compare-bd3}
\begin{aligned}
\Expect \left[ \abs{\smash{\phi_j}}^3 \max_{\abs{\alpha} \leq \abs{\smash{\phi_j}}} \abs{\vct{\xi}_j(\alpha) \cdot \vct{v}} \right]
	&\leq \Expect\left[ \abs{\smash{\phi_j}}^3 \max_{\abs{\alpha} \leq\abs{\smash{\phi_j}}}
	\left( \abs{\smash{\alpha v_j}} + \abs{ \vct{\xi}_j(0) \cdot \vct{v} } \right) \right] \\
	&= \Expect\left[ \abs{\smash{\phi_j}}^4 \abs{\smash{v_j}} \right]
	+ \left(\Expect \abs{\smash{\phi_j}}^3 \right)\left( \Expect \abs{\vct{\xi}_j(0)\cdot \vct{v}} \right) \\
	&\leq B^4 k^{-1/2} + B^3.
\end{aligned}
\end{equation}
We used the assumption that $\phi_j$ is independent from $\vct{\xi}_j(0)$
to factor the expectation in the second term in the second line.  The
bounds in the third line exploit the fact that $\abs{\smash{\phi_j}} \leq B$ and,
via the hypothesis~\eqref{eqn:TK-hyp-replace},
the fact that $\vct{v} \in T_J$.
We also relied on the estimate
$$
\Expect \abs{ \vct{\xi}_j(0) \cdot \vct{v} } = \Expect \abs{ \sum_{i < j} \phi_i v_i + \sum_{i > j} \psi_i v_i }
	\leq \norm{ \vct{v} } \leq 1
$$
Indeed, Jensen's inequality allows us to replace the expectation with the second moment,
which simplifies to $\norm{\vct{v}}$ because
$\{ \phi_1, \dots, \phi_{j-1}, \psi_{j+1}, \dots, \psi_n \}$ is an independent
family of standardized random variables.  Since $\vct{v} \in T_J$, the norm 
$\norm{\vct{v}}$ does not exceed one.

Continuing to the term in~\eqref{eqn:compare-bd2.5} that is quadratic in $\beta^2$,
we pursue the same approach to see that
\begin{equation} \label{eqn:compare-bd4}
\begin{aligned}
\Expect \left[ \abs{\smash{\phi_j}}^3 \max_{\abs{\alpha} \leq \abs{\smash{\phi_j}}} \abs{\vct{\xi}_j(\alpha) \cdot \vct{v}}^3 \right]
	&\leq \cnst{C} \Expect\left[ \abs{\smash{\phi_j}}^3 \max_{\abs{\alpha} \leq\abs{\smash{\phi_j}}}
	\left( \abs{\smash{\alpha v_j}}^3 + \abs{ \vct{\xi}_j(0) \cdot \vct{v} }^3 \right) \right] \\
	&\leq \cnst{C} \left( B^6 k^{-3/2} + B^6 \right).
\end{aligned}
\end{equation}
This bound relies on the Khintchine-type inequality
$$
\left( \Expect \abs{ \vct{\xi}_j(0) \cdot \vct{v} }^3 \right)^{1/3}
	\leq \cnst{C} B \norm{ \vct{v} }
	\leq \cnst{C} B.
$$
For example, see~\cite[Cor.~5.12]{Ver12:Introduction-Nonasymptotic}.
We remark that this estimate could be improved if we had additional
information about the distribution of the entries of $\mtx{\Phi}$.

Substituting the bounds~\eqref{eqn:compare-bd3} and~\eqref{eqn:compare-bd4}
into~\eqref{eqn:compare-bd2.5}, we obtain
$$
\begin{aligned}
\Expect \left[ \abs{\smash{\phi_j}}^3 \max_{\abs{\alpha} \leq \abs{\smash{\phi_j}}} \abs{ r_j'''(\alpha) } \right]
	&\leq \cnst{C} k^{-3/2} \left[ \beta \big(B^4 k^{-1/2} + B^3 \big)
	+ \beta^2 \big(B^6 k^{-3/2} + B^6 \big) \right] \\
	&\leq \cnst{C} (\beta B^4 + \beta^2 B^6) k^{-3/2}.
\end{aligned}
$$
This establishes the first branch of the claim~\eqref{eqn:compare-entry}.
The second branch follows from a similar argument, where we use explicit
values for the moments of a standard normal variable instead of the uniform
upper bound $B$.
\end{proof}

\subsubsection{Sublemma~\ref{slem:compare-one}: Derivative Calculations}

The final obstacle in the proof of Proposition~\ref{prop:partial-replacement}
is to bound the derivatives that are required in Sublemma~\ref{slem:compare-one}.
This argument uses some standard methods from statistical physics,
and it is similar with the approach in the paper~\cite{KM11:Applications-Lindeberg}.

\begin{sublemma}[Sublemma~\ref{slem:compare-one}: Derivatives] \label{slem:derivatives}
Adopt the notation and hypotheses of Proposition~\ref{prop:partial-replacement}
and Sublemma~\ref{slem:compare-one}.
Let $\vct{\xi} : \R \to \R^n$ be a linear function, so its derivative $\vct{\xi}' \in \R^n$ is a constant vector.
Define the function
$$
r(\alpha) := f\big(\vct{\xi}(\alpha)\big)
	= - \frac{1}{\beta} \log \sum_{\vct{t} \in T_J^\eps} \econst^{-\beta (\vct{\xi}(\alpha) \cdot \vct{t})^2 + q(\vct{t})}
$$
where $q : T_J^\eps \to \R$ is arbitrary.
The third derivative of this function satisfies
$$
\abs{r'''(\alpha)} \leq 48 \left( \max_{\vct{t} \in T_J^\eps} \abs{\vct{\xi}' \cdot \vct{t}}^3 \right)
	\left( \beta \Expect_{\vct{v}} \abs{ \vct{\xi}(\alpha) \cdot \vct{v} }
	+ \beta^2 \Expect_{\vct{v}} \abs{\vct{\xi}(\alpha) \cdot \vct{v}}^3 \right).
$$
In this expression, $\vct{v} \in T_J^\eps$ is a random vector that is independent from $\vct{\xi}(\alpha)$.
\end{sublemma}

\begin{proof}
Introduce a (parameterized) probability measure $\mu_\alpha$ on the set $T_J^\eps$:
$$
\mu_\alpha(\vct{t}) := \frac{1}{Z_\alpha} \econst^{ - \beta (\vct{\xi}(\alpha) \cdot \vct{t})^2 + q(\vct{t}) }
\quad\text{where}\quad
Z_\alpha := \sum_{\vct{t} \in T_J^\eps} \econst^{ - \beta (\vct{\xi}(\alpha) \cdot \vct{t})^2 + q(\vct{t}) }.
$$
We treat the normalizing factor $Z_\alpha$ as a function of $\alpha$,
and we write its derivatives as $Z_{\alpha}', Z_{\alpha}'', Z_{\alpha}'''$.
It is convenient to use the statistical mechanics notation for expectation
with respect to this measure.  For any function $h : T_J^\eps \to \R$,
$$
\muavg{ h(\vct{t}) } := \sum_{\vct{t} \in T_J^\eps} h(\vct{t}) \mu_\alpha(\vct{t}).
$$
For brevity, we always suppress the dependence of $\muavg{\cdot}$
on the parameter $\alpha$.  We often suppress the dependence of
$\vct{\xi}$ on $\alpha$ as well.

The function $r$ is proportional to the logarithm of the normalizing factor $Z_\alpha$:
$$
r(\alpha) = - \frac{1}{\beta} \log Z_\alpha.
$$
Thus, it is straightforward to express the third derivative of $r$ in terms
of the derivatives of $Z_{\alpha}$:
$$
r'''(\alpha) = - \frac{1}{\beta} \left( \frac{Z_\alpha'}{Z_\alpha} \right)''
	= - \frac{1}{\beta} \left( \frac{Z_\alpha''}{Z_\alpha} - \frac{(Z_\alpha')^2}{Z_\alpha^2} \right)'
	= - \frac{1}{\beta} \left( \frac{Z_\alpha'''}{Z_\alpha} - 3 \frac{Z_\alpha'}{Z_\alpha} \frac{Z_\alpha''}{Z_\alpha} + 2 \left(\frac{Z_\alpha'}{Z_\alpha}\right)^3 \right).
$$
By direct calculation, the derivative $Z_\alpha'$ satisfies
$$
\begin{aligned}
\frac{Z_\alpha'}{Z_\alpha}
	&= \frac{1}{Z_\alpha} \sum_{\vct{t} \in T_J^\eps} (- 2\beta )(\vct{\xi}' \cdot \vct{t}) (\vct{\xi}(\alpha) \cdot \vct{t})
	\econst^{- \beta (\vct{\xi}(\alpha) \cdot \vct{t})^2 + q(\vct{t})} \\
	&= - 2 \beta \muavg{ (\vct{\xi}' \cdot \vct{t}) (\vct{\xi} \cdot \vct{t}) }.
\end{aligned}
$$
The second derivative is
$$
\begin{aligned}
\frac{Z_\alpha''}{Z_\alpha} &= \frac{1}{Z_\alpha} \sum_{\vct{t} \in T_J^\eps} \left[ (- 2\beta )(\vct{\xi}' \cdot \vct{t})^2
	+ (2\beta)^2 (\vct{\xi}' \cdot \vct{t})^2(\vct{\xi}(\alpha) \cdot \vct{t})^2 \right]
	\econst^{- \beta (\vct{\xi}(\alpha) \cdot \vct{t})^2 + q(\vct{t})} \\
	&= - 2\beta \muavg{ (\vct{\xi}' \cdot \vct{t})^2 }
	+ 4\beta^2 \muavg{ (\vct{\xi}' \cdot \vct{t})^2(\vct{\xi}\cdot \vct{t})^2 }.
\end{aligned}
$$
The third derivative is
$$
\begin{aligned}
\frac{Z_\alpha'''}{Z_\alpha} &= \frac{1}{Z_\alpha} \sum_{\vct{t} \in T_J^\eps}
	\left[ 12 \beta^2 (\vct{\xi}' \cdot \vct{t})^3 (\vct{\xi}(\alpha) \cdot \vct{t})
	- 8 \beta^3 (\vct{\xi}' \cdot \vct{t})^3 (\vct{\xi}(\alpha) \cdot \vct{t})^3 \right]
	\econst^{- \beta (\vct{\xi}(\alpha) \cdot \vct{t})^2 + q(\vct{t})} \\
	&= 12 \beta^2 \muavg{ (\vct{\xi}' \cdot \vct{t})^3 (\vct{\xi} \cdot \vct{t}) }
	- 8 \beta^3 \muavg{ (\vct{\xi}' \cdot \vct{t})^3 (\vct{\xi} \cdot \vct{t})^3 } 
\end{aligned}
$$
We ascertain that
\begin{equation} \label{eqn:d3h-exact}
\begin{aligned}
r'''(\alpha) ={} & \beta \big[ -12 \muavg{ (\vct{\xi}' \cdot \vct{t})^3 (\vct{\xi} \cdot \vct{t}) }
	+ 12 \muavg{ (\vct{\xi}' \cdot \vct{t})^2 }
	\muavg{ (\vct{\xi}' \cdot \vct{t}) (\vct{\xi} \cdot \vct{t}) }
	  \big] \\
	&+ \beta^2 \big[ 8 \muavg{ (\vct{\xi}' \cdot \vct{t})^3 (\vct{\xi} \cdot \vct{t})^3 }
	- 24 \muavg{ (\vct{\xi}' \cdot \vct{t}) (\vct{\xi} \cdot \vct{t}) }
	\muavg{ (\vct{\xi}' \cdot \vct{t})^2(\vct{\xi} \cdot \vct{t})^2 }
	+ 16 \muavg{ (\vct{\xi}' \cdot \vct{t}) (\vct{\xi} \cdot \vct{t}) }^3
	 \big].
\end{aligned}
\end{equation}
This completes the calculation of the exact form of the third derivative of $r$.

Next, we simplify the formula~\eqref{eqn:d3h-exact} using basic probability
inequalities for the expectation $\muavg{\cdot}$.
First, consider the terms that are linear in $\beta$.  Observe that
$$
\abs{\muavg{ (\vct{\xi}' \cdot \vct{t})^3 (\vct{\xi} \cdot \vct{t}) }}
	\leq \muavg{ \abs{\vct{\xi}' \cdot \vct{t}}^3 \abs{\vct{\xi} \cdot \vct{t}} }
	\leq \bigg( \max_{\vct{t} \in T_J^\eps} \abs{\vct{\xi}' \cdot \vct{t}}^3 \bigg) \muavg{ \abs{\vct{\xi} \cdot \vct{t}} }.
$$
Indeed, Jensen's inequality allows us to draw the absolute value inside
the average, and we can invoke H{\"o}lder's inequality to pull out the maximum
of the first term.
Similarly, since $\muavg{1} = 1$,
$$ 
\abs{\muavg{ (\vct{\xi}' \cdot \vct{t})^2 } \muavg{ (\vct{\xi}' \cdot \vct{t})( \vct{\xi} \cdot \vct{t}) }}
	\leq \bigg( \max_{\vct{t} \in T_J^\eps} \abs{\vct{\xi}' \cdot \vct{t}}^3 \bigg)
	\muavg{ \abs{ \vct{\xi} \cdot \vct{t} } }
$$
Next, consider the terms that are quadratic in $\beta$.  The simplest is
$$
\abs{ \muavg{ (\vct{\xi}' \cdot \vct{t})^3 (\vct{\xi} \cdot \vct{t})^3 } }
	\leq \bigg( \max_{\vct{t} \in T_J^\eps} \abs{\vct{\xi}' \cdot \vct{t}}^3 \bigg)
	\muavg{ \abs{\vct{\xi} \cdot \vct{t}}^3 }.
$$
Using Jensen's inequality, we find that
$$
\abs{ \muavg{ (\vct{\xi}' \cdot \vct{t})( \vct{\xi} \cdot \vct{t}) }^3 }
	\leq \bigg( \max_{\vct{t} \in T_J^\eps} \abs{\vct{\xi}' \cdot \vct{t}}^3 \bigg)
	\muavg{ \abs{\vct{\xi} \cdot \vct{t} }}^3
	\leq \bigg( \max_{\vct{t} \in T_J^\eps} \abs{\vct{\xi}' \cdot \vct{t}}^3 \bigg)
	\muavg{ \abs{\vct{\xi} \cdot \vct{t} }^3 }.
$$
Last, using Lyapunov's inequality twice,
$$
\abs{\muavg{ (\vct{\xi}' \cdot \vct{t}) (\vct{\xi} \cdot \vct{t})}
\muavg{ (\vct{\xi}' \cdot \vct{t})^2 (\vct{\xi} \cdot \vct{t})^2 }}
	\leq \bigg( \max_{\vct{t} \in T_J^\eps} \abs{\vct{\xi}' \cdot \vct{t}}^3 \bigg)
	\muavg{ \abs{\vct{\xi}\cdot \vct{t}} }
	\muavg{ \abs{\vct{\xi}\cdot \vct{t}}^2 }
	\leq \bigg( \max_{\vct{t} \in T_J^\eps} \abs{\vct{\xi}' \cdot \vct{t}}^3 \bigg)
	\muavg{ \abs{\vct{\xi} \cdot \vct{t}}^3 }.
$$
Introduce the last five displays into~\eqref{eqn:d3h-exact} to arrive at the bound
\begin{equation} \label{eqn:d3h-bound}
\abs{r'''(\alpha)} \leq \bigg( \max_{\vct{t} \in T_J^\eps} \abs{\vct{\xi}' \cdot \vct{t}}^3 \bigg)
	\left(24 \beta \muavg{ \abs{ \vct{\xi} \cdot \vct{t}} }
	+ 48 \beta^2 \muavg{ \abs{\vct{\xi} \cdot \vct{t}}^3 } \right).
\end{equation}

Last, we want to replace the averages with respect to $\mu_\alpha$
by averages with respect to a simpler probability measure that
does not depend on $\vct{\xi}(\alpha)$.  This argument relies on a
correlation inequality.  Define another probability
measure $\mu_{*}$ on the set $T_J^\eps$:
$$
\mu_{*}(\vct{t}) := \frac{1}{Z_{*}} \econst^{q(\vct{t})}
\quad\text{where}\quad
Z_{*} := \sum_{\vct{t} \in T_J^\eps} \econst^{q(\vct{t})}.
$$
Write $\muavg{\cdot}_{*}$ for averages with respect to this new measure.
First, consider the average
$$
\begin{aligned}
\muavg{ \abs{\vct{\xi} \cdot \vct{t}} }
	&= \frac{1}{Z_\alpha} \sum_{\vct{t} \in T_J^\eps} \abs{\vct{\xi} \cdot \vct{t}}
	\econst^{-\beta (\vct{\xi} \cdot \vct{t})^2 + q(\vct{t})} \\
	&= \frac{Z_{*}}{Z_\alpha} \frac{1}{Z_{*}} \sum_{\vct{t} \in T_J^\eps} \abs{\vct{\xi} \cdot \vct{t}}
	\econst^{- \beta(\vct{\xi} \cdot \vct{t})^2} \econst^{q(\vct{t})} \\
	&= \frac{Z_{*}}{Z_\alpha} \muavg{ \abs{\vct{\xi} \cdot \vct{t}} \econst^{- \beta(\vct{\xi} \cdot \vct{t})^2} }_{*}.
\end{aligned}
$$
Let $X := \abs{\vct{\xi}\cdot \vct{t}}$ be the random variable obtained by pushing forward the measure $\mu_{*}$
from $T_J^\eps$ to the nonnegative real line.
Since $x \mapsto x$ is increasing and $x \mapsto \exp(- \beta x^2)$ is decreasing,
Chebyshev's association inequality~\cite[Thm.~2.14]{BLM13:Concentration-Inequalities}
provides that
$$
\muavg{ \abs{\vct{\xi} \cdot \vct{t}} \econst^{-\beta(\vct{\xi} \cdot \vct{t})^2} }_{*}
	\leq \muavg{ \abs{\vct{\xi} \cdot \vct{t}} }_{*} \muavg{ \econst^{-\beta (\vct{\xi} \cdot \vct{t})^2} }_{*}
	= \muavg{ \abs{\vct{\xi} \cdot \vct{t}} }_{*} 
	\frac{1}{Z_{*}} \sum_{\vct{t} \in T_J^\eps} \econst^{- \beta(\vct{\xi} \cdot \vct{t})^2 + q(\vct{t})}
	= \muavg{ \abs{\vct{\xi} \cdot \vct{t}} }_{*} \frac{Z_\alpha}{Z_{*}}.
$$
In summary,
\begin{equation} \label{eqn:mustar-abs}
\muavg{ \abs{\vct{\xi} \cdot \vct{t}} }
	\leq \muavg{ \abs{\vct{\xi} \cdot \vct{t}} }_{*}.
\end{equation}
The same argument shows that
\begin{equation} \label{eqn:mustar-cube}
\muavg{ \abs{\vct{\xi} \cdot \vct{t}}^3 }
	\leq \muavg{ \abs{\vct{\xi} \cdot \vct{t}}^3 }_{*}.
\end{equation}

Introduce~\eqref{eqn:mustar-abs} and~\eqref{eqn:mustar-cube}
into the bound~\eqref{eqn:d3h-bound} to reach the inequality
$$
r'''(\alpha) \leq 48 \bigg( \max_{\vct{t} \in T_J^\eps} \abs{\vct{\xi}' \cdot \vct{t}}^3 \bigg)
	\left( \beta \muavg{ \abs{ \vct{\xi} \cdot  \vct{t}} }_{*}
	+ \beta^2 \muavg{ \abs{\vct{\xi}\cdot \vct{t}}^3 }_{*} \right).
$$
The statement of the result follows when we reinterpret the averages with respect
to $\mu_{*}$ as expectations with respect to a random vector $\vct{v} \in T_J^\eps$
with distribution $Z_{*}^{-1} \econst^{-q(\vct{t})}$.
\end{proof}

\section{Theorem~\ref{thm:rsv-bdd}: Bounding the Restricted Singular Value by the Excess Width}

In Section~\ref{sec:replacement}, we showed that the restricted singular
value $\smin(\mtx{\Phi}; T_J)$ of the random matrix $\mtx{\Phi}$
does not change very much if we replace $\mtx{\Phi}$ with a
hybrid matrix $\mtx{\Psi}$, defined in~\eqref{eqn:rsv-hybrid},
that contains many standard normal
random variables.  Our next goal is to relate the restricted
singular value $\smin(\mtx{\Psi}; T_J)$ of the hybrid matrix
to the excess width of the set $T_J$.

\begin{proposition}[Theorem~\ref{thm:rsv-bdd}: Excess Width Bound] \label{prop:excess-width}
Let $\mtx{\Phi}$ be an $m \times n$ random matrix from Model~\ref{mod:bdd-mtx}
with magnitude bound $B$, and let $\mtx{\Gamma}$ be an $m \times n$ random matrix with
independent, standard normal entries.
Let $J$ be a subset of $\{1, \dots, n\}$ with cardinality $k$,
and let $T_J$ be a closed subset of $\mathsf{B}^n$.
Introduce the $m \times n$ random matrix $\mtx{\Psi} := \mtx{\Psi}(J)$
from~\eqref{eqn:rsv-hybrid}.
Then \begin{equation} \label{eqn:excess-width-lower}
\Expect \smin^2(\mtx{\Psi}; T_J) =
	\Expect \min_{\vct{t} \in T_J} \normsq{ \mtx{\Phi}_J \vct{t}_J + \mtx{\Gamma}_{J^c} \vct{t}_{J^c} }
	\geq \left( \coll{E}_m(T_J) - \cnst{C} B^2 \sqrt{k} \right)_+^2.
\end{equation}
Furthermore, if $T_J$ is convex and $k \geq m^{1/2}$,
\begin{equation} \label{eqn:excess-width-upper}
\Expect \smin^2(\mtx{\Psi}; T_J) =
	\Expect \min_{\vct{t} \in T_J} \normsq{ \mtx{\Phi}_J \vct{t}_J + \mtx{\Gamma}_{J^c} \vct{t}_{J^c} }
	\leq \left( \coll{E}_m(T_J) + \cnst{C} B^2 \sqrt{k} \right)_+^2.
\end{equation}
\end{proposition}

To obtain this result, we will invoke the Gaussian Minimax Theorem~\cite{Gor85:Some-Inequalities,Gor88:Milmans-Inequality}
to reduce the random matrix bounds to simpler bounds involving random vectors.  This approach has been
used to study the ordinary singular values of a Gaussian matrix~\cite[Sec.~2.3]{DS01:Local-Operator}.  It also plays a role in the analysis of restricted singular values of Gaussian matrices~\cite{Sto13:Regularly-Random,OTH13:Squared-Error,TOH15:Gaussian-Min-Max}.
The application here is complicated significantly by the presence
of the non-Gaussian matrix $\mtx{\Phi}_J$.

\subsection{Proof of Proposition~\ref{prop:excess-width}}

The proof of Proposition~\ref{prop:excess-width} involves \joelprev{several steps},
so it is helpful once again to summarize the calculations that are required.

First, let us explain the process by which we obtain the lower bound for a general closed subset $T_J$ of the Euclidean unit ball $\mathsf{B}^m$.  We will argue that
\begin{align*}
\Expect \min_{\vct{t} \in T_J} \normsq{ \mtx{\Phi}_J \vct{t}_J + \mtx{\Gamma}_{J^c} \vct{t}_{J^c} }
	&\geq \left( \Expect \min_{\vct{t} \in T_J} \norm{ \mtx{\Phi}_J \vct{t}_J + \mtx{\Gamma}_{J^c} \vct{t}_{J^c} } \right)_+^2 		&& \text{(Jensen)} \\
	&\geq \left( \Expect \min_{\vct{t} \in T_J}\left( \norm{ \mtx{\Phi}_J \vct{t}_J + \vct{h}\norm{\smash{\vct{t}_{J^c}}}}
	+ \vct{g}_{J^c} \cdot \vct{t}_{J^c} \right) - 2 \right)_+^2
		\qquad\qquad && \text{(Lemma~\ref{lem:apply-minmax-lower})}\\
	&\geq \left( \Expect \min_{\vct{t} \in T_J}\left( \sqrt{m} \norm{\vct{t}} + \vct{g}_{J^c} \cdot \vct{t}_{J^c} \right)
		- \cnst{C}B^2\sqrt{k} \right)_+^2
		&& \text{(Lemma~\ref{lem:remove-PhiK})}\\
	&\geq \left( \Expect \min_{\vct{t} \in T_J}\left( \sqrt{m} \norm{\vct{t}} + \vct{g} \cdot \vct{t} \right)
		- \cnst{C} B^2 \sqrt{k} \right)_+^2
		&& \text{(Lemma~\ref{lem:missing-coords-lower})} \\
	&= \left( \coll{E}_m(T_J) - \cnst{C} B^2 \sqrt{k} \right)_+^2.
		&& \text{(Definition~\ref{def:excess-width})}
\end{align*}
In the first line, we pass from the expected square to the square of the expectation,
which reduces the technical complexity of the rest of the argument.
Next, we replace the $m \times n$ Gaussian matrix $\mtx{\Gamma}$ with an expectation
over two independent standard normal vectors $\vct{h} \in \R^m$ and $\vct{g} \in \R^n$.
Third, we remove the remaining piece $\mtx{\Phi}_J$ of the original random matrix
to arrive at a formula involving only the random vector $\vct{g}$.  Finally,
we replace the missing coordinates in the random vector $\vct{g}$ to obtain
a bound in terms of the excess width $\coll{E}_m(T_J)$.
We arrive at the inequality~\eqref{eqn:excess-width-lower}.

The upper bound follows from a \joelprev{similar calculation}. Assuming that $T_J$ is convex, we may calculate that
\begin{align*}
\Expect \min_{\vct{t} \in T_J} \normsq{ \mtx{\Phi}_J \vct{t}_J + \mtx{\Gamma}_{J^c} \vct{t}_{J^c} }
	&\leq \left( \Expect \min_{\vct{t} \in T_J} \norm{ \mtx{\Phi}_J \vct{t}_J + \mtx{\Gamma}_{J^c} \vct{t}_{J^c} } + \cnst{C} B m^{1/4} \right)_+^2
		&& \text{(Lemma~\ref{lem:hybrid-rsv-poincare})} \\
	&\leq \left( \Expect \min_{\vct{t} \in T_J}\left( \norm{ \mtx{\Phi}_J \vct{t}_J + \vct{h}\norm{\smash{\vct{t}_{J^c}}}}
	+ \vct{g}_{J^c} \cdot \vct{t}_{J^c} \right) + \cnst{C} B m^{1/4} \right)_+^2
	 	&& \text{(Lemma~\ref{lem:apply-minmax-upper})} \\
	&\leq \left( \Expect \min_{\vct{t} \in T_J}\left( \sqrt{m} \norm{\vct{t}} + \vct{g}_{J^c} \cdot \vct{t}_{J^c} \right)
		+ \cnst{C}B^2\sqrt{k} \right)_+^2
		\qquad\qquad && \text{(Lemma~\ref{lem:remove-PhiK})}\\
	&\leq \left( \Expect \min_{\vct{t} \in T_J}\left( \sqrt{m} \norm{\vct{t}} + \vct{g} \cdot \vct{t} \right)
		+ \cnst{C}B^2\sqrt{k} \right)_+^2
		&& \text{(Lemma~\ref{lem:missing-coords-upper})} \\
	&= \left( \coll{E}_m(T_J)
		+ \cnst{C}B^2\sqrt{k} \right)_+^2.
		&& \text{(Definition~\ref{def:excess-width})}		
\end{align*}
The first step is a type of Poincar{\'e} inequality, which requires some specialized concentration results for the restricted singular value of the hybrid matrix.
The \joelprev{second step of this chain involves a convex duality argument
that is not present} in the proof of the lower bound~\eqref{eqn:excess-width-lower}.
We have used the assumption that $k \geq m^{1/2}$ to simplify the error term when
we pass from the second line to the third.
\joelprev{Altogether, this yields} the bound~\eqref{eqn:excess-width-upper}.

\subsection{Proposition~\ref{prop:excess-width}: Moment Comparison Inequality for the Hybrid RSV}

We require a moment comparison inequality to pass from the expected square of the restricted singular value of the hybrid matrix to the expectation of the restricted singular value itself.

\begin{lemma}[Proposition~\ref{prop:excess-width}: Moment Comparison]
\label{lem:hybrid-rsv-poincare}
Adopt the notation and hypotheses of Proposition~\ref{prop:excess-width}.
Then
$$
\Expect \min_{\vct{t} \in T_J} \normsq{ \mtx{\Phi}_J \vct{t}_J + \mtx{\Gamma}_{J^c} \vct{t}_{J^c} }
	\leq \left( \Expect \min_{\vct{t} \in T_J} \norm{ \mtx{\Phi}_J \vct{t}_J + \mtx{\Gamma}_{J^c} \vct{t}_{J^c} } + \cnst{C} B m^{1/4} \right)^2.
$$	
\end{lemma}

\begin{proof}
Define the random variable
$$
X := \min_{\vct{t} \in T_J} \norm{ \mtx{\Phi}_J \vct{t}_J + \mtx{\Gamma}_{J^c} \vct{t}_{J^c} }.
$$
We need a bound for $\Expect X^2$.  We obtain this result by applying a moment comparison inequality for $\mtx{\Phi}$, conditional on $\mtx{\Gamma}$, and then we apply a moment comparison inequality for $\mtx{\Gamma}$.

Sublemma~\ref{slem:hybrid-poincare}, applied conditionally, with the choice $\mtx{A} = \mtx{\Gamma}$, gives
$$
\Expect X^2 = \Expect \left[ \Expect\big[ X^2 \, \big\vert \, \mtx{\Gamma} \big] \right]
	\leq \Expect \left[ \big( \Expect[ X \,\vert\, \mtx{\Gamma} ] + \cnst{C}B m^{1/4} \big)^2 \right].
$$
The function $\mtx{\Gamma} \mapsto \Expect[ X \,\vert\, \mtx{\Gamma} ] + \cnst{C}B m^{1/4}$
is 1-Lipschitz, so the Gaussian variance inequality, Fact~\ref{fact:gauss-variance}, implies that
$$
\Expect \left[ \big( \Expect[ X \,\vert\, \mtx{\Gamma} ] + \cnst{C}B m^{1/4} \big)^2 \right]
	\leq \big( \Expect X + \cnst{C} B m^{1/4} \big)^2 + 1.
$$
Combine the last two displays, and adjust the constant to complete the proof.
\end{proof}

The proof of Lemma~\ref{lem:hybrid-rsv-poincare} requires a separate moment comparison result for a random variable that depends only on the original matrix $\mtx{\Phi}_J$.

\begin{sublemma}[Lemma~\ref{lem:hybrid-rsv-poincare}: Moment Comparison for Original Matrix]
	\label{slem:hybrid-poincare}
Adopt the notation and hypotheses of Lemma~\ref{lem:hybrid-rsv-poincare}.
For a fixed $m \times n$ matrix $\mtx{A}$,
\begin{equation} \label{eqn:hybrid-poincare}
\Expect \min_{\vct{t} \in T_J} \normsq{ \mtx{\Phi}_J \vct{t}_J + \mtx{A}_{J^c} \vct{t}_{J^c} }
	\leq \left( \Expect \min_{\vct{t} \in T_J} \norm{ \mtx{\Phi}_J \vct{t}_J + \mtx{A}_{J^c} \vct{t}_{J^c} } + \cnst{C} B m^{1/4} \right)^{2}.
\end{equation}
\end{sublemma}

\begin{proof}
Define the random variable
$$
X := \min_{\vct{t} \in T_J} \norm{ \mtx{\Phi}_J \vct{t}_J + \mtx{A}_{J^c} \vct{t}_{J^c} }.
$$
First, observe that
\begin{equation} \label{eqn:hybrid-poincare-pf1}
\sametprev{( \Expect X^2 )^{1/2} - \Expect X
	= \big((\Expect X^2)^{1/2} - \Expect X \big)_+}
	= \big(\Expect X - (\Expect X^2)^{1/2} \big)_-
	\leq \Expect \big(X - (\Expect X^2)^{1/2} \big)_-.
\end{equation}
The last inequality is Jensen's.
We can use concentration to bound this quantity.
\lang{Mutatis mutandis}, repeat the proof of equation~\eqref{eqn:rsv-lower-tail} from Proposition~\ref{prop:rsv-concentration} to see that, for all $\zeta \geq 0$,
$$
\Prob{ X^2 \leq \Expect X^2 - \cnst{C} B^2 \zeta } \leq \econst^{-\zeta^2 / m}.
$$
Invoke the subadditivity of the square root and change variables:
\begin{equation} \label{eqn:hybrid-poincare-pf2}
\Prob{ X \leq (\Expect X^2)^{1/2} - \cnst{C} B \zeta } \leq \econst^{-\zeta^4 / m}.
\end{equation}
Jensen's inequality and integration by parts deliver
$$
\Expect \big(X - (\Expect X^2)^{1/2} \big)_- =
\int_0^\infty \Prob{ X - ( \Expect X^2 )^{1/2} \leq - \zeta } \idiff{\zeta}
	\leq  \int_0^\infty \econst^{-\zeta^4/(\cnst{C} B^4 m)} \idiff{\zeta}
	\leq \cnst{C} B m^{1/4}.
$$
The third inequality follows from~\eqref{eqn:hybrid-poincare-pf2}.
Introduce the last display into~\eqref{eqn:hybrid-poincare-pf1} to reach
$$
(\Expect X^2)^{1/2} - \Expect X \leq \cnst{C} B m^{1/4}.
$$
Rearrange this relation to complete to proof of~\eqref{eqn:hybrid-poincare}.
\end{proof}

\subsection{Proposition~\ref{prop:excess-width}: The Role of the Gaussian Minimax Theorem}

To prove Proposition~\ref{prop:excess-width}, 
\joelprev{we must replace} the Gaussian matrix in the quantity
of interest with a pair of Gaussian vectors.  The key
to this argument is the following technical result.

\begin{lemma}[Proposition~\ref{prop:excess-width}: Application of Gaussian Minimax Theorem] \label{lem:my-comparison}
Adopt the notation and hypotheses of Proposition~\ref{prop:excess-width}.
Let $\mtx{A}$ be a fixed $m \times n$ matrix,
and let $\vct{g} \in \R^n$ and $\vct{h} \in \R^m$ be independent
standard normal random vectors.
Then, for all $\zeta \in \R$,
\begin{equation} \label{eqn:my-comparison-lower}
\Prob{ \min_{\vct{t} \in T_J} \norm{ \mtx{A}_J \vct{t}_J + \mtx{\Gamma}_{J^c} \vct{t}_{J^c} } \leq \zeta }
	\leq 2\, \Prob{ \min_{\vct{t} \in T_J} \left( \norm{ \mtx{A}_J \vct{t}_J +  \vct{h} \norm{\smash{\vct{t}_{J^c}}}  }
	+ \vct{g}_{J^c} \cdot \vct{t}_{J^c}\right) \leq \zeta }.
\end{equation}
Furthermore, if $T_J$ is convex and $\zeta\geq 0$,
\begin{equation} \label{eqn:my-comparison-upper}
\Prob{ \min_{\vct{t} \in T_J} \norm{ \mtx{A}_J \vct{t}_J + \mtx{\Gamma}_{J^c} \vct{t}_{J^c} } \geq \zeta } \leq 2\, \Prob{ \min_{\vct{t} \in T_J} \left( \norm{ \mtx{A}_J \vct{t}_J +  \vct{h} \norm{\smash{\vct{t}_{J^c}}}  }
	+ \vct{g}_{J^c} \cdot \vct{t}_{J^c}\right) \geq \zeta }.
\end{equation}
\joelprev{If $\zeta\leq 0$, the left-hand side is trivially equal to one.}
\end{lemma}

This result depends on the Gaussian Minimax Theorem~\cite{Gor85:Some-Inequalities};
see Fact~\ref{fact:gauss-minmax} for a statement.
Lemma~\ref{lem:my-comparison} is similar with early results of Gordon~\cite[Cor.~1.2]{Gor88:Milmans-Inequality}.
The detailed argument here is adapted from~\cite[Thm.~2.1]{TOH15:Gaussian-Min-Max};
see also Stojnic~\cite{Sto13:Regularly-Random}.

\begin{proof}
The basic idea is to express the quantity of interest as the value of a saddle-point problem:
$$
\min_{\vct{t} \in T_J} \norm{ \mtx{A}_J \vct{t}_J + \mtx{\Gamma}_{J^c} \vct{t}_{J^c} }
	= \min_{\vct{t} \in T_J} \max_{\vct{u} \in \mathsf{B}^m}
	\vct{u} \cdot \big( \mtx{A}_J \vct{t}_J + \mtx{\Gamma}_{J^c} \vct{t}_{J^c} \big).
$$
Then we apply the Gaussian Minimax Theorem
to obtain probabilistic lower bounds.  When $T_J$ is convex, we can also invoke convex duality
to interchange the minimum and maximum, which leads to complementary bounds.  To proceed with
this approach, however, it is convenient to work with a slightly different minimax problem.

Define the deterministic function
$$
\lambda(\vct{t}, \vct{u}) = \vct{u} \cdot (\mtx{A}_J \vct{t}_J)
\quad\text{for $\vct{t} \in T_J$ and $\vct{u} \in \mathsf{B}^{m}$.}
$$
Let $\gamma$ be a standard normal random variable, independent from everything else.
Introduce two centered Gaussian processes:
$$
X(\vct{t}, \vct{u}) := \vct{u} \cdot (\mtx{\Gamma}_{J^c} \vct{t}_{J^c})
	+ \norm{\vct{u}} \norm{\smash{\vct{t}_{J^c}}} \gamma
\quad\text{and}\quad
Y(\vct{t}, \vct{u}) := (\vct{u} \cdot \vct{h}) \norm{ \smash{\vct{t}_{J^c}} }
	+  \norm{ \vct{u} } (\vct{g} \cdot \vct{t}_{J^c})
$$
indexed by $\vct{t} \in T_J$ and $\vct{u} \in \mathsf{B}^{m}$.
Let us verify that these processes satisfy the conditions required
by the Gaussian Minimax Theorem, Fact~\ref{fact:gauss-minmax}.
First, for all parameters $\vct{t} \in T_J$ and $\vct{u} \in \mathsf{B}^{m}$,
\begin{equation} \label{eqn:minmax-cond1}
\Expect X(\vct{t}, \vct{u})^2 = 2 \normsq{\vct{u}} \normsq{\smash{\vct{t}_{J^c}}}
	= \Expect Y(\vct{t}, \vct{u})^2.
\end{equation}
Second, for all parameters $\vct{t}, \vct{t}' \in T_J$ and $\vct{u}, \vct{u}' \in \mathsf{B}^{m}$,
$$
\Expect\big[ X(\vct{t}, \vct{u}) X(\vct{t}', \vct{u}') \big]
	- \Expect\big[ Y(\vct{t}, \vct{u}) Y(\vct{t}', \vct{u}') \big]
	= \big( \norm{ \vct{u} } \norm{\smash{\vct{u}'}} - \vct{u} \cdot \vct{u}' \big)
	\big( \norm{ \smash{\vct{t}_{J^c}} } \norm{\smash{\vct{t}_{J^c}'} }
	- \vct{t}_{J^c} \cdot \vct{t}'_{J^c} \big).
$$
By the Cauchy--Schwarz inequality,
\begin{align}
\Expect\big[ X(\vct{t}, \vct{u}) X(\vct{t}, \vct{u}') \big]
	&= \Expect \big[ Y(\vct{t}, \vct{u}) Y(\vct{t}, \vct{u}') \big]; \label{eqn:minmax-cond2a} \\
\Expect\big[ X(\vct{t}, \vct{u}) X(\vct{t}', \vct{u}) \big]
	&= \Expect \big[ Y(\vct{t}, \vct{u}) Y(\vct{t}', \vct{u}) \big]; \label{eqn:minmax-cond2b} \\
\Expect\big[ X(\vct{t}, \vct{u}) X(\vct{t}', \vct{u}') \big]
	&\geq \Expect \big[ Y(\vct{t}, \vct{u}) Y(\vct{t}', \vct{u}') \big]. \label{eqn:minmax-cond3}
\end{align}
The formulas~\eqref{eqn:minmax-cond1},~\eqref{eqn:minmax-cond2a}, and~\eqref{eqn:minmax-cond3}
verify the conditions of the Gaussian Minimax Theorem.  Fact~\ref{fact:gauss-minmax} delivers
the bound
\begin{equation} \label{eqn:minmax-outcome-lower}
\Prob{ \min_{\vct{t} \in T_J} \max_{\vct{u} \in \mathsf{B}^m}
	\big( \lambda(\vct{t}, \vct{u}) + X(\vct{t}, \vct{u})  \big) > \zeta }
	\geq \Prob{ \min_{\vct{t} \in T_J} \max_{\vct{u} \in \mathsf{B}^m}
	\big( \lambda(\vct{t}, \vct{u}) + Y(\vct{t}, \vct{u}) \big) > \zeta }.
\end{equation}
This estimate does involve a small technicality.  We can only apply the Gaussian Minimax Theorem
to a finite subset of $T_J \times \mathsf{B}^m$, so we must make an approximation argument
to pass to the entire set.  We omit the details.

Now, let us determine the values of the saddle-point problems in~\eqref{eqn:minmax-outcome-lower}.  First,
\begin{equation}\label{eqn:max_over_ball}
\begin{aligned}
\min_{\vct{t} \in T_J} \max_{\vct{u} \in \mathsf{B}^m} \big( \lambda(\vct{t}, \vct{u}) + X(\vct{t}, \vct{u}) \big)
	&= \min_{\vct{t} \in T_J} \max_{\vct{u} \in \mathsf{B}^{m}} \big( \vct{u} \cdot ( \mtx{A}_J \vct{t}_J + \mtx{\Gamma}_{J^c} \vct{t}_{J^c}) + \norm{\vct{u}} \norm{ \smash{\vct{t}_{J^c}} } \gamma \big)\\
	&= \min_{\vct{t} \in T_J} \left( \norm{ \mtx{A}_J \vct{t}_J + \mtx{\Gamma}_{J^c} \vct{t}_{J^c} } + \norm{\smash{\vct{t}_{J^c}}} \gamma \right)_{\sametprev{+}}.
\end{aligned}
\end{equation}
\sametprev{Similarly,
$$
\min_{\vct{t} \in T_J} \max_{\vct{u} \in \mathsf{B}^m} \big( \lambda(\vct{t}, \vct{u}) + Y(\vct{t}, \vct{u}) \big)
	\geq \min_{\vct{t} \in T_J} \left( \norm{ \mtx{A}_J \vct{t}_J + \vct{h} \norm{\smash{\vct{t}_{J^c}}} }
	+ \vct{g}_{J^c} \cdot \vct{t}_{J^c} \right).
$$}

\sametprev{Next, we remove the term involving $\gamma$ from the right-hand side of~\eqref{eqn:max_over_ball}.
To do so, we condition on $\gamma > 0$ and $\gamma \leq 0$ and calculate that
$$
\begin{aligned}
\Prob{ \min_{\vct{t} \in T_J} \max_{\vct{u} \in \mathsf{B}^m}
	\big( \lambda(\vct{t}, \vct{u}) + X(\vct{t}, \vct{u})  \big) > \zeta }
	&= \Prob{ \min_{\vct{t} \in T_J} \max_{\vct{u} \in \mathsf{B}^{m}} \big( \norm{ \mtx{A}_J \vct{t}_J + \mtx{\Gamma}_{J^c} \vct{t}_{J^c} } + \norm{ \smash{\vct{t}_{J^c}} } \gamma \big)_+ > \zeta } \\
	&\leq \frac{1}{2} + \frac{1}{2} \Prob{ \min_{\vct{t} \in T_J} \left( \norm{ \mtx{A}_J \vct{t}_J + \mtx{\Gamma}_{J^c} \vct{t}_{J^c} } + \norm{\smash{\vct{t}_{J^c}}} \gamma \right)_+  > \zeta \, \Big\vert \, \gamma \leq 0 } \\
	&\leq \frac{1}{2} + \frac{1}{2} \Prob{\min_{\vct{t} \in T_J} \left( \norm{ \mtx{A}_J \vct{t}_J + \mtx{\Gamma}_{J^c} \vct{t}_{J^c} }  \right)_+  > \zeta }\\
	&= \frac{1}{2} + \frac{1}{2} \Prob{ \min_{\vct{t} \in T_J} \norm{ \mtx{A}_J \vct{t}_J + \mtx{\Gamma}_{J^c} \vct{t}_{J^c} } > \zeta }.
\end{aligned}
$$
On the other hand,
$$
\Prob{\min_{\vct{t} \in T_J} \max_{\vct{u} \in \mathsf{B}^m} \big( \lambda(\vct{t}, \vct{u}) + Y(\vct{t}, \vct{u}) \big)>\zeta}
	\geq \Prob{\min_{\vct{t} \in T_J} \left( \norm{ \mtx{A}_J \vct{t}_J + \vct{h} \norm{\smash{\vct{t}_{J^c}}} }
	+ \vct{g}_{J^c} \cdot \vct{t}_{J^c} \right)> \zeta}.
$$
Introduce the last two displays into~\eqref{eqn:minmax-outcome-lower} to reach
$$
\frac{1}{2} + \frac{1}{2} \Prob{ \min_{\vct{t} \in T_J} \norm{ \mtx{A}_J \vct{t}_J + \mtx{\Gamma}_{J^c} \vct{t}_{J^c} } > \zeta }
	\geq \Prob{ \min_{\vct{t} \in T_J} \big( \norm{ \mtx{A}_J \vct{t}_J + \vct{h} \norm{\smash{\vct{t}_{J^c}}} }
	+\vct{g}_{J^c} \cdot \vct{t}_{J^c} \big) > \zeta }.
$$}

\noindent
Take the complements of both probabilities and rearrange to conclude that~\eqref{eqn:my-comparison-lower} is correct.

The second result~\eqref{eqn:my-comparison-upper} requires an additional duality argument.
If we replace the function $\lambda$
and the random processes $X$ and $Y$ with their negations, all of the variance calculations above remain valid.
In particular, the relations~\eqref{eqn:minmax-cond1},~\eqref{eqn:minmax-cond2b}, and~\eqref{eqn:minmax-cond3}
permit us to apply Fact~\ref{fact:gauss-minmax} with the roles of $T_J$ and $\mathsf{B}^m$
reversed.  This step yields
$$
\Prob{ \min_{\vct{u} \in \mathsf{B}^m} \max_{\vct{t} \in T_J} \big(- \lambda(\vct{t}, \vct{u}) - X(\vct{t}, \vct{u})  \big) > -\zeta }
	\geq \Prob{ \min_{\vct{u} \in \mathsf{B}^m} \max_{\vct{t} \in T_J} \big( - \lambda(\vct{t}, \vct{u}) - Y(\vct{t}, \vct{u}) \big) > -\zeta }.
$$
Let us examine the saddle-point problems.  Since $- \lambda - X$ is bilinear and the sets $\mathsf{B}^m$ and $T_J$ are compact and convex, the Sion Minimax Theorem~\cite{Sio58:General-Minimax} allows us to interchange the minimum and maximum.  Thus
$$
\begin{aligned}
\min_{\vct{u} \in \mathsf{B}^m} \max_{\vct{t} \in T_J} \big( - \lambda(\vct{t}, \vct{u}) - X(\vct{t}, \vct{u})  \big)
	&= \max_{\vct{t} \in T_J} \min_{\vct{u} \in \mathsf{B}^m} \big( - \lambda(\vct{t}, \vct{u}) - X(\vct{t}, \vct{u}) \big) \\
	&= - \min_{\vct{t} \in T_J} \max_{\vct{u} \in \mathsf{B}^m} \big( \lambda(\vct{t}, \vct{u}) + X(\vct{t}, \vct{u})   \big) \\
	&= - \min_{\vct{t} \in T_J} \big( \norm{ \mtx{A}_J \vct{t}_J + \mtx{\Gamma}_{J^c} \vct{t}_{J^c} } + \norm{\smash{\vct{t}_{J^c}}} \gamma \big)_{\sametprev{+}}. 
\end{aligned}
$$
Similarly,
$$
\min_{\vct{u} \in \mathsf{B}^m} \max_{\vct{t} \in T_J} \big(  - \lambda(\vct{t}, \vct{u}) - Y(\vct{t}, \vct{u}) \big)
	=  - \min_{\vct{t} \in T_J} \big( \norm{ \mtx{A}_J \vct{t}_J + \vct{h} \norm{\smash{\vct{t}_{J^c}}} }
	+ \vct{g}_{J^c} \cdot \vct{t}_{J^c} \big)_{\sametprev{+}}.
$$
Combining the last three displays, we reach
$$
\Prob{ \min_{\vct{t} \in T_J} \big( \norm{ \mtx{A}_J \vct{t}_J + \mtx{\Gamma}_{J^c} \vct{t}_{J^c} } + \norm{\smash{\vct{t}_{J^c}}} \gamma \big)_+ < \zeta }
	\geq \Prob{ \min_{\vct{t} \in T_J} \big( \norm{ \mtx{A}_J \vct{t}_J + \vct{h} \norm{\smash{\vct{t}_{J^c}}} }
	+ \vct{g}_{J^c} \cdot \vct{t}_{J^c} \big)_+ < \zeta }.
$$
Proceeding as before, \sametprev{conditioning on the sign of $\gamma$}, we may remove the dependence on $\gamma$ from the left-hand side to obtain
$$
\Prob{ \min_{\vct{t} \in T_J} \norm{ \mtx{A}_J \vct{t}_J + \mtx{\Gamma}_{J^c} \vct{t}_{J^c} } \geq \zeta }
	\geq 2\Prob{ \min_{\vct{t} \in T_J} \big( \norm{ \mtx{A}_J \vct{t}_J + \vct{h} \norm{\smash{\vct{t}_{J^c}}} }
	+ \vct{g}_{J^c} \cdot \vct{t}_{J^c} \big)_+ \geq \zeta }.
$$
\sametprev{Clearly, this inequality is useful only for $\zeta\geq 0$ in which case the right hand side is equal to
$$
2\Prob{ \min_{\vct{t} \in T_J} \norm{ \mtx{A}_J \vct{t}_J + \vct{h} \norm{\smash{\vct{t}_{J^c}}} }
	+ \vct{g}_{J^c} \cdot \vct{t}_{J^c}  \geq \zeta }.
$$
We confirm that~\eqref{eqn:my-comparison-upper} is true.}
\end{proof}

\subsection{Proposition~\ref{prop:excess-width}: Reducing the Gaussian Matrix to Some Gaussian Vectors}

Our next goal is to convert the probability bounds from Lemma~\ref{lem:my-comparison}
into expectation bounds.  Lemma~\ref{lem:apply-minmax-lower} gives a lower bound that
is valid for every subset $T_J$ of the unit ball, and Lemma~\ref{lem:apply-minmax-upper}
gives an upper bound that is valid when $T_J$ is also convex.

\begin{lemma}[Proposition~\ref{prop:excess-width}: Reducing the Gaussian Matrix---Lower Bound] \label{lem:apply-minmax-lower}
Adopt the notation and hypotheses of Proposition~\ref{prop:excess-width}.
Let $\mtx{A}$ be a fixed $m \times n$ matrix,
and let $\vct{g} \in \R^n$ and $\vct{h} \in \R^m$
be independent standard normal vectors.
Then
\begin{equation} \label{eqn:remove-gauss-lower}
\Expect \min_{\vct{t} \in T_J} \norm{ \mtx{A}_J \vct{t}_J + \mtx{\Gamma}_{J^c} \vct{t}_{J^c} }
	\geq \Expect \min_{\vct{t} \in T_J} \left( \norm{ \mtx{A}_J \vct{t}_J +  \vct{h} \norm{\smash{\vct{t}_{J^c}}}  }
	+ \vct{g}_{J^c} \cdot \vct{t}_{J^c}\right)
	- 2.
\end{equation}
In particular,
$$
\Expect \min_{\vct{t} \in T_J} \norm{ \mtx{\Phi}_J \vct{t}_J + \mtx{\Gamma}_{J^c} \vct{t}_{J^c} }
	\geq \Expect \min_{\vct{t} \in T_J} \left( \norm{ \mtx{\Phi}_J \vct{t}_J +  \vct{h} \norm{\smash{\vct{t}_{J^c}}}  }
	+ \vct{g}_{J^c} \cdot \vct{t}_{J^c}\right)
	- 2.
$$
\end{lemma}

\begin{proof}
We make the abbreviations
$$
X := \min_{\vct{t} \in T_J} \norm{ \mtx{A}_J \vct{t}_J + \mtx{\Gamma}_{J^c} \vct{t}_{J^c} }
\quad\text{and}\quad
Y := \min_{\vct{t} \in T_J} \left( \norm{ \mtx{A}_J \vct{t}_J +  \vct{h} \norm{\smash{\vct{t}_{J^c}}}  }
	+ \vct{g}_{J^c} \cdot \vct{t}_{J^c}\right).
$$
First, note that
$$
\Expect Y - \Expect X \leq (\Expect Y - \Expect X)_+ = (\Expect X - \Expect Y)_-
	\leq \Expect (X - \Expect Y)_- .
$$
The last inequality is Jensen's.
To bound the right-hand side, we use integration by parts
and formula~\eqref{eqn:my-comparison-lower} from Lemma~\ref{lem:my-comparison}:
$$
\Expect (X - \Expect Y)_{-}
	= \int_{0}^\infty \Prob{ X - \Expect Y \leq - \zeta } \idiff{\zeta}
	\leq \int_{0}^\infty 2 \, \Prob{ Y - \Expect Y \leq - \zeta } \idiff{\zeta}
	= 2 \Expect (Y - \Expect Y)_-.
$$
Finally, we make the estimates
$$
\Expect[ (Y - \Expect Y)_- ]
	\leq \Var[ Y ]^{1/2}
	\leq 1.
$$
The last inequality follows from the Gaussian variance inequality, Fact~\ref{fact:gauss-variance}.
Indeed, the function $(\vct{g}, \vct{h}) \mapsto Y(\vct{g}, \vct{h})$ is 1-Lipschitz
because the set $T_J$ is contained in the unit ball.
In summary,
$$
\Expect Y - \Expect X \leq \Expect (X - \Expect Y)_-
	\leq 2 \Expect (Y - \Expect Y)_-
	\leq 2.
$$
We arrive the advertised bound~\eqref{eqn:remove-gauss-lower}.
\end{proof}

\begin{lemma}[Proposition~\ref{prop:excess-width}: Reducing the Gaussian Matrix---Upper Bound] \label{lem:apply-minmax-upper}
Adopt the notation and hypotheses of Proposition~\ref{prop:excess-width}.
Let $\mtx{A}$ be a fixed $m \times n$ matrix,
and let $\vct{g} \in \R^n$ and $\vct{h} \in \R^m$
be independent standard normal vectors.
Then
\begin{equation} \label{eqn:remove-gauss-upper}
\Expect \min_{\vct{t} \in T_J} \norm{ \mtx{A}_J \vct{t}_J + \mtx{\Gamma}_{J^c} \vct{t}_{J^c} }
	\leq \Expect \min_{\vct{t} \in T_J} \left( \norm{ \mtx{A}_J \vct{t}_J +  \vct{h} \norm{\smash{\vct{t}_{J^c}}}  }
	+ \vct{g}_{J^c} \cdot \vct{t}_{J^c}\right)
	+ 2.
\end{equation}
In particular,
$$
\Expect \min_{\vct{t} \in T_J} \norm{ \mtx{\Phi}_J \vct{t}_J + \mtx{\Gamma}_{J^c} \vct{t}_{J^c} }
	\leq \Expect \min_{\vct{t} \in T_J} \left( \norm{ \mtx{\Phi}_J \vct{t}_J +  \vct{h} \norm{\smash{\vct{t}_{J^c}}}  }
	+ \vct{g}_{J^c} \cdot \vct{t}_{J^c}\right)
	+ 2.
$$
\end{lemma}

\begin{proof}
We follow the same pattern as in Lemma~\ref{lem:apply-minmax-lower}.  Using the same notation, we calculate that
$$
\Expect X - \Expect Y \leq (\Expect X - \Expect Y)_+ \leq \Expect (X - \Expect Y)_+
	\leq 2 \Expect( Y - \Expect Y)_+
	\leq 2.
$$
In this case, we invoke~\eqref{eqn:my-comparison-upper}
from Lemma~\ref{lem:my-comparison} to obtain the penultimate inequality.
\end{proof}

\subsection{Proposition~\ref{prop:excess-width}: Removing the Remaining Part of the Original Random Matrix}

The next step in the proof of Proposition~\ref{prop:excess-width}
is to remove the remaining section of the random matrix $\mtx{\Phi}$
from the bounds in Lemmas~\ref{lem:apply-minmax-lower} and~\ref{lem:apply-minmax-upper}.

\begin{lemma}[Proposition~\ref{prop:excess-width}: Removing the Original Matrix] \label{lem:remove-PhiK}
Adopt the notation and hypotheses of Proposition~\ref{prop:excess-width}.
Let $\mtx{X}$ be an $m \times n$ random matrix with independent, standardized entries
that satisfy Assumption~\ref{eqn:subgauss-hyp}.
Then
$$
\abs{ \Expect \min_{\vct{t} \in T_J} \left( \norm{ \mtx{X}_J \vct{t}_J +  \vct{h} \norm{\smash{\vct{t}_{J^c}}} }
	+ \vct{g}_{J^c} \cdot \vct{t}_{J^c} \right)
	- \Expect \min_{\vct{t} \in T_J} \left( \sqrt{m} \norm{\vct{t}} + \vct{g}_{J^c} \cdot \vct{t}_{J^c} \right) }
	\leq \cnst{C} B^2 \sqrt{k}.
$$
In particular, this conclusion is valid when $\mtx{X} = \mtx{\Phi}$. \end{lemma}

\begin{proof}
Abbreviate $\mtx{\Psi} := \begin{bmatrix} \mtx{X}_J & \vct{h} \end{bmatrix}$.  
Observe that
$$
\mtx{X}_J \vct{t}_J + \vct{h} \norm{ \smash{\vct{t}_{J^c}} } = \mtx{\Psi} \begin{bmatrix}\vct{t}_J \\ \norm{\smash{\vct{t}_{J^c}}}
\end{bmatrix}
\quad\text{and}\quad
\norm{ \begin{bmatrix}\vct{t}_J \\ \norm{\smash{\vct{t}_{J^c}}} \end{bmatrix} } = \norm{\vct{t}}.
$$
Therefore, we have the deterministic bounds
\begin{equation} \label{eqn:bordered-sval-est}
\sigma_{\min}( \mtx{\Psi} ) \norm{\vct{t} }
	\leq \norm{ \mtx{X}_J \vct{t}_J + \vct{h} \norm{ \smash{\vct{t}_{J^c}} } }
	\leq \sigma_{\max}(  \mtx{\Psi} ) \norm{ \vct{t} }.
\end{equation}
The $m \times (k+1)$ random matrix $\mtx{\Psi}$ has independent, standardized entries
that satisfy~\eqref{eqn:subgauss-hyp}.
For all $\zeta \geq 0$,
Fact~\ref{fact:subgauss-mtx-tails} shows that its extreme singular values satisfy the probability bounds
\begin{equation} \label{eqn:subgauss-sval-prob}
\begin{aligned}
\Prob{ \sigma_{\max}(\mtx{\Psi}) \geq \sqrt{m} + \cnst{C}_0 B^2 \sqrt{k} + \zeta } &\leq \econst^{-\cnst{c}_0 \zeta^2 / B^4},
	\quad\text{and} \\
\Prob{ \sigma_{\min}(\mtx{\Psi}) \leq \sqrt{m} - \cnst{C}_0 B^2 \sqrt{k} - \zeta } &\leq \econst^{-\cnst{c}_0 \zeta^2 / B^4}.
\end{aligned}
\end{equation}
These inequalities allow us to treat the singular values of $\mtx{\Psi}$
as if they were equal to $\sqrt{m}$.

We may now perform the following estimates:
$$
\begin{aligned}
\Expect \min_{\vct{t} \in T_J} \big( \norm{\mtx{X}_J \vct{t}_J + \vct{h} \norm{\smash{\vct{t}_{J^c}}}}
	+ \vct{g}_{J^c} \cdot \vct{t}_{J^c} \big)
	&\leq \Expect \min_{\vct{t} \in T_J} \big( \sigma_{\max}(\mtx{\Psi}) \norm{\vct{t}}
	+ \vct{g}_{J^c} \cdot \vct{t}_{J^c} \big) \\
	&\leq \Expect \left[ \min_{\vct{t} \in T_J} \big( \sqrt{m} \norm{\vct{t}}
	+ \vct{g}_{J^c} \cdot \vct{t}_{J^c} \big)
	+ \big( \sigma_{\max}(\mtx{\Psi}) - \sqrt{m} \big)_+ \max_{\vct{t} \in T_J} \norm{\vct{t}} \right] \\
	&\leq \Expect \min_{\vct{t} \in T_J} \big( \sqrt{m} \norm{\vct{t}}
	+ \vct{g}_{J^c} \cdot \vct{t}_{J^c} \big) + \Expect \big( \sigma_{\max}(\mtx{\Psi}) - \sqrt{m} \big)_+.
\end{aligned}
$$
The first inequality is~\eqref{eqn:bordered-sval-est}.  Then we add and subtract $\sqrt{m}$
from the maximum singular value.  Last, we recall that $T_J$ is a subset of the unit ball.
Set $\alpha := \cnst{C}_0 B^2 \sqrt{k}$, and calculate that
\begin{equation} \label{eqn:expect-smax-sqrt}
\begin{aligned}
\Expect \big( \sigma_{\max}(\mtx{\Psi}) - \sqrt{m} \big)_+
	&= \int_0^\infty \Prob{ \sigma_{\max}(\mtx{\Psi}) - \sqrt{m} \geq \zeta }\idiff{\zeta} \\
	&\leq \int_0^\alpha \idiff{\zeta} + \int_0^\infty \Prob{ \sigma_{\max}(\mtx{\Psi}) - \sqrt{m} \geq \alpha + \zeta } \idiff{\zeta} \\
	&\leq \alpha + \int_0^\infty \econst^{-\cnst{c}_0 \zeta^2/B^4} \idiff{\zeta}
	= \alpha + \cnst{C} B^2.
\end{aligned}
\end{equation}
We split the integral at $\alpha$, and we change variables in the second integral.
For the first integral, we use the trivial bound of one on the probability.  Then
we invoke the probability inequality~\eqref{eqn:subgauss-sval-prob}.

Combining the last two displays and collecting constants, we arrive at
$$
\Expect \min_{\vct{t} \in T_J} \big( \norm{\mtx{X}_J \vct{t}_J + \vct{h} \norm{\smash{\vct{t}_{J^c}}}}
	+ \vct{g}_{J^c} \cdot \vct{t}_{J^c} \big)
	\leq \Expect \min_{\vct{t} \in T_J} \big( \sqrt{m} \norm{\vct{t}}
	+ \vct{g}_{J^c} \cdot \vct{t}_{J^c} \big) + \cnst{C} B^2 \sqrt{k}.
$$
An entirely similar argument delivers a matching lower bound.
Together, these estimates complete the proof.
\end{proof}

\subsection{Proposition~\ref{prop:excess-width}: Replacing the Coordinates Missing from the Excess Width}

The last step in the proof of Proposition~\ref{prop:excess-width}
is to examine the excess-width-like functional from Lemma~\ref{lem:remove-PhiK}
that involves the term $\vct{g}_{J^c} \cdot \vct{t}_{J^c}$.  We must show that
this term does not change very much if we reintroduce the coordinates listed in $J$.
Lemma~\ref{lem:missing-coords-lower} gives the easy proof of the lower bound.
Lemma~\ref{lem:missing-coords-upper} contains the upper bound. 
\begin{lemma}[Proposition~\ref{prop:excess-width}: Missing Coordinates---Lower Bound] \label{lem:missing-coords-lower}
Adopt the notation and hypotheses of Proposition~\ref{prop:excess-width}.  Then
$$
\Expect \min_{\vct{t} \in T_J} \big( \sqrt{m} \norm{\vct{t}} + \vct{g}_{J^c} \cdot \vct{t}_{J^c} \big)
\geq \Expect \min_{\vct{t} \in T_J} \big( \sqrt{m} \norm{\vct{t}} + \vct{g} \cdot \vct{t} \big).
$$
\end{lemma}

\begin{proof}
This result is an immediate consequence of Jensen's inequality:
$$
\begin{aligned}
\Expect \min_{\vct{t} \in T_J} \left( \sqrt{m} \norm{\vct{t}} + \vct{g}_{J^c} \cdot \vct{t}_{J^c} \right)
	&= \Expect  \min_{\vct{t} \in T_J} \left( \sqrt{m} \norm{\vct{t}} + \vct{g}_{J^c} \cdot \vct{t}_{J^c}
	+ \Expect_{\vct{g}_J} (\vct{g}_J \cdot \vct{t}_J) \right) \\
	&\geq \Expect \min_{\vct{t} \in T_J} \left( \sqrt{m} \norm{\vct{t}} + \vct{g}_{J^c} \cdot \vct{t}_{J^c}
	+ \vct{g}_J \cdot \vct{t}_J \right) \\
	&= \Expect \min_{\vct{t} \in T_J} \big( \sqrt{m} \norm{\vct{t}} + \vct{g} \cdot \vct{t} \big)
\end{aligned}
$$
We rely on the fact that $\vct{g}_J$ and $\vct{g}_{J^c}$ are independent
standard normal vectors.
\end{proof}

\joelprev{
\begin{lemma}[Proposition~\ref{prop:excess-width}: Missing Coordinates---Upper Bound] \label{lem:missing-coords-upper}
Adopt the notation and hypotheses of Proposition~\ref{prop:excess-width}.  Then
$$
\Expect \min_{\vct{t} \in T_J} \big( \sqrt{m} \norm{\vct{t}} + \vct{g}_{J^c} \cdot \vct{t}_{J^c} \big)
\leq \Expect \min_{\vct{t} \in T_J} \big( \sqrt{m} \norm{\vct{t}} + \vct{g} \cdot \vct{t} \big) + \sqrt{k}.
$$
\end{lemma}
\begin{proof}
This calculation is also easy:
$$
\begin{aligned}
\Expect \min_{\vct{t} \in T_J} \big( \sqrt{m} \norm{\vct{t}} + \vct{g}_{J^c} \cdot \vct{t}_{J^c} \big)
	&= \Expect \min_{\vct{t} \in T_J} \big( \sqrt{m} \norm{\vct{t}} + \vct{g} \cdot \vct{t} - \vct{g}_{J} \cdot \vct{t}_{J} \big) \\
	&\leq \Expect \min_{\vct{t} \in T_J} \big( \sqrt{m} \norm{\vct{t}} + \vct{g} \cdot \vct{t} \big) + \Expect \norm{ \smash{\vct{g}_{J}} } \\
	&\leq \Expect \min_{\vct{t} \in T_J} \big( \sqrt{m} \norm{\vct{t}} + \vct{g} \cdot \vct{t} \big) + \sqrt{k}.
\end{aligned}
$$
In the last step, we have used the fact that $\# J = k$.
\end{proof}}

\section{Proof of Corollary~\ref{cor:rsv-four} from Theorem~\ref{thm:rsv-bdd} by Truncation}
\label{sec:rsv-truncation}

In this section, we show how to establish Corollary~\ref{cor:rsv-four}
as a consequence of Theorem~\ref{thm:rsv-bdd}.  The proof depends on
a truncation argument.

\subsection{Proof of Corollary~\ref{cor:rsv-four}}

Fix parameters $p > 4$ and $\nu \geq 1$. Assume that $\mtx{\Phi}$
is an $m \times n$ matrix that follows Model~\ref{mod:p-mom-mtx}
with parameters $p$ and $\nu$.  Let $T$ be a compact subset of
the unit ball $\mathsf{B}^n$.  We will prove the lower bound for the
minimum singular value of $\mtx{\Phi}$ restricted to $T$.
When $T$ is convex, an entirely similar approach yields
the corresponding upper bound.

Fix a truncation parameter $R$ that satisfies $R^{p/2-1} \geq 2 \nu^{p/2}$.
Decompose the random matrix $\mtx{\Phi}$
as
$$
\mtx{\Phi} = \mtx{\Phi}_{\rm trunc} + \mtx{\Phi}_{\rm tail}
$$
by applying the truncation described below in Lemma~\ref{lem:trunc-rv}
separately to each entry of $\mtx{\Phi}$.  This procedure ensures that
$\mtx{\Phi}_{\rm trunc}$ contains independent, symmetric, standardized entries,
each bounded by $2R$.  In other words, $\mtx{\Phi}_{\rm trunc}$ follows
Model~\ref{mod:bdd-mtx} with $B = 2R$.
The tail $\mtx{\Phi}_{\rm tail}$ contains
independent, centered entries, each with variance bounded by $\cnst{C} \nu^p R^{2 - p}$
and whose $p$th moment is bounded by $(2 \nu)^p$.

We can control the restricted singular values of $\mtx{\Phi}$ using the
triangle inequality:
\begin{equation} \label{eqn:four-mom-pf1}
\smin(\mtx{\Phi}; T)
	= \min_{\vct{t} \in T} \norm{\mtx{\Phi} \vct{t}}
	\geq \min_{\vct{t} \in T} \norm{\mtx{\Phi}_{\rm trunc} \vct{t}}
	- \norm{ \mtx{\Phi}_{\rm tail} }
	= \smin(\mtx{\Phi}_{\rm trunc}; T) - \norm{ \mtx{\Phi}_{\rm tail} }.
\end{equation}
We bound the restricted singular value of the bounded matrix $\mtx{\Phi}_{\rm trunc}$
using Theorem~\ref{thm:rsv-bdd}.  To bound $\norm{\mtx{\Phi}_{\rm tail}}$,
we apply a simple norm estimate, Fact~\ref{fact:heavy-tail-norm}, based on the
matrix Rosenthal inequality~\cite[Thm.~I]{Tro15:Expected-Norm}.

Since $\mtx{\Phi}_{\rm trunc}$ follows Model~\ref{mod:bdd-mtx} with $B = 2R$,
Theorem~\ref{thm:rsv-bdd}\eqref{it:rsv-bdd-conc} and~\eqref{it:rsv-bdd-expect-lower}
give the probability bound
$$
\Prob{ \smin^2(\mtx{\Phi}_{\rm trunc}; T) \leq \big( \coll{E}_m(T) \big)_+^2 - \cnst{C} R^2 (m+n)^{0.92} 
	- \cnst{C} R^2 \zeta } \leq \econst^{-\zeta^2/(m+n)}.
$$
Select $\zeta = (m+n)^{0.92}$ to make the tail probability negligible:
$$
\Prob{ \smin^2(\mtx{\Phi}_{\rm trunc}; T) \leq \big( \coll{E}_m(T) \big)_+^2 - \cnst{C} R^2 (m+n)^{0.92} }
	\leq \econst^{-(m+n)^{0.84}}.
$$
Taking square roots inside the event, we reach
\begin{equation} \label{eqn:four-mom-pf2}
\Prob{ \smin(\mtx{\Phi}_{\rm trunc}; T) \leq \big( \coll{E}_m(T) \big)_+ - \cnst{C} R (m+n)^{0.46} }
	\leq \econst^{-(m+n)^{0.84}}.
\end{equation}
This step depends on the subadditivity of the square root.

Meanwhile, the entries of $\mtx{\Phi}_{\rm tail}$ are centered, have variances at most
$\cnst{C} \nu^p R^{2 - p}$, and have $p$th moments bounded by $(2\nu)^p$.
Therefore, we can apply the norm bound for heavy-tailed
random matrices, Fact~\ref{fact:heavy-tail-norm}, to see that
$$
\Prob{ \norm{ \mtx{\Phi}_{\rm tail} } \geq \cnst{C} \sqrt{ \nu^p R^{2-p} (m+n) \log(m+n) }
	+ \big( \cnst{C} \nu (m + n)^{2/p} \log(m+n)  \big) \, \zeta } \leq \zeta^{-p}.
$$
Define the positive quantity $\eps$ via the relation $4(1 + \eps) := p$.
Select $\zeta = (m+n)^{\eps/p}$ to obtain
\begin{equation} \label{eqn:four-mom-pf3}
\Prob{ \norm{ \mtx{\Phi}_{\rm tail} } \geq \cnst{C} \nu^{p/2} R^{1 - p/2} \sqrt{(m+n) \log(m+n)}
	+ \cnst{C} \nu (m+n)^{(2+\eps)/p} \log(m+n)} \leq (m+n)^{- \eps}.
\end{equation}
The key point here is that we can arrange for $\norm{\mtx{\Phi}_{\rm tail}}$
to have order $o(\sqrt{m+n})$ with high probability.

Combine~\eqref{eqn:four-mom-pf1},~\eqref{eqn:four-mom-pf2}, and~\eqref{eqn:four-mom-pf3}
to reach
\begin{multline*}
\mathbb{P}\bigg\{ \smin(\mtx{\Phi}; T) \leq \big(\coll{E}_m(T)\big)_+ - \cnst{C} R (m+n)^{0.46} \\
	- \cnst{C} \nu^{p/2} R^{1 - p/2} \sqrt{(m+n)\log(m+n)} - \cnst{C} \nu (m+n)^{(2+\eps)/p} \log(m+n) \bigg\} \\
	\leq \econst^{-(m+n)^{0.84}} + (m+n)^{-\eps}.
\end{multline*}
Set the truncation parameter $R = \nu (m+n)^{0.02/(1+\eps)}$
to equate the exponents on $m + n$ in the two terms that
depend on $R$.  Then simplify using $p = 4(1+\eps)$ to obtain
\begin{multline*}
\Prob{ \smin(\mtx{\Psi}; T) \leq \big(\coll{E}_m(T)\big)_+
	- \cnst{C} \nu \big( (m+n)^{0.5 - 0.02(1+2\eps)/(1+\eps)} + (m+n)^{0.5 - (\eps/4)/(1+\eps)} \big) \log(m+n) } \\
	\leq \econst^{-(m+n)^{0.84}} + (m+n)^{-\eps}. 
\end{multline*}
Note that both powers in the event are bounded away from $1/2$,
so we can absorb the logarithm by increasing the power slightly.
Furthermore, we can introduce a function $\kappa(p)$ that is
strictly positive for $p > 4$ to reach the inequality
$$
\Prob{ \smin(\mtx{\Psi}; T) \leq \big(\coll{E}_m(T)\big)_+
	- \cnst{C}_p \nu (m+n)^{0.5 - \kappa(p)} }
	\leq \econst^{-(m+n)^{0.84}} + (m+n)^{1 - p/4}. 
$$
The constant $\cnst{C}_p$ depends only on the parameter $p$.
The exponential vanishes faster than any polynomial,
so we can combine the terms on the right-hand side
to complete the proof of~\eqref{eqn:smin-four-lower}.

For the upper bound, we use Theorem~\ref{thm:rsv-bdd}\eqref{it:rsv-bdd-expect-cvx}
to control the expectation of the restricted singular value.
\joel{In this case, the error term in the expectation bound becomes
$R^4 (m+n)^{0.94}$.  This change presents no new difficulties, and we arrive at the
result~\eqref{eqn:smin-four-upper} by a slight modification of the argument.}

\subsection{Corollary~\ref{cor:rsv-four}: Truncation of Individual Random Variables}

In this section, we describe a truncation procedure for scalar random variables.
The arguments here are entirely standard, but we include details for completeness.

\begin{lemma}[Corollary~\ref{cor:rsv-four}: Truncation] \label{lem:trunc-rv}
Let $\phi$ be a random variable that satisfies the properties
listed in Model~\ref{mod:p-mom-mtx}.
Let $R$ be a parameter that satisfies $R^{p/2 - 1} \geq 2 \nu^{p/2}$.
Then we have the decomposition
\begin{equation} \label{eqn:psi-trunc-tail}
\phi = \phi_{\rm trunc} + \phi_{\rm tail}
\end{equation}
where
$$
\Expect \phi_{\rm trunc} = 0, \quad
\Expect \phi_{\rm trunc}^2 = 1, \quad\text{and}\quad
\abs{ \smash{\phi_{\rm trunc}} } \leq 2R
$$
and
$$
\Expect \phi_{\rm tail} = 0, \quad
\Expect \phi_{\rm tail}^2 \leq \frac{6\nu^p}{R^{p-2}}, \quad\text{and}\quad
\Expect \abs{\smash{\phi_{\rm tail}}}^p \leq (2\nu)^p.
$$
\end{lemma}

\begin{proof}
Define the random variable
$$
\phi_{\rm trunc} := \frac{\phi \mathbb{1}\{ \abs{\smash{\phi}} \leq R \}}{\alpha}
\quad\text{where}\quad
\alpha^2 := \Expect\big[ \phi^2 \mathbb{1}\{ \abs{\smash{\phi}} \leq R \} \big].
$$
Since $\phi$ is standardized and symmetric, $\phi_{\rm trunc}$ is also standardized and symmetric.  To ensure that the decomposition~\eqref{eqn:psi-trunc-tail} holds, we must set
\begin{equation} \label{eqn:phi-tail-defn}
\phi_{\rm tail} := \phi \mathbb{1}\{ \abs{\smash{\phi}} > R \} - \frac{1 - \alpha}{\alpha} \phi \mathbb{1}\{\abs{\smash{\phi}} \leq R\}.
\end{equation}
The random variable $\phi_{\rm tail}$ is also centered because of the symmetry of $\phi$.

To establish the other properties of $\phi_{\rm trunc}$ and $\phi_{\rm tail}$,
we need to calculate some expectations.  First, using integration by parts
and Markov's inequality,
\begin{equation} \label{eqn:phi2-tail}
\begin{aligned}
\Expect\big[ \phi^2 \mathbb{1}\{ \abs{\smash{\phi}} > R \} \big]
	&= \int_0^R 2 \zeta \Prob{ \abs{\smash{\phi}} > R } \idiff{\zeta}
	+ \int_R^\infty 2 \zeta \Prob{ \abs{\smash{\phi}} > \zeta } \idiff{\zeta} \\
	&\leq \int_0^R 2 \zeta \frac{\Expect \abs{\smash{\phi}}^p }{R^p} \idiff{\zeta}
	+ \int_R^\infty 2 \zeta \frac{\Expect \abs{\smash{\phi}}^p }{\zeta^p} \idiff{\zeta}
	\leq \frac{2 \nu^p}{R^{p-2}}.
\end{aligned}
\end{equation}
In the last step, we used the assumption that $p \geq 4$.
A similar calculation shows that
$$
\begin{aligned}
\alpha^2 = \Expect\big[ \phi^2 \mathbb{1}\{\abs{\smash{\phi}} \leq R \} \big]
	&= \int_0^R 2 \zeta \Prob{ \abs{\smash{\phi}} > \zeta } \idiff{\zeta}  \\
	&= \int_0^\infty 2 \zeta \Prob{ \abs{\smash{\phi}} > \zeta } \idiff{\zeta}
	- \int_R^\infty 2 \zeta \Prob{ \abs{\smash{\phi}} > \zeta } \idiff{\zeta} \\
	&\geq \Expect\big[ \phi^2 \big] - \frac{2\nu^p}{(p-2) R^{p-2}}
	\geq 1 - \frac{\nu^p}{R^{p-2}}. 
\end{aligned}
$$
The last relation holds because $\phi$ is standardized.  It follows that
\begin{equation} \label{eqn:trunc-var-bound}
\alpha \geq 1 - \frac{\nu^{p/2}}{R^{p/2-1}} \geq \frac{1}{2}.
\end{equation}
The last estimate holds because $p \geq 4$
and of the assumption $R^{p/2-1} \geq 2 \nu^{p/2}$.

We are now prepared to verify the uniform bound on $\phi_{\rm trunc}$:
$$
\abs{\smash{\phi_{\rm trunc}}}
	= \abs{ \frac{\phi \mathbb{1}\{\abs{\smash{\phi}} \leq R\}}{\alpha} }
	\leq \frac{R}{\alpha}
	\leq 2R.
$$
The last inequality follows from~\eqref{eqn:trunc-var-bound}.

Next, we need to bound the variance of $\phi_{\rm tail}$.  We have
$$
\begin{aligned}
\Expect\big[ \phi_{\rm tail}^2 \big]
	= \Expect\big[\phi^2 \mathbb{1}\{ \abs{\smash{\phi}} > R \} \big]
	+ \left(\frac{1-\alpha}{\alpha}\right)^2 \Expect\big[ \phi^2 \mathbb{1}\{ \abs{\smash{\phi}} \leq R \} \big]
	\leq \frac{2 \nu^p}{R^{p-2}} + \left( \frac{\nu^{p/2}/R^{p/2-1}}{1/2} \right)^2 \Expect\big[ \psi^2 \big]
	\leq \frac{6 \nu^p}{R^{p-2}}.
\end{aligned}
$$
The first identity holds because the two indicators are orthogonal random variables.
The second relation uses the expectation calculation~\eqref{eqn:phi2-tail} and
the estimate~\eqref{eqn:trunc-var-bound};
we have dropped the indicator in the second expectation.
The last estimate holds because $\phi$ is standardized.

Last, we need to check the moment inequality for $\phi_{\rm tail}$.
This estimate follows by applying the triangle inequality
to the definition~\eqref{eqn:phi-tail-defn}:
$$
\left( \Expect \abs{\smash{\phi_{\rm tail}}}^p \right)^{1/p}
	\leq \left( \Expect \abs{ \smash{\phi} }^p \right)^{1/p}
	+ \frac{1 - \alpha}{\alpha} \left( \Expect \abs{ \smash{\phi}}^p \right)^{1/p}
	= \frac{1}{\alpha} \left( \Expect \abs{ \smash{\phi}}^p \right)^{1/p}
	\leq 2 \nu.
$$
We have dropped the indicators after invoking the triangle inequality.
Finally, we introduced the estimate~\eqref{eqn:trunc-var-bound}.
\end{proof}

\newpage

\part{Universality of the Embedding Dimension: Proof of Theorem~\ref{thm:univ-embed}(b)}
\label{part:rap}

This part of the paper develops a condition under which the random projection
of a set fails with high probability.  This argument establishes
the second part of the universality law for the embedding dimension,
Theorem~\ref{thm:univ-embed}(b).

Section~\ref{sec:rap-bdd} contains the main technical result,
a condition under which a bounded random matrix maps a point
in a set to the origin.  Section~\ref{sec:rap-four} extends
this argument to the heavy-tailed random matrix model,
Model~\ref{mod:p-mom-mtx}.  In Section~\ref{sec:car-to-embed},
we use the latter result to derive Theorem~\ref{thm:univ-embed}(b).
The remaining parts of the section lay out the supporting argument.

\section{When Embedding Fails for a Bounded Random Matrix}
\label{sec:rap-bdd}

In this section, we introduce a \joelprev{functional} whose
value determines whether a linear transformation maps a
point in a set to the origin.  Then we present
the main technical result, which gives an estimate for this
functional evaluated on a random \joelprev{linear map} from Model~\ref{mod:bdd-mtx}.
The rest of the section outlines the main steps in the proof
of the result.

\subsection{The RAP Functional: Dual Condition for Failure}

To study when a \joelprev{linear map} maps a point a set to the origin,
we use an approach based on polarity.
Let us make the following definition.

\begin{definition}[Range Avoids Polar (RAP)]\label{rap def}
Let $K \subset \R^m$ be a closed, convex cone.
Let $\mtx{A}$ be an $m \times n$ matrix.
Define the quantity
$$
\tau_{\min}(\mtx{A}; K) := \min_{\norm{\vct{t}} = 1} \min_{\vct{s} \in K^\polar} \norm{ \vct{s} - \mtx{A}\vct{t} }.
$$
Note that the range of the inner minimum involves the polar $K^\polar$ of the cone.
We refer to $\tau_{\min}$ as the \term{RAP functional}.
\end{definition}

To see why the RAP functional is important,
consider a closed, spherically convex subset $\Omega$ of $\mathsf{S}^{m-1}$
\joelprev{for which $\cone(\Omega)$ is not a subspace.}
The second conclusion of Proposition~\ref{prop:annihilate} states that
$$
\tau_{\min}(\mtx{A}; \cone(\Omega)) > 0
\quad\text{implies}\quad
\vct{0} \in \mtx{A}(\Omega).
$$
\joelprev{The third conclusion of Proposition~\ref{prop:annihilate} 
gives a similar result in the case of a subspace.}
Therefore, we can obtain a sufficient condition that $\mtx{A}$
maps a point in $\Omega$ to zero by providing a lower bound
for the RAP functional.

\subsection{Theorem~\ref{thm:car-bdd}: Main Result for the Bounded Random Matrix Model}

The main technical result in this part of the paper is a theorem
on the behavior of the RAP functional of a bounded
random matrix.

\begin{theorem}[RAP: Bounded Random Matrix Model] \label{thm:car-bdd}
Place the following assumptions:

\begin{itemize}
\item	Let $m$ and $n$ be natural numbers with $m+n \leq \min\{m,n\}^{9/8}$.

\item	Let $K$ be a closed, convex cone in $\R^m$, and define $\Omega := K \cap \mathsf{S}^{m-1}$.

\item	Draw an $m \times n$ random matrix $\mtx{\Phi}$ from Model~\ref{mod:bdd-mtx}
with bound $B$.
\end{itemize}

\noindent
Then the squared RAP functional $\tau_{\min}^2(\mtx{\Phi}; K)$
has the following properties.

\begin{enumerate}
\item	\label{it:car-bdd-concentration}
\joel{The squared RAP functional deviates below its mean on
a scale of $B^2 \sqrt{m + n}$:}
$$
\begin{aligned}
\Prob{ \tau_{\min}^2(\mtx{\Phi}; K) \leq \Expect \tau_{\min}^2(\mtx{\Phi}; K)
	- \cnst{C}B^2 \zeta }
	&\leq \econst^{-\zeta^2/m} \end{aligned}	
$$

\item	\label{it:car-bdd-width}
The expected square of the RAP functional is bounded below:
$$
\begin{aligned}
\Expect \tau_{\min}^2(\mtx{\Phi}; K)
	&\geq \frac{\big( \coll{E}_n(\Omega) \big)_-^2}{\cnst{C} B^2 \log(m+n)}
	- \cnst{C} B^3 (m+n)^{0.95}.
\end{aligned}
$$
\end{enumerate}

\noindent
This result \emph{does} use the symmetry assumption in Model~\ref{mod:bdd-mtx}.
\end{theorem}

The proof of Theorem~\ref{thm:car-bdd} is \joelprev{long},
even though we can borrow a lot from the proof of Theorem~\ref{thm:rsv-bdd}.
This section contains an overview of the calculations that are required
with forward references to the detailed arguments.

\subsection{Proof of Theorem~\ref{thm:car-bdd}: \joel{Lower Tail Bound}}

\joel{Theorem~\ref{thm:car-bdd}\eqref{it:car-bdd-concentration} states
that the quantity $\tau_{\min}^2(\mtx{\Phi}; K)$ does not deviate
substantially below its mean.  The proof is similar to the proof
of the bound~\eqref{eqn:rsv-lower-tail} in Proposition~\ref{prop:rsv-concentration},
which shows that the squared restricted singular value $\sigma_{\min}^2(\mtx{\Phi}; K)$
does not deviate substantially below its mean.  We omit the repetitive details.}

\subsection{Proof of Theorem~\ref{thm:car-bdd}: Truncation and Dissection}

Let us proceed with the proof of Theorem~\ref{thm:car-bdd}\eqref{it:car-bdd-width}.
Define the sets
$$
S := \sametprev{K^\circ} \cap R \mathsf{B}^m
\quad\text{and}\quad
T := \mathsf{S}^{n-1}
$$
where the radius $R := \cnst{C}_{\rm rad} B^2 \sqrt{m + n}$ for some universal constant $\cnst{C}_{\rm rad}$.
Next, we construct a family of closed, convex subsets of $S$ and $T$.
For each $I \subset \{1, \dots, m\}$ and each $J \subset \{1, \dots, n \}$, define
$$
\begin{aligned}
S_I &:= \big\{ \vct{s} \in S : \abs{ s_i } \leq R (\#I)^{-1/2} \text{ for all $i \in I^c$} \big\}; \\
T_J &:= \big\{ \vct{t} \in T : \abs{\smash{t_j} } \leq (\#J)^{-1/2} \text{ for all $j \in J^c$} \big\}.
\end{aligned}
$$
Fix the cardinality parameter $k \in \{ 1, \dots, \min\{m,n\} \}$.
As in the proof of Theorem~\ref{thm:rsv-bdd}, we have the decompositions
$$
S = \bigcup_{\#I = k} S_I
\quad\text{and}\quad
T = \bigcup_{\#J = k} T_J.
$$
Furthermore,
\begin{equation} \label{eqn:IJ-count}
\# \{ S_I : \#I = k\} \times \# \{ T_J : \#J = k \} \leq \left( \frac{ \econst(m + n) }{ k } \right)^k. 
\end{equation}
We maintain the heuristic that the cardinality $k$ is much smaller than either
ambient dimension $m$ or $n$.

\subsection{Proof of Theorem~\ref{thm:car-bdd}: Lower Bound for RAP Functional}

To bound the quantity $\tau_{\min}(\mtx{\Phi}; K)$,
we must \sametprev{combine estimates from several technical results.}
We give an outline of the calculation here,
with the details postponed to a series of
propositions.

First, we must account for the error we incur when we truncate the cone $K^\polar$
to the wedge $S$.  Proposition~\ref{prop:car-cone-truncation} demonstrates that
\begin{equation} \label{eqn:car-bdd-pf1}
\Expect \tau_{\min}^2(\mtx{\Phi}; K)
	= \Expect \min_{\norm{\vct{t}} = 1} \min_{\vct{s} \in K^\polar} \normsq{ \vct{s} - \mtx{\Phi} \vct{t} }
	\geq \Expect \min_{\vct{t} \in T} \min_{\vct{s} \in S} \normsq{ \mtx{s} - \mtx{\Phi} \vct{t} }
	- \cnst{C}B^4.
\end{equation}
This inequality is based on an estimate for the norm of the point $\vct{s} \in K^\polar$ where the inner minimum is achieved, as well as a probability bound for the norm of the random matrix $\mtx{\Phi}$.

Next, we pass from the minimum over the full sets $S$ and $T$ to minima over their subsets
$S_I$ and $T_J$:
\begin{equation} \label{eqn:car-bdd-pf2}
\Expect \min_{\vct{t} \in T} \min_{\vct{s} \in S} \normsq{\vct{s} - \mtx{\Phi}\vct{t}}
	\geq \min_{\#I = \#J = k} \Expect  \min_{\vct{t} \in T_J} \min_{\vct{s} \in S_I}
	\normsq{ \vct{s} - \mtx{\Phi}\vct{t} }
	- \cnst{C} B^2 \sqrt{km \log((m+n)/k)}.
\end{equation}
The proof of this inequality hews to the argument in Proposition~\ref{prop:rsv-dissection}.
We just need to invoke the concentration inequality from Theorem~\ref{thm:car-bdd}\eqref{it:car-bdd-concentration}
in lieu of the concentration inequality from Proposition~\ref{prop:rsv-concentration},
and we take into account the bound~\eqref{eqn:IJ-count} on the number of subsets in the
decomposition.  Further details are omitted.

We are now prepared to perform the exchange argument to pass from the matrix $\mtx{\Phi}$
to a matrix $\mtx{\Psi}$ that contains many standard normal entries.  Fix subsets $I \subset \{1, \dots, m\}$
and $J \subset \{1, \dots, n\}$, each with cardinality $k$.  Introduce the random matrix
$$
\mtx{\Psi} := \mtx{\Psi}(I, J)
	:= \begin{bmatrix} \mtx{\Phi}_{IJ} & \mtx{\Phi}_{IJ^c} \\ \mtx{\Phi}_{I^c J} & \mtx{\Gamma}_{I^c J^c}
	\end{bmatrix}
$$
where $\mtx{\Gamma}$ is an $m \times n$ standard normal matrix.  Proposition~\ref{prop:partial-replacement-redux}
gives the bound  \begin{equation} \label{eqn:car-bdd-pf3}
\begin{aligned}
\Expect \min_{\vct{s} \in S_I} \min_{\vct{t} \in T_J} \normsq{ \vct{s} - \mtx{\Phi} \vct{t} }
	&\geq \Expect \min_{\vct{t} \in T_J} \min_{\vct{s} \in S_I}
	\normsq{ \vct{s} - \mtx{\Psi}(I, J) \vct{t} }
	- \frac{\cnst{C}B^3 (m+n)^{11/6} \log(mn)}{k} \\
	&\geq \Expect \min_{\norm{\vct{t}} = 1} \min_{\vct{s} \in K^\polar}
	\normsq{ \vct{s} - \mtx{\Psi}(I, J) \vct{t} }
	- \frac{\cnst{C}B^3 (m+n)^{11/6} \log(mn)}{k}
\end{aligned}
\end{equation}
The proof is similar with the proof of Proposition~\ref{prop:partial-replacement}.
We discretize both sets; we smooth the minima using the soft-min function;
and then we apply the Lindeberg principle.  The main distinction is that we
can replace even less of the matrix $\mtx{\Phi}$ than before.  The second line in~\eqref{eqn:car-bdd-pf3}
is an immediate consequence of the facts $S_I \subset S \subset K^\polar$ and $T_J \subset T = \mathsf{S}^{n-1}$.

To continue, we must identify a geometric functional that is hiding within the expression~\eqref{eqn:car-bdd-pf3}.
Write $\Omega := K \cap \mathsf{S}^{m-1}$.
Proposition~\ref{prop:car-width} demonstrates that
\begin{equation} \label{eqn:car-bdd-pf4}
\begin{aligned}
\Expect \min_{\norm{\vct{t}} = 1} \min_{\vct{s} \in K^\polar}
	\normsq{ \vct{s} - \mtx{\Psi}(I, J) \vct{t} }
	&\geq \left( \frac{\big( \coll{E}_n(\Omega) \big)_-}{\cnst{C} B \sqrt{\log m}}
	- \cnst{C}B^2\sqrt{k \log m} \right)_+^2. \\
\end{aligned}
\end{equation}
As in Proposition~\ref{prop:excess-width}, the main tool is the Gaussian Minimax Theorem,
Fact~\ref{fact:gauss-minmax}, which allows us to break down the standard normal matrix
$\mtx{\Gamma}$ into simpler quantities.  The proof requires some convex duality arguments,
as well as some delicate considerations that did not arise before.

Last, we linearize the function $(\cdot)_+^2$ in~\eqref{eqn:car-bdd-pf4}:
\begin{equation} \label{eqn:car-bdd-pf5}
\left( \frac{\big(\coll{E}_n(\Omega)\big)_-}{\cnst{C} B \sqrt{\log m}}
	- \cnst{C}B^2\sqrt{k \log m} \right)_+^2
	\geq \frac{\big( \coll{E}_n(\Omega) \big)_-^2}{\cnst{C} B^2 \log m}
	- \cnst{C}B\sqrt{k m}.
\end{equation}
We have employed the observation that
$$
\big( \coll{E}_n(\Omega) \big)_-
\leq \Expect \max_{\vct{s} \in \Omega} \vct{g} \cdot \vct{s}
\leq \Expect \norm{\smash{\vct{g}}}
\leq \sqrt{m}.
$$
Here, $\vct{g} \in \R^m$ is a standard normal vector.

Now, sequence the estimates~\eqref{eqn:car-bdd-pf1}, \eqref{eqn:car-bdd-pf2},
\eqref{eqn:car-bdd-pf3}, \eqref{eqn:car-bdd-pf4}, and~\eqref{eqn:car-bdd-pf5}
to arrive at
$$
\Expect \tau_{\min}^2(\mtx{\Phi}; K)
	\geq \frac{\big( \coll{E}_n(\Omega) \big)_-^2}{\cnst{C} B^2 \log m}
	- \cnst{C}B^3 \left( \frac{(m+n)^{11/6} \log(mn)}{k} 
	+ \sqrt{km \log((m+n)/k)} \right).
$$
We select $k = \lfloor (m+n)^{8/9} \rfloor$, which results in the bound
$$
\Expect \tau_{\min}^2(\mtx{\Phi}; K)
\geq \frac{\big( \coll{E}_n(\Omega) \big)_-^2}{\cnst{C} B^2 \log(m+n)}
	- \cnst{C} B^3 (m+n)^{17/18} \log(m+n).
$$
Note that $k \leq \min\{m,n\}$, as required, because we have assumed that
$m + n \leq \min\{m,n\}^{9/8}$.  We obtain the result quoted in Theorem~\ref{thm:car-bdd}\eqref{it:car-bdd-width}.

\section{When Embedding Fails for a Heavy-Tailed Random Matrix}
\label{sec:rap-four}

In this section, we extend Theorem~\ref{thm:car-bdd}
to the heavy-tailed matrix model, Model~\ref{mod:p-mom-mtx}.
In Section~\ref{sec:rsv-four-to-embed-univ}, we show how to
derive the second half of the universality result for the embedding
dimension, Theorem~\ref{thm:univ-embed}\eqref{eqn:univ-embed-succ}, as a consequence.

\subsection{Corollary~\ref{cor:car-four}: Main Result for the $p$-Moment Random Matrix Model}

The following corollary extends Theorem~\ref{thm:car-bdd}
to include random matrices drawn from Model~\ref{mod:p-mom-mtx}.

\begin{corollary}[RAP: $p$-Moment Random Matrix Model] \label{cor:car-four}
Fix parameters $p > 4$ and $\nu \geq 1$.  Place the following assumptions:

\begin{itemize}
\item	Let $m$ and $n$ be natural numbers with $m+n \leq \min\{m,n\}^{9/8}$.

\item	Let $K$ be a closed, convex cone in $\R^m$, and define $\Omega := K \cap \mathsf{S}^{m-1}$.

\item	Draw an $m \times n$ random matrix $\mtx{\Phi}$ that satisfies Model~\ref{mod:p-mom-mtx} with given $p$ and $\nu$.
\end{itemize}

\noindent
Then the RAP functional satisfies the probability bound
$$
\Prob{ \tau_{\min}(\mtx{\Phi}; K) \leq
	\frac{\big( \coll{E}_n(\Omega) \big)_-}{\cnst{C}_p \nu (m+n)^{(1 - \cnst{c}_p) \kappa(p)}}
	- \cnst{C}_p \nu^{3/2} (m+n)^{1/2 - \kappa(p)} }
	\leq \cnst{C}_{p} (m+n)^{1 - p/4}.
$$
The function $\kappa(p)$ is strictly positive for $p > 4$.
The strictly positive constants $\cnst{c}_p$ and $\cnst{C}_{p}$
depend only on $p$.
\end{corollary}

\noindent
The proof of Corollary~\ref{cor:car-four} follows from
Theorem~\ref{thm:car-bdd} and the same kind of truncation
argument that appears in Section~\ref{sec:rsv-truncation}.
We omit further details. 

\subsection{Proof of Theorem~\ref{thm:univ-embed}(\ref{eqn:univ-embed-fail}) from Corollary~\ref{cor:car-four}}
\label{sec:car-to-embed}

Theorem~\ref{thm:univ-embed}\eqref{eqn:univ-embed-fail} is an easy consequence of Corollary~\ref{cor:car-four}.  Let us restate the assumptions of the theorem:

\begin{itemize}
\item	$E$ is a compact subset of $\R^D$ that does not contain the origin.

\item	The statistical dimension of $E$ satisfies $\delta(E) \geq \varrho D$.

\item	The $d \times D$ random \joelprev{linear map} $\mtx{\Pi}$ follows Model~\ref{mod:p-mom-mtx}
with parameters $p > 4$ and $\nu \geq 1$.
\end{itemize}

\noindent
We must now consider the regime where the embedding dimension
$d \leq (1 - \eps) \, \delta(E)$.  We need to demonstrate that
\begin{equation} \label{eqn:embed-b-pf1}
\Prob{ \vct{0} \in \mtx{\Pi}(E) }
	= \Prob{ E \cap \nullsp(\mtx{\Pi}) \neq \emptyset }
	\geq 1 - \cnst{C}_p D^{1 - p/4}.
\end{equation}
As in Section~\ref{sec:rsv-four-to-embed-univ}, the probability
is a decreasing function of the embedding dimension,
so we may as well consider the case where $d = \lfloor (1 - \eps) \, \delta(E) \rfloor$.
It is easy to see that \joelprev{$d \leq D \leq \varrho^{-1} d$} because $\delta(E) \leq d$.
\joelprev{In particular, $d + D \leq \min\{d, D\}^{9/8}$.}

Introduce the spherical retraction $\Omega := \vct{\theta}(E)$.
\joelprev{If $\cone(\Omega)$ is a subspace, we replace $\Omega$ by
a subset $\Omega_0$ with the property that $\cone(\Omega_0)$
is a subspace of one dimension fewer than $\cone(\Omega)$.
Then we proceed with the argument.}

Proposition~\ref{prop:annihilate} and relation~\eqref{eqn:univ-embed-spherical}
show that \begin{equation}\label{useful rap converse}
\tau_{\min}(\mtx{\Pi}^\adj; \cone(\Omega)) > 0
\quad\text{implies}\quad
\vct{0} \in \mtx{\Pi}(\Omega)
\quad\text{implies}\quad
\vct{0} \in \mtx{\Pi}(E).
\end{equation}
Therefore, to verify~\eqref{eqn:embed-b-pf1},
it is enough to produce a high-probability lower bound
on the RAP functional $\tau_{\min}(\mtx{\Pi}^\adj; \cone(\Omega))$.
With the choices $\mtx{\Phi} = \mtx{\Pi}^\adj$ and $K = \cone(\Omega)$,
Corollary~\ref{cor:car-four} yields
$$
\Prob{ \tau_{\min}(\mtx{\Pi}^\adj; \cone(\Omega) )
	\geq \frac{\big( \coll{E}_d(\Omega) \big)_-}{\cnst{C}_p \nu (d+D)^{(1-\cnst{c}_p)\kappa(p)}}
	- \cnst{C}_p \nu^{3/2} (d + D)^{1/2 - \kappa(p)} }
	\geq 1 - \cnst{C}_p D^{1 - p/4}.
$$
We need to check that the lower bound
on the RAP functional is positive.  That is,
\begin{equation} \label{eqn:embed-b-pf1.5}
\big( \coll{E}_d(\Omega) \big)_-
	> \cnst{C}_p \nu^{5/2} (d + D)^{1/2 - \cnst{c}_p \kappa(p)}.
\end{equation}
Once again, this point follows from two relatively short calculations.

Since $\Omega$ is a subset of the unit sphere,
we quickly compute the excess width using~\eqref{eqn:excess-vs-sdim}:
\begin{equation} \label{eqn:embed-b-pf2}
\begin{aligned}
\big(\coll{E}_d(\Omega)\big)_-
	\geq \sqrt{\delta(\Omega) - 1} - \sqrt{d} 
	= \sqrt{\delta(E) - 1} - \sqrt{d} 
	\geq \big(1 - \sqrt{1 - \eps}\big) \sqrt{\delta(E) - 1}.
\end{aligned}
\end{equation}
The justifications are the same as in Section~\ref{sec:rsv-four-to-embed-univ}.

We can easily bound the dimensional term in~\eqref{eqn:embed-b-pf1.5}:
$$
(d + D)^{1/2 - \cnst{c}_p \kappa(p)} \leq \cnst{C} D^{1/2 - \cnst{c}_p \kappa(p)}
	\leq \cnst{C} \varrho^{-1/2} D^{-\cnst{c}_p \kappa(p)} \sqrt{\delta(E)}
	\leq \cnst{C} \varrho^{-1/2} D^{-\cnst{c}_p \kappa(p)} \sqrt{\delta(E) - 1}
$$
The first inequality holds because $d \leq D$,
and the last two relations both rely on the assumption $1 \ll D \leq \varrho^{-1} \delta(E)$.
Since $\cnst{c}_p \kappa(p)$ is positive,
we can find a number $N := N(p, \nu, \varrho, \eps)$ for which
$D \geq N$ implies that
$$
\cnst{C}_p \nu^{5/2} (d + D)^{1/2 -\cnst{c}_p\kappa(p)}
	< \big( \sqrt{1 + \eps} - 1 \big) \sqrt{\delta(E) - 1}.
$$
Combine the last result with~\eqref{eqn:embed-b-pf2} to
see that the claim~\eqref{eqn:embed-b-pf1.5} holds true.

\subsection{Proof of Proposition \ref{prop:univ-decode}: Application of RAP functional to decoding with structured errors}
\label{sec:decoding-success}

\joelprev{
This section establishes the success condition in Proposition~\ref{prop:univ-decode}.
More precisely, we show that
\begin{equation} \label{eqn:univ-decode-goal}
s/n < \psi_{\ell_1}^{-1}(1 - m/n) - o(1)
\quad\text{implies}\quad
\Omega \cap \range(\mtx{\Phi}) = \emptyset
\quad\text{with probability $1 - o(1)$.}
\end{equation}
The set $\Omega$ is derived from the descent cone of the $\ell_1$ norm,
and the calculation~\eqref{eqn:sdim-l1-polar} shows that
$$
\delta(\Omega^\polar) = n(1 - \psi_{\ell_1}(s/n) - o(1)).
$$
To begin, we assume that there are parameters $\varrho, \eps > 0$ for which
\begin{itemize}
\item	The statistical dimension $\delta(\Omega^{\circ}) \geq \varrho n$.
\item	The message length $m \leq (1 - \eps) \,\delta(\Omega^\circ)$.
\end{itemize}
As in Section~\ref{sec:car-to-embed}, these conditions imply that
$$
\Omega^\circ \cap \nullsp(\mtx{\Phi}^*) \neq \emptyset
\quad\text{with probability at least $1 - o(1)$.}
$$
Indeed, this statement corresponds to~\eqref{eqn:embed-b-pf1}
under the change of variables $d = m$ and $D = n$
and $E = \Omega^{\circ}$ and $\mtx{\Pi} = \mtx{\Phi}^*$.
By polarity, the conclusion of~\eqref{eqn:univ-decode-goal} holds.
}

\joelprev{
It remains to show that the parameter choice in~\eqref{eqn:univ-decode-goal}
implies the two assumptions we have made.  As in the proof of Proposition~\ref{prop:univ-decode},
the condition $s \leq (1 - \xi) n$ ensures that $\delta(\Omega^{\polar}) \geq \varrho n$
for some $\varrho > 0$.
Similarly, the condition $m \leq (1 - \eps) \delta(\Omega^{\circ})$ holds when
$$
m \leq (1 - \eps) \cdot n (1 - \psi_{\ell_1}(s/n) - o(1)).
$$
Equivalently, $$
s/n \leq \psi_{\ell_1}^{-1}(1 - m/((1-\eps)n) - o(1)).
$$
Since we can choose $\eps$ to be an arbitrarily small constant as $m \to \infty$, it suffices that
$$
s/n	\leq \psi_{\ell_1}^{-1}(1 - m/n - o(1)).
$$
This observation completes the argument.
}

\section{Theorem~\ref{thm:car-bdd}: Truncation of the Cone}

In this section, we argue that functional $\tau_{\min}(\mtx{\Phi}; K)$
does not change very much if we truncate the cone $K$.  Replacing the unbounded
set with a compact set allows us to develop discretization arguments.

\begin{proposition}[RAP: Truncation of Cone] \label{prop:car-cone-truncation}
Adopt the notation and hypotheses of Theorem~\ref{thm:car-bdd}.
Let $S := K^\polar \cap R \mathsf{B}^m$, where $R := \cnst{C}_{\rm rad}B^2 \sqrt{m + n}$.
Then
$$
\Expect  \min_{\norm{\vct{t}} = 1} \min_{\vct{s} \in K^\polar}
	\normsq{ \vct{s} - \mtx{\Phi} \vct{t}}
	\geq \Expect \min_{\norm{\vct{t}} = 1} \min_{\vct{s} \in S}
	\normsq{ \vct{s} - \mtx{\Phi} \vct{t}}
	- \cnst{C} B^4.
$$
\end{proposition}

\begin{proof}
Since $\vct{0} \in K^\polar$, it is easy to see that
$$
\min_{\vct{s} \in K^\polar} \norm{\vct{s} - \mtx{\Phi} \vct{t}}
	\leq \norm{\mtx{\Phi} \vct{t}}
	\leq \norm{\mtx{\Phi}}
	\quad\text{when $\norm{\vct{t}} = 1$.}
$$
Meanwhile, the triangle inequality gives the bound
$$
\norm{\vct{s} - \mtx{\Phi}\vct{t}}
	\geq \norm{\vct{s}} - \norm{\mtx{\Phi}\vct{t}}
	\geq \norm{\vct{s}} - \norm{\mtx{\Phi}}
	\quad\text{when $\norm{\vct{t}} = 1$.}
$$
It follows that
\begin{equation} \label{eqn:sstar-bdd}
\vct{s}_{\star} \in \argmin_{\vct{s} \in K^\polar} \norm{\vct{s} - \mtx{\Phi}\vct{t}}
	\quad\text{implies}\quad
	\norm{\vct{s}_{\star}} \leq 2 \norm{\mtx{\Phi}}
	\quad\text{when $\norm{\vct{t}} = 1$.}
\end{equation}
Since the norm of the random matrix $\mtx{\Phi}$
concentrates, the bound~\eqref{eqn:sstar-bdd} shows that the norm
of the minimizer $\vct{s}_{\star}$ is unlikely to be large.

For any positive parameter $R$, the observation~\eqref{eqn:sstar-bdd} allows us to calculate that
\begin{equation} \label{eqn:car-cone-trunc-bd}
\begin{aligned}
\Expect \tau_{\min}^2(\mtx{\Phi}; K) 
	&= \Expect  \min_{\norm{\vct{t}} = 1}\min_{\vct{s} \in K^\polar} \normsq{\vct{s} - \mtx{\Phi} \vct{t}} \\
	&\geq \Expect \left[ \min_{\norm{\vct{t}} = 1}\min_{\vct{s} \in K^\polar} \normsq{\vct{s} - \mtx{\Phi} \vct{t}}
	\mathbb{1}\{ \norm{\mtx{\Phi}} \leq R/2 \} \right] \\
	&= \Expect\left[ \min_{\norm{\vct{t}} = 1} \min_{\vct{s} \in K^\polar \cap R \mathsf{B}^m}
	\normsq{\vct{s} - \mtx{\Phi}\vct{t}} \mathbb{1}\{ \norm{\mtx{\Phi}} \leq R/2 \}\right] \\
	&= \Expect \min_{\norm{\vct{t}} = 1} \min_{\vct{s} \in K^\polar \cap R \mathsf{B}^m}
	\normsq{\vct{s} - \mtx{\Phi} \vct{t}}
	- \Expect \left[ \max_{\norm{\vct{t}} = 1} \max_{\vct{s} \in K^\polar \cap R \mathsf{B}^m}
	\normsq{\vct{s} - \mtx{\Phi} \vct{t}}
	\mathbb{1}\{ \norm{\mtx{\Phi}} > R/2 \} \right].
\end{aligned}
\end{equation}
To reach the last line, we write the indicator function in terms of its the complement.

To bound the second term on the right-hand side of~\eqref{eqn:car-cone-trunc-bd},
crude estimates suffice.
$$
\begin{aligned}
\Expect \left[ \max_{\norm{\vct{t}} = 1} \max_{\vct{s} \in K^\polar \cap R \mathsf{B}^m} \normsq{\vct{s} - \mtx{\Phi} \vct{t}}
	\mathbb{1}\{ \norm{\mtx{\Phi}} > R/2 \} \right]
	&\leq \Expect\left[ (R + \norm{\mtx{\Phi}})^2 \mathbb{1}\{ \norm{\mtx{\Phi}} > R/2 \} \right] \\ 
	&\leq 9 \Expect\left[ \normsq{\mtx{\Phi}} \mathbb{1}\{ \norm{\mtx{\Phi}} > R/2 \} \right] \\
	&\leq 9 \left( \Expect \norm{\mtx{\Phi}}^4 \right)^{1/2} \big( \Prob{ \norm{\mtx{\Phi}} > R/2 } \big)^{1/2}.
\end{aligned}
$$
The last inequality is Cauchy--Schwarz.
Since $\mtx{\Phi}$ satisfies the condition~\eqref{eqn:subgauss-hyp},
Fact~\ref{fact:subgauss-mtx-tails} implies that
$$
\Prob{ \norm{\mtx{\Phi}} > \cnst{C}_0 B^2 \sqrt{m+n} + \cnst{C}_0 B^2 \zeta } \leq \econst^{-\zeta^2}.
$$
In particular, using integration by parts,
$$
\left(\Expect \norm{\mtx{\Phi}}^4 \right)^{1/4} \leq \cnst{C} B^2 \sqrt{m+n}.
$$
Furthermore, there is a constant $\cnst{C}_{\rm rad}$ for which
$$
\Prob{ \norm{\mtx{\Phi}} > \half \cnst{C}_{\rm rad} B^2 \sqrt{m + n} }
	\leq (m+n)^{-2}.
$$
If we set $R := \cnst{C}_{\rm rad} B^2 \sqrt{m +n}$, then
\begin{equation} \label{eqn:car-cone-trunc-error}
\Expect \left[ \max_{\norm{\vct{t}} = 1} \max_{\vct{s} \in K^\polar \cap R \mathsf{B}^m}
	\normsq{\vct{s} - \mtx{\Phi} \vct{t}}
	\mathbb{1}\{ \norm{\mtx{\Phi}} > R/2 \} \right]
	\leq \cnst{C} B^4.
\end{equation}
Introduce the estimate~\eqref{eqn:car-cone-trunc-error} into~\eqref{eqn:car-cone-trunc-bd}
to complete the argument.
\end{proof}

\section{Theorem~\ref{thm:car-bdd}: Replacing Most Entries of the Random Matrix}

In this section, we show that we can replace most of the entries of a random matrix
$\mtx{\Phi}$ with standard normal variables without changing the value of the
functional $\Expect \tau_{\min}^2(\mtx{\Phi}; K)$ substantially

\begin{proposition}[RAP: Partial Replacement] \label{prop:partial-replacement-redux}
Let $\mtx{\Phi}$ be an $m \times n$ random matrix that satisfies Model~\ref{mod:bdd-mtx}
with magnitude bound $B$.  Fix the parameter $R :=\cnst{C}_{\rm rad} B^2\sqrt{ m + n }$.
Let $I$ be a subset of $\{1, \dots, m\}$ with cardinality $k$, and
let $S_I$ be a closed subset of $R \mathsf{B}^m$ for which
\begin{equation*} \label{eqn:SI-hyp-replace-redux}
\vct{s} \in S_I
\quad\text{implies}\quad
\abs{s_i} \leq R k^{-1/2}
\quad\text{for each index $i \in I^c$}.
\end{equation*}
Let $J$ be a subset of $\{1, \dots, n\}$ with cardinality $k$,
and let $T_J$ be a closed subset of $\mathsf{B}^n$ for which
\begin{equation*} \label{eqn:TJ-hyp-replace-redux}
\vct{t} \in T_J
\quad\text{implies}\quad
\abs{\smash{t_j}} \leq k^{-1/2}
\quad\text{for each index $j \in J^c$}.
\end{equation*}
Suppose that $\mtx{\Psi}$ is an $m \times n$ matrix with block form
\begin{equation} \label{eqn:car-hybrid}
\mtx{\Psi} := \mtx{\Psi}(I, J) := \begin{bmatrix} \mtx{\Phi}_{IJ} & \mtx{\Phi}_{IJ^c} \\
	\mtx{\Phi}_{I^c J} & \mtx{\Gamma}_{I^c J^c} \end{bmatrix}.
\end{equation}
Then
\begin{equation} \label{eqn:replacement-error-redux}
\abs{ \Expect \min_{\vct{t} \in T_J} \min_{\vct{s} \in S_I} \normsq{ \vct{s} - \mtx{\Phi} \vct{t} }
	- \Expect \min_{\vct{t} \in T_J} \min_{\vct{s} \in S_I} \normsq{ \vct{s} - \mtx{\Psi} \vct{t} } }
	\leq \frac{\cnst{C}B^3 (m+n)^{11/6} \log(mn)}{k}.
\end{equation}
As usual, $\mtx{\Gamma}$ is an $m \times n$ standard normal matrix.
\end{proposition}

\subsection{Proof of Proposition~\ref{prop:partial-replacement-redux}}

Fix the sets $I$ and $J$.
As in the proof of Proposition~\ref{prop:partial-replacement},
the error has three components: $$
\begin{aligned}
\abs{ \Expect \min_{\vct{t} \in T_J} \min_{\vct{s} \in S_I}
	\normsq{ \vct{s} - \mtx{\Phi} \vct{t} }
	\ -\  \Expect  \min_{\vct{t} \in T_J} \min_{\vct{s} \in S_I}
	\normsq{ \vct{s} - \mtx{\Psi} \vct{t} } }
	\quad&\leq\quad \cnst{C}mn \eps
	&&&& \text{(Lemma~\ref{lem:discretization-redux})} \\
	&+\quad \cnst{C} \beta^{-1} (m+n) \log(1/\eps)
	&&&& \text{(Lemma~\ref{lem:smoothing-redux})} \\
	&+\quad \cnst{C} B^3 mn \left(\frac{\beta R}{k^2} + \frac{\beta^2 R^3}{k^3}\right).
	&&&& \text{(Lemma~\ref{lem:exchange-redux})}
\end{aligned}
$$
The first error comes from discretizing the sets $S$ and $T$
at a level $\eps \in (0, 1]$.  The second error appears when
we replace the minima with a soft-min function with parameter $\beta > 0$.
The last error emerges from the Lindeberg exchange argument.

To complete the proof, we set $\eps = (mn)^{-1}$
to make the discretization error negligible.
Select the smoothing parameter so that $\beta^3 = k^3(m+n)/(B^3 R^3 mn)$.
We arrive at
\begin{multline*}
\abs{ \Expect \min_{\vct{t} \in T_J} \min_{\vct{s} \in S_I}
	\normsq{ \vct{s} - \mtx{\Phi} \vct{t} }
	\ -\  \Expect  \min_{\vct{t} \in T_J} \min_{\vct{s} \in S_I}
	\normsq{ \vct{s} - \mtx{\Psi} \vct{t} } } \\
	\leq \frac{\cnst{C} B^2 (m+n)^{1/3} (mn)^{2/3}}{k^2}
	+ \frac{\cnst{C} BR (m+n)^{2/3} (mn)^{1/3} \log(mn)}{k}.
\end{multline*}
Since $2\sqrt{mn} \leq m + n$ and $R = \cnst{C}_{\rm rad} B^2 \sqrt{m+n}$,
the second term dominates.  We reach the stated result.

\subsection{Proposition~\ref{prop:partial-replacement-redux}: Discretizing the Index Sets}

The first step in the proof of Proposition~\ref{prop:partial-replacement-redux}
is to replace the index sets by finite subsets.

\begin{lemma}[Proposition~\ref{prop:partial-replacement-redux}: Discretization] \label{lem:discretization-redux}
Adopt the notation and hypotheses of Proposition~\ref{prop:partial-replacement-redux}.
Fix a parameter $\eps \in (0,1]$.  Then $S_I$ contains a finite subset $S_I^\eps$
and $T_J$ contains a finite subset $T_J^\eps$ whose cardinalities satisfy
$$
\log(\# S_I^\eps) + \log(\# T_J^\eps) \leq (m + n) \log (3/\eps).
$$
Furthermore, these subsets have the property that
\begin{equation} \label{eqn:covering-error-redux}
\abs{ \Expect \min_{\vct{t} \in T_J} \min_{\vct{s} \in S_I} \normsq{ \vct{s} - \mtx{\Phi} \vct{t} }
	- \Expect \min_{\vct{t} \in T_J^\eps} \min_{\vct{s} \in S_I^\eps} \normsq{ \vct{s} - \mtx{\Phi} \vct{t} } }
	\leq \cnst{C} mn \eps.
\end{equation}
The bound~\eqref{eqn:covering-error-redux} also holds if we replace $\mtx{\Phi}$
by $\mtx{\Psi}$.
\end{lemma}

\begin{proof}
We choose $S_I^\eps$ to be an $(R \eps)$-covering of $S_I$,
and $T_J^\eps$ to be an $\eps$-covering of $T_J$.  Since $S_I$
is a subset of $R \mathsf{B}^m$ and $T_J$ is a subset of $\mathsf{B}^n$,
we can be sure that the coverings have cardinality
$\# S_I^\eps \leq (3/\eps)^m$ and $\# T_J^\eps \leq (3/\eps)^n$.
See~\cite[Lem.~5.2]{Ver12:Introduction-Nonasymptotic}.
The rest of the proof is essentially the same
as that of Lemma~\ref{lem:discretization}, so we omit the details.
\end{proof}

\subsection{Proposition~\ref{prop:partial-replacement-redux}: Smoothing the Minimum}

The next step in the proof of Proposition~\ref{prop:partial-replacement-redux}
is to pass from the minimum to the soft-min function.

\begin{lemma}[Proposition~\ref{prop:partial-replacement-redux}: Smoothing]
\label{lem:smoothing-redux}
Adopt the notation and hypotheses of Proposition~\ref{prop:partial-replacement-redux},
and let $S_I^\eps$ and $T_J^\eps$ be the sets introduced in Lemma~\ref{lem:discretization-redux}.
Fix a parameter $\beta > 0$, and introduce the function
\begin{equation} \label{eqn:soft-max-car}
F : \R^{m \times n} \to \R
\quad\text{where}\quad
F(\mtx{A}) := -\frac{1}{\beta} \log \sum_{\vct{s} \in S_I^\eps} \sum_{\vct{t} \in T_J^\eps}
	\econst^{ - \beta \normsq{\vct{s} - \mtx{A} \vct{t}} }.
\end{equation}
Then
\begin{equation} \label{eqn:smoothing-error-car}
\abs{ \Expect \min_{\vct{t} \in T_J^\eps} \min_{\vct{s} \in S_I^\eps} \normsq{ \vct{s} - \mtx{\Phi} \vct{t} }
	- \Expect F(\mtx{\Phi}) }
	\leq \frac{1}{\beta} \left( \log( \# S_I^\eps ) + \log( \# T_J^\eps ) \right).
\end{equation}
The estimate~\eqref{eqn:smoothing-error-car} also holds if we replace $\mtx{\Phi}$
by $\mtx{\Psi}$.
\end{lemma}

\begin{proof}
The proof is almost identical with that of Lemma~\ref{lem:smoothing}.
\end{proof}

\subsection{Proposition~\ref{prop:partial-replacement-redux}: Exchanging the Entries of the Random Matrix}

The main challenge in the proof of Proposition~\ref{prop:partial-replacement-redux}
is to exchange most of the entries of the random matrix $\mtx{\Phi}$
for the entries of $\mtx{\Psi}$.

\begin{lemma}[Proposition~\ref{prop:partial-replacement-redux}: Exchange] \label{lem:exchange-redux}
Adopt the notation and hypotheses of Proposition~\ref{prop:partial-replacement-redux},
and let $F$ be the function defined in Lemma~\ref{lem:smoothing-redux}.
Then
\begin{equation*} \label{eqn:total-exchange-error-redux}
\abs{ \Expect F(\mtx{\Phi}) - \Expect F( \mtx{\Psi} ) }
	\leq \cnst{C}B^3 mn \left( \frac{\beta R}{k^2}  + \frac{\beta^2 R^3}{k^3} \right). \end{equation*}
\end{lemma}

The proof is similar with Lemma~\ref{lem:exchange}.
This time, we replace only the rows of $\mtx{\Phi}$ listed in $I^c$.
We incur the same error for each of these $m - k$ rows, so it suffices to 
control the error in exchanging a single row.  The following sublemma achieves this goal.

\begin{sublemma}[Lemma~\ref{lem:exchange-redux}: Comparison for One Row] \label{slem:compare-one-redux}
Adopt the notation and hypotheses of Proposition~\ref{prop:partial-replacement-redux},
and let $S_I^\eps$ and $T_J^\eps$ be the sets defined in Lemma~\ref{lem:smoothing-redux}.
For $i \in I^c$, introduce the function
$$
f : \R^n \to \R
\quad\text{given by}\quad
f(\vct{a}) := - \frac{1}{\beta} \log \sum_{\vct{s} \in S_I^\eps} \sum_{\vct{t} \in T_J^\eps}
	\econst^{ - \beta (s_i - \vct{a} \cdot \vct{t})^2 + q(\vct{s}, \vct{t}) },
$$
where $q : S_I^\eps \times T_J^\eps \to \R$ is an arbitrary function.
Suppose that $\vct{\phi} \in \R^n$ is a random vector with independent, standardized entries
that are bounded in magnitude by $B$.
Suppose that $\vct{\psi} \in \R^n$ is a random vector with
$$
\vct{\psi}_J = \vct{\phi}_J
\quad\text{and}\quad
\vct{\psi}_{J^c} = \vct{\gamma}_{J^c},
$$
where $\vct{\gamma} \in \R^n$ is a standard normal vector.  Then
\begin{equation} \label{eqn:compare-one-redux}
\abs{ \Expect f(\vct{\phi}) - \Expect f( \vct{\psi}) }
	\leq 	\cnst{C}B^3 n \left(\frac{\beta R}{k^2} + \frac{\beta^2 R^3}{k^3} \right). \end{equation}
\end{sublemma}

The proof of this result is much the same as the proof of Sublemma~\ref{slem:compare-one}.
There are only two points that require care.  First, we use a slightly different result to compute
the derivatives.

\begin{sublemma}[Lemma~\ref{lem:exchange-redux}: Derivatives] \label{slem:derivatives-redux}
Adopt the notation and hypotheses of Proposition~\ref{prop:partial-replacement-redux}
and Sublemma~\ref{slem:compare-one-redux}.
Let $\vct{\xi} : \R \to \R^n$ be a linear function, so its derivative $\vct{\xi}' \in \R^n$ is a constant vector.
For $i \in I^c$, define the function
$$
r(\alpha) := f\big(\vct{\xi}(\alpha)\big)
	= - \frac{1}{\beta} \log \sum_{\vct{s} \in S_I^\eps} \sum_{\vct{t} \in T_J^\eps} \econst^{-\beta (s_i - \vct{\xi}(\alpha) \cdot \vct{t})^2 + q(\vct{s}, \vct{t})}
$$
where $q : S_I^\eps \times T_J^\eps \to \R$ is arbitrary.
The third derivative of this function satisfies
$$
\abs{r'''(\alpha)} \leq 48 \left( \max_{\vct{t} \in T_J^\eps} \abs{\vct{\xi}' \cdot \vct{t}}^3 \right)
	\left( \beta \Expect_{\vct{v}} \abs{ s_i - \vct{\xi}(\alpha) \cdot \vct{v} }
	+ \beta^2 \Expect_{\vct{v}} \abs{s_i - \vct{\xi}(\alpha) \cdot \vct{v}}^3 \right),
$$
where $\vct{v} \in T_J^\eps$ is a random vector that does not depend on $\vct{\xi}(\alpha)$.
\end{sublemma}

Second, when making further bounds on $\abs{r'''(\alpha)}$,
we need to exploit our control on the magnitude of $\vct{s}$
on the coordinates in $I^c$.  Note that
$$
\abs{ s_i - \vct{\xi}(\alpha) \cdot \vct{v} }
	\leq \abs{s_i} + \abs{\vct{\xi}(\alpha) \cdot \vct{v}}
	\leq Rk^{-1/2} + \abs{\vct{\xi}(\alpha) \cdot \vct{v}}
$$
The second inequality holds because $\abs{s_i} \leq R k^{-1/2}$ for each $i \in I^c$.
Similarly,
$$
\abs{ s_i - \vct{\xi}(\alpha) \cdot \vct{v} }^3
	\leq \cnst{C} \abs{s_i}^3 + \cnst{C} \abs{\vct{\xi}(\alpha) \cdot \vct{v}}^3
	\leq \cnst{C} Rk^{-3/2} + \cnst{C} \abs{\vct{\xi}(\alpha) \cdot \vct{v}}^3.
$$
Repeating the arguments from Sublemma~\ref{slem:compare-one}, we obtain bounds of the form
$$
\Expect\left[  \abs{\smash{\phi_j}}^3 \max_{\abs{\alpha} \leq \abs{\smash{\phi_j}}} \abs{r_j'''(\alpha)} \right]
	\leq \frac{\beta B^3 R}{k^{2}} + \frac{\cnst{C} \beta^2 B^3 R^3}{k^{3}}
	+ \frac{ \cnst{C} \beta B^4 + \cnst{C} \beta^2 B^6 }{ k^{3/2} } .
$$
The first two terms dominate the third because our choice
$R = \cnst{C}_{\rm rad} B^2 \sqrt{m+n}$ implies that $R \geq B k^{1/2}$.
We arrive at the statement of Sublemma~\ref{slem:compare-one-redux}.

\section{Theorem~\ref{thm:car-bdd}: Bounding the RAP Functional by the Excess Width}

The most difficult part of proving Theorem~\ref{thm:car-bdd} is to identify the
excess width $\coll{E}_n(\Omega)$ after we replace the original matrix
$\mtx{\Phi}$ by the hybrid matrix $\mtx{\Psi}$ defined in~\eqref{eqn:car-hybrid}.  The following result
does the job.

\begin{proposition}[Theorem~\ref{thm:car-bdd}: Excess Width Bound] \label{prop:car-width}
Adopt the notation and hypotheses of Proposition~\ref{prop:partial-replacement-redux}.
Let $\Omega := K \cap \mathsf{S}^{m-1}$.
Then
\begin{equation} \label{eqn:car-width-lower}
\Expect \min_{\norm{\vct{t}} = 1}  \min_{\vct{s} \in K^\polar} \normsq{ \vct{s} - \mtx{\Psi}(I, J) \vct{t} }
	\geq \left( \frac{\big( \coll{E}_n(\Omega) \big)_{-}}{\cnst{C}B \sqrt{\log m}} - \cnst{C} B^2 \sqrt{k \log m}
	\right)_+^2.
\end{equation}
The random matrix $\mtx{\Psi}$ is defined in~\eqref{eqn:car-hybrid}.
\end{proposition}

The proof of Proposition~\ref{prop:car-width} occupies the rest of this section.
At the highest level, the proof is similar with the argument underlying Proposition~\ref{prop:excess-width}.
We write the quantity of interest as a minimax, and then we apply the Gaussian Minimax Theorem to
replace the Gaussian matrix with a pair of Gaussian vectors.  Afterward, we analyze the resulting
expression to identify the Gaussian width; the new challenges appear in this step.

\subsection{Proof of Proposition~\ref{prop:car-width}}

Here is an overview of the calculations that we will perform;
the detailed justifications appear in the upcoming subsections.
Let us abbreviate $U := K \cap \mathsf{B}^{m}$.
We have the chain of inequalities
$$
\begin{aligned}
\Expect \min_{\norm{\vct{t}} = 1} \min_{\vct{s} \in K^\polar}  & \norm{ \vct{s} - \mtx{\Psi} \vct{t} } \\
	&\geq \Expect \min_{\norm{\vct{t}}=1} \max_{\vct{u} \in U} \vct{u} \cdot \mtx{\Psi} \vct{t}
		&& \text{(Lemma~\ref{lem:car-duality})} \\
	&= \Expect \min_{\norm{\vct{t}}=1} \max_{\vct{u} \in U}
	\begin{bmatrix} \vct{u}_I \\ \vct{u}_{I^c} \end{bmatrix} \cdot
	\begin{bmatrix} \mtx{\Phi}_{IJ} & \mtx{\Phi}_{I J^c} \\
	\mtx{\Phi}_{I^c J} & \mtx{\Gamma}_{I^c J^c} \end{bmatrix}
	\begin{bmatrix} \vct{t}_J \\ \vct{t}_{J^c} \end{bmatrix} \\
	&\geq \Expect \min_{\norm{\vct{t}}=1} \max_{\vct{u} \in U} \left(
	\vct{u} \cdot \begin{bmatrix} \mtx{\Phi}_{IJ} & \mtx{\Phi}_{IJ^c} \\
	\mtx{\Phi}_{I^cJ} & \mtx{0} \end{bmatrix} \vct{t}
	+ (\vct{u}_{I^c} \cdot \vct{g}_{I^c}) \norm{\smash{\vct{t}_{J^c}}}
	+ \norm{\vct{u}_{I^c}} (\vct{h}_{J^c} \cdot \vct{t}_{J^c}) \right)
	- 2 \quad
		&& \text{(Lemma~\ref{lem:apply-minmax-lower-redux})} \\
	&\geq \frac{\big( \coll{W}(U) - \sqrt{n} \big)_+}{\cnst{C}B\sqrt{\log m}}
	- \cnst{C} B^2 \sqrt{k \log m}
	&& \text{(Lemma~\ref{lem:car-width-id})}.
\end{aligned}
$$
Lemma~\ref{lem:car-duality} is a standard convex duality argument,
and the next line follows when we write out the quantity of interest more explicitly.
To reach the fourth line, we apply the Gaussian Minimax Theorem in the usual way
to replace the random
matrix $\mtx{\Gamma}$ with two standard normal vectors $\vct{g} \in \R^m$ and
$\vct{h} \in \R^n$.
In a rough sense, the remaining part of the random matrix $\mtx{\Phi}$ is negligible.
The term $(\vct{u}_{I^c} \cdot \vct{g}_{I^c})$
generates the Gaussian width $\coll{W}(U)$, defined in~\eqref{eqn:gauss-width},
while the term $(\vct{h}_{J^c} \cdot \vct{t}_{J^c})$
contributes a dimensional factor $- \sqrt{n}$.

Apply the increasing convex function $(\cdot)_+^2$ to the inequality 
in the last display, and invoke Jensen's inequality to
draw out the expectation.  Notice that
$$
\coll{W}(U) = \Expect \max_{\vct{u} \in K \cap \mathsf{B}^m} \vct{u} \cdot \vct{g}
	\geq \Expect \max_{\vct{u} \in K \cap \mathsf{S}^{m-1}} \vct{u} \cdot \vct{g}
	= \coll{W}(\Omega).
$$
Finally, $\big( \coll{E}_n(\Omega) \big)_- = \big(\coll{W}(\Omega) - \sqrt{n}\big)_+$
because of~\eqref{eqn:excess-vs-width}.  This point completes the proof.

\subsection{Proposition~\ref{prop:car-width}: Duality for the RAP Functional}

The first step in the argument is to apply the minimax inequality to pass to a saddle-point formulation
that is amenable to analysis with the Gaussian Minimax Theorem.

\begin{lemma}[Proposition~\ref{prop:car-width}: Duality] \label{lem:car-duality}
Adopt the notation and hypotheses of Proposition~\ref{prop:car-width}.
For any point $\vct{t} \in \R^n$,
$$
\min_{\vct{s} \in K^\polar} \norm{\vct{s} - \mtx{\Psi} \vct{t}}
	\geq \max_{\vct{u} \in U} \vct{u} \cdot \mtx{\Psi} \vct{t}
$$
where $U := K \cap \mathsf{B}^{m}$.
\end{lemma}

\begin{proof}
Write the norm as maximum:
$$
\min_{\vct{s} \in K^\polar} \norm{\vct{s} - \mtx{\Psi} \vct{t} }
	= \min_{\vct{s} \in K^\polar} \max_{\vct{u} \in \mathsf{B}^{m}} \vct{u} \cdot (\mtx{\Psi} \vct{t} - \vct{s}).
$$
The minimax inequality allows us to interchange the maximum and minimum:
$$
\min_{\vct{s} \in K^\polar} \norm{\vct{s} - \mtx{\Psi} \vct{t}}
	\geq \max_{\vct{u} \in \mathsf{B}^{m}} \min_{\vct{s} \in K^\polar} \vct{u} \cdot (\mtx{\Psi} \vct{t} - \vct{s})
	= \max_{\vct{u} \in \mathsf{B}^m} \Big( \vct{u} \cdot \mtx{\Psi} \vct{t}
	- \max_{\vct{s} \in K^\polar} \vct{u} \cdot \vct{s} \Big)
$$
The value of $\max_{\vct{s} \in K^\polar} \vct{u} \cdot \vct{s}$ equals zero when $\vct{u} \in (K^\polar)^{\polar} = K$; otherwise, it takes the value $+\infty$.  This step uses the assumption that $K$ is closed and convex.  We conclude that
$$
\min_{\vct{s} \in K^\polar} \norm{\vct{s} - \mtx{\Psi} \vct{t}}
	\geq \max_{\vct{u} \in K \cap \mathsf{B}^{m}} \vct{u} \cdot \mtx{\Psi} \vct{t}.
$$
This is the stated result.
\end{proof}

\subsection{Proposition~\ref{prop:car-width}: Reducing the Gaussian Matrix to Some Gaussian Vectors}

By an argument similar with the proof of Lemma~\ref{lem:apply-minmax-lower}, we can replace the Gaussian block of $\mtx{\Psi}$
with two Gaussian vectors.

\begin{lemma}[Proposition~\ref{prop:car-width}: Reducing the Gaussian Matrix]
\label{lem:apply-minmax-lower-redux}
Adopt the notation and hypotheses of Proposition~\ref{prop:car-width}.
Then
\begin{multline*}
\Expect \min_{\norm{\vct{t}}=1} \max_{\vct{u} \in U}
	\vct{u} \cdot \begin{bmatrix} \mtx{\Phi}_{IJ} & \mtx{\Phi}_{IJ^c} \\
	\mtx{\Phi}_{I^cJ} & \mtx{\Gamma}_{I^c J^c} \end{bmatrix} \vct{t}
	\geq \Expect \min_{\norm{\vct{t}}=1} \max_{\vct{u} \in U} \left(
	\vct{u} \cdot \begin{bmatrix} \mtx{\Phi}_{IJ} & \mtx{\Phi}_{IJ^c} \\
	\mtx{\Phi}_{I^cJ} & \mtx{0} \end{bmatrix} \vct{t}
	+ (\vct{u}_{I^c} \cdot \vct{g}_{I^c}) \norm{\smash{\vct{t}_{J^c}}}
	+ \norm{\vct{u}_{I^c}} (\vct{h}_{J^c} \cdot \vct{t}_{J^c}) \right)
	- 2,
\end{multline*}
where $\vct{g} \in \R^m$ and $\vct{h} \in \R^n$
be independent standard normal vectors.
\end{lemma}

\begin{proof}
There are no new ideas in this bound, so we refer the reader to Lemmas~\ref{lem:my-comparison}
and~\ref{lem:apply-minmax-lower} for the pattern of argument.
\end{proof}

\subsection{Proposition~\ref{prop:car-width}: Finding the Gaussian Width}

To prove Proposition~\ref{prop:car-width}, most of the difficulty arises
when we seek a good lower bound for the minimax problem that appears in
Lemma~\ref{lem:apply-minmax-lower-redux}.  We have the following result.

\begin{lemma}[Proposition~\ref{prop:car-width}: Finding the Gaussian Width] \label{lem:car-width-id}
Adopt the notation and hypotheses of Proposition~\ref{prop:car-width}.
Define the set $U := K \cap \mathsf{B}^m$.  Then
$$
\Expect \min_{\norm{\vct{t}}=1} \max_{\vct{u} \in U} \left(
	\vct{u} \cdot \begin{bmatrix} \mtx{\Phi}_{IJ} & \mtx{\Phi}_{IJ^c} \\
	\mtx{\Phi}_{I^c J} & \mtx{0} \end{bmatrix} \vct{t}
	+ (\vct{u}_{I^c} \cdot \vct{g}_{I^c}) \norm{\smash{\vct{t}_{J^c}}}
	+ \norm{\vct{u}_{I^c}} (\vct{h}_{J^c} \cdot \vct{t}_{J^c}) \right) \\
	\geq \frac{\coll{W}(U) - \sqrt{n}}{\cnst{C} B\sqrt{\log m}}
	- \cnst{C} B^2 \sqrt{k \log m}.
$$
\end{lemma}

\begin{proof}
The proof of this bound is lengthy, so we break the argument into several steps.
The overall result follows when we sequence the inequalities
in Sublemmas~\ref{slem:car-minmax-simple}, \ref{slem:car-minmax-simpler},
\ref{slem:prob-bounds}, and~\ref{slem:car-width-minmax} and consolidate
the error terms.
\end{proof}

\subsubsection{Lemma~\ref{lem:car-width-id}: Simplifying the Minimax I}

The first step in the proof of Lemma~\ref{lem:car-width-id}
is to simplify the minimax so we can identify the key terms.

\begin{sublemma}[Lemma~\ref{lem:car-width-id}: Simplifying the Minimax I] \label{slem:car-minmax-simple}
Adopt the notation and hypotheses of Lemma~\ref{lem:car-width-id}.  Then
\begin{multline*}
\Expect \min_{\norm{\vct{t}}=1} \max_{\vct{u} \in U} \left(
	\vct{u} \cdot \begin{bmatrix} \mtx{\Phi}_{IJ} & \mtx{\Phi}_{IJ^c} \\
	\mtx{\Phi}_{I^c J} & \mtx{0} \end{bmatrix} \vct{t}
	+ (\vct{u}_{I^c} \cdot \vct{g}_{I^c}) \norm{\smash{\vct{t}_{J^c}}}
	+ \norm{\vct{u}_{I^c}} (\vct{h}_{J^c} \cdot \vct{t}_{J^c}) \right) \\
	\geq \Expect \min_{\norm{\vct{t}} = 1} \max_{\vct{u} \in U} \left(
	\vct{u} \cdot \begin{bmatrix} \mtx{\Phi}_J & \vct{g} \end{bmatrix}
	\begin{bmatrix} \vct{t}_j \\ \norm{\smash{\vct{t}_{J^c}}} \end{bmatrix}
	- \sqrt{n} \norm{\smash{\vct{t}_{J^c}}} \right)_+ - \cnst{C} B^2\sqrt{k}.
\end{multline*}
\end{sublemma}

\begin{proof}
Let us introduce notation for the quantity of interest:
\begin{equation} \label{eqn:width-Q1}
Q_1 := \Expect \min_{\norm{\vct{t}}=1} \max_{\vct{u} \in U} \left(
	\vct{u} \cdot \begin{bmatrix} \mtx{\Phi}_{IJ} & \mtx{\Phi}_{IJ^c} \\
	\mtx{\Phi}_{I^c J} & \mtx{0} \end{bmatrix} \vct{t}
	+ (\vct{u}_{I^c} \cdot \vct{g}_{I^c}) \norm{\smash{\vct{t}_{J^c}}}
	+ \norm{\vct{u}_{I^c}} (\vct{h}_{J^c} \cdot \vct{t}_{J^c}) \right)_+. \\
\end{equation}
We can introduce the positive-part operator because the fact that $\vct{0} \in U$
ensures that the minimax is nonnegative.

The first step in the argument is to reintroduce the missing piece
of the random vector $\vct{g}$.
Adding and subtracting the quantity $(\vct{u}_I \cdot \vct{g}_I) \norm{\smash{\vct{t}_{J^c}}}$
inside the positive-part operator in~\eqref{eqn:width-Q1}, we obtain the bound
\begin{equation} \label{eqn:Eminmax-with-g}
\begin{aligned}
Q_1 &\geq \Expect \min_{\norm{\vct{t}} = 1} \max_{\vct{u} \in U} \left(
	\vct{u} \cdot \begin{bmatrix} \mtx{\Phi}_{IJ} & \mtx{\Phi}_{IJ^c} \\
	\mtx{\Phi}_{I^c J} & \mtx{0} \end{bmatrix} \vct{t}
	+ (\vct{u} \cdot \vct{g}) \norm{\smash{\vct{t}_{J^c}}}
	+ \norm{\vct{u}_{I^c}} (\vct{h}_{J^c} \cdot \vct{t}_{J^c}) \right)_+
	- \Expect \max_{\vct{u} \in U} \vct{u}_I \cdot \vct{g}_I \\
	&\geq \Expect \min_{\norm{\vct{t}} = 1} \max_{\vct{u} \in U} \left(
	\vct{u} \cdot \begin{bmatrix} \mtx{\Phi}_{IJ} & \mtx{\Phi}_{IJ^c} \\
	\mtx{\Phi}_{I^c J} & \mtx{0} \end{bmatrix} \vct{t}
	+ (\vct{u} \cdot \vct{g}) \norm{\smash{\vct{t}_{J^c}}}
	+ \norm{\vct{u}_{I^c}} (\vct{h}_{J^c} \cdot \vct{t}_{J^c}) \right)_+
	- \sqrt{k}.
\end{aligned}
\end{equation}
The second inequality holds because $\#I = k$ and $U$ is a subset of the unit ball.
This step is similar with the proof of Lemma~\ref{lem:missing-coords-lower}.

Next, we combine the terms in~\eqref{eqn:Eminmax-with-g}
involving $\mtx{\Phi}_{IJ^c}$ and the \emph{row} vector $\vct{h}_{J^c}$.
Since $\norm{\vct{u}} \leq 1$,
\begin{equation} \label{eqn:Phi-h-combo1}
\begin{aligned}
\vct{u}_I \cdot \mtx{\Phi}_{IJ^c} \vct{t}_{J^c}
	+ \norm{\vct{u}_{I^c}} (\vct{h}_{J^c} \cdot \vct{t}_{J^c})
	&= \begin{bmatrix} \vct{u}_I \\ \norm{\vct{u}_{I^c}} \end{bmatrix} \cdot \begin{bmatrix} \mtx{\Phi}_{IJ^c} \\ \vct{h}_{J^c} \end{bmatrix} \vct{t}_{J^c} \\
	&\geq - \norm{ \begin{bmatrix} \mtx{\Phi}_{IJ^c} \\ \vct{h}_{J^c} \end{bmatrix} } \norm{\smash{\vct{t}_{J^c}}} \\
	&\geq - \sqrt{n} \norm{\smash{\vct{t}_{J^c}}} -
	\left( \norm{ \begin{bmatrix} \mtx{\Phi}_{IJ^c} \\ \vct{h}_{J^c} \end{bmatrix} } - \sqrt{n}\right)_+.
\end{aligned}
\end{equation}
The $(k+1) \times (n-k)$ random matrix on the right-hand side has independent, standardized entries
that satisfy the subgaussian estimate~\eqref{eqn:subgauss-hyp} with bound $B$.
Repeating the calculations in \eqref{eqn:expect-smax-sqrt}, we see that
\begin{equation} \label{eqn:Phi-h-combo2}
\Expect \left( \norm{ \begin{bmatrix} \mtx{\Phi}_{IJ^c} \\ \vct{h}_{J^c} \end{bmatrix} } - \sqrt{n} \right)_+
	\leq \cnst{C} B^2 \sqrt{k}.
\end{equation}
Apply the estimate~\eqref{eqn:Phi-h-combo1} inside the minimax in~\eqref{eqn:Eminmax-with-g}
and then use~\eqref{eqn:Phi-h-combo2} to arrive at the lower bound
\begin{equation} \label{eqn:Eminmax-with-sqrt}
\begin{aligned}
Q_1 &\geq \Expect \min_{\norm{\vct{t}} = 1} \max_{\vct{u} \in U} \left(
	\vct{u} \cdot \begin{bmatrix} \mtx{\Phi}_{IJ} \\ \mtx{\Phi}_{I^cJ} \end{bmatrix} \vct{t}_J
	+ (\vct{u} \cdot \vct{g}) \norm{\smash{\vct{t}_{J^c}}} - \sqrt{n} \norm{\smash{\vct{t}_{J^c}}} \right)_+
	- \Expect \left( \norm{ \begin{bmatrix} \mtx{\Phi}_{IJ^c} \\ \vct{h}_{J^c} \end{bmatrix} } - \sqrt{n} \right)_+ \\
	&\geq \Expect \min_{\norm{\vct{t}} = 1} \max_{\vct{u} \in U} \left(
	\vct{u} \cdot \begin{bmatrix} \mtx{\Phi}_{J} & \vct{g} \end{bmatrix}
	\begin{bmatrix} \vct{t}_J \\ \norm{\smash{\vct{t}_{J^c}}} \end{bmatrix} - \sqrt{n} \norm{\smash{\vct{t}_{J^c}}} \right)_+
	- \cnst{C} B^2 \sqrt{k}.
\end{aligned}
\end{equation}
In the second line, we have simply consolidated terms.
\end{proof}

\subsubsection{Lemma~\ref{lem:car-width-id}: Simplifying the Minimax II}

The next step in the proof of Lemma~\ref{lem:car-width-id}
is to reduce the minimax problem in Sublemma~\ref{slem:car-minmax-simple}
to a scalar optimization problem.

\begin{sublemma}[Lemma~\ref{lem:car-width-id}: Simplifying the Minimax II] \label{slem:car-minmax-simpler}
Adopt the notation and hypotheses of Lemma~\ref{lem:car-width-id}.  Then
\begin{multline*}
\Expect \min_{\norm{\vct{t}} = 1} \max_{\vct{u} \in U} \left(
	\vct{u} \cdot \begin{bmatrix} \mtx{\Phi}_J & \vct{g} \end{bmatrix}
	\begin{bmatrix} \vct{t}_j \\ \norm{\smash{\vct{t}_{J^c}}} \end{bmatrix}
	- \sqrt{n} \norm{\smash{\vct{t}_{J^c}}} \right)_+ \\
	\geq \Expect\min_{\alpha \in [0, 1]} \max \Big\{ 0, 
	\big( \max_{\vct{u} \in U} \vct{u} \cdot \vct{g} - \sqrt{n} \big) \alpha,
	\min_{\norm{\vct{s}} = 1} \max_{\vct{u} \in U} \vct{u} \cdot
	\begin{bmatrix} \mtx{\Phi}_{J} & \vct{g} \end{bmatrix} \vct{s}
	- \sqrt{n} \alpha \Big\}
	- \sqrt{k}.
\end{multline*}
\end{sublemma}

\begin{proof}
Introduce the notation
\begin{equation} \label{eqn:car-width-Q2}
Q_2 := \max_{\vct{u} \in U} \left( \vct{u} \cdot \begin{bmatrix} \mtx{\Phi}_{J} & \vct{g} \end{bmatrix}
	\begin{bmatrix} \vct{t}_J \\ \norm{\smash{\vct{t}_{J^c}}} \end{bmatrix} - \sqrt{n} \norm{\smash{\vct{t}_{J^c}}} \right)_+.
\end{equation}
We will develop two lower bounds on the maximum by coupling
$\vct{u}$ to the random matrix in different ways.  Afterward,
we combine these results into a single bound.

In the first place, we can choose the \emph{row} vector $\vct{u}$ so that it depends only
on the remaining Gaussian vector:
$$
\vct{u}(\vct{g}) \in \argmax_{\vct{u} \in U} \vct{u} \cdot \vct{g}.
$$
Since $\norm{\vct{t}} = 1$, we obtain the bound
\begin{equation} \label{eqn:maxU-bound1}
\begin{aligned}
\max_{\vct{u} \in U} \left( \vct{u} \cdot \begin{bmatrix} \mtx{\Phi}_{J} & \vct{g} \end{bmatrix}
	\begin{bmatrix} \vct{t}_J \\ \norm{\smash{\vct{t}_{J^c}}} \end{bmatrix} - \sqrt{n} \norm{\smash{\vct{t}_{J^c}}} \right)
	&\geq \big( \vct{u}(\vct{g}) \cdot \vct{g} - \sqrt{n} \big) \norm{\smash{\vct{t}_{J^c}}}
	+ \vct{u}(\vct{g}) \cdot \mtx{\Phi}_{J} \vct{t}_J \\
	&\geq \big(\max_{\vct{u} \in U} \vct{u} \cdot \vct{g} - \sqrt{n} \big) \norm{\smash{\vct{t}_{J^c}}}
	- \norm{ \vct{u}(\vct{g}) \mtx{\Phi}_{J} }.
\end{aligned}
\end{equation}
The second term on the right-hand side of~\eqref{eqn:maxU-bound1} satisfies
\begin{equation} \label{eqn:maxU-bound1-expect}
\Expect \norm{ \vct{u}(\vct{g}) \mtx{\Phi}_{J} }
	\leq \left( \Expect \normsq{ \vct{u}(\vct{g}) \mtx{\Phi}_{J} } \right)^{1/2}
	\leq \sqrt{k}.
\end{equation}
Indeed, $\vct{u}(\vct{g})$ is a random vector with $\norm{\smash{\vct{u}(\vct{g})}} \leq 1$
that is stochastically independent from $\mtx{\Phi}_{J}$,
and the $m \times k$ random matrix $\mtx{\Phi}_{J}$
has independent, standardized entries.

The second bound is even simpler.  Since $\norm{\vct{t}} = 1$,
\begin{equation} \label{eqn:maxU-bound2}
\max_{\vct{u} \in U} \left( \vct{u} \cdot \begin{bmatrix} \mtx{\Phi}_{J} & \vct{g} \end{bmatrix}
	\begin{bmatrix} \vct{t}_J \\ \norm{\smash{\vct{t}_{J^c}}} \end{bmatrix} - \sqrt{n} \norm{\smash{\vct{t}_{J^c}}} \right)
	\geq \min_{\norm{\vct{s}} = 1} \max_{\vct{u} \in U} \vct{u} \cdot
	\begin{bmatrix} \mtx{\Phi}_{J} & \vct{g} \end{bmatrix} \vct{s}
	- \sqrt{n} \norm{\smash{\vct{t}_{J^c}}}.
\end{equation}
In this expression, the variable $\vct{s} \in \R^{k+1}$.

Introducing~\eqref{eqn:maxU-bound1} and~\eqref{eqn:maxU-bound2} into~\eqref{eqn:car-width-Q2},
we arrive at
\begin{equation} \label{eqn:Eminmax-three-branches}
\begin{aligned}
Q_2 &\geq \Expect \min_{\norm{\vct{t}} = 1} \max \Big\{ 0, 
	\big(\max_{\vct{u} \in U} \vct{u} \cdot \vct{g} - \sqrt{n} \big) \norm{\smash{\vct{t}_{J^c}}}
	- \norm{ \vct{u}(\vct{g}) \mtx{\Phi}_{J} },
	\min_{\norm{\vct{s}} = 1} \max_{\vct{u} \in U} \vct{u} \cdot
	\begin{bmatrix} \mtx{\Phi}_{J} & \vct{g} \end{bmatrix} \vct{s}
	- \sqrt{n} \norm{\smash{\vct{t}_{J^c}}} \Big\} \\
	&\geq \Expect\min_{\alpha \in [0, 1]} \max \Big\{ 0, 
	\big( \max_{\vct{u} \in U} \vct{u} \cdot \vct{g} - \sqrt{n} \big) \alpha,
	\min_{\norm{\vct{s}} = 1} \max_{\vct{u} \in U} \vct{u} \cdot
	\begin{bmatrix} \mtx{\Phi}_{J} & \vct{g} \end{bmatrix} \vct{s}
	- \sqrt{n} \alpha \Big\}
	- \sqrt{k}.
\end{aligned}
\end{equation}
The zero branch in the maximum accounts for the positive-part operator
in~\eqref{eqn:car-width-Q2}.
The second line follows from~\eqref{eqn:maxU-bound1-expect}.
We have also introduced a new parameter $\alpha$ to stand in for $\norm{\smash{\vct{t}_{J^c}}}$.
\end{proof}

\subsubsection{Lemma~\ref{lem:car-width-id}: Probabilistic Bounds}

The last major step in Lemma~\ref{lem:car-width-id}
is to develop probabilistic bounds for the terms
that arise in Sublemma~\ref{slem:car-minmax-simpler}.

\begin{sublemma}[Lemma~\ref{lem:car-width-id}: Probabilistic Bounds] \label{slem:prob-bounds}
Adopt the notation and hypotheses of Lemma~\ref{lem:car-width-id}.  Then
\begin{multline*}
\Expect\min_{\alpha \in [0, 1]} \max \Big\{ 0, 
	\big( \max_{\vct{u} \in U} \vct{u} \cdot \vct{g} - \sqrt{n} \big) \alpha,
	\min_{\norm{\vct{s}} = 1} \max_{\vct{u} \in U} \vct{u} \cdot
	\begin{bmatrix} \mtx{\Phi}_{J} & \vct{g} \end{bmatrix} \vct{s}
	- \sqrt{n} \alpha \Big\} \\
	\geq \frac{1}{2} \min_{\alpha \in [0,1]} \max\left\{ (\coll{W}(U) - \sqrt{n} - 2) \alpha,
	\frac{\coll{W}(U) - 2}{2B\sqrt{\log m}}
	- \sqrt{n} \alpha \right\}
	- \cnst{C} B^2 \sqrt{k \log m}.
\end{multline*}
\end{sublemma}

\begin{proof}
Introduce the notation
$$
Q_3 := \Expect\min_{\alpha \in [0, 1]} \max \Big\{ 0, 
	\big( \max_{\vct{u} \in U} \vct{u} \cdot \vct{g} - \sqrt{n} \big) \alpha,
	\min_{\norm{\vct{s}} = 1} \max_{\vct{u} \in U} \vct{u} \cdot
	\begin{bmatrix} \mtx{\Phi}_{J} & \vct{g} \end{bmatrix} \vct{s}
	- \sqrt{n} \alpha \Big\}.
$$
We assume that $\coll{W}(U) \geq 2 + \sqrt{n}$, which is permitted
because the final result would otherwise become vacuous.

The next stage in the argument is to introduce probabilistic bounds for the
branches of the maximum and use these to control the expectation.
It is convenient to abbreviate
$$
X := \max_{\vct{u} \in U} \vct{u} \cdot \vct{g} - \sqrt{n}
\quad\text{and}\quad
Y := \min_{\norm{\vct{s}} = 1} \max_{\vct{u} \in U} \vct{u} \cdot \begin{bmatrix} \mtx{\Phi}_{J} & \vct{g} \end{bmatrix} \vct{s}.
$$
Note that $\Expect X = \coll{W}(U) - \sqrt{n}$.  Since $\vct{g} \mapsto X(\vct{g})$ is 1-Lipschitz,
the Gaussian concentration inequality, Fact~\ref{fact:gauss-concentration}, implies that
$$
\Prob{ X \geq \coll{W}(U) - \sqrt{n} - 2 } \geq 3/4.
$$
On the other hand, Sublemma~\ref{slem:empirical-width} will demonstrate that
$$
\Prob{ Y \geq \frac{\coll{W}(U)}{\cnst{C} B \sqrt{\log m}} - \cnst{C}B \sqrt{k\log m} } \geq 3/4.
$$
Therefore, taking complements and a union bound,
\begin{equation} \label{eqn:prob-minXY}
\Prob{ X > \coll{W}(U) - \sqrt{n} - 2 \quad\text{and}\quad Y > \frac{\coll{W}(U)}{\cnst{C} B \sqrt{\log m}} - \cnst{C}B\sqrt{k\log m} }
	\geq 1/2.
\end{equation}
For each nonnegative random variable $Z$ and each number $L > 0$,
it holds that $\Expect Z \geq L \Prob{ Z > L }$.  Using the
estimate~\eqref{eqn:Eminmax-three-branches} and the probability bound~\eqref{eqn:prob-minXY},
we find that
\begin{equation} \label{eqn:Eminmax-pen}
\begin{aligned}
Q_3 &\geq \frac{1}{2} \min_{\alpha \in [0,1]} \max\left\{ 0, (\coll{W}(U) - \sqrt{n} - 2) \alpha,
	\left( \frac{\coll{W}(U)}{\cnst{C} B\sqrt{\log m}} - \cnst{C}B\sqrt{k \log m} \right)
	- \sqrt{n} \alpha \right\}
	- \sqrt{k} \\
	&\geq \frac{1}{2} \min_{\alpha \in [0,1]} \max\left\{ (\coll{W}(U) - \sqrt{n} - 2) \alpha,
	\frac{\coll{W}(U) - 2}{\cnst{C} B\sqrt{\log m}}
	- \sqrt{n} \alpha \right\}
	- \cnst{C} B^2 \sqrt{k \log m}.
\end{aligned}
\end{equation}
Once again, we have used shift-invariance of the maximum to combine the error terms.
For convenience, we have also dropped the zero branch
of the maximum and introduced the number two into the numerator in the second branch.
\end{proof}

\subsubsection{Lemma~\ref{lem:car-width-id}: Solving the Scalar Minimax Problem}

The final step in the proof of Lemma~\ref{lem:car-width-id}
is to solve the scalar minimax problem that emerges
in Sublemma~\ref{slem:prob-bounds}.

\begin{sublemma}[Lemma~\ref{lem:car-width-id}: Solving the Minimax Problem]
\label{slem:car-width-minmax}
Adopt the notation and hypotheses of Lemma~\ref{lem:car-width-id}.  Then
$$
\min_{\alpha \in [0,1]} \max\left\{ (\coll{W}(U) - \sqrt{n} - 2) \alpha,
	\frac{\coll{W}(U) - 2}{\cnst{C} B\sqrt{\log m}}
	- \sqrt{n} \alpha \right\}
	\geq \frac{\coll{W}(U) - \sqrt{n} - 2}{\cnst{C} B \sqrt{\log m}}.
$$
\end{sublemma}

\begin{proof}
The first branch of the maximum is increasing in $\alpha$ while the second branch is decreasing in $\alpha$,
so the minimum occurs when the two branches are equal, provided that this situation occurs when $\alpha \in [0,1]$.
Setting the branches equal, we identify the point $\alpha_{\star}$ where the saddle value is achieved.
$$
\alpha_{\star} := \frac{b}{a+c}
\quad\text{where}\quad
a := \coll{W}(U) - \sqrt{n} - 2 \quad\text{and}\quad 
b := \frac{\coll{W}(U) - 2}{\cnst{C}B\sqrt{\log m}} \quad\text{and}\quad
c := \sqrt{n}.
$$
We quickly verify that $\alpha_\star \in [0,1]$, so the minimax takes the value
$$
\frac{ab}{a+c}
	= \frac{ \coll{W}(U) - \sqrt{n} - 2}{\coll{W}(U) - 2} \times
	\frac{\coll{W}(U) - 2}{2B\sqrt{\log m}}
	= \frac{\coll{W}(U) - \sqrt{n} - 2}{\cnst{C}B \sqrt{\log m}}.
$$
This is the required estimate.
\end{proof}

\subsubsection{Sublemma~\ref{slem:prob-bounds}: Probabilistic Lower Bound for Bilinear Minimax}

In this section, we explain how to obtain a lower bound for the
minimax problem in~\eqref{eqn:maxU-bound2} in terms of the
Gaussian width of the set $U$.

\begin{sublemma}[Sublemma~\ref{slem:prob-bounds}: Probability Bound for Bilinear Minimax] \label{slem:matrix-width}
Assume that $k < m$, and let $\mtx{X}$ be an $m \times k$ random matrix
that satisfies Model~\ref{mod:bdd-mtx}
with bound $B$.  Let $\vct{g} \in \R^m$ be standard normal.
Let $U$ be a subset of the unit ball in $\R^m$.  Then
$$
\Prob{ \min_{\norm{\vct{s}} = 1} \max_{\vct{u} \in U} \vct{u} \cdot \begin{bmatrix} \mtx{X} & \vct{g} \end{bmatrix} \vct{s}
	> \frac{\coll{W}(U)}{\cnst{C} B \sqrt{\log m}} - \cnst{C}B \sqrt{k\log m}  } \geq 3/4.
$$
\end{sublemma}

\begin{proof}
Fix $\eps = m^{-1}$.  Let $\coll{N}$ be an $\eps$-net for the unit sphere in $\R^{k+1}$.
The cardinality of the net satisfies $\log( \coll{N} ) \leq (k+1) \log(3m)$ by the
standard volumetric argument~\cite[Lem.~5.2]{Ver12:Introduction-Nonasymptotic}

We can estimate the quantity of interest below by discretizing the parameter $\vct{s}$.
Since $\coll{N}$ and $U$ are subsets of the unit ball, we have the bound
\begin{equation} \label{eqn:matrix-width-pf1}
\min_{\norm{\vct{s}} = 1} \max_{\vct{u} \in U} \vct{u} \cdot \begin{bmatrix} \mtx{X} & \vct{g} \end{bmatrix} \vct{s}
	\geq \min_{\vct{s} \in \coll{N}} \max_{\vct{u}\in U} \vct{u} \cdot \begin{bmatrix} \mtx{X} & \vct{g} \end{bmatrix} \vct{s}
	- \norm{ \begin{bmatrix} \mtx{X} & \vct{g} \end{bmatrix} } \eps
\end{equation}
We will establish a probabilistic lower bound for the right-hand side
of~\eqref{eqn:matrix-width-pf1}.

First, we develop a probability bound for the second term on
the right-hand side of~\eqref{eqn:matrix-width-pf1}.
A simple spectral norm estimate suffices.
The $m \times k$ random matrix $\mtx{X}$ has standardized entries
and $\vct{g} \in \R^m$ is standard normal, so
$$
\Expect \norm{ \begin{bmatrix} \mtx{X} & \vct{g} \end{bmatrix} }
	\leq \Expect \fnorm{ \begin{bmatrix} \mtx{X} & \vct{g} \end{bmatrix} } \leq \sqrt{(k+1)m}.
$$
As usual, $\fnorm{\cdot}$ denotes the Frobenius norm.
Markov's inequality now implies that
$$
\Prob{ \norm{ \begin{bmatrix} \mtx{X} & \vct{g} \end{bmatrix} } \geq 6 \sqrt{(k+1) m} } \geq 5/6.
$$
It follows that
\begin{equation} \label{eqn:matrix-width-pf1.5}
\Prob{ \norm{ \begin{bmatrix} \mtx{X} & \vct{g} \end{bmatrix} } \eps \geq 6 } \geq 5/6.
\end{equation}
We have used the facts that $\eps = m^{-1}$ and $k < m$.

Let us turn to the second quantity on the right-hand side of~\eqref{eqn:matrix-width-pf1}.
We develop a strong probability bound for each fixed point $\vct{s} \in \coll{N}$, and
we extend it to the entire net using the union bound.  For technical reasons,
it is easier to treat the random matrix $\mtx{X}$ and the random vector $\vct{g}$ separately.

Fix a point $\vct{s} \in \coll{N}$, and decompose it as
$\vct{s} = \begin{bmatrix} \vct{s}_1 & s_{k+1} \end{bmatrix}$
where $\vct{s}_1 \in \R^k$.  Construct a random vector
$\vct{u} \in U$ that satisfies
$$
\vct{u}(\mtx{X}) \in \argmax_{\vct{u} \in U} \vct{u} \cdot \mtx{X}\vct{s}_1.
$$
We may calculate that
\begin{equation} \label{eqn:matrix-width-pf2}
\begin{aligned}
\max_{\vct{u} \in U} \vct{u} \cdot \begin{bmatrix} \mtx{X} & \vct{g} \end{bmatrix} \vct{s}
	&\geq \vct{u}(\mtx{X}) \cdot \begin{bmatrix} \mtx{X} & \vct{g} \end{bmatrix}
	\begin{bmatrix} \vct{s}_1 \\ s_{k+1} \end{bmatrix} \\
	&= \vct{u}(\mtx{X}) \cdot \mtx{X} \vct{s}_1 + (\vct{u}(\mtx{X}) \cdot \vct{g}) s_{k+1} \\
	&\geq \max_{\vct{u} \in S} \vct{u} \cdot \mtx{X} \vct{s}_1 - \abs{ \vct{u}(\mtx{X}) \cdot \vct{g} }.
\end{aligned}
\end{equation}
The last estimate relies on the fact that $\abs{s_{k+1}} \leq 1$ because $\norm{\vct{s}} = 1$.

The second term on the right-hand side of~\eqref{eqn:matrix-width-pf2} is easy to handle
using the Gaussian concentration inequality, Fact~\ref{fact:gauss-concentration}:
\begin{equation} \label{eqn:matrix-width-pf3}
\Prob{ \abs{\vct{u}(\mtx{X}) \cdot \vct{g}} \geq \zeta } \leq 2 \econst^{-\zeta^2/2}
\end{equation}
Indeed, $\vct{g} \mapsto \vct{u}(\mtx{X}) \cdot \vct{g}$ is a 1-Lipschitz function
with mean zero because the random vector $\vct{u}(\mtx{X})$ is stochastically independent
from $\vct{g}$ and has norm bounded by one.

We can interpret the first term on the right-hand side of~\eqref{eqn:matrix-width-pf2}
as an ``empirical width.''  It takes \joelprev{some work} to compare this
quantity with the Gaussian width.  Sublemma~\ref{slem:empirical-width} contains
a bound for the expectation, and Sublemma~\ref{slem:empirical-width-conc}
contains a tail bound.  Together, they deliver the probability inequality
\begin{equation} \label{eqn:matrix-width-pf4}
\Prob{ \max_{\vct{u}\in U} \vct{u} \cdot \mtx{X}\vct{s}_1
	\leq \frac{\coll{W}(U)}{\cnst{C} B \sqrt{\log m}} - \zeta }
	\leq \econst^{-\zeta^2/(8B^2)}
	\quad\text{for each $\vct{s} \in \coll{N}$.}
\end{equation}
In other words, the empirical width of $U$ is comparable with the Gaussian width,
modulo a logarithmic factor.

Introduce the two probability bounds~\eqref{eqn:matrix-width-pf3} and~\eqref{eqn:matrix-width-pf4}
into the deterministic estimate~\eqref{eqn:matrix-width-pf2}.  We arrive at
$$
\Prob{ \max_{\vct{u} \in U} \vct{u} \cdot \begin{bmatrix} \mtx{X} & \vct{g} \end{bmatrix} \vct{s} 
	\leq \frac{\coll{W}(U)}{\cnst{C} B \sqrt{\log m}} - 2 \zeta }
	\leq 3 \econst^{-\zeta^2/(8B^2)}
	\quad\text{for each $\vct{s} \in \coll{N}$.}
$$
Finally, we take a union bound over $\vct{s} \in \coll{N}$ to obtain an estimate that is uniform over the net.
Recall that $\log(\#\coll{N}) \leq (k+1) \log(3m)$ and select $\zeta = 4 B \sqrt{(k+1)\log(3m)}$
to reach
\begin{multline} \label{eqn:matrix-width-pf5}
\Prob{ \min_{\vct{s} \in \coll{N}} \max_{\vct{u} \in U} \vct{u} \cdot \mtx{X}\vct{s}_1
	\leq \frac{\coll{W}(U)}{\cnst{C} B \sqrt{\log m}} - \cnst{C}B \sqrt{(k+1) \log(3m)} } \\
	\leq 3 (\# \coll{N}) \econst^{-\zeta^2/(8B^2)}
	\leq 3 \econst^{-(k+1) \log (3m)}
	= 3 (3m)^{-(k+1)}
	\leq 1/18.
\end{multline}
The numerical estimate holds because $1 \leq k < m$.

The two probability bounds~\eqref{eqn:matrix-width-pf1.5} and~\eqref{eqn:matrix-width-pf5}
hold simultaneously with probability at least $3/4$.  Therefore, we can substitute these
results into~\eqref{eqn:matrix-width-pf1} and adjust constants to obtain the stated bound.
\end{proof}

\subsubsection{Sublemma~\ref{slem:prob-bounds}: Lower Estimate for the Empirical Width}

The next sublemma demonstrates that the Gaussian width
of a set is not more than a logarithmic factor larger
than the empirical width of the set as computed with
bounded random variables.  This is the only step in
the argument for bounded random matrices that requires
the symmetry assumption.

\begin{sublemma}[Sublemma~\ref{slem:prob-bounds}: Empirical Width Bound] \label{slem:empirical-width}
Adopt the notation and hypotheses of Sublemma~\ref{slem:matrix-width},
and let $\vct{s}$ be a fixed unit-norm vector.
Then
$$
\Expect \max_{\vct{u} \in U} \vct{u} \cdot \mtx{X} \vct{s}
	\geq \frac{\coll{W}(U)}{\cnst{C} B \sqrt{\log m}}.
$$
\end{sublemma}

\begin{proof}[Proof of Sublemma~\ref{slem:empirical-width}]
Define the random vector $\vct{v} := (V_1, \dots, V_n) := \mtx{X}\vct{s}$.
Our goal is to compare the empirical width of the set $U$ computed using the vector
$\vct{v}$ with the Gaussian width of the set.

First, we develop a lower bound on the first moment of the entries of $\vct{v}$.
Fix an index $i$.  Since the entries of $\mtx{X}$ are independent and symmetric
$$
\Expect \abs{V_i} = \Expect \abs{ \sum_{j=1}^k X_{ij} s_j }
	= \Expect \abs{ \sum_{j=1}^k \eta_{j} X_{ij} s_j }
$$
where $\{\eta_j\}$ is an independent family of Rademacher random variables, independent from $\mtx{X}$.
The Khintchine inequality~\cite{LO94:Best-Constant} allows us to compare the first moment
with the second moment:
$$
\Expect \abs{V_i} \geq \frac{1}{\sqrt{2}} \Expect_{\mtx{X}} \left( \Expect_{\vct{\eta}}
	\abssq{\sum_{j=1}^k \eta_j X_{ij} s_j} \right)^{1/2}
	= \frac{1}{\sqrt{2}} \Expect \left( \sum_{j=1}^k \abssq{\smash{X_{ij}}} \abssq{\smash{s_j}} \right)^{1/2}.
$$
Since $\vct{s}$ has unit norm, we can regard the sum as a weighted average, and we can invoke Jensen's inequality
to draw the average out of the square root:
$$
\Expect \abs{V_i} \geq \frac{1}{\sqrt{2}} \sum_{j=1}^k \big(\Expect \abs{\smash{X_{ij}}}\big) \abssq{\smash{s_j}}.
$$
Last, note that $1 = \Expect \abssq{\smash{X_{ij}}} \leq B \Expect \abs{\smash{X_{ij}}}$
because the entries of $\mtx{X}$ are standardized and bounded by $B$.  Thus,
$$
\Expect \abs{V_i} \geq \frac{1}{\sqrt{2}} \sum_{j=1}^k \frac{1}{B} \abssq{\smash{s_j}}
	= \frac{1}{B\sqrt{2}}.
$$

Let us bound the width-like functional below.
Since $\vct{v}$ has independent, symmetric coordinates,
$$
\Expect \max_{\vct{u} \in U} \vct{u} \cdot \mtx{X}\vct{s}
	= \Expect \max_{\vct{u} \in U} \vct{u} \cdot \vct{v}
	= \Expect \max_{\vct{u} \in U} \sum_{i=1}^n  V_i u_i
	= \Expect \max_{\vct{u} \in U} \sum_{i=1}^n \eta_i\abs{V_i} u_i 
$$
where, again, $\{\eta_i\}$ is an independent family of Rademacher random variables,
independent from $\vct{u}$.  Using a corollary of the contraction principle~\cite[Lem.~4.5]{LT91:Probability-Banach},
$$
\Expect \max_{\vct{u} \in U} \vct{u} \cdot \mtx{X}\vct{s}
	\geq \min\nolimits_i (\Expect \abs{V_i}) \, \Expect_{\vct{\eta}} \max_{\vct{u} \in U} \sum_{i=1}^n \eta_i u_i
	\geq \frac{1}{B\sqrt{2}} \Expect \max_{\vct{u} \in U} \sum_{i=1}^n \eta_i u_i.
$$
Applying the contraction principle again~\cite[Eqn.~(4.9)]{LT91:Probability-Banach},
we can randomize the sum with independent Gaussian variables:
$$
\Expect \max_{\vct{u} \in U} \vct{u} \cdot \mtx{X}\vct{s}
	\geq \frac{1}{2B \sqrt{\log n}} \Expect \max_{\vct{u}\in U} \sum_{i=1}^n g_i u_i 
$$
Here, $\vct{g} := (g_1, \dots, g_n)$ is a standard normal vector.
Identify the Gaussian width to complete the proof.
\end{proof}

\subsubsection{Sublemma~\ref{slem:prob-bounds}: \joel{Lower Tail of the Empirical Width}}

Last, we present a concentration inequality for the empirical width.

\begin{sublemma}[Sublemma~\ref{slem:prob-bounds}: \joel{Lower Tail of Empirical Width}] \label{slem:empirical-width-conc}
Adopt the notation and hypotheses of Sublemma~\ref{slem:matrix-width}.
Let $\vct{s}$ be a fixed unit-norm vector.
Then
$$
\Prob{ \max_{\vct{u} \in U} \vct{u} \cdot \mtx{X}\vct{s}
	\leq \Expect \max_{\vct{u} \in U} \vct{u} \cdot \mtx{X}\vct{s} - \zeta }
	\leq \econst^{-\zeta^2/8B^2}.
$$
\end{sublemma}
\joel{
The most direct proof of this result relies on a version of
Talagrand's inequality~\cite[Thm.~8.6]{BLM13:Concentration-Inequalities}
obtained from the transportation cost method.}
\joel{
\begin{fact}[Talagrand's Inequality] \label{fact:talagrand}
Let $\mathcal{X}$ be a metric space.
Suppose that $f : \mathcal{X}^p \to \R$ fulfills the one-sided Lipschitz bound
$$
f(\vct{a}) - f(\vct{z}) \leq \sum_{i=1}^p c_i(\vct{a}) \mathbb{1}_{a_i \neq z_i}
\quad\text{for all $\vct{a}, \vct{z} \in \mathcal{X}^p$,} 
$$
where $c_i : \R^p \to \R$ are auxiliary functions that satisfy
$$
\sum_{i=1}^p c_i^2(\vct{a}) \leq v
\quad\text{for all $\vct{a} \in \mathcal{X}^p$.}
$$
Let $(X_1, \dots, X_p)$ be an independent sequence of random variables
taking values in $\mathcal{X}$, and define
$$
Y := f( X_1, \dots, X_p ).
$$
Then, for all $\zeta \geq 0$,
$$
\begin{aligned}
\Prob{ Y - \Expect Y \geq \zeta } &\leq \econst^{-\zeta^2/(2v)}, \quad\text{and} \\
\Prob{ Y - \Expect Y \leq - \zeta } &\leq \econst^{-\zeta^2/(2v)}.
\end{aligned}
$$
\end{fact}}

\joel{
\begin{proof}[Proof of Sublemma~\ref{slem:empirical-width-conc}]
Let $\mathcal{X}$ be the interval $[-B, B]$ equipped with the Euclidean
distance.  For a matrix $\mtx{A} \in \mathcal{X}^{m \times k}$,
introduce the function
$$
f(\mtx{A}) := \max_{\vct{u} \in U} \vct{u} \cdot \mtx{A} \vct{s}.
$$
Select a point $\vct{t} \in \argmax_{\vct{u} \in U} \vct{u} \cdot \mtx{A}\vct{s}$.
Then
$$
\begin{aligned}
f(\mtx{A}) - f(\mtx{Z})
	&\leq \vct{t} \cdot \mtx{A} \vct{s} - \vct{t} \cdot \mtx{Z} \vct{s} \\
	&= \sum_{i=1}^m \sum_{j=1}^k t_i s_j (a_{ij} - z_{ij}) \\
	&\leq \sum_{i=1}^m \sum_{j=1}^k 2B \abs{t_i s_j} \mathbb{1}_{a_{ij} \neq z_{ij}}.
\end{aligned}
$$
We used the fact that the entries of $\mtx{A}$ and $\mtx{Z}$ are bounded in magnitude by $B$.
With the choice $c_{ij}(\mtx{A}) = 2B \abs{t_i s_j}$, we see that
$$
\sum_{i=1}^m \sum_{j=1}^{k} c_{ij}^2(\mtx{A}) \leq 4B^2 \sum_{i,j} \abssq{t_i s_j} \leq 4B^2
$$
because $\vct{t}, \vct{s}$ both belong to the Euclidean unit ball.
Invoke Fact~\ref{fact:talagrand} to complete the argument.
\end{proof}}

\newpage

\part{Back Matter}
\label{part:back-matter}

Two appendices contain statements of some results that we use throughout
the paper.  Appendix~\ref{app:gauss} presents some facts about Gaussian
analysis, while Appendix~\ref{app:norm} describes some spectral bounds for
random matrices with independent entries.  We conclude with acknowledgments
and a list of works cited.

\appendix

\section{Tools from Gaussian Analysis}
\label{app:gauss}

We make extensive use of methods from Gaussian analysis to provide
precise information about the behavior of functions of Gaussian random variables.
These results come up in many places in the paper, so we have collected
them here.

\subsection{Concentration for Gaussian Lipschitz Functions}

We begin with two concentration results that apply to a Lipschitz function
of independent Gaussian variables.
Recall that a function $f : \R^n \to \R$ has Lipschitz constant $L$ when
$$
\abs{ f( \vct{a} )  - f( \vct{b} ) } \leq L \norm{ \vct{a} - \vct{b} }
\quad\text{for all $\vct{a}, \vct{b} \in \R^n$.}
$$
We also say, more briefly, that $f$ is $L$-Lipschitz.
The first result~\cite[Thm.~1.6.4]{Bog98:Gaussian-Measures}
gives a bound on the variance of a Lipschitz function.
The second result~\cite[Thm.~1.7.6]{Bog98:Gaussian-Measures}
provides a normal concentration inequality for Lipschitz functions.

\begin{fact}[Gaussian Variance Inequality] \label{fact:gauss-variance}
Suppose that $f : \R^n \to \R$ has Lipschitz constant $L$.
Let $\vct{\gamma} \in \R^n$ be a standard normal random vector.  Then
$$
\Var[ f( \vct{\gamma} ) ] \leq L.
$$
Equivalently,
$$
\Expect f(\vct{\gamma})^2 \leq \big( \Expect f(\vct{\gamma}) \big)^2 + L.
$$
\end{fact}

\begin{fact}[Gaussian Concentration Inequality] \label{fact:gauss-concentration}
Suppose that $f : \R^n \to \R$ has Lipschitz constant $L$.
Let $\vct{\gamma} \in \R^n$ be a standard normal random vector.  Then, for all $\zeta \geq 0$,
$$
\begin{aligned}
\Prob{ f(\vct{\gamma}) \geq \Expect f(\vct{\gamma}) + \zeta } &\leq \econst^{-\zeta^2/2}, \quad\text{and} \\
\Prob{ f(\vct{\gamma}) \leq \Expect f(\vct{\gamma}) - \zeta } &\leq \econst^{-\zeta^2/2}.
\end{aligned}
$$
\end{fact}

\subsection{The Gaussian Minimax Theorem}

To compute the expectations of certain functions of Gaussian random variables,
we depend on a comparison principle due to Gordon~\cite[Thm.~1.1]{Gor85:Some-Inequalities}.

Let $S$ be an abstract set.  A family $\{ Z_{s} : s \in S \}$ of real random variables
is called a \term{centered Gaussian process} when each element $Z_s$ has mean zero
and each finite subcollection $\{ Z_{s_1}, \dots, Z_{s_n} \}$ has a jointly Gaussian
distribution.

\begin{fact}[Gaussian Minimax Theorem] \label{fact:gauss-minmax}
Let $T$ and $U$ be finite sets.  Consider two centered Gaussian
processes $\{ X_{tu} \}$ and $\{ Y_{tu} \}$, indexed over $T \times U$.
For all choices of indices, suppose that
$$
\begin{cases}
\Expect X_{tu}^2 = \Expect Y_{tu}^2 \\
\Expect X_{tu} X_{tu'} \leq \Expect Y_{tu} Y_{tu'} \\
\Expect X_{tu} X_{t'u'} \geq \Expect Y_{tu} Y_{t'u'} & \text{when $t \neq t'$}.
\end{cases}
$$
Then, for all real numbers $\lambda_{tu}$ and $\zeta$,
$$
\Prob{ \min_{t \in T} \max_{u \in U} \big( \lambda_{tu} + X_{tu} \big) \geq \zeta }
	\geq \Prob{ \min_{t \in T} \max_{u \in U} \big( \lambda_{tu} + Y_{tu} \big) \geq \zeta }.
$$
\end{fact}

\noindent
Fact~\ref{fact:gauss-minmax} extends to infinite index sets $T, U$ by approximation.

\section{Spectral Bounds for Random Matrices}
\label{app:norm}

Our argument also depends heavily on some non-asymptotic bounds
for the spectrum of a random matrix with independent entries.
These results only give rough estimates, but they are adequate
for our purposes.

The first result gives tail bounds for the extreme singular values
of a rectangular matrix with independent, subgaussian entries.

\begin{fact}[Subgaussian Matrix: Tail Bounds] \label{fact:subgauss-mtx-tails}
Let $\mtx{X}$ be an $d_1 \times d_2$ random matrix with
independent, standardized entries that are uniformly subgaussian:
\begin{equation} \label{eqn:subgauss-hyp}
\sup_{p \geq 1} \frac{1}{\sqrt{p}}\left( \Expect \abs{\smash{X_{ij}}}^p \right)^{1/p} \leq B. 
\end{equation}
Then the largest singular value $\smax(\mtx{X})$ and the $d_2$-th largest singular
value $\smin(\mtx{X})$ satisfy the bounds
$$
\begin{aligned}
\Prob{ \smax(\mtx{X}) > \sqrt{d_1} + \cnst{C} B^2 \sqrt{d_2} + \cnst{C} B^2 \zeta }
	&\leq \econst^{- \zeta^2} \\
\Prob{ \smin(\mtx{X}) < \sqrt{d_1} - \cnst{C} B^2 \sqrt{d_2} - \cnst{C} B^2 \zeta }
	&\leq \econst^{- \zeta^2}.
\end{aligned}
$$
\end{fact}

\noindent
This result follows from~\cite[Thm.~5.39]{Ver12:Introduction-Nonasymptotic}
when we track the role of the subgaussian constant through the proof.

The second result gives a tail bound for the norm of a matrix
with independent entries that may only have two moments;
it is based on the matrix Rosenthal inequality~\cite[Thm.~1.1]{Tro15:Expected-Norm}
and a standard concentration inequality~\cite[Thm.~15.5]{BLM13:Concentration-Inequalities}.

\begin{fact}[Heavy-Tailed Matrix: Norm Bound] \label{fact:heavy-tail-norm}
Fix a parameter $p \in [2, \log(d_1 + d_2)]$.
Let $\mtx{X}$ be a $d_1 \times d_2$ random matrix with independent entries
that have the following properties.

\begin{itemize}
\item	The entries are centered: $\Expect X_{ij} = 0$.

\item	The variances of the entries are uniformly bounded: $\Var(X_{ij}) \leq \alpha$.

\item	The entries have uniformly bounded $p$th moments:
$\Expect \abs{\smash{X_{ij}}}^p \leq \nu^p$.
\end{itemize}

\noindent
Then
$$
\Prob{ \norm{\mtx{X}} \geq \cnst{C}\sqrt{\alpha (d_1 + d_2) \log(d_1+d_2)} + \big( \cnst{C} \nu (d_1 + d_2)^{2/p}\log(d_1 + d_2)\big)\, \zeta}
	\leq \zeta^{-p}.
$$
\end{fact}

\begin{proof}[Proof Sketch]
Write the random matrix as a sum of independent random matrices:
$$
\mtx{X} = \sum_{i=1}^{d_1} \sum_{j=1}^{d_2} X_{ij} \mathbf{E}_{ij},
$$
where $\mathbf{E}_{ij}$ is the $d_1 \times d_2$ matrix with a one in the
$(i, j)$ position and zeros elsewhere.
A straightforward application of
the matrix Rosenthal inequality~\cite[Thm.~I]{Tro15:Expected-Norm}
yields
$$
\begin{aligned}
\Expect \norm{ \mtx{X} }
	&\leq \cnst{C} \sqrt{\alpha (d_1 + d_2) \log(d_1+d_2)}
	+ \cnst{C} \big(\Expect \max\nolimits_{ij} \abs{\smash{X_{ij}}}^p \big)^{1/p} \log(d_1 + d_2) \\
	&\leq \cnst{C} \sqrt{\alpha (d_1 + d_2) \log(d_1+d_2)}
	+ \cnst{C} \nu (d_1 d_2)^{1/p} \log(d_1 + d_2) \\
	&\leq \cnst{C} \sqrt{\alpha (d_1 + d_2) \log(d_1+d_2)}
	+ \cnst{C} \nu (d_1 + d_2)^{2/p} \log(d_1 + d_2).
\end{aligned}
$$
The second line follows when we replace the maximum by a sum
and exploit the uniform moment estimate.
The third line is just the inequality between the geometric and
arithmetic means.

A standard concentration inequality for moments~\cite[Thm.~15.5]{BLM13:Concentration-Inequalities} gives
$$
\left[ \Expect (\norm{\mtx{X}} - \Expect \norm{\mtx{X}})_+^p \right]^{1/p}
	\leq \cnst{C} \sqrt{p} \big( \Expect V_+^{p/2} \big)^{1/p}
$$
In this expression, the variance parameter
$$
V_+ := \sum_{ij} \Expect \big[ \big(\norm{\mtx{X}} - \norm{\smash{\mtx{X}^{(ij)}}} \big)_+^2 \, \big\vert\, \mtx{X} \big].
$$
The $(i, j)$ entry of $\mtx{X}^{(ij)}$ is an independent copy of the corresponding entry of $\mtx{X}$; the remaining entries of the two matrices are the same.
Applying the usual method~\cite[Ex.~3.14]{BLM13:Concentration-Inequalities},
we see that
$$
V_+ \leq \cnst{C} \max\nolimits_{ij} \Expect\big[ \big(X_{ij} - X_{ij}'\big)^2 \, \big\vert\, \mtx{X} \big].
$$
Applying the same considerations as in the last paragraph, we obtain
$$
\big( \Expect V_+^{p/2} \big)^{1/p} \leq \cnst{C} \nu (d_1 + d_2)^{2/p}.
$$
Combine these results and apply Markov's inequality to obtain
the tail bound
$$
\Prob{ \norm{ \mtx{X} } \geq \Expect \norm{\mtx{X}} + \cnst{C} \nu \sqrt{p} (d_1 + d_2)^{2/p} \zeta } \leq \zeta^{-p}.
$$
Introduce the estimate for the expected norm to complete the argument.
\end{proof}

\section*{Acknowledgments}

The authors would like to thank David Donoho, Surya Ganguli, Babak Hassibi,
Michael McCoy, Andreas Maurer, Andrea Montanari, Ivan Nourdin, Giovanni Peccati,
Adrian R{\"o}llin, Jared Tanner, Christos Thrampoulidis, and Madeleine Udell
for helpful conversations.
\joelprev{We also thank the anonymous reviewers and the editors for their careful reading and suggestions.}
SO was generously supported by the Simons Institute for the Theory of Computing and NSF award CCF-1217058.
JAT gratefully acknowledges support from ONR award N00014-11-1002 and the Gordon \& Betty Moore Foundation.

\bibliographystyle{myalpha}

\begin{thebibliography}{DDW{\etalchar{+}}07}

\bibitem[ALMT14]{ALMT14:Living-Edge}
D.~Amelunxen, M.~Lotz, M.~B. McCoy, and J.~A. Tropp.
\newblock Living on the edge: phase transitions in convex programs with random
  data.
\newblock {\em Inf. Inference}, 3(3):224--294, 2014.

\bibitem[BDN15]{BDN15:Toward-Unified}
J.~Bourgain, S.~Dirksen, and J.~Nelson.
\newblock Toward a unified theory of sparse dimensionality reduction in
  {E}uclidean space.
\newblock {\em Geom. Funct. Anal.}, 25(4):1009--1088, 2015.

\bibitem[BGVV14]{BGVV14:Geometry-Isotropic}
S.~Brazitikos, A.~Giannopoulos, P.~Valettas, and B.-H. Vritsiou.
\newblock {\em Geometry of isotropic convex bodies}, volume 196 of {\em
  Mathematical Surveys and Monographs}.
\newblock American Mathematical Society, Providence, RI, 2014.

\bibitem[BLM13]{BLM13:Concentration-Inequalities}
S.~Boucheron, G.~Lugosi, and P.~Massart.
\newblock {\em Concentration inequalities}.
\newblock Oxford University Press, Oxford, 2013.
\newblock A nonasymptotic theory of independence, With a foreword by Michel
  Ledoux.

\bibitem[BLM15]{BLM15:Universality-Polytope}
M.~Bayati, M.~Lelarge, and A.~Montanari.
\newblock Universality in polytope phase transitions and message passing
  algorithms.
\newblock {\em Ann. Appl. Probab.}, 25(2):753--822, 2015.

\bibitem[BM12]{BM12:LASSO-Risk}
M.~Bayati and A.~Montanari.
\newblock The {LASSO} risk for {G}aussian matrices.
\newblock {\em IEEE Trans. Inform. Theory}, 58(4):1997--2017, 2012.

\bibitem[Bog98]{Bog98:Gaussian-Measures}
V.~I. Bogachev.
\newblock {\em Gaussian measures}, volume~62 of {\em Mathematical Surveys and
  Monographs}.
\newblock American Mathematical Society, Providence, RI, 1998.

\bibitem[BS10]{BS10:Spectral-Analysis}
Z.~Bai and J.~W. Silverstein.
\newblock {\em Spectral analysis of large dimensional random matrices}.
\newblock Springer Series in Statistics. Springer, New York, second edition,
  2010.

\bibitem[BY93]{BY93:Limit-Smallest}
Z.~D. Bai and Y.~Q. Yin.
\newblock Limit of the smallest eigenvalue of a large-dimensional sample
  covariance matrix.
\newblock {\em Ann. Probab.}, 21(3):1275--1294, 1993.

\bibitem[CDS98]{CDS98:Atomic-Decomposition}
S.~S. Chen, D.~L. Donoho, and M.~A. Saunders.
\newblock Atomic decomposition by basis pursuit.
\newblock {\em SIAM J. Sci. Comput.}, 20(1):33--61, 1998.

\bibitem[Cha06]{Cha06:Generalization-Lindeberg}
S.~Chatterjee.
\newblock A generalization of the {L}indeberg principle.
\newblock {\em Ann. Probab.}, 34(6):2061--2076, 2006.

\bibitem[CRPW12]{CRPW12:Convex-Geometry}
V.~Chandrasekaran, B.~Recht, P.~A. Parrilo, and A.~S. Willsky.
\newblock The convex geometry of linear inverse problems.
\newblock {\em Found. Comput. Math.}, 12(6):805--849, 2012.

\bibitem[CRT06]{CRT06:Robust-Uncertainty}
E.~J. Cand{\`e}s, J.~Romberg, and T.~Tao.
\newblock Robust uncertainty principles: exact signal reconstruction from
  highly incomplete frequency information.
\newblock {\em IEEE Trans. Inform. Theory}, 52(2):489--509, 2006.

\bibitem[CRTV05]{CRTV05:Error-Correction}
E.~Cand{\`e}s, M.~Rudelson, T.~Tao, and R.~Vershynin.
\newblock Error correction via linear programming.
\newblock In {\em Foundations of Computer Science, 2005. FOCS 2005. 46th Annual
  IEEE Symposium on}, pages 668--681, Oct 2005.

\bibitem[CSPW11]{CSPW11:Rank-Sparsity-Incoherence}
V.~Chandrasekaran, S.~Sanghavi, P.~A. Parrilo, and A.~S. Willsky.
\newblock Rank-sparsity incoherence for matrix decomposition.
\newblock {\em SIAM J. Optim.}, 21(2):572--596, 2011.

\bibitem[CW13]{CW13:Low-Rank-Approximation}
K.~L. Clarkson and D.~P. Woodruff.
\newblock Low rank approximation and regression in input sparsity time.
\newblock In {\em S{TOC}'13---{P}roceedings of the 2013 {ACM} {S}ymposium on
  {T}heory of {C}omputing}, pages 81--90. ACM, New York, 2013.

\bibitem[DDW{\etalchar{+}}07]{DDW+07:Smashed-Filter}
M.~A. Davenport, M.~F. Duarte, M.~B. Wakin, J.~N. Laska, D.~Takhar, K.~F.
  Kelly, and R.~G. Baraniuk.
\newblock The smashed filter for compressive classification and target
  recognition.
\newblock In {\em Proc. SPIE}, volume 6498, pages 64980H--64980H--12, 2007.

\bibitem[DGM13]{DGM13:Phase-Transition}
D.~L. Donoho, M.~Gavish, and A.~Montanari.
\newblock The phase transition of matrix recovery from {G}aussian measurements
  matches the minimax {MSE} of matrix denoising.
\newblock {\em Proc. Natl. Acad. Sci. USA}, 110(21):8405--8410, 2013.

\bibitem[DH01]{DH01:Uncertainty-Principles}
D.~L. Donoho and X.~Huo.
\newblock Uncertainty principles and ideal atomic decomposition.
\newblock {\em IEEE Trans. Inform. Theory}, 47(7):2845--2862, 2001.

\bibitem[DJM13]{DJM13:Accurate-Prediction}
D.~L. Donoho, I.~Johnstone, and A.~Montanari.
\newblock Accurate prediction of phase transitions in compressed sensing via a
  connection to minimax denoising.
\newblock {\em IEEE Trans. Inform. Theory}, 59(6):3396--3433, 2013.

\bibitem[Don06a]{Don06:Compressed-Sensing}
D.~L. Donoho.
\newblock Compressed sensing.
\newblock {\em IEEE Trans. Inform. Theory}, 52(4):1289--1306, 2006.

\bibitem[Don06b]{Don06:Most-Large-II}
D.~L. Donoho.
\newblock For most large underdetermined systems of linear equations the
  minimal {$l\sb 1$}-norm solution is also the sparsest solution.
\newblock {\em Comm. Pure Appl. Math.}, 59(6):797--829, 2006.

\bibitem[Don06c]{Don06:High-Dimensional-Centrally}
D.~L. Donoho.
\newblock High-dimensional centrally symmetric polytopes with neighborliness
  proportional to dimension.
\newblock {\em Discrete Comput. Geom.}, 35(4):617--652, 2006.

\bibitem[DS01]{DS01:Local-Operator}
K.~R. Davidson and S.~J. Szarek.
\newblock Local operator theory, random matrices and {B}anach spaces.
\newblock In {\em Handbook of the geometry of {B}anach spaces, {V}ol. {I}},
  pages 317--366. North-Holland, Amsterdam, 2001.

\bibitem[DT06]{DT06:Thresholds-Recovery}
D.~Donoho and J.~Tanner.
\newblock Thresholds for the recovery of sparse solutions via l1 minimization.
\newblock In {\em Information Sciences and Systems, 2006 40th Annual Conference
  on}, pages 202--206, March 2006.

\bibitem[DT09a]{DT09:Observed-Universality}
D.~Donoho and J.~Tanner.
\newblock Observed universality of phase transitions in high-dimensional
  geometry, with implications for modern data analysis and signal processing.
\newblock {\em Philos. Trans. R. Soc. Lond. Ser. A Math. Phys. Eng. Sci.},
  367(1906):4273--4293, 2009.
\newblock With electronic supplementary materials available online.

\bibitem[DT09b]{DT09:Counting-Faces}
D.~L. Donoho and J.~Tanner.
\newblock Counting faces of randomly projected polytopes when the projection
  radically lowers dimension.
\newblock {\em J. Amer. Math. Soc.}, 22(1):1--53, 2009.

\bibitem[DT10]{DT10:Counting-Faces}
D.~L. Donoho and J.~Tanner.
\newblock Counting the faces of randomly-projected hypercubes and orthants,
  with applications.
\newblock {\em Discrete Comput. Geom.}, 43(3):522--541, 2010.

\bibitem[EK12]{EK12:Compressed-Sensing}
Y.~C. Eldar and G.~Kutyniok, editors.
\newblock {\em Compressed sensing}.
\newblock Cambridge University Press, Cambridge, 2012.
\newblock Theory and applications.

\bibitem[Faz02]{Faz02:Matrix-Rank}
M.~Fazel.
\newblock {\em Matrix rank minimization with applications}.
\newblock PhD thesis, Stanford, 2002.

\bibitem[FM14]{FM14:Corrupted-Sensing}
R.~Foygel and L.~Mackey.
\newblock Corrupted sensing: novel guarantees for separating structured
  signals.
\newblock {\em IEEE Trans. Inform. Theory}, 60(2):1223--1247, 2014.

\bibitem[FR13]{FR13:Mathematical-Introduction}
S.~Foucart and H.~Rauhut.
\newblock {\em A mathematical introduction to compressive sensing}.
\newblock Applied and Numerical Harmonic Analysis. Birkh\"auser/Springer, New
  York, 2013.

\bibitem[GG15]{GG15:Simplicity-Complexity}
P.~Gao and S.~Ganguli.
\newblock On simplicity and complexity in the brave new world of large-scale
  neuroscience.
\newblock {\em Curr. Opinion Neurobiology}, 32:148--155, 2015.

\bibitem[GNP14]{GNP14:Gaussian-Phase}
L.~Goldstein, I.~Nourdin, and G.~Peccati.
\newblock {G}aussian phase transitions and conic intrinsic volumes: {S}teining
  the {S}teiner formula.
\newblock Available at \url{http://arXiv.org/abs/1411.6265}, Nov. 2014.

\bibitem[Gor85]{Gor85:Some-Inequalities}
Y.~Gordon.
\newblock Some inequalities for {G}aussian processes and applications.
\newblock {\em Israel J. Math.}, 50(4):265--289, 1985.

\bibitem[Gor88]{Gor88:Milmans-Inequality}
Y.~Gordon.
\newblock On {M}ilman's inequality and random subspaces which escape through a
  mesh in {${\bf R}\sp n$}.
\newblock In {\em Geometric aspects of functional analysis (1986/87)}, volume
  1317 of {\em Lecture Notes in Math.}, pages 84--106. Springer, Berlin, 1988.

\bibitem[Gru07]{Gru07:Convex-Discrete}
P.~M. Gruber.
\newblock {\em Convex and discrete geometry}, volume 336 of {\em Grundlehren
  der Mathematischen Wissenschaften [Fundamental Principles of Mathematical
  Sciences]}.
\newblock Springer, Berlin, 2007.

\bibitem[HMT11]{HMT11:Finding-Structure}
N.~Halko, P.~G. Martinsson, and J.~A. Tropp.
\newblock Finding structure with randomness: probabilistic algorithms for
  constructing approximate matrix decompositions.
\newblock {\em SIAM Rev.}, 53(2):217--288, 2011.

\bibitem[JM14]{JM14:Hypothesis-Testing}
A.~Javanmard and A.~Montanari.
\newblock Hypothesis testing in high-dimensional regression under the
  {G}aussian random design model: asymptotic theory.
\newblock {\em IEEE Trans. Inform. Theory}, 60(10):6522--6554, 2014.

\bibitem[Kle55]{Kle55:Separation-Properties}
V.~L. Klee, Jr.
\newblock Separation properties of convex cones.
\newblock {\em Proc. Amer. Math. Soc.}, 6:313--318, 1955.

\bibitem[KM11]{KM11:Applications-Lindeberg}
S.~B. Korada and A.~Montanari.
\newblock Applications of the {L}indeberg principle in communications and
  statistical learning.
\newblock {\em IEEE Trans. Inform. Theory}, 57(4):2440--2450, 2011.

\bibitem[KN14]{KN14:Sparser-Johnson-Lindenstrauss}
D.~M. Kane and J.~Nelson.
\newblock Sparser {J}ohnson-{L}indenstrauss transforms.
\newblock {\em J. ACM}, 61(1):Art. 4, 23, 2014.

\bibitem[KY14]{KY14:Anisotropic-Local}
A.~Knowles and J.~Yin.
\newblock Anisotropic local laws for random matrices.
\newblock Available at \url{http://arXiv.org/abs/1410.3516}, Nov. 2014.

\bibitem[Lin22]{Lin22:Eine-Neue}
J.~W. Lindeberg.
\newblock Eine neue {H}erleitung des {E}xponentialgesetzes in der
  {W}ahrscheinlichkeitsrechnung.
\newblock {\em Math. Z.}, 15(1):211--225, 1922.

\bibitem[LO94]{LO94:Best-Constant}
R.~Lata{\l}a and K.~Oleszkiewicz.
\newblock On the best constant in the {K}hinchin-{K}ahane inequality.
\newblock {\em Studia Math.}, 109(1):101--104, 1994.

\bibitem[LT11]{LT91:Probability-Banach}
M.~Ledoux and M.~Talagrand.
\newblock {\em Probability in {B}anach spaces}.
\newblock Classics in Mathematics. Springer-Verlag, Berlin, 2011.
\newblock Isoperimetry and processes, Reprint of the 1991 edition.

\bibitem[Mah11]{Mah11:Randomized-Algorithms}
M.~W. Mahoney.
\newblock Randomized algorithms for matrices and data.
\newblock {\em Foundations and Trends{\textregistered} in Machine Learning},
  3(2):123--224, 2011.

\bibitem[Mas00]{Mas00:About-Constants}
P.~Massart.
\newblock About the constants in {T}alagrand's concentration inequalities for
  empirical processes.
\newblock {\em Ann. Probab.}, 28(2):863--884, 2000.

\bibitem[Mau12]{Mau12:Thermodynamics-Concentration}
A.~Maurer.
\newblock Thermodynamics and concentration.
\newblock {\em Bernoulli}, 18(2):434--454, 2012.

\bibitem[Men10]{Men10:Empirical-Processes}
S.~Mendelson.
\newblock Empirical processes with a bounded {$\psi\sb 1$} diameter.
\newblock {\em Geom. Funct. Anal.}, 20(4):988--1027, 2010.

\bibitem[Men14]{Men13:Remark-Diameter}
S.~Mendelson.
\newblock A remark on the diameter of random sections of convex bodies.
\newblock In B.~Klartag and E.~Milman, editors, {\em Geometric Aspects of
  Functional Analysis}, volume 2116 of {\em LNM}, pages 395--404. Springer,
  2014.

\bibitem[MOO10]{MOO10:Noise-Stability}
E.~Mossel, R.~O'Donnell, and K.~Oleszkiewicz.
\newblock Noise stability of functions with low influences: invariance and
  optimality.
\newblock {\em Ann. of Math. (2)}, 171(1):295--341, 2010.

\bibitem[MPTJ07]{MPT07:Reconstruction-Subgaussian}
S.~Mendelson, A.~Pajor, and N.~Tomczak-Jaegermann.
\newblock Reconstruction and subgaussian operators in asymptotic geometric
  analysis.
\newblock {\em Geom. Funct. Anal.}, 17(4):1248--1282, 2007.

\bibitem[MT13]{MT13:Achievable-Performance}
M.~B. McCoy and J.~A. Tropp.
\newblock The achievable performance of convex demixing.
\newblock Available at \url{http://arXiv.org/abs/arXiv:1309.7478}, Sep. 2013.

\bibitem[MT14a]{MT14:Steiner-Formulas}
M.~B. McCoy and J.~A. Tropp.
\newblock From {S}teiner formulas for cones to concentration of intrinsic
  volumes.
\newblock {\em Discrete Comput. Geom.}, 51(4):926--963, 2014.

\bibitem[MT14b]{MT14:Sharp-Recovery}
M.~B. McCoy and J.~A. Tropp.
\newblock Sharp recovery bounds for convex demixing, with applications.
\newblock {\em Found. Comput. Math.}, 14(3):503--567, 2014.

\bibitem[NN13]{NN13:OSNAP-Faster}
J.~Nelson and H.~L. Nguyen.
\newblock O{SNAP}: faster numerical linear algebra algorithms via sparser
  subspace embeddings.
\newblock In {\em 2013 {IEEE} 54th {A}nnual {S}ymposium on {F}oundations of
  {C}omputer {S}cience---{FOCS} 2013}, pages 117--126. IEEE Computer Soc., Los
  Alamitos, CA, 2013.

\bibitem[OH10]{OH10:New-Null-Space}
S.~Oymak and B.~Hassibi.
\newblock New null space results and recovery thresholds for matrix rank
  minimization.
\newblock Available at \url{http://arXiv.org/abs/1011.6326}, Nov. 2010.

\bibitem[OH13]{OH13:Asymptotically-Exact}
S.~Oymak and B.~Hassibi.
\newblock Asymptotically exact denoising in relation to compressed sensing.
\newblock Available at \url{http://arXiv.org/abs/1305.2714}, May 2013.

\bibitem[OTH13]{OTH13:Squared-Error}
S.~Oymak, C.~Thrampoulidis, and B.~Hassibi.
\newblock The squared-error of generalized lasso: A precise analysis.
\newblock In {\em Communication, Control, and Computing (Allerton), 2013 51st
  Annual Allerton Conference on}, pages 1002--1009, Oct 2013.
\newblock Available at \url{http://arXiv.org/abs/1311.0830}.

\bibitem[PW15]{PW15:Randomized-Sketches}
M.~Pilanci and M.~Wainwright.
\newblock Randomized sketches of convex programs with sharp guarantees.
\newblock {\em Information Theory, IEEE Transactions on}, 61(9):5096--5115,
  Sept 2015.

\bibitem[Rot73]{Rot73:Certain-Limit}
V.~I. Rotar'.
\newblock Certain limit theorems for polynomials of degree two.
\newblock {\em Teor. Verojatnost. i Primenen.}, 18:527--534, 1973.

\bibitem[RV08]{RV08:Sparse-Reconstruction}
M.~Rudelson and R.~Vershynin.
\newblock On sparse reconstruction from {F}ourier and {G}aussian measurements.
\newblock {\em Comm. Pure Appl. Math.}, 61(8):1025--1045, 2008.

\bibitem[San52]{San52:Integral-Geometry}
L.~A. Santal{\'o}.
\newblock Integral geometry in spaces of constant curvature.
\newblock {\em Repub. Argentina. Publ. Comision Nac. Energia Atomica. Ser.
  Mat.}, 1(1):68, 1952.

\bibitem[San76]{San76:Integral-Geometry}
L.~A. Santal{\'o}.
\newblock {\em Integral geometry and geometric probability}.
\newblock Addison-Wesley Publishing Co., Reading, Mass.-London-Amsterdam, 1976.
\newblock With a foreword by Mark Kac, Encyclopedia of Mathematics and its
  Applications, Vol. 1.

\bibitem[Sar06]{Sar06:Improved-Approximation}
T.~Sarlos.
\newblock Improved approximation algorithms for large matrices via random
  projections.
\newblock In {\em Foundations of Computer Science, 2006. FOCS '06. 47th Annual
  IEEE Symposium on}, pages 143--152, Oct 2006.

\bibitem[Sch50]{Sch50:Gesammelte}
L.~Schl{\"a}fli.
\newblock {\em Gesammelte mathematische {A}bhandlungen. {B}and {I}}.
\newblock Verlag Birkh\"auser, Basel, 1950.

\bibitem[Sio58]{Sio58:General-Minimax}
M.~Sion.
\newblock On general minimax theorems.
\newblock {\em Pacific J. Math.}, 8:171--176, 1958.

\bibitem[Sto09]{Sto09:Various-Thresholds}
M.~Stojnic.
\newblock Various thresholds for $\ell_1$-optimization in compressed sensing.
\newblock Available at \url{http://arXiv.org/abs/0907.3666}, July 2009.

\bibitem[Sto13]{Sto13:Regularly-Random}
M.~Stojnic.
\newblock Regularly random duality.
\newblock Available at \url{http://arXiv.org/abs/1303.7295}, Mar. 2013.

\bibitem[SW08]{SW08:Stochastic-Integral}
R.~Schneider and W.~Weil.
\newblock {\em Stochastic and integral geometry}.
\newblock Probability and its Applications (New York). Springer-Verlag, Berlin,
  2008.

\bibitem[TAH15]{TAH15:High-Dimensional-Error}
C.~Thrampoulidis, E.~Abbasi, and B.~Hassibi.
\newblock High-dimensional error anlaysis of regularized $m$-estimators.
\newblock Forthcoming, Nov. 2015.

\bibitem[Tao12]{Tao12:Topics-Random}
T.~Tao.
\newblock {\em Topics in random matrix theory}, volume 132 of {\em Graduate
  Studies in Mathematics}.
\newblock American Mathematical Society, Providence, RI, 2012.

\bibitem[TH15]{TH15:Isotropically-Random}
C.~Thrampoulidis and B.~Hassibi.
\newblock Isotropically random orthogonal matrices: Performance of lasso and
  minimum conic singular values.
\newblock In {\em Information Theory (ISIT), 2015 IEEE International Symposium
  on}, pages 556--560, June 2015.
\newblock Available at \url{http://arXiv.org/abs/1503.07236}.

\bibitem[Tib96]{Tib96:Regression-Shrinkage}
R.~Tibshirani.
\newblock Regression shrinkage and selection via the lasso.
\newblock {\em J. Roy. Statist. Soc. Ser. B}, 58(1):267--288, 1996.

\bibitem[Tik15]{Tik15:Limit-Smallest}
K.~Tikhomirov.
\newblock The limit of the smallest singular value of random matrices with
  i.i.d. entries.
\newblock {\em Adv. Math.}, 284:1--20, 2015.

\bibitem[TOH15]{TOH15:Gaussian-Min-Max}
C.~Thrampoulidis, S.~Oymak, and B.~Hassibi.
\newblock The {G}aussian min--max theorem in the presence of convexity.
\newblock In {\em Proceedings of The 28th Conference on Learning Theory
  (COLT)}, Jul. 2015.
\newblock Available at \url{http://arXiv.org/abs/1408.4837}.

\bibitem[Tro59]{Tro59:Elementary-Proof}
H.~F. Trotter.
\newblock An elementary proof of the central limit theorem.
\newblock {\em Arch. Math.}, 10:226--234, 1959.

\bibitem[Tro06]{Tro06:Just-Relax}
J.~A. Tropp.
\newblock Just relax: convex programming methods for identifying sparse signals
  in noise.
\newblock {\em IEEE Trans. Inform. Theory}, 52(3):1030--1051, 2006.

\bibitem[Tro15a]{Tro15:Universality-Code}
J.~A. Tropp.
\newblock Code for reproducing figures from {O}ymak \& {T}ropp,
  \textsl{Universality Laws for Randomized Dimension Reduction, with
  Applications}, 2015.ropp, \textsl{Universality Laws for Randomized Dimension
  Reduction, with Applications}, 2015.
\newblock Available at \url{http://users.cms.caltech.edu/~jtropp}, Nov. 2015.

\bibitem[Tro15b]{Tro15:Convex-Recovery}
J.~A. Tropp.
\newblock Convex recovery of a structured signal from independent random
  measurements.
\newblock In G.~Pfander, editor, {\em Sampling Theory: A Rennaissance}.
  Birkh{\"a}user Verlag, 2015.
\newblock Available at \url{http://arXiv.org/abs/1405.1102}.

\bibitem[Tro15c]{Tro15:Expected-Norm}
J.~A. Tropp.
\newblock The expected norm of a sum of independent random matrices: An
  elementary approach.
\newblock In {\em High-Dimensional Probability VII}, Carg{\`e}se, June 2015.
\newblock To appear. Available at \url{http://arXiv.org/abs/1506.04711}.

\bibitem[TV15]{TV15:Random-Matrices}
T.~Tao and V.~Vu.
\newblock Random matrices: universality of local spectral statistics of
  non-{H}ermitian matrices.
\newblock {\em Ann. Probab.}, 43(2):782--874, 2015.

\bibitem[Ver12]{Ver12:Introduction-Nonasymptotic}
R.~Vershynin.
\newblock Introduction to the non-asymptotic analysis of random matrices.
\newblock In {\em Compressed sensing}, pages 210--268. Cambridge Univ. Press,
  Cambridge, 2012.

\bibitem[Ver15]{Ver15:Estimation-High}
R.~Vershynin.
\newblock Estimation in high dimensions: A geometric perspective.
\newblock In G.~Pfander, editor, {\em Sampling Theory: A Rennaissance}.
  Birkh\"auser Verlag, 2015.

\bibitem[Wen62]{Wen62:Problem-Geometric}
J.~G. Wendel.
\newblock A problem in geometric probability.
\newblock {\em Math. Scand.}, 11:109--111, 1962.

\bibitem[Woo14]{Woo14:Sketching-Tool}
D.~P. Woodruff.
\newblock Sketching as a tool for numerical linear algebra.
\newblock {\em Foundations and Trends{\textregistered} in Theoretical Computer
  Science}, 10(1--2):1--157, 2014.

\bibitem[Yin86]{Yin86:Limiting-Spectral}
Y.~Q. Yin.
\newblock Limiting spectral distribution for a class of random matrices.
\newblock {\em J. Multivariate Anal.}, 20(1):50--68, 1986.

\end{thebibliography}
\newcommand{\etalchar}[1]{$^{#1}$}

\end{document}